\documentclass[a4paper,reqno,11pt]{amsart}

\usepackage[english]{babel}
\usepackage{amsmath}
\usepackage{amssymb}
\usepackage{enumerate}
\usepackage{ifthen}
\usepackage{bbm}
\usepackage{color}
\usepackage{latexsym}
\provideboolean{shownotes} 
\setboolean{shownotes}{true}
\usepackage{hyperref}
\usepackage{graphicx}
\usepackage{geometry}

\newcommand{\margnote}[1]{
\ifthenelse{\boolean{shownotes}}%
{\marginpar{\raggedright\tiny\texttt{#1}}}%
{}%
}

\newcommand{\hole}[1]{
\ifthenelse{\boolean{shownotes}}%
{\begin{center} \fbox{ \rule {.25cm}{0cm}
\rule[-.1cm]{0cm}{.4cm} \parbox{.85\textwidth}{\begin{center}
\texttt{#1}\end{center}} \rule {.25cm}{0cm}}\end{center}}
{}
}

\newtheorem{theorem}{Theorem}[section]
\newtheorem{proposition}[theorem]{Proposition}
\newtheorem{lemma}[theorem]{Lemma}

\newtheorem{definition}[theorem]{Definition}

\theoremstyle{remark}
\newtheorem{remark}[theorem]{Remark}

\newcommand{\R}{\mathbb{R}}

\newcommand{\dive}{\mathop{\mathrm {div}}}

\newcommand{\trace}{\mathrm{Tr}}
\newcommand{\We}{We}
\newcommand{\re}{Re}
\newcommand{\cof}{\mathrm{Cof}}
\numberwithin{equation}{section}

\begin{document}
\title[]{Splash singularities for a general Oldroyd model with finite Weissenberg number}

\author{Elena Di Iorio}
\address{GSSI - Gran Sasso Science Institute\\
67100 L'Aquila - 
Italy}

\email{elena.diiorio@gssi.it}

\author{Pierangelo Marcati}
\address{GSSI - Gran Sasso Science Institute\\
67100 L'Aquila -
Italy}
\email{pierangelo.marcati@gssi.it}

\author{Stefano Spirito}
\address{Department of Information Engineering, Computer Science and Mathematics\\ University of L'Aquila\\
67100 L'Aquila - Italy}
\email{stefano.spirito@univaq.it}

\begin{abstract}
In this paper we study a 2D free-boundary Oldroyd-B model which describes the evolution of a viscoelastic fluid. We prove the existence of splash singularities, namely points where the free-boundary remains smooth but self-intersects. This paper extends the previous result obtained for infinite Weissenberg number by the authors in \cite{DMS1}, \cite{DMS2} to the more realistic physical case of any finite Weissenberg number. The main difficulty faced in this paper is due to the non linear balance law of the elastic tensor, which cannot be reduced, as in the case of infinite Weissenberg, to a transport equation for the deformation gradient. Overcoming this difficulty requires a very accurate local existence theorem in terms of dependence on the  Weissenberg number. The method in this case is based on the combined use of conformal transformations and lagrangian coordinates, whose formulation must however take into account the general balance law of the elastic tensor and its dependence on the Weissenberg number.  The existence of the splash singularities is therefore guaranteed by an adequate choice of initial data, depending also on the elastic tensor, combined with stability estimates.
\end{abstract}

\maketitle 

\section{Introduction}

The aim of this paper is to study the existence of splash singularities,  namely a point where the interface self-intersects, for the free-boundary value problem of a general Oldroyd-B model. This model describes the behaviour of viscoelastic fluids, which unlike the classical Newtonian fluids, exhibit both viscous and elastic behaviours.\\
This paper extends in a nontrivial way our previous papers  \cite{DMS1}, \cite{DMS2} and provides an estimate of the 
blow-up time in terms of the Weissenberg number (namely a dimensionless number given by the ratio between the elastic forces and the viscous forces). As we stated in the abstract the main difficulty in this general case is due to the non linear balance law of the elastic tensor, which cannot be reduced, as in the case of infinite Weissenberg, to a transport equation for the deformation gradient. Our model behaves like a viscous fluid \cite{CCFGG2} for a very small Weissenberg number and like the models analyzed  in \cite{DMS1}, \cite{DMS2} for very large  Weissenberg numbers.

We focus our analysis on the case of incompressible viscoelastic fluids and we choose a constant density normalized to $\bar{\rho}\equiv 1$. Then the conservation laws of mass and momentum for a general material are given by
\begin{equation}\label{conservation-eq1}
\left\{\begin{array}{lll}
\dive u=0\\[2mm]
\partial_t u+u\cdot\nabla u=\dive\sigma.
\end{array}\right.
\end{equation}
\medskip

Viscoelastic fluids are known to be characterized by the following constitutive equation for the stress tensor $\sigma$,

\begin{equation}\label{extra-str}
\sigma=-p\mathcal{I}+\tau \hspace{0.3cm}\textrm{and}\hspace{0.3cm}\tau=\mu_s(\nabla u+\nabla u^T)+\tau_p,
\end{equation}

\noindent where $\tau$ is called extra-stress and $\tau_p$ is the polymeric stress tensor related to the elastic behaviour. 
In order to have a closed system we need to formulate an equation for $\tau$.\par

 In literature there are several formulations for the tensor $\tau$, both in differential and integral forms and they describe viscoelastic models with various distinct properties, see \cite{FGO} and \cite{R}. In particular, since viscoelastic materials have both elastic and viscous properties, a natural approach to understand their behaviours is through an ideal model made of springs, for the elastic part and dashpots, for the viscous one. We can combine them in a number of different ways to simulate various possible material responses. The simplest models in viscoelasticity are given by a serial or a parallel connection of springs and dashpots, the former representing Maxwell fluids, while the latter represents Kelvin-Voigt solids. The simplest constitutive relationship for Maxwell fluids is obtained as a linear combination of the constitutive relations for the elasticity and the viscosity terms. Precisely, 

\begin{equation}\label{Maxwell}
\tau+\lambda\partial_t\tau=\mu_0(\nabla u+\nabla u^T),
\end{equation}
where $\mu_0=\mu_s+\mu_p$ denotes the material viscosity, $\mu_s$ the solvent viscosity and $\mu_p$ the polymeric viscosity respectively. The quantity $\lambda$ has the dimension of time and it represents the \textit{relaxation time}, which is the typical time for the system to return to the equilibrium after a typical deformation, \cite{OP} and \cite{R}. This parameter is proportional to the Weissenberg number $\We$, which measures the ratio of the elastic and the viscous forces and it appears instead of $\lambda$ when we use dimensionless variables, see \eqref{undimsys}.\par

We remark that the Maxwell viscoelastic equations \eqref{Maxwell} are not frame-indifferent, which means they are not invariant under proper time dependent rotations. To recover this property, see \cite{FGO}, the following upper convected derivative, 
$$\partial^{uc}_t \tau=\partial_t\tau+u\cdot\nabla\tau-\nabla u\tau-\tau\nabla u^T,$$
was introduced in his 1950 celebrated paper by Oldoryd \cite {Old}.  Then the equation for $\tau$ provided by the general Oldoryd model is given by the following

\begin{equation}\label{UCO}
\tau+\lambda\partial^{uc}_t\tau=\mu_0(\nabla u+\nabla u^T +\lambda_s\partial^{uc}_t(\nabla u+\nabla u^T)),
\end{equation}
where $\lambda_s=\frac{\mu_s}{\mu_0}\lambda$ is the characteristic retardation time for the fluid. If we plug in \eqref{UCO} the extra-stress tensor given in \eqref{extra-str}, we get that $\tau_p$, the stress associated to the elastic part, satisfies the upper convected Maxwell equation
\begin{equation}\label{UCMM}
\lambda\partial_t^{uc}\tau_p+\tau_p=\mu_p(\nabla u+(\nabla u)^T).
\end{equation}
In order to have a closed system in $u$ and $\tau_p$, we combine  \eqref {UCMM} with \eqref{conservation-eq1} and \eqref{extra-str}, namely we put together the conservation of mass and momentum and the constitutive law, expressed in terms of $\tau_p$ to obtain
\begin{equation}\label{Nav}
\partial_t u+u\cdot\nabla u+\nabla p-\mu_s\Delta u=\dive{\tau_p}
\end{equation}
Therefore a closed system is given by coupling  \eqref{UCMM} and \eqref{Nav}.
 \medskip

To describe the moving boundary we use the particle-trajectory mapping, namely the flux $X(t,\alpha)$ provided by the following system of ODEs of the integral curves of the velocity vector field $u$. 

\begin{equation}\label{ODEflux}
\left\{\begin{array}{lll}
\displaystyle\frac{d}{dt}{X}(t,\alpha)=u(t,X(t,\alpha))\\ [3mm]
X(0,\alpha)=\alpha,
\end{array}\right.
\end{equation}

\noindent where $\alpha\in \R^2$ denotes the material point in the reference configuration, also called Lagrangian particle marker.
Given the initial domain  $\Omega_0 \subset \R^2$, the moving domain is then given by   $\Omega(t)=X(t,\Omega_0)$ and hence in the related free-boundary problem one of the unknown is given by $\partial \Omega(t)$ itself. The motion of the free-boundary  is determined through boundary conditions given by the  balance of the force fields at the interface.  
\\ \
\\
 In this paper we consider the free boundary initial value  problem the general Oldroyd-B system  endowed with boundary  conditions

\begin{equation}\label{lambda-Eulsys}
\left\{\begin{array}{lll}
\partial_t u+u\cdot\nabla u+\nabla p-\mu_s\Delta u=\dive{\tau_p}\hspace{2cm}\textrm{in}\hspace{0.1cm}\Omega(t)\\[3mm]
\lambda\partial_t^{uc}\tau_p+\tau_p=\mu_p(\nabla u+ \nabla u^T)\hspace{2.98cm}\textrm{in}\hspace{0.1cm}\Omega(t)\\[3mm]
\displaystyle\dive u=0 \hspace{6.35cm}\textrm{in}\hspace{0.1cm}\Omega(t)\\ [3mm]
\displaystyle (-p\mathcal{I}+\mu_s(\nabla u+\nabla u^{T})+\tau_p)n=0 \hspace{2.15cm}\textrm{on}\hspace{0.1cm}\partial\Omega(t)\\ [3mm]
\displaystyle u(0,\alpha)=u_0(\alpha),\hspace{0.1cm} \tau_p(0,\alpha)=\tau_0(\alpha)\hspace{2.4cm} \textrm{in}\hspace{0.3cm}\Omega_0,\\
\end{array}\right.
\end{equation}
\medskip

\noindent We use the notation $(\dive M)_j=\sum_{i}\partial_i M_{ij}$, for any matrix valued function $M$. Since our goal is to have an explicit dependence of all the estimates on the Weissenberg number,  we prefer to rewrite the previous system \eqref{lambda-Eulsys} in dimensionless variables. We set

\begin{align*}
&X=\frac{X^*}{L},\quad u=\frac{u^*}{U},\quad t=t^*\frac{U}{L},\quad \tau_p=\frac{\tau_p^*L}{\mu_0 U}, \quad p=p^*\frac{L}{\mu_0 U},
\end{align*}
where $U, L$ represent the typical velocity and length of the flow. \\ 
Moreover we define
\begin{itemize}
\item the Weissenberg number $\We=\frac{U\lambda}{L}$, 
\item the Reynolds number $\re=\bar{\rho}\frac{UL}{\mu_0}$, 
\end{itemize}
where  $\bar{\rho}=1$, $\frac{\mu_p}{\mu_0}=\kappa$. The stars are attached to denote the dimensional variables. \\

\noindent Hence the system \eqref{lambda-Eulsys} becomes 

\begin{equation}\label{undimsys}
\left\{\begin{array}{lll}
\re(\partial_t u+u\cdot\nabla u)+\nabla p=(1-\kappa)\Delta u+\dive\tau_p\hspace{2cm}\textrm{in}\hspace{0.1cm}\Omega(t)\\[3mm]
\We\hspace{0.1cm} \partial_t^{uc}\tau_p+\tau_p=\kappa(\nabla u+\nabla u^T)\hspace{4.25cm}\textrm{in}\hspace{0.1cm}\Omega(t)\\[3mm]
\dive u=0\hspace{7.95cm}\textrm{in}\hspace{0.1cm}\Omega(t)\\[3mm]
(-p\mathcal{I}+(1-\kappa)(\nabla u+\nabla u^T)+\tau_p)n=0\hspace{2.95cm}\textrm{on}\hspace{0.1cm}\partial\Omega(t)\\[3mm]
u(0,\alpha)=u_0(\alpha),\quad \tau_p(0,\alpha)=\tau_0(\alpha)\hspace{3.7cm}\textrm{in}\hspace{0.1cm}\Omega_0.
\end{array}\right.
\end{equation}

\noindent This system is supplemented with the following compatibility conditions

\begin{equation}\label{undim-compatibility}
\left\{\begin{array}{lll}
\displaystyle \dive{u_0}=0\hspace{6cm}\textrm{ in } \Omega_0 \\ [1mm]
\displaystyle n_0^{\perp}((1-\kappa)(\nabla u_0 +\nabla  u_0^T) + \tau_0)n_0=0 \hspace{1.5cm} \textrm{ on } \partial\Omega_0.\\
\end{array}\right.
\end{equation}
\medskip

\noindent The main goal is to prove existence of splash singularity for the above system. The definition of splash sigularity involves the interface $\partial\Omega(t)$, which is a smooth, simple and closed curve and it is parametrized by $z:\mathbb{R}/\mathbb{Z}\rightarrow\mathbb{R}^2$, that  satisfies the arc-chord condition 
$|z(\alpha)-z(\alpha')|\geq K\|\alpha-\alpha'\|,$ for $\alpha, \alpha'\in\mathbb{R}/\mathbb{Z}$ and $K$ is the chord-arc constant. A rigorous definition is the following, see \cite{CCFGG2}.

\begin{definition}
 $z(\alpha)=(z_1(\alpha),z_2(\alpha))$ is a splash curve if
\begin{enumerate}
\item $z_1(\alpha), z_2(\alpha)$ are smooth functions and $2\pi$-periodic.
\item $z(\alpha)$ satisfies the arc-chord condition at every point except at $\alpha_1$ and $\alpha_2$, with $\alpha_1< \alpha_2$ where $z(\alpha_1)=z(\alpha_2)$ and $|\partial_{\alpha} z(\alpha_1)|,|\partial_{\alpha} z(\alpha_2)|>0$. This means $z(\alpha_1)=z(\alpha_2)$, but if we remove either a neighborhood of $\alpha_1$ or a neighborhood of $\alpha_2$ in parameter space, then the arc-chord condition holds.
\item The curve $z(\alpha)$ separates the domain into two regions; a connected region with the fluid and a vacuum region (not necessarily connected). We choose the parametrization in such a way that the normal vector $n=\frac{(-\partial_{\alpha}z_2(\alpha),\partial_{\alpha}z_1(\alpha))}{|\partial_{\alpha}z(\alpha)|}$ points in the vacuum region. The interface to be a part of the region with the fluid.
\item We can choose a branch of the function $P$, a conformal map that we will define later, on the fluid region such that the curve $\tilde{z}(\alpha)=(\tilde{z_1}(\alpha), \tilde{z_2}(\alpha))=P(z(\alpha))$ satisfies
\begin{enumerate}
\item [a.] $\tilde{z_1}(\alpha)$ and $\tilde{z_2}(\alpha)$ are smooth and $2\pi$-periodic.
\item [b.] $\tilde{z}$ is a closed contour.
\item [c.]$\tilde{z}$ satifies the arc-chord condition.
\end{enumerate}
\end{enumerate} 
\end{definition}

\noindent   Our approach to the analysis of the free boundary problem for the \eqref{undimsys}  takes inspiration from \cite{B} and \cite{CCFGG2} for the fluid part  and from our analysis of the Oldroyd-B model at infinite Weissenberg number in \cite{DMS1} and \cite{DMS2} for the elastic part. The idea for proving the existence of splash singularities in the general Oldroyd-B model \eqref{undimsys}  is inspired by the geometric construction  done in \cite{CCFGG2} for the Navier-Stokes equation and by our results in  \cite{DMS1} and \cite{DMS2} which take into account the balance laws of the elastic tensor.\par

Let us now proceed to summarize the key points of our geometric approach.
\begin{itemize}
\item  Let the initial domain  $\Omega_0$ be a domain of splash type as in fig. \ref{fig:1} (b), then let us introduce a conformal transformation $P$ in the complex plane, with the property to transform a non-splash  type domain $\tilde \Omega_0 $, into a domain of splash type  
$P^{-1}(\tilde{\Omega}_0)=\Omega_0$,  
as shown in fig. \ref{fig:1}(b).

\item If $\{\tilde{\Omega}_0, \tilde{u}_0, \tilde{p}_0, \tilde{\tau}_0\}$ are sufficiently smooth  we can apply the conformal transformation to \eqref{undimsys} and we can prove in  tilde plane the local existence of a smooth solution  $$\{\tilde{\Omega}(t), \tilde{u}(t,\cdot), \tilde{p}(t,\cdot), \tilde{\tau}_p(t,\cdot)\}$$ for $t\in[0,T],$ with a sufficiently small $T$.

\item  We make a suitable choice of the initial velocity, in particular $\tilde{u}(0,\tilde{z}_1)\cdot n>0,$ $\tilde{u}(0,\tilde{z}_2)\cdot n>0$ in order to guarantee that the interface moves from a configuration as in \ref{fig:1}(b) towards a configuration as in \ref{fig:1} (c), namely 
  there exists $\bar{t}\in(0,T]$ such that $P^{-1}(\partial\tilde{\Omega}(\bar{t}))$ is \textit{self-intersecting}, as in the case fig. \ref{fig:1} (c). 
  This solution using  \ref{fig:1} (b) as initial datum,  is well defined only in the tilde complex plane and cannot be mapped back by $P^{-1}$ into a solution in the non-tilde complex plane, hence  it cannot be used  to  show the existence of a splash singularity.
  
\item To solve the problem in the non-tilde domain, we consider a one-parameter family of initial data $\{\tilde{\Omega}_{\varepsilon}(0),\tilde{u}_{\varepsilon}(0),\tilde{p}_{\varepsilon}, \tilde{\tau}_{p,\varepsilon}(0)\}$ in the configuration \ref{fig:1} (a),   with $\tilde{\Omega}_{\varepsilon}(0)=\tilde{\Omega}_0+\varepsilon b$ and  $|b|=1$, such that  $P^{-1}(\partial\tilde{\Omega}_{\varepsilon}(0))$ is regular, then there exists a local in time smooth solution $\{{\Omega}_{\varepsilon}(t), {u}_{\varepsilon}(t,\cdot), {p}_{\varepsilon}(t,\cdot), {\tau}_{p,\varepsilon}(t,\cdot)\}$, in the non tilde complex plane.

\item We then prove a stability result  which allows us to say
\begin{center}
$\textrm{dist}(\partial\tilde{\Omega}_{\varepsilon}(\bar{t}),\partial\tilde{\Omega}(\bar{t}))\leq \varepsilon\quad\textrm{ hence }\quad
P^{-1}(\partial\tilde{\Omega}_{\varepsilon}(\bar{t}))\sim P^{-1}(\partial\tilde{\Omega}(\bar{t}))$
\end{center}
\medskip

\noindent and so $P^{-1}(\partial\tilde{\Omega}_{\varepsilon}(\bar{t}))$ self-intersects.
\item Since $P^{-1}(\tilde{\Omega}_{\varepsilon}(0))$ is regular of type fig. \ref{fig:1}(a) and $P^{-1}(\tilde{\Omega}_{\varepsilon}(\bar{t}))$ is self-intersecting domain of type fig. \ref{fig:1}(c), then there exists a time $t^*$ such that $P^{-1}(\tilde{\Omega}_{\varepsilon}(t^*))$ has a splash singularity.
\end{itemize}

\begin{figure}
\centering
\includegraphics[width=0.9\textwidth] {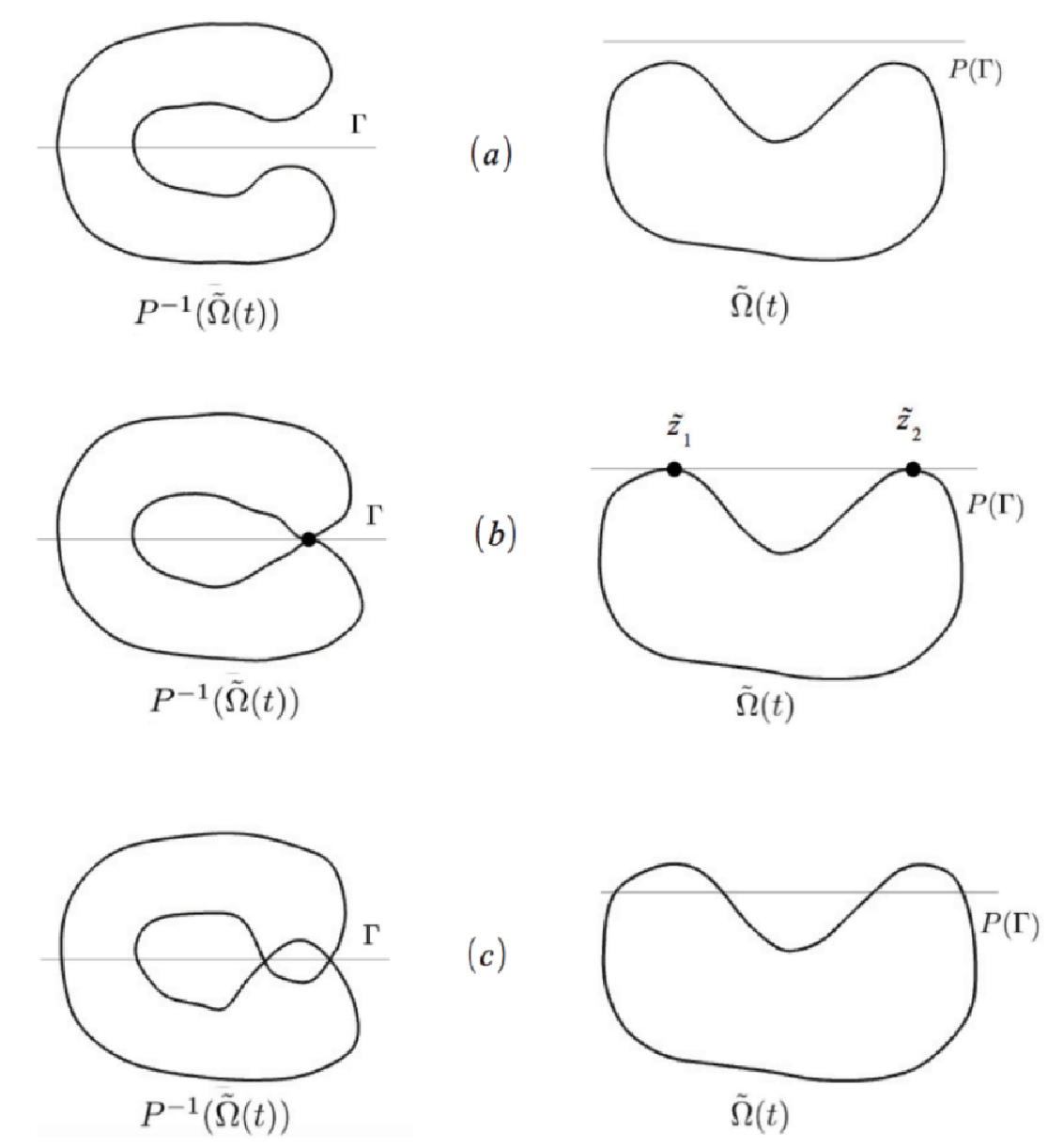}
\caption{Possibilities for $P^{-1}(\tilde{\Omega}(t))$}
\label{fig:1}     
\end{figure}

\subsection{Outline of the paper} 
In the Section \ref{sec:2}  we prove, through a fixed point argument, a local existence result for the conformal-lagragian system associated to \eqref{undimsys} and as byproduct of our estimates, we get that the local existence time $T$ can be estimated from above by $\frac{\We}{1+\We}$. \\
In the Section \ref{sec:3} we prove the stability estimates, with respect to a suitable one parameter family of perturbations of the initial splash domain. \\
Finally, in the Section \ref{sec:4} we show the existence of a splash singularity, by means of the previously mentioned geometric argument.

\section{Local existence for the system in conformal-lagrangian coordinates}\label{sec:2}
We focus on the analysis of the system \eqref{undimsys} but we want to rewrite it in conformal-lagrangian coordinates in order to use a fixed point argument for proving the local existence. The first step is to pass in conformal coordinates. Let us apply the conformal map $P(z)=\tilde{z}$, for $z\in\mathbb{C} \setminus \Gamma$, defined as a branch of $\sqrt{z}$, where $\Gamma$ is a line, passed through the splash point, for details see \cite{CCFGG2} and \cite{DMS1}; and the change of coordinates from $\Omega\rightarrow\tilde{\Omega}=P(\Omega)$. The conformal velocity field and elastic stress tensor are defined as follows

\begin{align*}
&\tilde{u}(t,\tilde{X})=u(t,P^{-1}(\tilde{X})), \hspace{0.5cm}\textrm{hence}\hspace{0.3cm}  u(t,X)=\tilde{u}(t,P(X)),\\[1mm]
&\tilde{\tau}_p(t,\tilde{X})=\tau_p(t,P^{-1}(\tilde{X})),\hspace{0.35cm}\textrm{hence}\hspace{0.2cm} \tau_p(t,X)=\tilde{\tau}_p(t,P(X)).
\end{align*}

\noindent Defining $J^P_{kj}=\partial_{X_j}P_k(P^{-1}(\tilde{X}))$ and $\displaystyle Q^2=\left|\frac{dP}{dz}(P^{-1}(\tilde{X}))\right|^2$, then the new system in conformal coordinates is the following in $[0,T]\times \tilde{\Omega}(t)$.

\begin{equation}
\left\{\begin{array}{lll}
\re\hspace{0.1cm}(\partial_t\tilde{u}+(J^P\tilde{u}\cdot \nabla)\tilde{u})-(1-\kappa)Q^2\Delta\tilde{u}+(J^P)^T\nabla\tilde{p}=\trace(\nabla\tilde{\tau}_pJ^P)\\[2mm]
\partial_t\tilde{\tau}_p+(J^P\tilde{u}\cdot\nabla)\tilde{\tau}_p-\nabla\tilde{u}J^P\tilde{\tau}_p-\tilde{\tau}_p(\nabla\tilde{u}J^P)^T=\\[2mm]
\displaystyle\hspace{5cm}=\frac{\kappa(\nabla\tilde{u}J^P+(\nabla\tilde{u}J^P)^T)-\tilde{\tau}_p}{\We}\\[2mm]
\trace(\nabla\tilde{u}J^P)=0\\[2mm]
(-\tilde{p}\mathcal{I}+(1-\kappa)(\nabla\tilde{u}J^P+(\nabla\tilde{u}J^P)^T)+\tilde{\tau}_p)\tilde{n}=0\\[2mm]
\tilde{u}(0,\tilde{\alpha})=\tilde{u}_0,\quad \tilde{\tau}_p(0,\tilde{\alpha})=\tilde{\tau}_0.
\end{array}\right.
\end{equation}
\medskip

\noindent The second step is to pass in lagrangian coordinates in order to fix the domain. We look at the flux equation

\begin{equation}\label{Conf-flux}
\left\{\begin{array}{lll}
\displaystyle \frac{d}{dt}\tilde{X}(t,\tilde{\alpha})=J^P(\tilde{X}(t,\tilde{\alpha}))\tilde{u}(t,\tilde{X}(t,\tilde{\alpha}))\hspace{1cm}\textrm{in}\hspace{0.3cm}\tilde{\Omega}(t)\\[3mm]
\displaystyle\tilde{X}(0,\tilde{\alpha})=\tilde{\alpha}\hspace{5.18cm} \textrm{in}\hspace{0.3cm}\tilde{\Omega}(0),
\end{array}\right.
\end{equation}
\medskip

\noindent by using the new lagrangian velocity, pressure and elastic stress

\begin{equation}\label{conf-lag-var}
\left\{\begin{array}{lll}
\tilde{v}(t,\tilde{\alpha})=\tilde{u}(t, \tilde{X}(t,\tilde{\alpha}))\\[2mm]
\tilde{q}(t,\tilde{\alpha})=\tilde{p}(t, \tilde{X}(t,\tilde{\alpha}))\\[2mm]
\tilde{\mathbf{T}}_p(t, \tilde{\alpha})=\tilde{\tau}_p(t,\tilde{X}(t,\tilde{\alpha})),
\end{array}\right.
\end{equation}
\medskip

we get the new conformal-lagrangian system in $[0,T]\times\tilde{\Omega}_0$, written as follows.

\begin{equation}\label{undim-Conf-Lag-Sys}
\left\{\begin{array}{lll}
\re\hspace{0.1cm} \partial_t \tilde{v}-(1-\kappa)Q^2(\tilde{X})\tilde{\zeta}\nabla(\tilde{\zeta}\nabla\tilde{v})+(J^P(\tilde{X}))^T\tilde{\zeta}\nabla \tilde{q}=\trace(\tilde{\zeta}\nabla \tilde{\mathbf{T}}_pJ^P(\tilde{X}))\\[2mm]
\partial_t\tilde{\mathbf{T}}_p-J^P(\tilde{X})\tilde{\zeta}\nabla\tilde{v}\tilde{\mathbf{T}}_p-\tilde{\mathbf{T}}_p(J^P(\tilde{X})\tilde{\zeta}\nabla\tilde{v})^T=\\[2mm]
\displaystyle\hspace{5cm}=\frac{\kappa\left(J^P(\tilde{X})\tilde{\zeta}\nabla\tilde{v}+(J^P(\tilde{X})\tilde{\zeta}\nabla \tilde{v})^T\right)-\tilde{\mathbf{T}}_p}{\We}\\[2mm]
\trace(\nabla \tilde{v}\tilde{\zeta}J^P(\tilde{X}))=0\\[2mm]
(-\tilde{q}\mathcal{I}+(1-\kappa)(\nabla \tilde{v}\tilde{\zeta}J^P(\tilde{X}) +(\nabla\tilde{v}\tilde{\zeta}J^P(\tilde{X}))^{T})+ \tilde{\mathbf{T}}_p)(J^P(\tilde{X}))^{-1}\nabla_{\Lambda} \tilde{X} \tilde{n}_0=0 \\[3mm]
\tilde{v}(0,\tilde{\alpha})=\tilde{v}_0(\tilde{\alpha}),\hspace{0.3cm}\tilde{\mathbf{T}}_p(0,\tilde{\alpha})=\tilde{\mathbf{T}}_0(\tilde{\alpha}),
\end{array}\right.
\end{equation}
\medskip

where $\tilde{\zeta}(t,\tilde{\alpha})=(\nabla \tilde{X})^{-1}(t,\tilde{\alpha})$ and $\nabla_{\Lambda} \tilde{X}=-\Lambda\nabla\tilde{ X}\Lambda$, with $\Lambda=\left(\begin{matrix} 0 & -1 \\ 1& 0\end{matrix}\right)$.

\subsection{Iterative scheme for \eqref{undim-Conf-Lag-Sys}}
\noindent We prove a local existence result throughout a fixed point argument. Thus the idea is to separate the equations for the elastic stress tensor $\tilde{\mathbf{T}}_p$ from the equations for the velocity and the pressure. By using Picard iterations, we get the following systems

\begin{equation}\label{Re-v-conf-iterative}
\left\{\begin{array}{lll}
\displaystyle\re\hspace{0.1cm} \partial_t \tilde{v}^{(n+1)}-(1-\kappa)Q^2\Delta \tilde{v}^{(n+1)}+(J^P)^T\nabla \tilde{q}^{(n+1)}={\tilde{f}}^{(n)}\\[3mm]

\displaystyle\trace(\nabla\tilde{v}^{(n+1)}J^P)={\tilde{g}}^{(n)}\\[3mm]

\displaystyle [-\tilde{q}^{(n+1)}\mathcal{I}+(1-\kappa) (\nabla\tilde{ v}^{(n+1)}J^P+(\nabla\tilde{v}^{(n+1)}J^P)^T)](J^P)^{-1}\tilde{n}_0={\tilde{h}}^{(n)}\\[3mm]
\displaystyle \tilde{v} (0,\tilde{\alpha})=\tilde{v}_0(\tilde{\alpha}),
\end{array}\right.
\end{equation}

\noindent where ${\tilde{f}}^{(n)}, {\tilde{g}}^{(n)}, {\tilde{h}}^{(n)}$ are defined as follows

\begin{align*} 
&{\tilde{f}}^{(n)}=-(1-\kappa)Q^2\Delta\tilde{v}^{(n)}+(J^P)^T\nabla\tilde{q}^{(n)}+(1-\kappa)Q^2(\tilde{X}^{(n)})\tilde{\zeta}^{(n)}\nabla(\tilde{\zeta}^{(n)}\nabla\tilde{v}^{(n)})\\
&\hspace{1cm}-J^P(\tilde{X}^{(n)})^T\tilde{\zeta}^{(n)}\nabla\tilde{q}^{(n)}+\trace(\tilde{\zeta}^{(n)}\nabla\tilde{\mathbf{T}}_p^{(n)}J^P(\tilde{X}^{(n)})),\\[3mm]
&{\tilde{g}}^{(n)}=\trace(\nabla\tilde{v}^{(n)}J^P)-\trace(\nabla\tilde{v}^{(n)}\tilde{\zeta}^{(n)}J^P(\tilde{X}^{(n)})),\\[3mm]
&{\tilde{h}}^{(n)}=-\tilde{q}^{(n)}(J^P)^{-1}\tilde{n}_0+\tilde{q}^{(n)}(J^P(\tilde{X}^{(n)}))^{-1}\nabla_{\Lambda}\tilde{X}^{(n)}\tilde{n}_0+
(1-\kappa)\nabla\tilde{v}^{(n)}\tilde{n}_0\\
&\hspace{1cm}+(1-\kappa)(\nabla\tilde{v}^{(n)}J^P)^T(J^P)^{-1}\tilde{n_0}-(1-\kappa)\nabla\tilde{v}^{(n)}\tilde{\zeta}^{(n)}\nabla_{\Lambda}\tilde{X}^{(n)}\tilde{n}_0\\
&\hspace{1cm}-(1-\kappa)(\nabla\tilde{v}^{(n)}\tilde{\zeta}^{(n)}J^P(\tilde{X}^{(n)}))^T(J^P(\tilde{X}^{(n)}))^{-1}\nabla_{\Lambda}\tilde{X}^{(n)}\tilde{n}_0\\
&\hspace{1cm}-\tilde{\mathbf{T}}_p^{(n)}(J^P(\tilde{X}^{(n)}))^{-1}\nabla_{\Lambda}\tilde{X}^{(n)}\tilde{n}_0,
\end{align*}
\medskip

\noindent The elastic stress tensor satisfies the following ODE.

\begin{equation}\label{We-G-conf-iterative}
\left\{\begin{array}{lll}
\displaystyle\partial_t \tilde{\mathbf{T}}_p^{(n+1)}-J^P(\tilde{X}^{(n)}) \tilde{\zeta}^{(n)}\nabla\tilde{v}^{(n)}\tilde{\mathbf{T}}_p^{(n)}- \tilde{\mathbf{T}}_p^{(n)}\left(J^P(\tilde{X}^{(n)})\tilde{\zeta}^{(n)}\nabla\tilde{v}^{(n)}\right)^T=\\[2mm]
\displaystyle-\frac{\tilde{\mathbf{T}}_p^{(n)}}{\We}
+\frac{\kappa\left(J^P(\tilde{X}^{(n)})\tilde{\zeta}^{(n)}\nabla\tilde{v}^{(n)}J+(J^P(\tilde{X}^{(n)})\tilde{\zeta}^{(n)}\nabla \tilde{v}^{(n)})^T\right)}{\We}\\[7mm]
\tilde{\mathbf{T}}_p(0, \tilde{\alpha})= \tilde{T}_0.
\end{array}\right.
\end{equation}
\bigskip

\noindent Moreover the flux satisfies

\begin{equation}\label{Lag-X-conf-iterative}
\left\{\begin{array}{lll}
\displaystyle\frac{d}{dt}\tilde{X}^{(n+1)}(t,\tilde{\alpha})= J^P(\tilde{X}^{(n)}(t,\tilde{\alpha})) \tilde{v}^{(n)}(t,\tilde{\alpha})\\[5mm]
\displaystyle \tilde{X}(0, \tilde{\alpha})= \tilde{\alpha}.
\end{array}\right.
\end{equation}
\medskip

\noindent We study separately the three systems \eqref{Re-v-conf-iterative}, \eqref{We-G-conf-iterative} and \eqref{Lag-X-conf-iterative}. We solve system \eqref{Re-v-conf-iterative}, by using the idea and the results of \cite{B} and \cite{CCFGG2}. In particular the general linearized system we consider is the following

\begin{equation}\label{v-conf-lag-sys}
\left\{\begin{array}{lll}
\displaystyle\re\hspace{0.1cm} \partial_t \tilde{v}-(1-\kappa)Q^2\Delta \tilde{v}+(J^P)^T \nabla \tilde{q}=\tilde{f} \hspace{2.9cm}\textrm{in}\hspace{0.3cm}(0,T)\times\tilde{\Omega}_0\\[2mm]
\displaystyle \trace(\nabla \tilde{v} J^P)=\tilde{g}\hspace{6.95cm}\textrm{in}\hspace{0.3cm}(0,T)\times\tilde{\Omega}_0\\[2mm]
\displaystyle [-\tilde{q}\mathcal{I}+(1-\kappa)(\nabla \tilde{v} J^P+(\nabla \tilde{v} J^P)^T))]\frac{(J^P)^T}{Q^2} \tilde{n}=\tilde{h}   \hspace{1.2cm}\textrm{on}\hspace{0.3cm}(0,T)\times\partial\tilde{\Omega}_0\\[2mm]
\displaystyle \tilde{v}(0,\tilde{\alpha})=\tilde{v}_0(\tilde{\alpha})\hspace{6.85cm}\textrm{on}\hspace{0.3cm}\{t=0\}\times\tilde{\Omega}_0.
\end{array}\right.
\end{equation}
\medskip

\noindent This system is supplemented with compatibility conditions 

\begin{equation}\label{compcond}
\left\{\begin{array}{lll}
\displaystyle \trace(\nabla \tilde{v}_0 J^P)=\tilde{g}(0)\hspace{8.2cm}\textrm{in}\hspace{0.1cm}\tilde{\Omega}_0\\[3mm]
\displaystyle \{(1-\kappa)(\nabla \tilde{v}_0 J^P+(\nabla \tilde{v}_0 J^P)^T)(J^P)^{-1}\tilde{n}\}_{tang}=\{\tilde{h}(0)(J^P)^{-1}\tilde{n}\}_{tang}\hspace{0.3cm}\textrm{on}\hspace{0.1cm}\partial\tilde{\Omega}_0.
\end{array}\right.
\end{equation}

\noindent In order to analyze \eqref{v-conf-lag-sys}, we introduce the following function space of the solution $X_0$ and the function space of the data $Y_0$, for their precise definition see Appendix \ref{appendix}.

\begin{align*}
&X_0:=\left\lbrace (\tilde{v}, \tilde{q})\in\mathcal{K}^{s+1}_{(0)}\times \mathcal{K}^{s}_{pr(0)}\right\rbrace,\\[2mm]
& Y_0:=\{(\tilde{f},\tilde{g},\tilde{h})\in\mathcal{K}^{s-1}_{(0)}\times \mathcal{\bar{K}}^{s}_{(0)}\times \mathcal{K}^{s-\frac{1}{2}}_{(0)}([0,T];\partial\tilde{\Omega}): \hspace{0.1cm}(\ref{compcond})\hspace{0.1cm} \textrm{are satisfied}\},
\end{align*}
\medskip

\noindent and a linear operator  $L: X_0\rightarrow Y_0$, related to the system (\ref{v-conf-lag-sys}) by 
\begin{equation}\label{L}
L(\tilde{v},\tilde{q})=(\tilde{f},\tilde{g},\tilde{h}, \tilde{v}_0).
\end{equation}
\noindent The well-posedness of the system (\ref{v-conf-lag-sys}) is guaranteed by the invertibility of the operator $L$, proved in \cite[Theorem 4.3]{B}.  For our system \eqref{Re-v-conf-iterative} we want to use the invertibility of $L$ but we do not have zero initial velocity. For this reason we need to define a new velocity field $\tilde{w}=\tilde{v}-\tilde{\phi}$, where 

$$\phi=\tilde{v}_0+\frac{1}{\re}t\left((1-\kappa)Q^2\Delta \tilde{v}_0-(J^P)^T\nabla \tilde{q}_{\phi}+\trace(\nabla\tilde{\mathbf{T}}_0J^P)\right)=\tilde{v}_0+t\hat{\phi}.
$$
\medskip

\noindent We choose $\tilde{q}_{\phi}$ such that $\partial_t\tilde{v}^{(n+1)}(0,\tilde{\alpha})=\partial_t\phi^{(n+1)}(0,\tilde{\alpha})$. In this way $\partial_t\tilde{w}(0,\tilde{\alpha})=0$ and $\tilde{w}(0,\tilde{\alpha})=0$. Indeed $\tilde{q}_{\phi}$ satisfies 

\begin{equation}\label{q_phi}
\left\{\begin{array}{lll}
-Q^2\Delta\tilde{q}_{\phi}=\re\left(\trace(\nabla\tilde{v}_0J^P\nabla\tilde{v}_0J^P)-\trace(\nabla(\trace(\nabla\mathbf{\tilde{T}}_0J^P))J^P)\right)\hspace{0.6cm}\textrm{in}\hspace{0.3cm}\tilde{\Omega}_0\\[3mm]
\tilde{q}_{\phi}\tilde{n}_0=\left((1-\kappa)(J^P\nabla\tilde{v}_0+(J^P\nabla\tilde{v}_0)^T)+\mathbf{\tilde{T}}_0\right)(J^P)^{-1}\tilde{n}_0\hspace{1.8cm}\textrm{on}\hspace{0.3cm}\partial\tilde{\Omega}_0
\end{array}\right.
\end{equation}
\medskip

\noindent The new system for $\tilde{w}$ is the following

\begin{equation}\label{systemW}
\left\{\begin{array}{lll}
\displaystyle\re \partial_t\tilde{w}^{(n+1)}-(1-\kappa)Q^2\Delta\tilde{w}^{(n+1)}+(J^P)^T\nabla\tilde{q}_w^{(n+1)}=\tilde{f}^{(n)}+\tilde{f}^L_{\phi}\\ [3mm]
\displaystyle\trace(\nabla\tilde{w}^{(n+1)}J^P)=\tilde{g}^{(n)}+\tilde{g}^L_{\phi}\\ [3mm]
\displaystyle [-\tilde{q}_w^{(n+1)}\mathcal{I}+(1-\kappa)((\nabla\tilde{w}^{(n+1)}J^P)+(\nabla\tilde{w}^{(n+1)}J^P)^T)](J^P)^{-1}\tilde{n_0}=\\[1mm]
\hspace{1cm}=\tilde{h}^{(n)}+\tilde{h}^L_{\phi}\\ [3mm]
\displaystyle \tilde{w}^{(n+1)}(0,\tilde{\alpha})=0,
\end{array}\right.
\end{equation}

\noindent where we split $\tilde{f}^{(n)}, \tilde{g}^{(n)}$ and $\tilde{h}^{(n)}$  as follows

\begin{equation}\label{Re-split}
\begin{split}
&\tilde{f}^{(n)}_w=(1-\kappa)Q^2(\tilde{X}^{(n)})\tilde{\zeta}^{(n)}\partial(\tilde{\zeta}^{(n)}\partial\tilde{w}^{(n)})- (1-\kappa)Q^2\Delta\tilde{w}^{(n)},\\[2mm]
&\tilde{f}_{\phi}=(1-\kappa)Q^2(\tilde{X}^{(n)})\tilde{\zeta}^{(n)}\partial(\tilde{\zeta}^{(n)}\partial\phi)- (1-\kappa)Q^2\Delta\phi,\\[2mm]
&\tilde{f}^{(n)}_q=(J^P)^T\nabla\tilde{q}^{(n)}-J^P(\tilde{X}^{(n)})^T\tilde{\zeta}^{(n)}\nabla\tilde{q}^{(n)},\\[2mm]
&\tilde{f}^{(n)}_T=\trace(\tilde{\zeta}^{(n)}\nabla\tilde{\mathbf{T}}_p^{(n)}J^P(\tilde{X^{(n)}})).\\[4mm]
&\tilde{g}^{(n)}_w=\trace(\nabla\tilde{w}^{(n)}J^P)-\trace(\nabla\tilde{w}^{(n)}\tilde{\zeta}^{(n)}J^P(\tilde{X}^{(n)})),\\[2mm]
&\tilde{g}^{(n)}_{\phi}=\trace(\nabla\phi J^P)-\trace(\nabla\phi\tilde{\zeta}^{(n)}J^P(\tilde{X}^{(n)})).\\[5mm]
&\tilde{h}^{(n)}_w=(1-\kappa)\nabla\tilde{w}^{(n)}\tilde{n}_0-(1-\kappa)\nabla\tilde{w}^{(n)}\tilde{\zeta}^{(n)}\nabla_{\Lambda}\tilde{X}^{(n)}\tilde{n}_0,\\[2mm]
&\tilde{h}^{(n)}_{w^T}=(1-\kappa)(\nabla\tilde{w}J^P)^T(J^P)^{-1}\tilde{n}_0\\
&\hspace{1cm}-(1-\kappa)(\nabla\tilde{w}^{(n)}\tilde{\zeta}^{(n)}J^P(\tilde{X}^{(n)}))^T(J^P)^{-1}\nabla_{\Lambda}\tilde{X}^{(n)}\tilde{n}_0,\\[2mm]
&\tilde{h}^{(n)}_{\phi}=(1-\kappa)\nabla\phi\tilde{n}_0-(1-\kappa)\nabla\phi\tilde{\zeta}^{(n)}\nabla_{\Lambda}\tilde{X}^{(n)}\tilde{n}_0,\\[2mm]
&\tilde{h}^{(n)}_{\phi^T}=(1-\kappa)(\nabla\phi J^P)^T(J^P)^{-1}\tilde{n}_0\\
&\hspace{1cm}-(1-\kappa)(\nabla\phi\tilde{\zeta}^{(n)}J^P(\tilde{X}^{(n)}))^T(J^P)^{-1}\nabla_{\Lambda}\tilde{X}^{(n)}\tilde{n}_0,\\[2mm]
&\tilde{h}^{(n)}_{q}=\tilde{q}^{(n)}(J^P(\tilde{X}^{(n)}))^{-1}\nabla_{\Lambda}\tilde{X}^{(n)}\tilde{n}_0-\tilde{q}^{(n)}(J^P)^{-1}\tilde{n}_0,\\[2mm]
&\tilde{h}^{(n)}_T=-\tilde{\mathbf{T}}_p^{(n)}(J^P(\tilde{X}^{(n)}))^{-1}\nabla_{\Lambda}\tilde{X}^{(n)}\tilde{n}_0.
\end{split}
\end{equation}

\noindent Moreover the presence of the new function $\phi$ introduces new terms in the RHS of system \eqref{Re-v-conf-iterative}

\begin{align*}
&\tilde{f}^L_{\phi}=-\re\hspace{0.1cm}\partial_t\phi+(1-\kappa)Q^2\Delta\phi-(J^P)^T\tilde{q}_{\phi},\\[2mm]
&\tilde{g}^L_{\phi}=-\trace(\nabla\phi J^P),\\[2mm]
&\tilde{h}^L_{\phi}=\tilde{q}_{\phi}(J^P)^{-1}\tilde{n}_0-(1-\kappa)(\nabla\phi J^P+(\nabla\phi J^P)^T)(J^P)^{-1}\tilde{n}_0.
\end{align*}

\noindent Furthermore, in writing $\tilde{g}^{(n)}$, we make some adjustements in order to satisfy $(\tilde{g}^{(n)},\partial_t\tilde{g}^{(n)})_{|t=0}=(0,0)$. We set

\begin{align*}
&\bar{g}^{(n)}=\tilde{g}^{(n)}+\trace(\nabla\phi\tilde{\zeta}_{\phi}J^P_{\phi})-\trace(\nabla\phi J^P),\\[2mm]
&\bar{g}^L_{\phi}=\tilde{g}^L_{\phi}-\trace(\nabla\phi\tilde{\zeta}_{\phi}J^P_{\phi})+\trace(\nabla\phi J^P),
\end{align*}
\medskip

\noindent where $\tilde{\zeta}_{\phi}=\mathcal{I}-t(\nabla(J^P \tilde{v}_0))$ and $(J^P_{\phi})_{ij}=J^P_{ij}+t\partial_kJ^P_{ij}J^P_{kl}\tilde{v}_{0,l}$. Now we can define properly the operator $L$ associated to the system \eqref{systemW}, from $X_0$ into $Y_0$.

\begin{align*}
&L(\tilde{w}^{(n+1)},\tilde{q}^{(n+1)})=\left(\tilde{f}^{(n)}-\tilde{f}_{T_0},\bar{g}^{(n)}, \tilde{h}^{(n)}-\tilde{h}_{T_0}\right)+\left(\tilde{f}_{\phi}^L+\tilde{f}_{T_0}, \bar{g}_{\phi}^L, \tilde{h}^L_{\phi}+\tilde{h}_{T_0} \right),
\end{align*}
\noindent where $\tilde{f}_{T_0}=\trace(\nabla\tilde{T}_0J^P)$ and $\tilde{h}_{T_0}=-\tilde{T}_0(J^P)^{-1}\tilde{n}_0$.
\medskip

\noindent Concerning the systems  \eqref{We-G-conf-iterative} and \eqref{Lag-X-conf-iterative}, they are ODEs thus we obtain an explicit formula for $\tilde{\mathbf{T}}_p^{(n+1)}$ and $\tilde{X}^{(n+1)}$. The main result of this Section is  the following 

\begin{theorem}\label{local-existence conf-lag}
Let $2<s<\frac{5}{2}$, $1<\gamma<s-1$. If $\tilde{v}_0$ and $\tilde{\mathbf{T}}_0$ belong to $H^r(\tilde{\Omega}_0)$ for $r$ big enough. Then there exists a time $T=T(\re,\kappa,\We,\tilde{v}_0, \tilde{\mathbf{T}}_0)$ sufficiently small and a solution 
$\{\tilde{X}-\hat{X}, \tilde{w},\tilde{q}_w, \tilde{\mathbf{T}}_p-\tilde{\mathbf{T}}_0-t\hat{\mathbf{T}}\}\in\mathcal{F}^{s+1,\gamma}\times\mathcal{K}^{s+1}_{(0)}\times\mathcal{K}^s_{pr(0)}\times\mathcal{F}^{s,\gamma-1}_{(0)}$ on $(0,T]\times\tilde{\Omega}_0$.
Moreover, there is an explicit dependence of $T$ on $\We$, namely $T< C\left(\frac{\We}{1+\We}\right)$.
\end{theorem}
\medskip

\noindent The main problem we have to face is that $\tilde{X}-\tilde{\alpha}$ and $\tilde{\mathbf{T}}_p-\tilde{\mathbf{T}}_0$ do not belong to the space $H^2_{(0)}([0,T])$ because their time derivatives are not zero at time zero. For that reason we introduce $\hat{X}=\tilde{\alpha}+tJ^P\tilde{v}_0$ and $\hat{\mathbf{T}}=J^P\nabla\tilde{v}_0\mathbf{\tilde{T}}_0+\mathbf{\tilde{T}}_0(J^P\nabla\tilde{v}_0)^T-\frac{1}{\We}\mathbf{\tilde{T}}_0+\frac{\kappa}{\We}\left(J^P\nabla\tilde{v}_0+(J^P\nabla\tilde{v}_0)^T\right).$\\

\noindent Theorem \ref{local-existence conf-lag} is a consequence of the following proposition.

\begin{proposition}(\textbf{Iterative bounds}) \label{fixed point}
For $2<s<\frac{5}{2}$, $1<\gamma<s-1$ and for  $\tilde{X}^{(n)}-\tilde{\alpha}, \tilde{X}^{(n-1)}-\tilde{\alpha}\in B_1$,
$(\tilde{w}^{(n)},\tilde{q}_w^{(n)}), (\tilde{w}^{(n-1)}, \tilde{q}_w^{(n)})\in B_2$ and $\tilde{\mathbf{T}}_p^{(n)}-\tilde{\mathbf{T}}_0, \tilde{\mathbf{T}}_p^{(n-1)}-\tilde{\mathbf{T}}_0\in B_3$, where $B_1, B_2, B_3$ are balls that we will define later. Then it follows

\begin{align*}
&\left\|\tilde{w}^{(n+1)}-\tilde{w}^{(n)}\right\|_{\mathcal{K}^{s+1}_{(0)}}+\left\|\tilde{q}_w^{(n+1)}-\tilde{q}_w^{(n)}\right\|_{\mathcal{K}^{s}_{pr(0)}}+ \left\|\tilde{X}^{(n+1)}-\tilde{X}^{(n)}\right\|_{\mathcal{F}^{s+1,\gamma}}+\left\|\tilde{\mathbf{T}}_p^{(n+1)}-\tilde{\mathbf{T}}_p^{(n)}\right\|_{\mathcal{F}^{s,\gamma-1}}\\
&\leq C(\tilde{v}_0, \tilde{\mathbf{T}}_0,\re,\kappa)\left(1+\frac{1}{\We}\right)T^{\mu}\left( \left\|\tilde{X}^{(n)}-\tilde{X}^{(n-1)}\right\|_{\mathcal{F}^{s+1,\gamma}}+\left\|\tilde{w}^{(n)}-\tilde{w}^{(n-1)}\right\|_{\mathcal{K}^{s+1}_{(0)}}\right.\\
&\left.+\left\|\tilde{q}^{(n)}-\tilde{q}^{(n-1)}\right\|_{\mathcal{K}^{s}_{pr(0)}}+\left\|\tilde{\mathbf{T}}_p^{(n)}-\tilde{\mathbf{T}}_p^{(n-1)}\right\|_{\mathcal{F}^{s,\gamma-1}}\right).
\end{align*}
\end{proposition}
\medskip

\noindent The proof of Proposition \ref{fixed point} is obtained by estimating the velocity and the pressure, the flux and finally the elastic stress tensor separately. Before starting with all the computations, we stress that in these results a lot of terms of the type $A\cdot B^T$ or $A^T\cdot B$ will appear and we are not going to estimate all of them since the transpose does not affect the estimates.

\subsection{Estimate for the Conformal Lagrangian velocity and pressure}
\noindent We start with the analysis of the system \eqref{systemW}.

\begin{proposition}\label{estimate-conf-lag-(v,q)}
Given $2<s<\frac{5}{2}$, $\tilde{v}_0\in H^r(\tilde{\Omega}_0)$, $\tilde{\mathbf{T}}_0\in H^r(\tilde{\Omega}_0)$, for $r$ big enough and $T>0$ small enough we have
\begin{enumerate}

\item Let $\tilde{X}^{(n)}-\hat{X}\in\mathcal{F}^{s+1,\gamma}$, $\tilde{w}^{(n)}\in\mathcal{K}^{s+1}_{(0)}$, and $\tilde{q}_w^{(n)}\in \mathcal{K}^s_{pr(0)}$  and such that 

\begin{description}

\item [(a)] $\tilde{X}^{(n)}-\hat{X}\in\left\lbrace \tilde{X}-\hat{X}\in\mathcal{F}^{s+1,\gamma} :\left\|\tilde{X}-\tilde{\alpha}-\int_0^t J^P\nabla\phi\,d\tau\right\|_{\mathcal{F}^{s+1,\gamma}} \leq N \right\rbrace\equiv B_1,$\\[3mm]

\item [(b)]  $(\tilde{w}^{(n)},\tilde{q}_w^{(n)})\in\left\lbrace (\tilde{w},\tilde{q})\in\mathcal{K}^{s+1}_{(0)}\times\mathcal{K}^s_{pr(0)}:  \right.\\
\left.\hspace{1cm}\left\|(\tilde{w},\tilde{q})-L^{-1}(\tilde{f}_{\phi}^L+\tilde{f}_{T_0},\bar{g}_{\phi}^L,\tilde{h}_{\phi}^L+\tilde{h}_{T_0}))\right\|_{\mathcal{K}^{s+1}_{(0)}\times \mathcal{K}^s_{pr(0)}}\leq N\right\rbrace\equiv B_2.$\\[3mm]

\item [(c)]  $\tilde{\mathbf{T}}_p^{(n)}-\tilde{\mathbf{T}}_0-t\hat{\mathbf{T}}\in\left\lbrace \tilde{\mathbf{T}}_p-\tilde{\mathbf{T}}_0-t\hat{\mathbf{T}}\in\mathcal{F}^{s,\gamma-1}: \left\|\tilde{\mathbf{T}}_p-\tilde{\mathbf{T}}_0-\int_0^t \hat{\mathbf{T}}_{\phi}\,d\tau\right\|_{\mathcal{F}^{s,\gamma-1}} \leq N\right\rbrace\\
\equiv B_3,$
\end{description}
\medskip

\noindent where 
\begin{align*}
&\hat{\mathbf{T}}_{\phi}=J^P\nabla\phi\tilde{\mathbf{T}}_0+\tilde{\mathbf{T}}_0(J^P\nabla\phi)^T-\frac{1}{\We}\tilde{\mathbf{T}}_0+\frac{\kappa}{\We}\left(J^P\nabla\phi+(J^P\nabla\phi)^T\right)
\end{align*}

\noindent Then

\begin{equation}\label{part1-lag-(v,q)}
(\tilde{w}^{(n+1)},\tilde{q}_w^{(n+1)})\in B_2.
\end{equation}
\medskip

\item Let $\tilde{X}^{(n)}-\tilde{\alpha}, \tilde{X}^{(n-1)}-\tilde{\alpha}\in B_1$, $(\tilde{w}^{(n)}, \tilde{q}_w^{(n)}), (\tilde{w}^{(n-1)}, \tilde{q}_w^{(n-1)}) \in B_2$ and $\tilde{\mathbf{T}}_p^{(n)}-\tilde{\mathbf{T}}_0, \tilde{\mathbf{T}}_p^{(n-1)}-\tilde{\mathbf{T}}_0\in B_3$. Then for a suitable  $\varrho>0$

\begin{equation}\label{part2-lag-(v,q)}
\begin{split}
&\left\|\tilde{w}^{(n+1)}-\tilde{w}^{(n)}\right\|_{\mathcal{K}^{s+1}_{(0)}}+\left\|\tilde{q}_w^{(n+1)}-\tilde{q}_w^{(n)}\right\|_{\mathcal{K}^{s}_{pr(0)}}\leq C(\tilde{v}_0,\tilde{\mathbf{T}}_0,\re,\kappa)\left(1+\frac{1}{\We}\right) T^{\varrho}\\
&\cdot\left( \left\|\tilde{X}^{(n)}-\tilde{X}^{(n-1)}\right\|_{\mathcal{F}^{s+1,\gamma}}+\left\|\tilde{w}^{(n)}-\tilde{w}^{(n-1)}\right\|_{\mathcal{K}^{s+1}_{(0)}}+\left\|\tilde{q}_w^{(n)}-\tilde{q}_w^{(n-1)}\right\|_{\mathcal{K}^{s}_{pr(0)}}+ \left\|\tilde{\mathbf{T}}_p^{(n)}-\tilde{\mathbf{T}}_p^{(n-1)}\right\|_{\mathcal{F}^{s,\gamma-1}}\right).
\end{split}
\end{equation}
\end{enumerate}
\end{proposition}

\begin{proof}
The first part of this proof, concerning the boundedness of $(\tilde{w}^{(n+1)},\tilde{q}_w^{(n+1)})$ can be proved in the same way as \cite[Proposition 5.4]{CCFGG2}. The difference concerns the choice of the parameters of the ball $B_2$ and the presence of  new terms,  $\tilde{f}^{(n)}_{T}-\tilde{f}_{T_0}$ and  $\tilde{h}^{(n)}_{T}-\tilde{h}_{T_0}$ . We have to show

\begin{align*}
&\left\|(\tilde{w}^{(n+1)},\tilde{q}_w^{(n+1)})-L^{-1}(\tilde{f}_{\phi}^L+\tilde{f}_{T_0}, \bar{g}_{\phi}^L, \tilde{h}_{\phi}^L+\tilde{h}_{T_0})\right\|_{X_0}\leq C \left\|L^{-1}(\tilde{f}^{(n)}-\tilde{f}_{T_0},\hspace{0.2cm} \bar{g}^{(n)},\hspace{0.2cm}\tilde{h}^{(n)}-\tilde{h}_{T_0}, \tilde{v}_0)\right\|_{X_0}\\[2mm]
&\hspace{0.5cm}\leq C\left(\left\|\tilde{f}^{(n)}-\tilde{f}_{T_0}\right\|_{\mathcal{K}^{s-1}_{(0)}}+\left\|\bar{g}^{(n)}\right\|_{\mathcal{\bar{K}}^{s}_{(0)}}+\left\|\tilde{h}^{(n)}-\tilde{h}_{T_0}\right\|_{\mathcal{K}^{s-\frac{1}{2}}_{(0)}}\right)
\end{align*}
\medskip

\noindent Thus it is sufficient to prove

\begin{align*}
&\|\tilde{f}^{(n)}-\tilde{f}_{T_0}\|_{\mathcal{K}^{s-1}_{(0)}}\leq C\left(\tilde{v}_0,\tilde{\mathbf{T}}_0,\re,\kappa, \|\tilde{w}^{(n)}\|_{\mathcal{K}^{s+1}_{(0)}}, \|\tilde{q}_w^{(n)}\|_{\mathcal{K}^{s}_{pr(0)}},\|\tilde{X}^{(n)}-\hat{X}\|_{\mathcal{F}^{s+1,\gamma}},\right.\\
&\hspace{3cm}\left.\|\tilde{\mathbf{T}}_p-\tilde{\mathbf{T}}_0-t\hat{\mathbf{T}}\|_{\mathcal{F}^{s,\gamma-1}}\right)T^{\delta},\\[2mm]
&\|\bar{g}^{(n)}\|_{\mathcal{\bar{K}}^{s}_{(0)}}\leq C\left(\tilde{v}_0,\tilde{\mathbf{T}}_0,\|\tilde{w}^{(n)}\|_{\mathcal{K}^{s+1}_{(0)}},\|\tilde{X}^{(n)}-\hat{X}\|_{\mathcal{F}^{s+1,\gamma}}\right)T^{\theta},\\[2mm]
&\|\tilde{h}^{(n)}-\tilde{h}_{T_0}\|_{\mathcal{K}^{s-\frac{1}{2}}_{(0)}}\leq C\left(\tilde{v}_0,\tilde{\mathbf{T}}_0,\kappa, \|\tilde{w}^{(n)}\|_{\mathcal{K}^{s+1}_{(0)}}, \|\tilde{q}_w^{(n)}\|_{\mathcal{K}^{s}_{pr}(0)},\|\tilde{X}^{(n)}-\hat{X}\|_{\mathcal{F}^{s+1,\gamma}},\right.\\
&\hspace{3cm}\left.\|\tilde{\mathbf{T}}_p-\tilde{\mathbf{T}}_0-t\hat{\mathbf{T}}\|_{\mathcal{F}^{s,\gamma-1}}\right)T^{\beta},
\end{align*}
\medskip

\noindent We have to estimate $\tilde{f}^{(n)}-\tilde{f}_{T_0},$ $\bar{g}^{(n)},$ and $\tilde{h}^{(n)}-\tilde{h}_{T_0}$. In particular all the terms in \eqref{Re-split}. We observe that these terms have already been estimated in \cite[Proposition 5.4]{CCFGG2} except for $\tilde{f}^{(n)}_{T}-\tilde{f}_{T_0}$ and $\tilde{h}^{(n)}_T-\tilde{h}_{T_0}$. The result already obtained can be resumed as follows

\begin{equation}\label{their-estim}
\begin{split}
&\|\tilde{f}^{(n)}_w\|_{\mathcal{K}^{s+1}_{(0)}}+\|\tilde{f}^{(n)}_{\phi}\|_{\mathcal{K}^{s+1}_{(0)}}+\|\tilde{f}^{(n)}_q\|_{\mathcal{K}^{s+1}_{(0)}}\leq C(\tilde{v}_0,\tilde{\mathbf{T}}_0,\re,\kappa) T^{\delta'},\\[2mm]
&\|\tilde{g}_w\|_{\mathcal{\bar{K}}^s_{(0)}}+\|\tilde{g}_{\phi}\|_{\mathcal{\bar{K}}^s_{(0)}}\leq C(\tilde{v}_0)T^{\theta},\\[2mm]
&\|\tilde{h}^{(n)}_w\|_{\mathcal{K}^{s-\frac{1}{2}}_{(0)}}+ \|\tilde{h}^{(n)}_{w^T}\|_{\mathcal{K}^{s-\frac{1}{2}}_{(0)}}+\|\tilde{h}^{(n)}_{\phi}\|_{\mathcal{K}^{s-\frac{1}{2}}_{(0)}}+\|\tilde{h}^{(n)}_{\phi^T}\|_{\mathcal{K}^{s-\frac{1}{2}}_{(0)}}+\|\tilde{h}^{(n)}_q\|_{\mathcal{K}^{s-\frac{1}{2}}_{(0)}}\\[1mm]
&\hspace{7cm}\leq C(\tilde{v}_0,\tilde{\mathbf{T}}_0,\re,\kappa) T^{\beta'},
\end{split}
\end{equation}
\medskip

\noindent for all $\tilde{X}^{(n)}-\hat{X}\in B_1, (\tilde{w^{(n)}},\tilde{q}_w^{(n)})\in B_2$ and $\tilde{\mathbf{T}}_p-\tilde{\mathbf{T
}}_0-t\hat{\mathbf{T}}\in B_3$.
\medskip

\noindent\underline{\textbf{Estimate for $\tilde{f}^{(n)}_T-\tilde{f}_{T_0}$}}\\

\noindent  The term we have to estimate $\tilde{f}^{(n)}_T-\tilde{f}_{T_0}= \trace(\tilde{\zeta}^{(n)}\nabla\tilde{\mathbf{T}}_p^{(n)}J^P(\tilde{X}^{(n)}))-\trace(\nabla\tilde{\mathbf{T}}_0J^P)$ belongs to the space $\mathcal{K}^{s-1}_{(0)}([0,T];\tilde{\Omega}_0)$, then we have

\begin{align*}
\|\tilde{f}^{(n)}_T-\tilde{f}_{T_0}\|_{\mathcal{K}^{s-1}_{(0)}}\leq \|\tilde{f}^{(n)}_T-\tilde{f}_{T_0}\|_{L^2H^{s-1}}+\|\tilde{f}^{(n)}_T-\tilde{f}_{T_0}\|_{H^{\frac{s-1}{2}_{(0)}}L^2}=I_1+I_2.
\end{align*}

\noindent We start with the estimate in $L^2H^{s-1}$ and we split $I_1$ as follows

\begin{align*}
I_1&=\trace(\tilde{\zeta}^{(n)}(\nabla\tilde{\mathbf{T}}_p^{(n)}-\nabla\tilde{\mathbf{T}}_0)J^P(\tilde{X}^{(n)}))+\trace(\tilde{\zeta}^{(n)}\nabla\tilde{\mathbf{T}}_0(J^P(\tilde{X}^{(n)})-J^P))\\
&+\trace((\tilde{\zeta}^{(n)}-\mathcal{I})\nabla\tilde{\mathbf{T}}_0 J^P)=I_{1,1}+I_{1,2}+I_{1,3}.
\end{align*}

\noindent We will show the estimate of $I_{1,1}$ since the other estimates can be deduced from it. We use lemma \ref{Jp-est} and lemma \ref{zeta-est}. In addition we need to use the estimate for the flux \eqref{flux-estim} and fot the elastic stress tensor \eqref{elastic-estim}.

\begin{align*}
&\|I_{1,1}\|_{L^2H^{s-1}}\leq \|\tilde{\zeta}^{(n)}\|_{L^{\infty}H^{s-1}}\|\nabla\tilde{\mathbf{T}}_p^{(n)}-\nabla\tilde{\mathbf{T}}_0\|_{L^2H^{s-1}}\|J^P(\tilde{X}^{(n)})\|_{L^{\infty}H^{s-1}}\\[2mm]
&\leq C(\tilde{v}_0)\|\tilde{X}^{(n)}-\tilde{\alpha}\|_{L^{\infty}H^{s}}T^{\frac{1}{2}}\|\tilde{\mathbf{T}}_p^{(n)}-\tilde{\mathbf{T}}_0\|_{L^{\infty}H^{s}}\leq C(\tilde{v}_0,\tilde{\mathbf{T}}_0,\kappa) \left(1+\frac{1}{\We}\right) T^{\frac{3}{4}}.
\end{align*}
\medskip

\noindent The estimate of $I_2$ is more complicated since each term  has to satisfy the zero conditions at time zero required by lemma \ref{lem5}. For this important reason $I_2$ becomes

\begin{align*}
I_2&=\trace((\tilde{\zeta}^{(n)}-\mathcal{I})(\nabla\tilde{\mathbf{T}}_p^{(n)}-\nabla\tilde{\mathbf{T}}_0)(J^P(\tilde{X}^{(n)})-J^P))\\[1mm]
&+\trace((\tilde{\zeta}^{(n)}-\mathcal{I})\nabla\tilde{\mathbf{T}}_0(J^P(\tilde{X}^{(n)})-J^P))\\[1mm]
&+\trace((\tilde{\zeta}^{(n)}-\mathcal{I})(\nabla\tilde{\mathbf{T}}_p^{(n)}-\nabla\tilde{\mathbf{T}}_0)J^P)+\trace((\tilde{\zeta}^{(n)}-\mathcal{I})\nabla\tilde{\mathbf{T}}_0J^P)\\[1mm]
&+\trace((\nabla\tilde{\mathbf{T}}_p^{(n)}-\nabla\tilde{\mathbf{T}}_0)(J^P(\tilde{X}^{(n)})-J^P))+\trace(\nabla\tilde{\mathbf{T}}_0(J^P(\tilde{X}^{(n)})-J^P))\\[1mm]
&+\trace((\nabla\tilde{\mathbf{T}}_p^{(n)}-\nabla\tilde{\mathbf{T}}_0)J^P)=\sum_{i=1}^7 I_{2,i}.
\end{align*}
\medskip

\noindent We show the estimates of $I_{2,1}$ and $I_{2,6}$. For $I_{2,1}$ we use lemma \ref{lem3}, lemma \ref{lem5}. Moreover we need lemma \ref{Jp-est} and lemma \ref{zeta-est} and to conclude the estimate we use the estimate \eqref{elastic-estim} for the elastic tensor. For $I_{2,6}$ is a bit easier. We use the property of the space $H^{\frac{s-1}{2}}_{(0)}([0,T])$ which allows us to apply lemma \ref{lem3} in order to separte the term $\nabla\tilde{\mathbf{T}}_0$ which does not depend on time and to conclude we apply the estimate for the flux \eqref{flux-estim}.

\begin{align*}
&\|I_{2,1}\|_{H^{\frac{s-1}{2}}_{(0)}L^2}\leq \|(\tilde{\zeta}^{(n)}-\mathcal{I})(J^P(\tilde{X}^{(n)})-J^P)\|_{H^{\frac{s-1}{2}}_{(0)}H^{1+\eta}}\|\nabla\tilde{\mathbf{T}}_p^{(n)}-\nabla\tilde{\mathbf{T}}_0\|_{H^{\frac{s-1}{2}}_{(0)}L^2}\\[2mm]
&\hspace{1cm}\leq \|\tilde{\zeta}^{(n)}-\mathcal{I}\|_{H^{\frac{s-1}{2}}_{(0)}H^{1+\eta}}\|J^P(\tilde{X}^{(n)})-J^P\|_{H^{\frac{s-1}{2}}_{(0)}H^{1+\eta}}\|\tilde{\mathbf{T}}_p^{(n)}-\tilde{\mathbf{T}}_0\|_{H^{\frac{s-1}{2}}_{(0)}H^1}\\[2mm]
&\hspace{1cm}\leq C(\tilde{v}_0,\tilde{\mathbf{T}}_0,\kappa)\left(1+\frac{1}{\We}\right)T^{\delta_1}\\[5mm]
&\|I_{2,6}\|_{H^{\frac{s-1}{2}}_{(0)}L^2}\leq \|\nabla\tilde{\mathbf{T}}_0\|_{L^2}\|J^P(\tilde{X}^{(n)})-J^P\|_{H^{\frac{s-1}{2}}_{(0)}H^{1+\eta}}\leq C(\tilde{\mathbf{T}}_0)\|\tilde{X}^{(n)}-\tilde{\alpha}\|_{H^{\frac{s-1}{2}}_{(0)}H^{1+\eta}}\\[2mm]
&\hspace{1cm}\leq C(\tilde{v}_0,\tilde{\mathbf{T}}_0) T^{\delta_6}.
\end{align*}
\medskip

\noindent For the other terms we have $\|I_{2,i}\|_{H^{\frac{s-1}{2}}_{(0)}L^2}\leq C(\tilde{v}_0,\tilde{\mathbf{T}}_0,\kappa)\left(1+\frac{1}{\We}\right)T^{\delta_i}$, for $i=3,5,7$ and $\|I_{2,i}\|_{H^{\frac{s-1}{2}}_{(0)}L^2}\leq C(\tilde{v}_0,\tilde{\mathbf{T}}_0) T^{\delta_i}$, for $i=2,4$.
\bigskip

\noindent\underline{\textbf{Estimate for $\tilde{h}^{(n)}_T-\tilde{h}_{T_0}$}}\\

\noindent The term  $\tilde{h}_T^{(n)}-\tilde{h}_{T_0}=-\tilde{\mathbf{T}}_p^{(n)}J^P(\tilde{X}^{(n)})^{-1}\nabla_{\Lambda}\tilde{X}^{(n)}\tilde{n}_0+\tilde{\mathbf{T}}_0(J^P)^{-1}\tilde{n}_0$ is related to the boundary. So its norm can be written as follows

\begin{align*}
&\|\tilde{h}_T^{(n)}-\tilde{h}_{T_0}\|_{\mathcal{K}^{s-\frac{1}{2}}_{(0)}}\leq \|\tilde{h}_T^{(n)}-\tilde{h}_{T_0}\|_{L^2H^{s-\frac{1}{2}}}+\|\tilde{h}_T^{(n)}-\tilde{h}_{T_0}\|_{H^{\frac{s}{2}-\frac{1}{4}}_{(0)}L^2}=I_1+I_2
\end{align*}
\medskip

\noindent As for the estimate of $\tilde{f}^{(n)}_T-\tilde{f}_{T_0}$ also in this case the estimate of $I_1$ is easier than the estimate of  $I_2$. First we deal with $I_1$ and we split it in the following way

\begin{align*}
I_1&=(\tilde{\mathbf{T}}_0-\tilde{\mathbf{T}}_p^{(n)})J^P(\tilde{X}^{(n)})^{-1}\nabla_{\Lambda}\tilde{X}^{(n)}\tilde{n}_0+\tilde{\mathbf{T}}_0((J^P)^{-1}-J^P(\tilde{X}^{(n)})^{-1})\nabla_{\Lambda}\tilde{X}^{(n)}\tilde{n}_0\\
&+\tilde{\mathbf{T}}_0(J^P)^{-1}(\mathcal{I}-\nabla_{\Lambda}\tilde{X}^{(n)})\tilde{n}_0=I_{1,1}+I_{1,2}+I_{1,3}.
\end{align*}

\noindent The estimate of $I_{1,1}$ is based on the use of the trace theorem \ref{parabolic-trace}, lemma \ref{Jp-est} and lemma \ref{zeta-est}. Moreover we need to use \eqref{flux-estim} and \eqref{elastic-estim}. For $I_{1,2}$ we use the same lemmas ad $I_{1,1}$ but we do not need \eqref{elastic-estim}.

\begin{align*}
&\|I_{1,1}\|_{L^2H^{s-\frac{1}{2}}}\leq \|\tilde{\mathbf{T}}_0-\tilde{\mathbf{T}}_p^{(n)}\|_{L^2H^{s-\frac{1}{2}}}\|J^P(\tilde{X}^{(n)})^{-1}\|_{L^{\infty}H^{s-\frac{1}{2}}}\|\nabla_{\Lambda}\tilde{X}^{(n)}\|_{L^{\infty}H^{s-\frac{1}{2}}}\\[2mm]
&\hspace{2.1cm}\leq C(\tilde{v}_0)T^{\frac{1}{2}}\|\tilde{\mathbf{T}}_p^{(n)}-\tilde{\mathbf{T}}_0\|_{L^{\infty}H^{s-1}}\leq C(\tilde{v}_0,\tilde{\mathbf{T}}_0,\kappa)\left(1+\frac{1}{\We}\right)T^{\frac{3}{4}},\\[5mm]
&\|I_{1,2}\|_{L^2H^{s-\frac{1}{2}}}\leq \|\tilde{\mathbf{T}}_0\|_{H^{s-\frac{1}{2}}}\|(J^P)^{-1}-J^P(\tilde{X}^{(n)})^{-1})\|_{L^2H^{s-\frac{1}{2}}}\|\nabla_{\Lambda}\tilde{X}^{(n)}\|_{L^{\infty}H^{s-\frac{1}{2}}}\\[2mm]
&\hspace{2.1cm}\leq C(\tilde{v}_0,\tilde{\mathbf{T}}_0)T^{\frac{1}{2}}\|\tilde{X}^{(n)}-\tilde{\alpha}\|_{L^{\infty}H^{s-1}}\leq C(\tilde{v}_0,\tilde{\mathbf{T}}_0)T^{\frac{3}{4}}.
\end{align*}

\noindent The estimate of $I_{1,3}$ is exactly as $I_{1,2}$. So we can pass to estimate the term $I_2$, which needs more splittings in order to satisfy the properties required for the space $H^{\frac{s}{2}-\frac{1}{4}}_{(0)}$. Thus we have the following terms

\begin{align*}
I_2=&(\tilde{\mathbf{T}}_0-\tilde{\mathbf{T}}_p^{(n)})(J^P(\tilde{X}^{(n)})^{-1}-(J^P)^{-1})(\nabla_{\Lambda}\tilde{X}^{(n)}-\mathcal{I})\tilde{n}_0\\
&-\tilde{\mathbf{T}}_0(J^P(\tilde{X}^{(n)})^{-1}-(J^P)^{-1})(\nabla_{\Lambda}\tilde{X}^{(n)}-\mathcal{I})\tilde{n}_0+(\tilde{\mathbf{T}}_0-\tilde{\mathbf{T}}_p^{(n)})(J^P)^{-1}(\nabla_{\Lambda}\tilde{X}^{(n)}-\mathcal{I})\tilde{n}_0\\
&-\tilde{\mathbf{T}}_0(J^P)^{-1}(\nabla_{\Lambda}\tilde{X}^{(n)}-\mathcal{I})\tilde{n}_0+(\tilde{\mathbf{T}}_0-\tilde{\mathbf{T}}_p^{(n)})(J^P(\tilde{X}^{(n)})^{-1}-(J^P)^{-1})\tilde{n}_0\\
&+(\tilde{\mathbf{T}}_0-\tilde{\mathbf{T}}_p^{(n)})(J^P)^{-1}\tilde{n}_0=\sum_{i=1}^6 I_{2,i}.
\end{align*}

\noindent We show the estimate of $I_{2,1}$ and $I_{2,2}$. For the first term we use lemma \ref{lem4} with $\frac{1}{q}=\frac{1}{2}-\eta$, the trace theorem and lemma \ref{lem5}, finally lemma \ref{Jp-est}, lemma \ref{zeta-est} and \eqref{flux-estim}, \eqref{elastic-estim} give the result. The term $I_{2,2}$ can be estimatated by using the same lemmas as $I_{2,1}$ but we have to take into account some properties of the $H^{\frac{s}{2}-\frac{1}{4}}_{(0)}(0,T])$ in order to estimate the term that does not depend on time $\tilde{\mathbf{T}}_0$.

\begin{align*}
&\|I_{2,1}\|_{H^{\frac{s}{2}-\frac{1}{4}}_{(0)}L^2}\leq \|\tilde{\mathbf{T}}_0-\tilde{\mathbf{T}}_p^{(n)}\|_{H^{\frac{s}{2}-\frac{1}{4}}_{(0)}H^{\frac{1}{2}-\eta}}\|(J^P(\tilde{X}^{(n)})^{-1}-(J^P)^{-1})(\nabla_{\Lambda}\tilde{X}^{(n)}-\mathcal{I})\|_{H^{\frac{s}{2}-\frac{1}{4}}_{(0)}H^{\frac{1}{2}+\eta}}\\[2mm]
&\hspace{0.3cm}\leq \|\tilde{\mathbf{T}}_0-\tilde{\mathbf{T}}_p^{(n)}\|_{H^{\frac{s}{2}-\frac{1}{4}}_{(0)}H^{1-\eta}}\|J^P(\tilde{X}^{(n)})^{-1}-(J^P)^{-1}\|_{H^{\frac{s}{2}-\frac{1}{4}}_{(0)}H^{1+\eta}}\|\nabla_{\Lambda}\tilde{X}^{(n)}-\mathcal{I}\|_{H^{\frac{s}{2}-\frac{1}{4}}_{(0)}H^{1+\eta}}\\[2mm]
&\hspace{0.3cm}\leq C(\tilde{v}_0,\tilde{\mathbf{T}}_0,\kappa)\left(1+\frac{1}{\We}\right)T^{\beta_1},\\[5mm]
&\|I_{2,2}\|_{H^{\frac{s}{2}-\frac{1}{4}}_{(0)}L^2}\leq \|\tilde{\mathbf{T}}_0\|_{H^{\frac{1}{2}-\eta}}\|(J^P(\tilde{X}^{(n)})^{-1}-(J^P)^{-1})(\nabla_{\Lambda}\tilde{X}^{(n)}-\mathcal{I})\|_{H^{\frac{s}{2}-\frac{1}{4}}_{(0)}H^{\frac{1}{2}+\eta}}\\[2mm]
&\hspace{0.3cm}\leq C(\tilde{\mathbf{T}}_0)\|J^P(\tilde{X}^{(n)})^{-1}-(J^P)^{-1}\|_{H^{\frac{s}{2}-\frac{1}{4}}_{(0)}H^{1+\eta}}\|\nabla_{\Lambda}\tilde{X}^{(n)}-\mathcal{I}\|_{H^{\frac{s}{2}-\frac{1}{4}}_{(0)}H^{1+\eta}}\\[2mm]
&\hspace{0.3cm}\leq  C(\tilde{\mathbf{T}}_0)\|\tilde{X}^{(n)}-\tilde{\alpha}\|_{H^{\frac{s}{2}-\frac{1}{4}}_{(0)}H^{1+\eta}}\leq C(\tilde{v}_0,\tilde{\mathbf{T}}_0) T^{\beta_2}.
\end{align*}
\medskip

\noindent In conclusion, for $i=3,5,6$ we have $\|I_{2,i}\|_{H^{\frac{s}{2}-\frac{1}{4}}_{(0)}L^2}\leq C(\tilde{v}_0,\tilde{\mathbf{T}}_0,\kappa)\left(1+\frac{1}{\We}\right)T^{\beta_i}$ as $I_{2,1}$ and $\|I_{2,4}\|_{H^{\frac{s}{2}-\frac{1}{4}}_{(0)}L^2}\leq C(\tilde{v}_0,\tilde{\mathbf{T}}_0) T^{\beta_4}$.\\

\noindent The first point of the proposition holds by choosing $\delta=\min\{\delta', \frac{3}{4},\delta_i\}$, for $i=1,\ldots,7$ and $\beta=\min\{\beta', \frac{3}{4}, \beta_i\}$, for $i=1,\ldots,6$. Then by gather the estimates for $\tilde{f}^{(n)}_T-\tilde{f}_{T_0}$ and the estimates for $\tilde{h}^{(n)}_T-\tilde{h}_{T_0}$ with \eqref{their-estim}, we get the final result.
\bigskip

\noindent Now we can prove the second part of the proposition, then we must estimate the differences. In according to the proof above, it is enough to estimate

\begin{align*}
&\|\tilde{f}^{(n)}-\tilde{f}^{(n-1)}\|_{\mathcal{K}^{s-1}_{(0)}}\leq C(\tilde{v}_0,\tilde{\mathbf{T}}_0,\re,\kappa)T^{\delta}\left( \|\tilde{w}^{(n)}-\tilde{w}^{(n-1)}\|_{\mathcal{K}^{s+1}_{(0)}}+\|\tilde{q}_w^{(n)}-\tilde{q}_w^{(n-1)}\|_{\mathcal{K}^{s}_{pr(0)}}\right.\\[2mm]
&\hspace{0.5cm}\left.+\|\tilde{\mathbf{T}}_p^{(n)}-\tilde{\mathbf{T}}_p^{(n-1)}\|_{\mathcal{F}^{s,\gamma-1}}+\|\tilde{X}^{(n)}-\tilde{X}^{(n-1)}\|_{\mathcal{F}^{s+1,\gamma}}\right),\\[5mm]
&\|\tilde{g}^{(n)}-\tilde{g}^{(n-1)}\|_{\mathcal{\bar{K}}^{s}_{(0)}}\leq C(\tilde{v}_0,\tilde{\mathbf{T}}_0)T^{\theta}\left(\|\tilde{w}^{(n)}-\tilde{w}^{(n-1)}\|_{\mathcal{K}^{s+1}_{(0)}}+\|\tilde{X}^{(n)}-\tilde{X}^{(n-1)}\|_{\mathcal{F}^{s+1,\gamma}}\right),\\[5mm]
&\|\tilde{h}^{(n)}-\tilde{h}^{(n-1)}\|_{\mathcal{K}^{s-\frac{1}{2}}_{(0)}}\leq C(\tilde{v}_0,\tilde{\mathbf{T}}_0,\re,\kappa) T^{\beta}\left( \|\tilde{w}^{(n)}-\tilde{w}^{(n-1)}\|_{\mathcal{K}^{s+1}_{(0)}}+\|\tilde{q}_w^{(n)}-\tilde{q}_w^{(n-1)}\|_{\mathcal{K}^{s}_{pr(0)}}\right.\\[2mm]
&\hspace{0.5cm}\left.+\|\tilde{\mathbf{T}}_p^{(n)}-\tilde{\mathbf{T}}_p^{(n-1)}\|_{\mathcal{F}^{s,\gamma-1}}+\|\tilde{X}^{(n)}-\tilde{X}^{(n-1)}\|_{\mathcal{F}^{s+1,\gamma}}\right).
\end{align*}
\medskip

\textbf{\underline{Estimate for $\tilde{f}^{(n)}-\tilde{f}^{(n-1)}$}}\\

\noindent This difference can be estimated as follows

\begin{align*}
\|\tilde{f}^{(n)}-\tilde{f}^{(n-1)}\|_{\mathcal{K}^{s-1}_{(0)}}&\leq \|\tilde{f}^{(n)}_w-\tilde{f}^{(n-1)}_w\|_{\mathcal{K}^{s-1}_{(0)}}+\|\tilde{f}^{(n)}_{\phi}-\tilde{f}^{(n-1)}_{\phi}\|_{\mathcal{K}^{s-1}_{(0)}}+\|\tilde{f}^{(n)}_q-\tilde{f}^{(n-1)}_q\|_{\mathcal{K}^{s-1}_{(0)}}\\[2mm]
&+\|\tilde{f}^{(n)}_T-\tilde{f}^{(n-1)}_T\|_{\mathcal{K}^{s-1}_{(0)}}
\end{align*}

\noindent  Some of these estimates have been proved in details in \cite[Proposition 5.4]{CCFGG2} and we resume here the results

\begin{equation}\label{their-f}
\begin{split}
&\|\tilde{f}^{(n)}_w-\tilde{f}^{(n-1)}_w\|_{\mathcal{K}^{s-1}_{(0)}}\leq C(\tilde{v}_0, \tilde{\mathbf{T}}_0,\re,\kappa)T^{\delta'}\left(\|\tilde{X}^{(n)}-\tilde{X}^{(n-1)}\|_{\mathcal{F}^{s+1,\gamma}}+\|\tilde{w}^{(n)}-\tilde{w}^{(n-1)}\|_{\mathcal{K}^{s+1}_{(0)}}\right),\\[2mm]
&\|\tilde{f}^{(n)}_{\phi}-\tilde{f}^{(n-1)}_{\phi}\|_{\mathcal{K}^{s-1}_{(0)}}\leq C(\tilde{v}_0, \tilde{\mathbf{T}}_0,\kappa)T^{\delta''}\|\tilde{X}^{(n)}-\tilde{X}^{(n-1)}\|_{\mathcal{F}^{s+1,\gamma}},\\[2mm]
&\|\tilde{f}^{(n)}_q-\tilde{f}^{(n-1)}_q\|_{\mathcal{K}^{s-1}_{(0)}}\leq C(\tilde{v}_0, \tilde{\mathbf{T}}_0,\re,\kappa)T^{\delta'''}\left(\|\tilde{X}^{(n)}-\tilde{X}^{(n-1)}\|_{\mathcal{F}^{s+1,\gamma}}+\|\tilde{q}^{(n)}_w-\tilde{q}^{(n-1)}_w\|_{\mathcal{K}^{s}_{pr(0)}}\right).
\end{split}
\end{equation}

\noindent We want to show the resulting estimates related to $\tilde{f}^{(n)}_T-\tilde{f}^{(n-1)}_T$, in both $L^2H^{s-1}$ and $H^{\frac{s-1}{2}}_{(0)}L^2$. For the estimates in $L^2H^{s-1}-$norm we split as follows

\begin{align*}
&d_{1,T}^{\tilde{f}}=\trace\left((\tilde{\zeta}^{(n)}-\tilde{\zeta}^{(n-1)})(\nabla\tilde{\mathbf{T}}_p^{(n)}-\nabla\tilde{\mathbf{T}}_0)J^P(\tilde{X}^{(n)})\right),\\[1mm]
&d_{2,T}^{\tilde{f}}=\trace\left((\tilde{\zeta}^{(n)}-\tilde{\zeta}^{(n-1)})\nabla\tilde{\mathbf{T}}_0J^P(\tilde{X}^{(n)})\right),\\[1mm]
&d_{3,T}^{\tilde{f}}=\trace\left(\tilde{\zeta}^{(n-1)}(\nabla\tilde{\mathbf{T}}_p^{(n)}-\nabla\tilde{\mathbf{T}}_p^{(n-1)})J^P(\tilde{X}^{(n)})\right),\\[1mm]
&d_{4,T}^{\tilde{f}}=\trace\left(\tilde{\zeta}^{(n-1)}(\nabla\tilde{\mathbf{T}}_p^{(n-1)}-\nabla\tilde{\mathbf{T}}_0)(J^P(\tilde{X}^{(n)})-J^P(\tilde{X}^{(n-1)}))\right),\\[1mm]
&d_{5,T}^{\tilde{f}}=\trace\left(\tilde{\zeta}^{(n-1)}\nabla\tilde{\mathbf{T}}_0(J^P(\tilde{X}^{(n)}-J^P(\tilde{X}^{(n-1)}))\right)
\end{align*}
\medskip

\noindent We show the estimate of $d_{1,T}^{\tilde{f}}$ and $d_{3,T}^{\tilde{f}}$, which give us the desired differences by means of  lemma \ref{Jp-est},  lemma \ref{zeta-dif-est} and \eqref{flux-estim}, \eqref{elastic-estim}.

\begin{align*}
&\|d_{1,T}^{\tilde{f}}\|_{L^2H^{s-1}}\leq \|\tilde{\zeta}^{(n)}-\tilde{\zeta}^{(n-1)}\|_{L^2H^{s-1}}\|\nabla\tilde{\mathbf{T}}_p^{(n)}-\nabla\tilde{\mathbf{T}}_0\|_{L^2H^{s-1}}\|J^P(\tilde{X}^{(n)})\|_{L^{\infty}H^{s-1}}\\[1mm]
&\hspace{0.5cm}\leq C(\tilde{v}_0)\|\tilde{X}^{(n)}-\tilde{X}^{(n-1)}\|_{L^{\infty}H^{s+1}}T^{\frac{1}{2}}\|\tilde{\mathbf{T}}_p^{(n)}-\tilde{\mathbf{T}}_0\|_{L^{\infty}H^{s-1}}\\[1mm]
&\hspace{0.5cm}\leq C(\tilde{v}_0,\tilde{\mathbf{T}}_0) \left(1+\frac{1}{\We}\right)T^{\frac{3}{4}}\|\tilde{X}^{(n)}-\tilde{X}^{(n-1)}\|_{\mathcal{F}^{s+1,\gamma}},\\[3mm]
&\|d_{3,T}^{\tilde{f}}\|_{L^2H^{s-1}}\leq\|\tilde{\zeta}^{(n-1)}\|_{L^{\infty}H^{s-1}}\|\nabla\tilde{\mathbf{T}}_p^{(n)}-\nabla\tilde{\mathbf{T}}_p^{(n-1)}\|_{L^2H^{s-1}}\|J^P(\tilde{X}^{(n)})\|_{L^{\infty}H^{s-1}}\\[1mm]
&\hspace{0.5cm}\leq C(\tilde{v}_0)T^{\frac{1}{2}}\|\tilde{\mathbf{T}}_p^{(n)}-\tilde{\mathbf{T}}_p^{(n-1)}\|_{L^{\infty}H^{s}}\leq C(\tilde{v}_0)T^{\frac{3}{4}}\|\tilde{\mathbf{T}}_p^{(n)}-\tilde{\mathbf{T}}_p^{(n-1)}\|_{\mathcal{F}^{s,\gamma-1}}.
\end{align*}
\noindent The remaining terms give the same results of $d_{1,T}^{\tilde{f}}$ by using also lemma \ref{Jp-dif-est} and lemma \ref{zeta-est}. For the $H^{\frac{s-1}{2}}_{(0)}L^2-$norm we split the difference in another way because in this space, in order to have constants independent of time, all terms are required to be zero at $t=0$. 
\begin{align*}
&d_{1,T}^{\tilde{f}}=\trace\left((\tilde{\zeta}^{(n)}-\tilde{\zeta}^{(n-1)})(\nabla\tilde{\mathbf{T}}_p^{(n)}-\nabla\tilde{\mathbf{T}}_0)(J^P(\tilde{X}^{(n)})-J^P)\right),\\
&d_{2,T}^{\tilde{f}}=\trace\left((\tilde{\zeta}^{(n)}-\tilde{\zeta}^{(n-1)})\nabla\tilde{\mathbf{T}}_0(J^P(\tilde{X}^{(n)})-J^P)\right)\\
&d_{3,T}^{\tilde{f}}=\trace\left((\tilde{\zeta}^{(n)}-\tilde{\zeta}^{(n-1)})(\nabla\tilde{\mathbf{T}}_p^{(n)}-\nabla\tilde{\mathbf{T}}_0)J^P\right),\\
&d_{4,T}^{\tilde{f}}=\trace\left((\tilde{\zeta}^{(n)}-\tilde{\zeta}^{(n-1)})\nabla\tilde{\mathbf{T}}_0 J^P\right)\\
&d_{5,T}^{\tilde{f}}=\trace\left((\tilde{\zeta}^{(n-1)}-\mathcal{I})(\nabla\tilde{\mathbf{T}}_p^{(n)}-\nabla\tilde{\mathbf{T}}_p^{(n-1)})(J^P(\tilde{X}^{(n)})-J^P)\right),\\
&d_{6,T}^{\tilde{f}}=\trace\left((\nabla\tilde{\mathbf{T}}_p^{(n)}-\nabla\tilde{\mathbf{T}}_p^{(n-1)})(J^P(\tilde{X}^{(n)})-J^P)\right),\\
&d_{7,T}^{\tilde{f}}=\trace\left((\tilde{\zeta}^{(n-1)}-\mathcal{I})(\nabla\tilde{\mathbf{T}}_p^{(n)}-\nabla\tilde{\mathbf{T}}_p^{(n-1)})J^P\right),\\
&d_{8,T}^{\tilde{f}}=\trace\left((\nabla\tilde{\mathbf{T}}_p^{(n)}-\nabla\tilde{\mathbf{T}}_p^{(n-1)})J^P\right),\\
&d_{9,T}^{\tilde{f}}=\trace\left((\tilde{\zeta}^{(n-1)}-\mathcal{I})(\nabla\tilde{\mathbf{T}}_p^{(n-1)}-\nabla\tilde{\mathbf{T}}_0)(J^P(\tilde{X}^{(n)})-J^P(\tilde{X}^{(n-1)}))\right),\\
&d_{10,T}^{\tilde{f}}=\trace\left((\nabla\tilde{\mathbf{T}}_p^{(n-1)}-\nabla\tilde{\mathbf{T}}_0)(J^P(\tilde{X}^{(n)})-J^P(\tilde{X}^{(n-1)}))\right),\\[1mm]
&d_{11,T}^{\tilde{f}}=\trace\left((\tilde{\zeta}^{(n-1)}-\mathcal{I})\nabla\tilde{\mathbf{T}}_0(J^P(\tilde{X}^{(n)})-J^P(\tilde{X}^{(n-1)}))\right),\\
&d_{12,T}^{\tilde{f}}=\trace\left(\nabla\tilde{\mathbf{T}}_0(J^P(\tilde{X}^{(n)})-J^P(\tilde{X}^{(n-1)}))\right).
\end{align*}

\noindent We show the estimate of $d_{1,T}^{\tilde{f}}$ by using lemma \ref{lem3}, lemma \ref{lem5}. In addition lemma \ref{Jp-est}, lemma \ref{zeta-dif-est} and estimate \eqref{elastic-estim}. Finally we conclude with lemma \ref{lem2}

\begin{align*}
&\|d_{1,T}^{\tilde{f}}\|_{H^{\frac{s-1}{2}}_{(0)}L^2}\leq \|(\tilde{\zeta}^{(n)}-\tilde{\zeta}^{(n-1)})(J^P(\tilde{X}^{(n)})-J^P)\|_{H^{\frac{s-1}{2}}_{(0)}H^{1+\eta}}\|\nabla\tilde{\mathbf{T}}_p^{(n)}-\nabla\tilde{\mathbf{T}}_0\|_{H^{\frac{s-1}{2}}_{(0)}L^2}\\[1mm]
&\hspace{0.5cm}\leq \|\tilde{\zeta}^{(n)}-\tilde{\zeta}^{(n-1)}\|_{H^{\frac{s-1}{2}}_{(0)}H^{1+\eta}}\|J^P(\tilde{X}^{(n)})-J^P\|_{H^{\frac{s-1}{2}}_{(0)}H^{1+\eta}}\|\tilde{\mathbf{T}}_p^{(n)}-\tilde{\mathbf{T}}_0\|_{H^{\frac{s-1}{2}}_{(0)}H^1}\\[1mm]
&\hspace{0.5cm}\leq C(\tilde{v}_0,\tilde{\mathbf{T}}_0,\kappa)\left(1+\frac{1}{\We}\right)\|\tilde{X}^{(n)}-\tilde{X}^{(n-1)}\|_{H^{\frac{s-1}{2}}_{(0)}H^{2+\eta}}\\[1mm]
&\hspace{0.5cm}\leq C(\tilde{v}_0,\tilde{\mathbf{T}}_0,\kappa)\left(1+\frac{1}{\We}\right)\left\|\partial_t\int_0^t\tilde{X}^{(n)}-\tilde{X}^{(n-1)}\right\|_{H^{\frac{s-1}{2}+\delta_1-\delta_1}_{(0)}H^{2+\eta}}\\[1mm]
&\hspace{0.5cm}\leq C(\tilde{v}_0,\tilde{\mathbf{T}}_0,\kappa)\left(1+\frac{1}{\We}\right) T^{\delta_1}\|\tilde{X}^{(n)}-\tilde{X}^{(n-1)}\|_{\mathcal{F}^{s+1,\gamma}}.
\end{align*}

\noindent The estimate of $d_{3,T}^{\tilde{f}}$ is exactly as $d_{1,T}^{\tilde{f}}$, while $\|d_{i,T}^{\tilde{f}}\|_{H^{\frac{s-1}{2}}_{(0)}L^2}\leq  C(\tilde{v}_0,\tilde{\mathbf{T}}_0)T^{\delta_i}\|\tilde{X}^{(n)}-\tilde{X}^{(n-1)}\|_{\mathcal{F}^{s+1,\gamma}},$  for $i=2,4$. Furthermore, for $i=9,10$ by using lemma \ref{Jp-dif-est} and lemma \ref{zeta-est} instead of  lemma \ref{Jp-est} and lemma \ref{zeta-dif-est}, respectively we have $\|d_{i,T}^{\tilde{f}}\|_{H^{\frac{s-1}{2}}_{(0)}L^2}\leq  C(\tilde{v}_0,\tilde{\mathbf{T}}_0,\kappa)\left(1+\frac{1}{\We}\right) T^{\delta_i}\|\tilde{X}^{(n)}-\tilde{X}^{(n-1)}\|_{\mathcal{F}^{s+1,\gamma}}$ and for $i=11,12$, $\|d_{i,T}^{\tilde{f}}\|_{H^{\frac{s-1}{2}}_{(0)}L^2}\leq  C(\tilde{v}_0,\tilde{\mathbf{T}}_0)T^{\delta_i}\|\tilde{X}^{(n)}-\tilde{X}^{(n-1)}\|_{\mathcal{F}^{s+1,\gamma}}.$  For the remaining terms we focus on $d_{5,T}^{\tilde{f}}$, we use lemma \ref{lem3}, lemma \ref{lem5}. Then lemma \ref{Jp-est} and lemma \ref{zeta-est} and to conclude lemma \ref{lem2}.

\begin{align*}
&\|d_{5,T}^{\tilde{f}}\|_{H^{\frac{s-1}{2}}_{(0)}L^2}\leq \|(\tilde{\zeta}^{(n-1)}-\mathcal{I})(J^P(\tilde{X}^{(n)})-J^P)\|_{H^{\frac{s-1}{2}}_{(0)}H^{1+\eta}}\|\nabla\tilde{\mathbf{T}}_p^{(n)}-\nabla\tilde{\mathbf{T}}_p^{(n-1)}\|_{H^{\frac{s-1}{2}}_{(0)}L^2}\\[1mm]
&\leq \|\tilde{\zeta}^{(n-1)}-\mathcal{I}\|_{H^{\frac{s-1}{2}}_{(0)}H^{1+\eta}}\|J^P(\tilde{X}^{(n)})-J^P\|_{H^{\frac{s-1}{2}}_{(0)}H^{1+\eta}}\|\tilde{\mathbf{T}}_p^{(n)}-\tilde{\mathbf{T}}_p^{(n-1)}\|_{H^{\frac{s-1}{2}}_{(0)}H^1}\\[1mm]
&\leq C(\tilde{v}_0)\left\|\partial_t\int_0^t\tilde{\mathbf{T}}_p^{(n)}-\tilde{\mathbf{T}}_p^{(n-1)}\right\|_{H^{\frac{s-1}{2}+\delta_5+\delta_5}_{(0)}H^1}\leq C(\tilde{v}_0)T^{\delta_5}\|\tilde{\mathbf{T}}_p^{(n)}-\tilde{\mathbf{T}}_p^{(n-1)}\|_{\mathcal{F}^{s,\gamma-1}}.
\end{align*} 
\medskip

\noindent For $i=6,7,8$ we have $\|d_{i,T}^{\tilde{f}}\|_{H^{\frac{s-1}{2}}_{(0)}L^2}\leq C(\tilde{v}_0)T^{\delta_i}\|\tilde{\mathbf{T}}_p^{(n)}-\tilde{\mathbf{T}}_p^{(n-1)}\|_{\mathcal{F}^{s,\gamma-1}}.$ Finally  by choosing $\delta=\min\{\delta',\delta'',\delta''',\frac{3}{4},\delta_i\}$, for $i=1,\ldots,12$ we get the desired result.
\medskip

\noindent\underline{\textbf{Estimate for $\tilde{g}^{(n)}-\tilde{g}^{(n-1)}$}}\\

\noindent For the estimate of this difference, we have

\begin{align*}
&d_{1}^{\tilde{g}}=-\trace((\nabla\tilde{w}^{(n)}-\nabla\tilde{w}^{(n-1)})(\tilde{\zeta}^{(n)}-\mathcal{I})J^P(\tilde{X}^{(n)})),\\
&d_{2}^{\tilde{g}}=-\trace(\nabla\tilde{w}^{(n-1)}(\tilde{\zeta}^{(n)}-\tilde{\zeta}^{(n-1)})J^P(\tilde{X}^{(n)})),\\
&d_{3}^{\tilde{g}}=-\trace(\nabla\tilde{w}^{(n-1)}\tilde{\zeta}^{(n-1)}(J^P(\tilde{X}^{(n)})-J^P(\tilde{X}^{(n-1)}))),\\
&d_{4}^{\tilde{g}}=-\trace((\nabla\tilde{w}^{(n)}-\nabla\tilde{w}^{(n-1)})(J^P(\tilde{X}^{(n)})-J^P)),\\
&d_{5}^{\tilde{g}}=-\trace(\nabla{\phi}(\tilde{\zeta}^{(n)}-\tilde{\zeta}^{(n-1)})J^P(\tilde{X}^{(n)})),\\
&d_{6}^{\tilde{g}}=-\trace(\nabla{\phi}\tilde{\zeta}^{(n-1)}(J^P(\tilde{X}^{(n)})-J^P(\tilde{X}^{(n-1)}))).
\end{align*}

\noindent By using the estimate obtained in \cite{CCFGG2} and by observing that the only difference is due to the fact that in our $\phi$ there is the presence of $\tilde{\mathbf{T}}_0$, but it does not depends on time, we get the final result

\begin{equation}\label{tilde(g_n-g_n-1)}
\begin{split}
&\|\tilde{g}^{(n)}-\tilde{g}^{(n-1)}\|_{\mathcal{\bar{K}}^{s}_{(0)}}\leq C(\tilde{v}_0,\tilde{\mathbf{T}}_0)\left(1+\frac{1}{\We}\right)T^{\theta}\left(\|\tilde{w}^{(n)}-\tilde{w}^{(n-1)}\|_{\mathcal{K}^{s+1}_{(0)}}+\|\tilde{X}^{(n)}-\tilde{X}^{(n-1)}\|_{\mathcal{F}^{s+1,\gamma}}\right).
\end{split}
\end{equation}
\medskip

\noindent\underline{\textbf{Estimate for $\tilde{h}^{(n)}-\tilde{h}^{(n-1)}$}}\\

\noindent For the estimate of this difference we separate the terms depending on $\tilde{w}, \tilde{q}$ and $\tilde{\mathbf{T}}_p$. We notice that the only term which needs a detailed estimate is the one depending on $\tilde{\mathbf{T}}_p$, since the others have already been estimated in \cite{CCFGG2}. We resume here their results

\begin{equation}\label{their-h}
\begin{split}
&\|\tilde{h}^{(n)}_{w}-\tilde{h}^{(n-1)}_{w}\|_{\mathcal{K}^{s-\frac{1}{2}}_{(0)}}\leq C(\tilde{v}_0,\kappa)T^{\beta'}\left(\|\tilde{w}^{(n)}-\tilde{w}^{(n-1)}\|_{\mathcal{K}^{s+1}_{(0)}}+\|\tilde{X}^{(n)}-\tilde{X}^{(n-1)}\|_{\mathcal{F}^{s+1,\gamma}}\right),\\[1mm]
&\|\tilde{h}^{(n)}_{w^T}-\tilde{h}^{(n-1)}_{w^T}\|_{\mathcal{K}^{s-\frac{1}{2}}_{(0)}}\leq C(\tilde{v}_0,\kappa)T^{\beta''}\left(\|\tilde{w}^{(n)}-\tilde{w}^{(n-1)}\|_{\mathcal{K}^{s+1}_{(0)}}+\|\tilde{X}^{(n)}-\tilde{X}^{(n-1)}\|_{\mathcal{F}^{s+1,\gamma}}\right),\\[1mm]
&\|\tilde{h}^{(n)}_{\phi}-\tilde{h}^{(n-1)}_{\phi}\|_{\mathcal{K}^{s-\frac{1}{2}}_{(0)}}\leq C(\tilde{v}_0,\tilde{\mathbf{T}}_0,\re,\kappa)T^{\beta'''}\|\tilde{X}^{(n)}-\tilde{X}^{(n-1)}\|_{\mathcal{F}^{s+1,\gamma}},\\[1mm]
&\|\tilde{h}^{(n)}_{\phi^T}-\tilde{h}^{(n-1)}_{\phi^T}\|_{\mathcal{K}^{s-\frac{1}{2}}_{(0)}}\leq C(\tilde{v}_0,\tilde{\mathbf{T}}_0,\re,\kappa)T^{\beta^{iv}}\|\tilde{X}^{(n)}-\tilde{X}^{(n-1)}\|_{\mathcal{F}^{s+1,\gamma}},\\[1mm]
&\|\tilde{h}^{(n)}_{q}-\tilde{h}^{(n-1)}_{q}\|_{\mathcal{K}^{s-\frac{1}{2}}_{(0)}}\leq C(\tilde{v}_0,\kappa)T^{\beta^{v}}\left(\|\tilde{q}^{(n)}_w-\tilde{q}^{(n-1)}_w\|_{\mathcal{K}^{s}_{pr(0)}}+\|\tilde{X}^{(n)}-\tilde{X}^{(n-1)}\|_{\mathcal{F}^{s+1,\gamma}}\right).
\end{split}
\end{equation}

\begin{remark}
We notice that $\phi=\tilde{v}_0+\frac{1}{\re}t\left((1-\kappa)Q^2\Delta \tilde{v}_0-(J^P)^T\nabla \tilde{q}_{\phi}\right.$\\
$\left.+\trace(\nabla\tilde{\mathbf{T}}_0J^P)\right)$ and differently with respect to the definition of $\phi$ in \cite{CCFGG2}, it depends on $\tilde{\mathbf{T}}_0$ but it does not affect the estimates except for the constant which will depend on it.
\end{remark}

\noindent Now we show the explicit estimates for the part related to the elastic stress tensor  $\tilde{h}_T^{(n)}-\tilde{h}_T^{(n-1)}$. First we deal with the $L^2H^{s-\frac{1}{2}}-$norm and we split the difference as follows

\begin{align*}
&d_{1,T}^{\tilde{h}}=(\tilde{\mathbf{T}}_p^{(n-1)}-\tilde{\mathbf{T}}_p^{(n)})J^P(\tilde{X}^{(n-1)})^{-1}\nabla_{\Lambda}\tilde{X}^{(n-1)}\tilde{n}_0,\\
&d_{2,T}^{\tilde{h}}=(\tilde{\mathbf{T}}_p^{(n)}-\tilde{\mathbf{T}}_0)(J^P(\tilde{X}^{(n-1)})^{-1}-J^P(\tilde{X}^{(n)})^{-1})\nabla_{\Lambda}\tilde{X}^{(n-1)}\tilde{n}_0,\\
&d_{3,T}^{\tilde{h}}=\tilde{\mathbf{T}}_0(J^P(\tilde{X}^{(n-1)})^{-1}-J^P(\tilde{X}^{(n)})^{-1})\nabla_{\Lambda}\tilde{X}^{(n-1)}\tilde{n}_0,\\
&d_{4,T}^{\tilde{h}}=(\tilde{\mathbf{T}}_p^{(n)}-\tilde{\mathbf{T}}_0)J^P(\tilde{X}^{(n)})^{-1}(\nabla_{\Lambda}\tilde{X}^{(n-1)}-\nabla_{\Lambda}\tilde{X}^{(n)})\tilde{n}_0,\\
&d_{5,T}^{\tilde{h}}=\tilde{\mathbf{T}}_0J^P(\tilde{X}^{(n)})^{-1}(\nabla_{\Lambda}\tilde{X}^{(n-1)}-\nabla_{\Lambda}\tilde{X}^{(n)})\tilde{n}_0.
\end{align*}

\noindent We show the estimates of $d_{1,T}^{\tilde{h}}$ and $d_{2,T}^{\tilde{h}}$ which give an idea also for the other terms. We use the trace theorem, lemma \ref{Jp-est} or \ref{Jp-dif-est} and lemma \ref{zeta-dif-est} and in the end \eqref{flux-estim} and \eqref{elastic-estim}.

\begin{align*}
&\|d_{1,T}^{\tilde{h}}\|_{L^2H^{s-\frac{1}{2}}}\leq \|\tilde{\mathbf{T}}_p^{(n-1)}-\tilde{\mathbf{T}}_p^{(n)}\|_{L^{2}H^{s-\frac{1}{2}}}\|J^P(\tilde{X}^{(n-1)})^{-1}\|_{L^{\infty}H^{s-\frac{1}{2}}}\|\nabla_{\Lambda}\tilde{X}^{(n-1)}\|_{L^{\infty}H^{s-\frac{1}{2}}}\\[2mm]
&\hspace{0.5cm}\leq T^{\frac{1}{2}}\|\tilde{\mathbf{T}}_p^{(n-1)}-\tilde{\mathbf{T}}_p^{(n)}\|_{L^{\infty}H^{s}}\|J^P(\tilde{X}^{(n-1)})^{-1}\|_{L^{\infty}H^{s}}\|\nabla_{\Lambda}\tilde{X}^{(n-1)}\|_{L^{\infty}H^{s}}\\[2mm]
&\hspace{0.5cm}\leq C(\tilde{v}_0)T^{\frac{3}{4}} \|\tilde{\mathbf{T}}_p^{(n)}-\tilde{\mathbf{T}}_p^{(n-1)}\|_{\mathcal{F}^{s,\gamma-1}},\\[4mm]
&\|d_{2,T}^{\tilde{h}}\|_{L^2H^{s-\frac{1}{2}}}\leq \|\tilde{\mathbf{T}}_p^{(n)}-\tilde{\mathbf{T}}_0\|_{L^2H^{s-\frac{1}{2}}}\|J^P(\tilde{X}^{(n-1)})^{-1}-J^P(\tilde{X}^{(n)})^{-1}\|_{L^{\infty}H^{s-\frac{1}{2}}}\|\nabla_{\Lambda}\tilde{X}^{(n-1)}\|_{L^{\infty}H^{s-\frac{1}{2}}}\\[2mm]
&\hspace{0.5cm}\leq T^{\frac{1}{2}} \|\tilde{\mathbf{T}}_p^{(n)}-\tilde{\mathbf{T}}_0\|_{L^{\infty}H^{s}}\|J^P(\tilde{X}^{(n-1)})^{-1}-J^P(\tilde{X}^{(n)})^{-1}\|_{L^{\infty}H^{s}}\|\nabla_{\Lambda}\tilde{X}^{(n-1)}\|_{L^{\infty}H^{s}}\\[2mm]
&\hspace{0.5cm}\leq C(\tilde{v}_0, \tilde{\mathbf{T}}_0,\kappa)\left(1+\frac{1}{\We}\right)T^{\frac{3}{4}}\|\tilde{X}^{(n)}-\tilde{X}^{(n-1)}\|_{\mathcal{F}^{s,\gamma-1}}.
\end{align*}

\noindent For $d_{4,T}^{\tilde{h}}$ the estimate is exactly as $d_{2,T}^{\tilde{h}}$ but with the application of lemma \ref{zeta-dif-est} instead of lemma \ref{Jp-dif-est}, while for $i=3,5$ we have $\|d_{i,T}^{\tilde{h}}\|_{L^2H^{s-\frac{1}{2}}}\leq C(\tilde{v}_0, \tilde{\mathbf{T}}_0) T^{\frac{3}{4}}\cdot\|\tilde{X}^{(n)}-\tilde{X}^{(n-1)}\|_{\mathcal{F}^{s,\gamma-1}}.$ Now we approach the estimate of the $H^{\frac{s}{2}-\frac{1}{4}}_{(0)}L^2-$norm. We split the difference as follows

\begin{align*}
&d_{1,T}^{\tilde{h}}=(\tilde{\mathbf{T}}_p^{(n-1)}-\tilde{\mathbf{T}}_p^{(n)})(J^P(\tilde{X}^{(n-1)})^{-1}-(J^P)^{-1})(\nabla_{\Lambda}\tilde{X}^{(n-1)}-\mathcal{I})\tilde{n}_0,\\
&d_{2,T}^{\tilde{h}}=(\tilde{\mathbf{T}}_p^{(n-1)}-\tilde{\mathbf{T}}_p^{(n)})(J^P)^{-1}(\nabla_{\Lambda}\tilde{X}^{(n-1)}-\mathcal{I})\tilde{n}_0,\\
&d_{3,T}^{\tilde{h}}=(\tilde{\mathbf{T}}_p^{(n-1)}-\tilde{\mathbf{T}}_p^{(n)})(J^P(\tilde{X}^{(n-1)})^{-1}-(J^P)^{-1})\tilde{n}_0,\\
&d_{4,T}^{\tilde{h}}=(\tilde{\mathbf{T}}_p^{(n-1)}-\tilde{\mathbf{T}}_p^{(n)})(J^P)^{-1}\tilde{n}_0,\\
&d_{5,T}^{\tilde{h}}=(\tilde{\mathbf{T}}_p^{(n)}-\tilde{\mathbf{T}}_0)(J^P(\tilde{X}^{(n-1)})^{-1}-J^P(\tilde{X}^{(n)})^{-1})(\nabla_{\Lambda}\tilde{X}^{(n-1)}-\mathcal{I})\tilde{n}_0,\\
&d_{6,T}^{\tilde{h}}=\tilde{\mathbf{T}}_0(J^P(\tilde{X}^{(n-1)})^{-1}-J^P(\tilde{X}^{(n)})^{-1})(\nabla_{\Lambda}\tilde{X}^{(n-1)}-\mathcal{I})\tilde{n}_0,\\
&d_{7,T}^{\tilde{h}}=(\tilde{\mathbf{T}}_p^{(n)}-\tilde{\mathbf{T}}_0)(J^P(\tilde{X}^{(n-1)})^{-1}-J^P(\tilde{X}^{(n)})^{-1})\tilde{n}_0,\\
&d_{8,T}^{\tilde{h}}=\tilde{\mathbf{T}}_0(J^P(\tilde{X}^{(n-1)})^{-1}-J^P(\tilde{X}^{(n)})^{-1})\tilde{n}_0,\\
&d_{9,T}^{\tilde{h}}=(\tilde{\mathbf{T}}_p^{(n)}-\tilde{\mathbf{T}}_0)(J^P(\tilde{X}^{(n)})^{-1}-(J^P)^{-1})(\nabla_{\Lambda}\tilde{X}^{(n-1)}-\nabla_{\Lambda}\tilde{X}^{(n)})\tilde{n}_0,\\
&d_{10,T}^{\tilde{h}}=\tilde{\mathbf{T}}_0(J^P(\tilde{X}^{(n)})^{-1}-(J^P)^{-1})(\nabla_{\Lambda}\tilde{X}^{(n-1)}-\nabla_{\Lambda}\tilde{X}^{(n)})\tilde{n}_0,\\
&d_{11,T}^{\tilde{h}}=(\tilde{\mathbf{T}}_p^{(n)}-\tilde{\mathbf{T}}_0)(J^P)^{-1}(\nabla_{\Lambda}\tilde{X}^{(n-1)}-\nabla_{\Lambda}\tilde{X}^{(n)})\tilde{n}_0,\\
&d_{12,T}^{\tilde{h}}=\tilde{\mathbf{T}}_0(J^P)^{-1}(\nabla_{\Lambda}\tilde{X}^{(n-1)}-\nabla_{\Lambda}\tilde{X}^{(n)})\tilde{n}_0.
\end{align*}

\noindent As before, we estimate only the relevant terms. We start with $d_{1,T}^{\tilde{h}}$, by using lemma \ref{lem4}, with $\frac{1}{q}=\frac{1}{2}-\eta$, trace theorem \ref{parabolic-trace}, lemma \ref{lem5}, lemma \ref{Jp-est} and lemma \ref{zeta-est} and finally lemma \ref{lem2}. 

\begin{align*}
&\|d_{1,T}^{\tilde{h}}\|_{H^{\frac{s}{2}-\frac{1}{4}}_{(0)}L^2}\leq \|\tilde{\mathbf{T}}_p^{(n-1)}-\tilde{\mathbf{T}}_p^{(n)}\|_{H^{\frac{s}{2}-\frac{1}{4}}_{(0)}H^{\frac{1}{2}-\eta}}\|(J^P(\tilde{X}^{(n-1)})^{-1}-(J^P)^{-1})(\nabla_{\Lambda}\tilde{X}^{(n-1)}-\mathcal{I})\|_{H^{\frac{s}{2}-\frac{1}{4}}_{(0)}H^{\frac{1}{2}+\eta}}\\[2mm]
\end{align*}
\begin{align*}
&\leq \|\tilde{\mathbf{T}}_p^{(n-1)}-\tilde{\mathbf{T}}_p^{(n)}\|_{H^{\frac{s}{2}-\frac{1}{4}}_{(0)}H^{1-\eta}}\|(J^P(\tilde{X}^{(n-1)})^{-1}-(J^P)^{-1})(\nabla_{\Lambda}\tilde{X}^{(n-1)}-\mathcal{I})\|_{H^{\frac{s}{2}-\frac{1}{4}}_{(0)}H^{1+\eta}}\\[2mm]
&\leq\|\tilde{\mathbf{T}}_p^{(n-1)}-\tilde{\mathbf{T}}_p^{(n)}\|_{H^{\frac{s}{2}-\frac{1}{4}}_{(0)}H^{1-\eta}}\|J^P(\tilde{X}^{(n-1)})^{-1}-(J^P)^{-1}\|_{H^{\frac{s}{2}-\frac{1}{4}}_{(0)}H^{1+\eta}}\|\nabla_{\Lambda}\tilde{X}^{(n-1)}-\mathcal{I}\|_{H^{\frac{s}{2}-\frac{1}{4}}_{(0)}H^{1+\eta}}\\[2mm]
&\leq C(\tilde{v}_0)\left\|\partial_t \int_0^t \tilde{\mathbf{T}}_p^{(n-1)}-\tilde{\mathbf{T}}_p^{(n)}\right\|_{H^{\frac{s}{2}-\frac{1}{4}+\beta_1-\beta_1}_{(0)}H^{1-\eta}}\leq C(\tilde{v}_0) T^{\beta_1}\| \tilde{\mathbf{T}}_p^{(n)}-\tilde{\mathbf{T}}_p^{(n-1)}\|_{\mathcal{F}^{s,\gamma-1}}.
\end{align*}
\medskip

\noindent Thus $\|d_{i,T}^{\tilde{h}}\|_{H^{\frac{s}{2}-\frac{1}{4}}_{(0)}L^2}\leq C(\tilde{v}_0) T^{\beta_i}\| \tilde{\mathbf{T}}_p^{(n)}-\tilde{\mathbf{T}}_p^{(n-1)}\|_{\mathcal{F}^{s,\gamma-1}},$ for $i=2,3,4$. Furthermore in order to get the other estimates we consider $d_{5,T}^{\tilde{h}}$ and for this term we use lemma \ref{lem4}, with $\frac{1}{q}=\frac{1}{2}-\eta$, trace theorem and lemma \ref{lem5} and to finish lemma \ref{Jp-dif-est}, lemma \ref{zeta-est} and lemma \ref{lem2}.

\begin{align*}
&\|d_{5,T}^{\tilde{h}}\|_{H^{\frac{s}{2}-\frac{1}{4}}_{(0)}L^2}\\[1mm]
&\hspace{0.5cm}\leq\|\tilde{\mathbf{T}}_p^{(n)}-\tilde{\mathbf{T}}_0\|_{H^{\frac{s}{2}-\frac{1}{4}}_{(0)}H^{\frac{1}{2}-\eta}}\cdot\|(J^P(\tilde{X}^{(n-1)})^{-1}-J^P(\tilde{X}^{(n)})^{-1})(\nabla_{\Lambda}\tilde{X}^{(n-1)}-\mathcal{I})\|_{H^{\frac{s}{2}-\frac{1}{4}}_{(0)}H^{\frac{1}{2}+\eta}}\\[2mm]
&\hspace{0.5cm}\leq \|\tilde{\mathbf{T}}_p^{(n)}-\tilde{\mathbf{T}}_0\|_{H^{\frac{s}{2}-\frac{1}{4}}_{(0)}H^{1-\eta}}|(J^P(\tilde{X}^{(n-1)})^{-1}-J^P(\tilde{X}^{(n)})^{-1})(\nabla_{\Lambda}\tilde{X}^{(n-1)}-\mathcal{I})\|_{H^{\frac{s}{2}-\frac{1}{4}}_{(0)}H^{1+\eta}}\\[2mm]
&\hspace{0.5cm}\leq \|\tilde{\mathbf{T}}_p^{(n)}-\tilde{\mathbf{T}}_0\|_{H^{\frac{s}{2}-\frac{1}{4}}_{(0)}H^{1-\eta}}\|J^P(\tilde{X}^{(n-1)})^{-1}-J^P(\tilde{X}^{(n)})^{-1}\|_{H^{\frac{s}{2}-\frac{1}{4}}_{(0)}H^{1+\eta}}\|\nabla_{\Lambda}\tilde{X}^{(n-1)}-\mathcal{I}\|_{H^{\frac{s}{2}-\frac{1}{4}}_{(0)}H^{1+\eta}}\\[2mm]
&\hspace{0.5cm}\leq C(\tilde{v}_0,\tilde{\mathbf{T}}_0,\kappa)\left(1+\frac{1}{\We}\right) \left\|\partial_t\int_0^t \tilde{X}^{(n)}-\tilde{X}^{(n-1)}\right\|_{H^{\frac{s}{2}-\frac{1}{4}+\beta_5-\beta_5}_{(0)}H^{1+\eta}}\\[2mm]
&\hspace{0.5cm}\leq C(\tilde{v}_0,\tilde{\mathbf{T}}_0,\kappa) \left(1+\frac{1}{\We}\right)T^{\beta_5}\|\tilde{X}^{(n)}-\tilde{X}^{(n-1)}\|_{\mathcal{F}^{s+1,\gamma}}.
\end{align*}
\medskip

\noindent Finally, $\|d_{i,T}^{\tilde{h}}\|_{H^{\frac{s}{2}-\frac{1}{4}}_{(0)}L^2}\leq C(\tilde{v}_0,\tilde{\mathbf{T}}_0) T^{\beta_i}\|\tilde{X}^{(n)}-\tilde{X}^{(n-1)}\|_{\mathcal{F}^{s+1,\gamma}}$, for $i=6,8,10,12$ and \\
$\|d_{i,T}^{\tilde{h}}\|_{H^{\frac{s}{2}-\frac{1}{4}}_{(0)}L^2}\leq C(\tilde{v}_0,\tilde{\mathbf{T}}_0,\kappa) \left(1+\frac{1}{\We}\right)T^{\beta_i}\|\tilde{X}^{(n)}-\tilde{X}^{(n-1)}\|_{\mathcal{F}^{s+1,\gamma}},$ for $i=7,9,11$.  By choosing 
$\beta=\min\{\beta',\beta'',\beta'',\beta ^{iv}, \beta^v,\frac{3}{4},\beta_i\}$, for $i=1,\ldots,12$ also the second part of the proposition holds.
\end{proof}
\medskip

\subsection{Estimate for the Conformal Lagrangian flux}

\noindent The equation for the conformal lagrangian flux is given by  \eqref{Conf-flux}. In order to prove theorem \ref{local-existence conf-lag} we need iterative bounds also for the flux. Here we state the proposition, without the proof, since it is exactly the same as in \cite[Proposition 5.3]{CCFGG2}.

\begin{proposition}\label{estimate-conf-lag-flux}
For $2<s<\frac{5}{2}$ and $T>0$ small enough depending on $N$, the radius of the balls and $ \tilde{v}_0$, we have

\begin{enumerate}
\item Let $\tilde{X}^{(n)}-\hat{X}\in \mathcal{F}^{s+1,\gamma}$, $\tilde{w}^{(n)}\in\mathcal{K}^{s+1}_{(0)}$ and $\tilde{q}_w^{(n)}\in\mathcal{K}^s_{pr(0)}$ and such that

\begin{description}{}

\item [(a)] $\tilde{X}^{(n)}-\hat{X}\in B_1,$\\

\item[(b)] $(\tilde{w}^{(n)},\tilde{q}_w^{(n)})\in B_2.$
\end{description}
\medskip

\noindent Then, $\tilde{X}^{(n+1)}-\hat{X}\in B_1$.
\bigskip

\item Let $\tilde{X}^{(n)}-\tilde{\alpha}, \tilde{X}^{(n-1)}-\tilde{\alpha}\in B_1$ and $(\tilde{w}^{(n)}, \tilde{q}^{(n)}), (\tilde{w}^{(n-1)}, \tilde{q}^{(n-1)})\in B_2.$ Then

\begin{align*}
&\left\|\tilde{X}^{(n+1)}-\tilde{X}^{(n)}\right\|_{\mathcal{F}^{s+1,\gamma}}\leq C(\tilde{v}_0)T^{\eta}\left(\left\|\tilde{w}^{(n)}-\tilde{w}^{(n-1)}\right\|_{\mathcal{K}^{s+1}_{(0)}}+\left\|\tilde{X}^{(n)}-\tilde{X}^{(n-1)}\right\|_{\mathcal{F}^{s+1,\gamma}}\right).
\end{align*}
\end{enumerate}
\end{proposition}

\subsection{Estimate for the Conformal Lagrangian elastic stress tensor}
\noindent The elastic stress tensor $\tilde{\mathbf{T}}_p$ satisfies the equation \eqref{We-G-conf-iterative}. So we have an explicit formula for this term

\begin{align*}
&\tilde{\mathbf{T}}_p^{(n+1)}=\tilde{\mathbf{T}}_0+\int_0^t \left(J^P(\tilde{X}^{(n)})\tilde{\zeta}^{(n)}\nabla\tilde{v}^{(n)} \tilde{\mathbf{T}}_p^{(n)}\right)+\int_0^t \tilde{\mathbf{T}}_p^{(n)}\left(J^P(\tilde{X}^{(n)}\tilde{\zeta}^{(n)}\nabla\tilde{v}^{(n)})\right)^T-\frac{1}{\We}\int_0^t\tilde{\mathbf{T}}_p^{(n)}\\
&\hspace{0.5cm}+\frac{\kappa}{\We}\int_0^t\left(J^P(\tilde{X}^{(n)})\tilde{\zeta}^{(n)}\nabla\tilde{v}^{(n)}+\left(J^P(\tilde{X}^{(n)}\tilde{\zeta}^{(n)}\nabla \tilde{v}^{(n)})\right)^T\right)
\end{align*}

\noindent The following proposition gives the estimate for the conformal lagrangian elastic tensor.

 \begin{proposition}\label{We-G-conf-lag-estimate}
 For $2<s<\frac{5}{2}$ and $T>0$ small enough, depending on  $N$, the radius of the balls, on $\tilde{v}_0$ and on $\tilde{\mathbf{T}}_0$, we have
 
 \begin{enumerate}
 
\item Let $\tilde{X}^{(n)}-\hat{X}\in\mathcal{F}^{s+1,\gamma}, \tilde{w}^{(n)}\in\mathcal{K}^{s+1}_{(0)}$ and $\tilde{\mathbf{T}}_p^{(n)}-\tilde{\mathbf{T}}_0-t\hat{\mathbf{T}}\in\mathcal{F}^{s,\gamma-1}$ such that

\begin{description}
 \item[(a)]  $\tilde{X}^{(n)}-\hat{X}\in B_1$,\\
 
 \item[(b)] $(\tilde{w}^{(n)},\tilde{q}_w^{(n-1)})\in B_2$,\\
 
 \item[(c)] $\tilde{\mathbf{T}}_p^{(n)}-\tilde{\mathbf{T}}_0-t\hat{\mathbf{T}}\in B_3$.\\

\noindent Then $\tilde{\mathbf{T}}_p^{(n+1)}-\tilde{\mathbf{T}}_0-t\hat{\mathbf{T}}\in B_3$
\end{description}
\medskip

\item Let   $\tilde{X}^{(n)}-\tilde{\alpha}, \tilde{X}^{(n-1)}-\tilde{\alpha}\in B_1$, $(\tilde{w}^{(n)},\tilde{q}_w), (\tilde{w}^{(n-1)},\tilde{q}_w^{(n-1)})\in B_2$ and $\tilde{\mathbf{T}}_p^{(n)}-\tilde{\mathbf{T}}_0, \tilde{\mathbf{T}}_p^{(n-1)}-\tilde{\mathbf{T}}_0\in B_3$. Then for a suitable $\delta>0$ we have

\begin{equation}\label{G-conf-lag-part2}
\begin{split}
&\left\|\tilde{\mathbf{T}}_p^{(n+1)}-\tilde{\mathbf{T}}_p^{(n)}\right\|_{\mathcal{F}^{s,\gamma-1}}\leq C(N, \tilde{v}_0, \tilde{\mathbf{T}}_0,\kappa)\left(1+\frac{1}{\We}\right)T^{\delta}\\[1mm]
&\cdot\left( \left\|\tilde{\mathbf{T}}_p^{(n)}-\tilde{\mathbf{T}}_p^{(n-1)}\right\|_{\mathcal{F}^{s,\gamma-1}}+\left\|\tilde{w}^{(n)}-\tilde{w}^{(n-1)}\right\|_{\mathcal{K}^{s+1}_{(0)}}+\left\|\tilde{X}^{(n)}-\tilde{X}^{(n-1)}\right\|_{\mathcal{ F}^{s+1,\gamma}}\right) 
\end{split}
\end{equation}
\end{enumerate}
\end{proposition}

\begin{proof}
\noindent Let us recall the definition of the ball $B_3$.
$$B_3:=\left\lbrace \tilde{\mathbf{T}}_p-\tilde{\mathbf{T}}_0-t\hat{\mathbf{T}}\in\mathcal{F}^{s,\gamma-1} : \left\|\tilde{\mathbf{T}}_p-\tilde{\mathbf{T}}_0-\int_0^t \hat{\mathbf{T}}_{\phi}\,d\tau\right\|_{\mathcal{F}^{s,\gamma-1}} \leq N\right\rbrace,$$

\noindent where 
\begin{align*}
&\hat{\mathbf{T}}=J^P\nabla\tilde{v}_0\mathbf{\tilde{T}}_0+\mathbf{\tilde{T}}_0(J^P\nabla\tilde{v}_0)^T-\frac{1}{\We}\mathbf{\tilde{T}}_0+\frac{\kappa}{\We}\left(J^P\nabla\tilde{v}_0+(J^P\nabla\tilde{v}_0)^T\right)\\
&\hat{\mathbf{T}}_{\phi}=J^P\nabla\phi\tilde{\mathbf{T}}_0+\tilde{\mathbf{T}}_0(J^P\nabla\phi)^T-\frac{1}{\We}\tilde{\mathbf{T}}_0+\frac{\kappa}{\We}\left(J^P\nabla\phi+(J^P\nabla\phi)^T\right).
\end{align*}
\medskip

\noindent We have to estimate $\tilde{\mathbf{T}}^{(n+1)}_p-\tilde{\mathbf{T}}_0-\int_0^t \hat{\mathbf{T}}_{\phi}$ in both $L^{\infty}_{\frac{1}{4}}H^s$ and $H^2_{(0)}H^{\gamma-1}$. At first we rewrite the term as follows

\begin{equation}\label{split-T}
\begin{split}
&\tilde{\mathbf{T}}^{(n+1)}_p-\tilde{\mathbf{T}}_0-\int_0^t \hat{\mathbf{T}}_{\phi}=\int_0^t J^P(\tilde{X}^{(n)})\tilde{\zeta}^{(n)}\nabla\tilde{w}^{(n)}\tilde{\mathbf{T}}_p^{(n)}\\
&\hspace{0.5cm}+\int_0^t J^P(\tilde{X}^{(n)})\tilde{\zeta}^{(n)}\nabla\phi\tilde{\mathbf{T}}_p^{(n)}-J^P\nabla\phi\tilde{\mathbf{T}}_0+\int_0^t \tilde{\mathbf{T}}_p^{(n)}\left(J^P(\tilde{X}^{(n)})\tilde{\zeta}^{(n)}\nabla\tilde{w}^{(n)}\right)^T\\
&\hspace{0.5cm}+\int_0^t  \tilde{\mathbf{T}}_p^{(n)}\left(J^P(\tilde{X}^{(n)})\tilde{\zeta}^{(n)}\nabla\phi\right)^T-\tilde{\mathbf{T}}_0(J^P\nabla\phi)^T+\frac{1}{\We}\int_0^t \tilde{\mathbf{T}}_0-\tilde{\mathbf{T}}_p^{(n)}\\
&\hspace{0.5cm}+\frac{\kappa}{\We}\int_0^t J^P(\tilde{X}^{(n)})\tilde{\zeta}^{(n)}\nabla\tilde{w}^{(n)}+\frac{\kappa}{\We}\int_0^t J^P(\tilde{X}^{(n)})\tilde{\zeta}^{(n)}\nabla\phi-J^P\nabla\phi\\
&\hspace{0.5cm}+\frac{\kappa}{\We}\int_0^t \left(J^P(\tilde{X}^{(n)})\tilde{\zeta}^{(n)}\nabla\tilde{w}^{(n)}\right)^T+\frac{\kappa}{\We}\int_0^t \left(J^P(\tilde{X}^{(n)})\tilde{\zeta}^{(n)}\nabla\phi\right)^T-(J^P\nabla\phi)^T=\sum_{i=1}^9 I_i.
\end{split}
\end{equation}

\noindent The estimates in $L^{\infty}_{\frac{1}{4}}H^s$ are based on the use of lemma \ref{Jp-est}, lemma \ref{zeta-est} and estimates \eqref{flux-estim}, \eqref{elastic-estim}.  

\begin{align*}
&\|I_1\|_{L^{\infty}_{\frac{1}{4}}H^s}\leq\sup_{t\in[0,T]}t^{-\frac{1}{4}}\int_0^t\| J^P(\tilde{X}^{(n)})\tilde{\zeta}^{(n)}\nabla\tilde{w}^{(n)}(\tilde{\mathbf{T}}_p^{(n)}-\tilde{\mathbf{T}}_0)+J^P(\tilde{X}^{(n)})\tilde{\zeta}^{(n)}\nabla\tilde{w}^{(n)}\tilde{\mathbf{T}}_0\|_{H^s}\\[1mm]
&\leq T^{\frac{1}{4}}\left(\|J^P(\tilde{X}^{(n)})\tilde{\zeta}^{(n)}\nabla\tilde{w}^{(n)}(\tilde{\mathbf{T}}_p^{(n)}-\tilde{\mathbf{T}}_0)\|_{L^2H^s}+\|J^P(\tilde{X}^{(n)})\tilde{\zeta}^{(n)}\nabla\tilde{w}^{(n)}\tilde{\mathbf{T}}_0\|_{L^2H^s}\right)\\[1mm]
&\leq T^{\frac{1}{4}}\left(\|J^P(\tilde{X}^{(n)})\|_{L^{\infty}H^s}\|\tilde{\zeta}^{(n)}\|_{L^{\infty}H^s}\|\nabla\tilde{w}^{(n)}\|_{L^2H^s}\|\tilde{\mathbf{T}}_p^{(n)}-\tilde{\mathbf{T}}_0\|_{L^{\infty}H^s}\right.\\[1mm]
&\hspace{0.5cm}\left.+\|J^P(\tilde{X}^{(n)})\|_{L^{\infty}H^s}\|\tilde{\zeta}^{(n)}\|_{L^{\infty}H^s}\|\nabla\tilde{w}^{(n)}\|_{L^2H^s}\|\tilde{\mathbf{T}}_0\|_{H^s}\right)\\[2mm]
&\leq C(\tilde{v}_0,\tilde{\mathbf{T}}_0,\kappa)\left(1+\frac{1}{\We}\right)T^{\frac{1}{4}}.
\end{align*}
\medskip

\noindent For $I_2$ we use the definition of $\phi$ and we have 

\begin{align*}
&I_{2,1}=\int_0^t J^P(\tilde{X}^{(n)})\tilde{\zeta}^{(n)}\nabla\tilde{v}_0\tilde{\mathbf{T}}_p^{(n)}-J^P\nabla\tilde{v}_0\tilde{\mathbf{T}}_0=\int_0^t (J^P(\tilde{X}^{(n)})-J^P)\tilde{\zeta}^{(n)}\nabla\tilde{v}_0\tilde{\mathbf{T}}_p^{(n)}\\[1mm]
&\hspace{0.5cm}+\int_0^t J^P(\tilde{\zeta}^{(n)}-\mathcal{I})\nabla\tilde{v}_0\tilde{\mathbf{T}}_p^{(n)}+\int_0^t J^P\nabla\tilde{v}_0(\tilde{\mathbf{T}}_p^{(n)}-\tilde{\mathbf{T}}_0),\\[2mm]
&I_{2,2}=\int_0^t J^P(\tilde{X}^{(n)})\tilde{\zeta}^{(n)}t\nabla\hat{\phi}\tilde{\mathbf{T}}_p^{(n)}-J^P t\nabla\hat{\phi}\tilde{\mathbf{T}}_0=\int_0^t (J^P(\tilde{X}^{(n)})-J^P)\tilde{\zeta}^{(n)}t\nabla\hat{\phi}\tilde{\mathbf{T}}_p^{(n)}\\[1mm]
&\hspace{0.5cm}+\int_0^t J^P(\tilde{\zeta}^{(n)}-\mathcal{I})t\nabla\hat{\phi}\tilde{\mathbf{T}}_p^{(n)}+\int_0^t J^P t\nabla\hat{\phi}(\tilde{\mathbf{T}}_p^{(n)}-\tilde{\mathbf{T}}_0).
\end{align*}
\medskip

\noindent For the estimates we use the same lemmas as for $I_1$. First of all we deal separately which each term of $I_{2,1}$ and $I_{2,2}$.

\begin{align*}
&\left\|\int_0^t (J^P(\tilde{X}^{(n)})-J^P)\tilde{\zeta}^{(n)}\nabla\tilde{v}_0\tilde{\mathbf{T}}_p^{(n)}\right\|_{L^{\infty}_{\frac{1}{4}}H^s}\\[1mm]
&\hspace{1cm}\leq T^{\frac{1}{4}}\|(J^P(\tilde{X}^{(n)})-J^P)\tilde{\zeta}^{(n)}\nabla\tilde{v}_0(\tilde{\mathbf{T}}_p^{(n)}-\tilde{\mathbf{T}}_0)\|_{L^2H^s}+T^{\frac{1}{4}}\|(J^P(\tilde{X}^{(n)})-J^P)\tilde{\zeta}^{(n)}\nabla\tilde{v}_0\tilde{\mathbf{T}}_0\|_{L^2H^s}\\[2mm]
&\hspace{1cm}\leq T^{\frac{3}{4}}\|J^P(\tilde{X}^{(n)})-J^P\|_{L^{\infty}H^s}\|\tilde{\zeta}^{(n)}\|_{L^{\infty}H^s}\|\nabla\tilde{v}_0\|_{H^s}\|\tilde{\mathbf{T}}_p^{(n)}-\tilde{\mathbf{T}}_0\|_{L^{\infty}H^s}\\[2mm]
&\hspace{1cm}+T^{\frac{3}{4}}\|J^P(\tilde{X}^{(n)})-J^P\|_{L^{\infty}H^s}\|\tilde{\zeta}^{(n)}\|_{L^{\infty}H^s}\|\nabla\tilde{v}_0\|_{H^s}\|\tilde{\mathbf{T}}_0\|_{H^s}\\[2mm]
&\hspace{1cm}\leq C(\tilde{v}_0,\tilde{\mathbf{T}}_0,\kappa)\left(1+\frac{1}{\We}\right)T^{\frac{3}{4}},\\[4mm]
&\left\|\int_0^t J^P(\tilde{\zeta}^{(n)}-\mathcal{I})\nabla\tilde{v}_0\tilde{\mathbf{T}}_p^{(n)}\right\|_{L^{\infty}_{\frac{1}{4}}H^s}\\[2mm]
&\hspace{1cm}\leq T^{\frac{1}{4}}\|J^P(\tilde{\zeta}^{(n)}-\mathcal{I})\nabla\tilde{v}_0(\tilde{\mathbf{T}}_p^{(n)}-\tilde{\mathbf{T}}_0)\|_{L^2H^s}+T^{\frac{1}{4}}\|J^P(\tilde{\zeta}^{(n)}-\mathcal{I})\nabla\tilde{v}_0\tilde{\mathbf{T}}_0\|_{L^2H^s}\\[2mm]
&\hspace{1cm}\leq C T^{\frac{3}{4}}\|\tilde{\zeta}^{(n)}-\mathcal{I}\|_{L^{\infty}H^s}\|\nabla\tilde{v}_0\|_{H^s}\|\tilde{\mathbf{T}}_p^{(n)}-\tilde{\mathbf{T}}_0\|_{L^{\infty}H^s}+C T^{\frac{3}{4}}\|\tilde{\zeta}^{(n)}-\mathcal{I}\|_{L^{\infty}H^s}\|\nabla\tilde{v}_0\|_{H^s}\|\tilde{\mathbf{T}}_0\|_{H^s}\\[2mm]
&\hspace{1cm}\leq C(\tilde{v}_0,\tilde{\mathbf{T}}_0,\kappa)\left(1+\frac{1}{\We}\right)T^{\frac{3}{4}},\\[4mm]
&\left\|\int_0^t J^P\nabla\tilde{v}_0(\tilde{\mathbf{T}}_p^{(n)}-\tilde{\mathbf{T}}_0)\right\|_{L^{\infty}_{\frac{1}{4}}H^s}\leq C T^{\frac{3}{4}}\|\nabla\tilde{v}_0\|_{H^s}\|\tilde{\mathbf{T}}_p^{(n)}-\tilde{\mathbf{T}}_0\|_{L^{\infty}H^s}\leq C(\tilde{v}_0,\tilde{\mathbf{T}}_0,\kappa)\left(1+\frac{1}{\We}\right)T^{\frac{3}{4}},\\[4mm]
&\left\|\int_0^t (J^P(\tilde{X}^{(n)})-J^P)\tilde{\zeta}^{(n)}t\nabla\hat{\phi}\tilde{\mathbf{T}}_p^{(n)}\right\|_{L^{\infty}_{\frac{1}{4}}H^s}\\[1mm]
&\hspace{1cm}\leq  T^{\frac{1}{4}} \|(J^P(\tilde{X}^{(n)})-J^P)\tilde{\zeta}^{(n)}t\nabla\hat{\phi}(\tilde{\mathbf{T}}_p^{(n)}-\tilde{\mathbf{T}}_0)\|_{L^2H^s}+\|J^P(\tilde{X}^{(n)})-J^P)\tilde{\zeta}^{(n)}t\nabla\hat{\phi}\tilde{\mathbf{T}}_0\|_{L^2H^s}\\[2mm]
&\hspace{0.5cm}\leq T^{\frac{1}{4}}\|J^P(\tilde{X}^{(n)})-J^P\|_{L^{\infty}H^s}\|\tilde{\zeta}^{(n)}\|_{L^{\infty}H^s}\|t\nabla\hat{\phi}\|_{L^2H^s}\|\tilde{\mathbf{T}}_p^{(n)}-\tilde{\mathbf{T}}_0\|_{L^{\infty}H^s}\\[2mm]
&\hspace{1cm}+T^{\frac{1}{4}}\|J^P(\tilde{X}^{(n)})-J^P\|_{L^{\infty}H^s}\|\tilde{\zeta}^{(n)}\|_{L^{\infty}H^s}\|t\nabla\hat{\phi}\|_{L^2H^s}\|\tilde{\mathbf{T}}_0\|_{H^s}\\[2mm]
&\hspace{1cm}\leq C(\tilde{v}_0,\tilde{\mathbf{T}}_0,\kappa)\left(1+\frac{1}{\We}\right)T^{\frac{7}{4}},\\[2mm]
&\left\|\int_0^t J^P(\tilde{\zeta}^{(n)}-\mathcal{I})t\nabla\hat{\phi}\tilde{\mathbf{T}}_p^{(n)}\right\|_{L^{\infty}_{\frac{1}{4}}H^s}\\[1mm]
&\hspace{1cm}\leq T^{\frac{1}{4}} \|J^P(\tilde{\zeta}^{(n)}-\mathcal{I})t\nabla\hat{\phi}(\tilde{\mathbf{T}}_p^{(n)}-\tilde{\mathbf{T}}_0)\|_{L^{\infty}H^s}+T^{\frac{1}{4}} \|J^P(\tilde{\zeta}^{(n)}-\mathcal{I})t\nabla\hat{\phi}\tilde{\mathbf{T}}_0\|_{L^{\infty}H^s}\\[2mm]
&\hspace{1cm}\leq C T^{\frac{1}{4}} \|\tilde{\zeta}^{(n)}-\mathcal{I}\|_{L^{\infty}H^s}\|t\nabla\hat{\phi}\|_{L^2H^s}\|\tilde{\mathbf{T}}_p^{(n)}-\tilde{\mathbf{T}}_0\|_{L^{\infty}H^s}\\[2mm]
&\hspace{1cm}+C T^{\frac{1}{4}} \|\tilde{\zeta}^{(n)}-\mathcal{I}\|_{L^{\infty}H^s}\|t\nabla\hat{\phi}\|_{L^2H^s}\|\tilde{\mathbf{T}}_0\|_{H^s}\\[2mm]
&\hspace{1cm}\leq C(\tilde{v}_0,\tilde{\mathbf{T}}_0,\kappa)\left(1+\frac{1}{\We}\right)T^{\frac{7}{4}},\\[4mm]
\end{align*}
\begin{align*}
&\left\|\int_0^t J^P t\nabla\hat{\phi}(\tilde{\mathbf{T}}_p^{(n)}-\tilde{\mathbf{T}}_0)\right\|_{L^{\infty}_{\frac{1}{4}}H^s} \leq C T^{\frac{1}{4}} \|t\nabla\hat{\phi}(\tilde{\mathbf{T}}_p^{(n)}-\tilde{\mathbf{T}}_0)\|_{L^2H^s}\\[2mm]
&\hspace{1cm}\leq C T^{\frac{1}{4}}\|t\nabla\hat{\phi}\|_{L^2H^s}\|\tilde{\mathbf{T}}_p^{(n)}-\tilde{\mathbf{T}}_0\|_{L^{\infty}H^s}\leq  C(\tilde{v}_0,\tilde{\mathbf{T}}_0,\kappa)\left(1+\frac{1}{\We}\right)T^{\frac{7}{4}}.
\end{align*}

\noindent The estimates of $I_3, I_4$ are the same as $I_1, I_2$. We show the estimate of $I_5$, by using definition of $L^{\infty}_{\frac{1}{4}}$-norm and estiamate \eqref{elastic-estim}.

\begin{align*}
\|I_5\|_{L^{\infty}_{\frac{1}{4}}H^s}\leq \frac{T^{\frac{1}{4}}}{\We}\|\tilde{\mathbf{T}}_0-\tilde{\mathbf{T}}_p^{(n)}\|_{L^2H^s}\leq C(\tilde{v}_0,\tilde{\mathbf{T}}_0,\kappa)\left(1+\frac{1}{\We}\right)T.
\end{align*}

\noindent The estimates of $I_6, I_7, I_8$ and $I_9$ can be obtained in the same way as the integrals before so we pass to the estimate of the $H^2_{(0)}H^{\gamma-1}$-norm. We start by rewriting the integrals $I_1,\ldots,I_9$ in \eqref{split-T} in a more convenient way. 

\begin{align*}
I_1&=\int_0^t (J^P(\tilde{X}^{(n)})-J^P)(\tilde{\zeta}^{(n)}-\mathcal{I})\nabla\tilde{w}^{(n)}(\tilde{\mathbf{T}}_p^{(n)}-\tilde{\mathbf{T}}_0)+\int_0^t J^P(\tilde{\zeta}^{(n)}-\mathcal{I})\nabla\tilde{w}^{(n)}(\tilde{\mathbf{T}}_p^{(n)}-\tilde{\mathbf{T}}_0)\\[1mm]
&+\int_0^t (J^P(\tilde{X}^{(n)})-J^P)(\tilde{\zeta}^{(n)}-\mathcal{I})\nabla\tilde{w}^{(n)}\tilde{\mathbf{T}}_0+\int_0^t J^P(\tilde{\zeta}^{(n)}-\mathcal{I})\nabla\tilde{w}^{(n)}\tilde{\mathbf{T}}_0\\[1mm]
&+\int_0^t (J^P(\tilde{X}^{(n)})-J^P)\nabla\tilde{w}^{(n)}(\tilde{\mathbf{T}}_p^{(n)}-\tilde{\mathbf{T}}_0)+\int_0^t J^P\nabla\tilde{w}^{(n)}\tilde{\mathbf{T}}_0=\sum_{i=1}^6  I_{1,i}.
\end{align*}

\noindent We show the estimate of $I_{1,1}$, for the others we have just to apply less inequalities. In order to deal with this estimate we need to use lemma \ref{lem2} with $\varepsilon=0$ then we can use lemma \ref{lem3} with $\gamma>1$, lemma \ref{lem5}, lemma \ref{lem1} and lemma \ref{Jp-est}, lemma \ref{zeta-est}  and to conclude the estimates for the flux and for the elastic stress tensor \eqref{flux-estim} and \eqref{elastic-estim}.

\begin{align*}
&\|I_{1,1}\|_{H^2_{(0)}H^{\gamma-1}}\leq \|(J^P(\tilde{X}^{(n)})-J^P)(\tilde{\zeta}^{(n)}-\mathcal{I})\nabla\tilde{w}^{(n)}(\tilde{\mathbf{T}}_p^{(n)}-\tilde{\mathbf{T}}_0)\|_{H^1_{(0)}H^{\gamma-1}}\\[2mm]
&\hspace{0.5cm}\leq \|J^P(\tilde{X}^{(n)})-J^P\|_{H^1_{(0)}H^{\gamma}}\|(\tilde{\zeta}^{(n)}-\mathcal{I})\nabla\tilde{w}^{(n)}(\tilde{\mathbf{T}}_p^{(n)}-\tilde{\mathbf{T}}_0)\|_{H^1_{(0)}H^{\gamma-1}}\\[2mm]
&\hspace{0.5cm}\leq \|\tilde{X}^{(n)}-\tilde{\alpha}\|_{H^1_{(0)}H^{\gamma}}\|\tilde{\zeta}^{(n)}-\mathcal{I}\|_{H^1_{(0)}H^{\gamma-1}}\|\tilde{w}^{(n)}\|_{H^1_{(0)}H^{\gamma}}\|\tilde{\mathbf{T}}_p^{(n)}-\tilde{\mathbf{T}}_0\|_{H^1_{(0)}H^{\gamma-1}}\\[2mm]
&\hspace{0.5cm}\leq C(\tilde{v}_0,\tilde{\mathbf{T}}_0, \kappa)\left(1+\frac{1}{\We}\right) T^{\delta_1}.
\end{align*}
\medskip

\noindent   The estimates of  $I_{1,2}, I_{1,5}$ are exactly as $I_{1,1}$, on the other hand the remaining terms have a slight difference that concerns the absence of $\tilde{\mathbf{T}}_p^{(n)}-\tilde{\mathbf{T}}_0$ and then the absence of the Weissenberg number in the estimates, as we show below

\begin{align*}
&\|I_{1,i}\|_{H^2_{(0)}H^{\gamma-1}}\leq C(\tilde{v}_0,\tilde{\mathbf{T}}_0, \kappa)\left(1+\frac{1}{\We}\right) T^{\delta_i},\hspace{0.2cm}\textrm{for}\hspace{0.3cm} i=2,5\\[2mm]
&\|I_{1,i}\|_{H^2_{(0)}H^{\gamma-1}}\leq C(\tilde{v}_0,\tilde{\mathbf{T}}_0)T^{\delta_i},\hspace{0.2cm}\textrm{for}\hspace{0.3cm} i=3,4,6.
\end{align*}
\medskip

\noindent In the same way we have to split $I_{2}$

\begin{align*}
I_2&=\int_0^t (J^P(\tilde{X}^{(n)})-J^P)(\tilde{\zeta}^{(n)}-\mathcal{I})\nabla\tilde{v}_0(\tilde{\mathbf{T}}_p^{(n)}-\tilde{\mathbf{T}}_0)+\int_0^t J^P(\tilde{\zeta}^{(n)}-\mathcal{I})\nabla\tilde{v}_0(\tilde{\mathbf{T}}_p^{(n)}-\tilde{\mathbf{T}}_0)\\[1mm]
&+\int_0^t (J^P(\tilde{X}^{(n)})-J^P)(\tilde{\zeta}^{(n)}-\mathcal{I})\nabla\tilde{v}_0\tilde{\mathbf{T}}_0+\int_0^t J^P(\tilde{\zeta}^{(n)}-\mathcal{I})\nabla\tilde{v}_0\tilde{\mathbf{T}}_0\\[1mm]
&+\int_0^t (J^P(\tilde{X}^{(n)})-J^P)\nabla\tilde{v}_0\tilde{\mathbf{T}}_0+\int_0^t (J^P(\tilde{X}^{(n)})-J^P)(\tilde{\zeta}^{(n)}-\mathcal{I})t\nabla\hat{\phi}(\tilde{\mathbf{T}}_p^{(n)}-\tilde{\mathbf{T}}_0)\\[1mm]
&+\int_0^t J^P(\tilde{\zeta}^{(n)}-\mathcal{I})t\nabla\hat{\phi}(\tilde{\mathbf{T}}_p^{(n)}-\tilde{\mathbf{T}}_0)+\int_0^t (J^P(\tilde{X}^{(n)})-J^P)(\tilde{\zeta}^{(n)}-\mathcal{I})t\nabla\hat{\phi}\tilde{\mathbf{T}}_0\\[1mm]
&+\int_0^t J^P(\tilde{\zeta}^{(n)}-\mathcal{I})t\nabla\hat{\phi}\tilde{\mathbf{T}}_0+\int_0^t (J^P(\tilde{X}^{(n)})-J^P)t\nabla\hat{\phi}\tilde{\mathbf{T}}_0=\sum_{i=1}^{10} I_{2,i}
\end{align*}
\medskip

\noindent For the estimates of these terms we use the lemmas already mentioned for $I_1$ and we will show the estimates of $I_{2,3}$ and $I_{2,6}$, by assuming $\tilde{v}_0$ and $\tilde{\mathbf{T}}_0$ regular enough.

\begin{align*}
&\|I_{2,3}\|_{H^{2}_{(0)}H^{\gamma-1}}=\left\|\int_0^t  (J^P(\tilde{X}^{(n)})-J^P)(\tilde{\zeta}^{(n)}-\mathcal{I})\nabla\tilde{v}_0\tilde{\mathbf{T}}_0\right\|_{H^{2}_{(0)}H^{\gamma-1}}\\[2mm]
&\hspace{0.5cm}\leq \|(J^P(\tilde{X}^{(n)})-J^P)(\tilde{\zeta}^{(n)}-\mathcal{I})\nabla\tilde{v}_0\tilde{\mathbf{T}}_0\|_{H^{1}_{(0)}H^{\gamma-1}}\\[2mm]
&\hspace{0.5cm}\leq  \|J^P(\tilde{X}^{(n)})-J^P\|_{H^1_{(0)}H^{\gamma}}\|(\tilde{\zeta}^{(n)}-\mathcal{I})\nabla\tilde{v}_0\tilde{\mathbf{T}}_0\|_{H^{1}_{(0)}H^{\gamma-1}}\\[2mm]
&\hspace{0.5cm}\leq \|\tilde{X}^{(n)}-\tilde{\alpha}\|_{H^1_{(0)}H^{\gamma}}\|\tilde{\zeta}^{(n)}-\mathcal{I}\|_{H^1_{(0)}H^{\gamma-1}}\|\nabla\tilde{v}_0\tilde{\mathbf{T}}_0\|_{H^{\gamma}} \leq C(\tilde{v}_0,\tilde{\mathbf{T}}_0) T^{\beta_3} \\[4mm]
&\|I_{2,6}\|_{H^{2}_{(0)}H^{\gamma-1}}=\left\|\int_0^t (J^P(\tilde{X}^{(n)})-J^P)(\tilde{\zeta}^{(n)}-\mathcal{I})t\nabla\hat{\phi}(\tilde{\mathbf{T}}_p^{(n)}-\tilde{\mathbf{T}}_0)\right\|_{H^{1}_{(0)}H^{\gamma-1}}\\[2mm]
&\hspace{0.5cm}\leq \|(J^P(\tilde{X}^{(n)})-J^P)(\tilde{\zeta}^{(n)}-\mathcal{I})t\nabla\hat{\phi}(\tilde{\mathbf{T}}_p^{(n)}-\tilde{\mathbf{T}}_0)\|_{H^{1}_{(0)}H^{\gamma-1}}\\[2mm]
&\hspace{0.5cm}\leq \|J^P(\tilde{X}^{(n)})-J^P\|_{H^{1}_{(0)}H^{\gamma}}\|(\tilde{\zeta}^{(n)}-\mathcal{I})t\nabla\hat{\phi}(\tilde{\mathbf{T}}_p^{(n)}-\tilde{\mathbf{T}}_0)\|_{H^{1}_{(0)}H^{\gamma-1}}\\[2mm]
&\hspace{0.5cm}\leq \|\tilde{X}^{(n)}-\tilde{\alpha}\|_{H^{1}_{(0)}H^{\gamma}}\|\tilde{\zeta}^{(n)}-\mathcal{I}\|_{H^{1}_{(0)}H^{\gamma-1}}\|t\|_{H^1_{(0)}}\|\hat{\phi}\|_{H^{\gamma}}\|\tilde{\mathbf{T}}_p^{(n)}-\tilde{\mathbf{T}}_0\|_{H^{1}_{(0)}H^{\gamma-1}}\\[2mm]
&\hspace{0.5cm}\leq C(\tilde{v}_0,\tilde{\mathbf{T}}_0,\re,\kappa)\left(1+\frac{1}{\We}\right) T^{\beta_5}.
\end{align*}

\noindent For all the remaining terms we have $\|I_{2,i}\|_{H^{2}_{(0)}H^{\gamma-1}}\leq C(\tilde{v}_0,\tilde{\mathbf{T}}_0) T^{\beta_i}$, for $i=4,5,8,9,10$ and $\|I_{2,i}\|_{H^{2}_{(0)}H^{\gamma-1}}\leq C(\tilde{v}_0,\tilde{\mathbf{T}}_0,\kappa)\left(1+\frac{1}{\We}\right) T^{\beta_i}$, for $i=1,2,7$. Now, for the sake of simplicity we omit the splitting of $I_3$ and $I_4$, since it can be notice that these terms are similar to $I_1$ and $I_2$, respectively. Thus the way to split and estimate is the same. We summarize the result below

\begin{align*}
&\|I_3\|_{H^{2}_{(0)}H^{\gamma-1}}\leq C(\tilde{v}_0,\tilde{\mathbf{T}}_0,\kappa)\left(1+\frac{1}{\We}\right) T^{\mu_1},\\[2mm]
&\|I_4\|_{H^{2}_{(0)}H^{\gamma-1}}\leq C(\tilde{v}_0,\tilde{\mathbf{T}}_0,\kappa)\left(1+\frac{1}{\We}\right) T^{\mu_2}.
\end{align*}

\noindent It remains to estimate  $I_5$, by using lemma \ref{lem2} with $\varepsilon=0$ and the estimate \eqref{elastic-estim}.

\begin{align*}
&\frac{1}{\We}\left\|\int_0^t \tilde{\mathbf{T}}_p^{(n)}-\tilde{\mathbf{T}}_0\right\|_{H^{2}_{(0)}H^{\gamma-1}}\leq \frac{1}{\We}\|\tilde{\mathbf{T}}_p^{(n)}-\tilde{\mathbf{T}}_0\|_{H^{1}_{(0)}H^{\gamma-1}}\leq C(\tilde{v}_0,\tilde{\mathbf{T}}_0,\kappa)\frac{1}{\We} T^{\mu_3}.
\end{align*}
\medskip

\noindent For $I_6$ we consider the following splitting

\begin{align*}
I_6&=\frac{\kappa}{\We}\int_0^t (J^P(\tilde{X}^{(n)})-J^P)(\tilde{\zeta}^{(n)}-\mathcal{I})\nabla\tilde{w}^{(n)}+\frac{\kappa}{\We}\int_0^t J^P(\tilde{\zeta}^{(n)}-\mathcal{I})\nabla\tilde{w}^{(n)}\\[2mm]
&+\frac{\kappa}{\We}\int_0^t (J^P(\tilde{X}^{(n)})-J^P)\nabla\tilde{w}^{(n)}+\frac{\kappa}{\We}\int_0^t J^P\nabla\tilde{w}^{(n)}=\sum_{i=1}^4 I_{6,i}.
\end{align*}
\medskip

\noindent We show how to manage $I_{6,1}$ by using lemma \ref{lem2},  lemma \ref{lem3} with $\gamma>1$, lemma \ref{lem5}, lemma \ref{lem1} and lemma \ref{Jp-est}, lemma \ref{zeta-est}. To conclude we use the estimates for the flux and for the elastic stress tensor \eqref{flux-estim}, \eqref{elastic-estim}.

\begin{align*}
&\|I_{61}\|_{H^2_{(0)}H^{\gamma-1}}\leq \frac{\kappa}{\We}\|(J^P(\tilde{X}^{(n)})-J^P)(\tilde{\zeta}^{(n)}-\mathcal{I})\nabla\tilde{w}^{(n)}\|_{H^1_{(0)}H^{\gamma-1}}\\[2mm]
&\hspace{0.5cm}\leq \frac{\kappa}{\We}\|J^P(\tilde{X}^{(n)})-J^P\|_{H^1_{(0)}H^{\gamma}}\|(\tilde{\zeta}^{(n)}-\mathcal{I})\nabla\tilde{w}^{(n)}\|_{H^1_{(0)}H^{\gamma-1}}\\[2mm]
&\hspace{0.5cm}\leq \frac{\kappa}{\We}\|\tilde{X}^{(n)}-\tilde{\alpha}\|_{H^1_{(0)}H^{\gamma}}\|\tilde{\zeta}^{(n)}-\mathcal{I}\|_{H^1_{(0)}H^{\gamma-1}}\|\tilde{w}^{(n)}\|_{H^1_{(0)}H^{\gamma}}\leq C(\tilde{v}_0,\kappa)\frac{1}{\We} T^{\varrho_1}.
\end{align*}
\medskip

\noindent For $i=2,3,4$, we have $\|I_{6,i}\|_{H^2_{(0)}H^{\gamma-1}}\leq C(\tilde{v}_0,\kappa)\frac{1}{\We} T^{\varrho_i}$. For $I_7$, we  have the following terms

\begin{align*}
I_7&=\frac{\kappa}{\We}\int_0^t (J^P(\tilde{X}^{(n)})-J^P)(\tilde{\zeta}^{(n)}-\mathcal{I})\nabla\tilde{v}_0+\frac{\kappa}{\We}\int_0^t J^P(\tilde{\zeta}^{(n)}-\mathcal{I})\nabla\tilde{v}_0\\[2mm]
&+\frac{\kappa}{\We}\int_0^t (J^P(\tilde{X}^{(n)})-J^P)\nabla\tilde{v}_0+\frac{\kappa}{\We}\int_0^t (J^P(\tilde{X}^{(n)})-J^P)(\tilde{\zeta}^{(n)}-\mathcal{I})t\nabla\hat{\phi}\\[2mm]
&+\frac{\kappa}{\We}\int_0^t J^P(\tilde{\zeta}^{(n)}-\mathcal{I})t\nabla\hat{\phi}+\frac{\kappa}{\We}\int_0^t (J^P(\tilde{X}^{(n)})-J^P)t\nabla\hat{\phi}=\sum_{i=1}^6 I_{7,i}.
\end{align*}

\noindent We prove the results for $I_{7,1}$ and $I_{7,4}$, in the same way as $I_6$.

\begin{align*}
&\|I_{7,1}\|_{H^{2}_{(0)}H^{\gamma-1}}\leq \frac{\kappa}{\We} \|(J^P(\tilde{X}^{(n)})-J^P)(\tilde{\zeta}^{(n)}-\mathcal{I})\nabla\tilde{v}_0\|_{H^1_{(0)}H^{\gamma-1}}\\[2mm]
&\hspace{0.5cm}\leq \frac{\kappa}{\We}\|J^P(\tilde{X}^{(n)})-J^P\|_{H^1_{(0)}H^{\gamma}}\|(\tilde{\zeta}^{(n)}-\mathcal{I})\nabla\tilde{v}_0\|_{H^1_{(0)}H^{\gamma-1}}\\[2mm]
&\hspace{0.5cm}\leq C(\tilde{v}_0,\kappa)\frac{1}{\We}\|\tilde{X}^{(n)}-\tilde{\alpha}\|_{H^{1}_{(0)}H^{\gamma}}\|\tilde{\zeta}^{(n)}-\mathcal{I}\|_{H^1_{(0)}H^{\gamma-1}}\|\nabla\tilde{v}_0\|_{H^{\gamma}}\leq  C(\tilde{v}_0,\kappa)\frac{1}{\We}T^{\eta_1}\\[5mm]
&\|I_{7,4}\|_{H^{2}_{(0)}H^{\gamma-1}}\leq \frac{\kappa}{\We}\| (J^P(\tilde{X}^{(n)})-J^P)(\tilde{\zeta}^{(n)}-\mathcal{I})t\nabla\hat{\phi}\|_{H^1_{(0)}H^{\gamma-1}}\\[2mm]
&\hspace{0.5cm}\leq\frac{\kappa}{\We}\|J^P(\tilde{X}^{(n)})-J^P\|_{H^1_{(0)}H^{\gamma}}\|(\tilde{\zeta}^{(n)}-\mathcal{I})t\nabla\hat{\phi}\|_{H^1_{(0)}H^{\gamma-1}}\\[2mm]
&\hspace{0.5cm}\leq C(\tilde{v}_0,\kappa)\frac{1}{\We}\|\tilde{X}^{(n)}-\tilde{\alpha}\|_{H^{1}_{(0)}H^{\gamma}}\|\tilde{\zeta}^{(n)}-\mathcal{I}\|_{H^1_{(0)}H^{\gamma-1}}\|t\nabla\hat{\phi}\|_{H^1_{(0)}H^{\gamma-1}}\\[2mm]
&\hspace{0.5cm}\leq  C(\tilde{v}_0,\tilde{T}_0,\re,\kappa)\frac{1}{\We}T^{\eta_4}.
\end{align*}
\medskip

\noindent For $I_{7,2}, I_{7,3}$ the estimate is the same as $I_{7,1}$ and for $I_{7,5}, I_{7,6}$ is the same as $I_{7,4}$. Furthermore, the integrals $I_8$ and $I_9$ can be splitted and estimated in a similas manner as $I_6$ and $I_7$, since they are just transposed. So we will skip all the computations and we outline just the final results

\begin{align*}
&\|I_8\|_{H^{2}_{(0)}H^{\gamma-1}}\leq C(\tilde{v}_0,\kappa)\frac{1}{\We}T^{\sigma_1},\\[2mm]
&\|I_9\|_{H^{2}_{(0)}H^{\gamma-1}}\leq C(\tilde{v}_0,\tilde{T}_0,\re,\kappa)\frac{1}{\We}T^{\sigma_2}.
\end{align*}
\medskip

\noindent To conclude the proof of the first part we choose\\
$\delta=\min\{\frac{1}{4},\delta_i, \beta_j,\mu_1,\mu_2,\mu_3,\varrho_k,\eta_h,\sigma_1,\sigma_2\},$ for $i=1,\ldots,6$, $j=1,\ldots,10$, $k=1,\ldots,4$ and $h=1,\ldots,6$.
\bigskip

\noindent At this point, we consider the difference  $\tilde{\mathbf{T}}_p^{(n+1)}-\tilde{\mathbf{T}}_p^{(n)}$, namely
\begin{align*}
&\tilde{\mathbf{T}}_p^{(n+1)}-\tilde{\mathbf{T}}_p^{(n)}=\int_0^t \left(J^P(\tilde{X}^{(n)})\tilde{\zeta}^{(n)}\nabla\tilde{w}^{(n)} \tilde{\mathbf{T}}_p^{(n)}-J^P(\tilde{X}^{(n-1)})\tilde{\zeta}^{(n-1)}\nabla\tilde{w}^{(n-1)} \tilde{\mathbf{T}}_p^{(n-1)}\right)\\[1mm]
&+\int_0^t \left(J^P(\tilde{X}^{(n)})\tilde{\zeta}^{(n)}\nabla\tilde{v}_0 \tilde{\mathbf{T}}_p^{(n)}-J^P(\tilde{X}^{(n-1)})\tilde{\zeta}^{(n-1)}\nabla\tilde{v}_0 \tilde{\mathbf{T}}_p^{(n-1)}\right)\\[2mm]
&+\int_0^t \left(J^P(\tilde{X}^{(n)})\tilde{\zeta}^{(n)}t\nabla\hat{\phi} \tilde{\mathbf{T}}_p^{(n)}-J^P(\tilde{X}^{(n-1)})\tilde{\zeta}^{(n-1)}t\nabla\hat{\phi} \tilde{\mathbf{T}}_p^{(n-1)}\right)\\[2mm]
&+\int_0^t\left(\tilde{\mathbf{T}}_p^{(n)}\left(J^P(\tilde{X}^{(n)})\tilde{\zeta}^{(n)}\nabla\tilde{w}^{(n)}\right)^T-\tilde{\mathbf{T}}_p^{(n-1)}\left(J^P(\tilde{X}^{(n-1)})\tilde{\zeta}^{(n-1)}\nabla\tilde{w}^{(n-1)}\right)^T\right)\\[2mm]
&+\int_0^t\left(\tilde{\mathbf{T}}_p^{(n)}\left(J^P(\tilde{X}^{(n)})\tilde{\zeta}^{(n)}\nabla\tilde{v}_0\right)^T-\tilde{\mathbf{T}}_p^{(n-1)}\left(J^P(\tilde{X}^{(n-1)})\tilde{\zeta}^{(n-1)}\nabla\tilde{v}_0\right)^T\right)\\[2mm]
\end{align*}
\begin{align*}
&+\int_0^t\left(\tilde{\mathbf{T}}_p^{(n)}\left(J^P(\tilde{X}^{(n)})\tilde{\zeta}^{(n)}t\nabla\hat{\phi}\right)^T-\tilde{\mathbf{T}}_p^{(n-1)}\left(J^P(\tilde{X}^{(n-1)})\tilde{\zeta}^{(n-1)}t\nabla\hat{\phi}\right)^T\right)\\[2mm]
&-\frac{1}{\We}\int_0^t \left(\tilde{\mathbf{T}}_p^{(n)}-\tilde{\mathbf{T}}_p^{(n-1)}\right)+\frac{\kappa}{\We}\int_0^t\left(J^P(\tilde{X}^{(n)})\tilde{\zeta}^{(n)}\nabla\tilde{w}^{(n)}-J^P(\tilde{X}^{(n-1)})\tilde{\zeta}^{(n-1)}\nabla\tilde{w}^{(n-1)}\right)\\[2mm]
&+\frac{\kappa}{\We}\int_0^t\left(J^P(\tilde{X}^{(n)})\tilde{\zeta}^{(n)}\nabla\tilde{v}_0-J^P(\tilde{X}^{(n-1)})\tilde{\zeta}^{(n-1)}\nabla\tilde{v}_0\right)\\[2mm]
&+\frac{\kappa}{\We}\int_0^t\left(J^P(\tilde{X}^{(n)})\tilde{\zeta}^{(n)}t\nabla\hat{\phi}-J^P(\tilde{X}^{(n-1)})\tilde{\zeta}^{(n-1)}t\nabla\hat{\phi}\right)\\[2mm]
&+\frac{\kappa}{\We}\int_0^t\left(\left(J^P(\tilde{X}^{(n)})\tilde{\zeta}^{(n)}\nabla\tilde{w}^{(n)}\right)^T-\left(J^P(\tilde{X}^{(n-1)})\tilde{\zeta}^{(n-1)}\nabla\tilde{w}^{(n-1)}\right)^T\right)\\[2mm]
&+\frac{\kappa}{\We}\int_0^t\left(\left(J^P(\tilde{X}^{(n)})\tilde{\zeta}^{(n)}\nabla\tilde{v}_0\right)^T-\left(J^P(\tilde{X}^{(n-1)})\tilde{\zeta}^{(n-1)}\nabla\tilde{v}_0\right)^T\right)\\[2mm]
&+\frac{\kappa}{\We}\int_0^t\left(\left(J^P(\tilde{X}^{(n)})\tilde{\zeta}^{(n)}t\nabla\hat{\phi}\right)^T-\left(J^P(\tilde{X}^{(n-1)})\tilde{\zeta}^{(n-1)}\nabla\hat{\phi}\right)^T\right)=\sum_{i=1}^{13} I_i
\end{align*}

\noindent  We deal with the first term in $L^{\infty}_{\frac{1}{4}}H^s$-norm and in order to be able to estimate correctly we split as follows

\begin{align*}
&I_1=\int_0^t (J^P(\tilde{X}^{(n)})-J^P(\tilde{X}^{(n-1)}))\tilde{\zeta}^{(n)}\nabla\tilde{w}^{(n)} (\tilde{\mathbf{T}}_p^{(n)}-\tilde{\mathbf{T}}_0)\\[1mm]
&+\int_0^t (J^P(\tilde{X}^{(n)})-J^P(\tilde{X}^{(n-1)}))\tilde{\zeta}^{(n)}\nabla\tilde{w}^{(n)}\tilde{\mathbf{T}}_0\\[1mm]
&+\int_0^t J^P(\tilde{X}^{(n-1)})(\tilde{\zeta}^{(n)}-\tilde{\zeta}^{(n-1)})\nabla\tilde{w}^{(n)} (\tilde{\mathbf{T}}_p^{(n)}-\tilde{\mathbf{T}}_0)\\[1mm]
&+\int_0^t J^P(\tilde{X}^{(n-1)})(\tilde{\zeta}^{(n)}-\tilde{\zeta}^{(n-1)})\nabla\tilde{w}^{(n)}\tilde{\mathbf{T}}_0\\[1mm]
&+\int_0^t J^P(\tilde{X}^{(n-1)})\tilde{\zeta}^{(n-1)}(\nabla\tilde{w}^{(n)}-\nabla\tilde{w}^{(n-1)})(\tilde{\mathbf{T}}_p^{(n)}-\tilde{\mathbf{T}}_0)\\[1mm]
&+\int_0^t J^P(\tilde{X}^{(n-1)})\tilde{\zeta}^{(n-1)}(\nabla\tilde{w}^{(n)}-\nabla\tilde{w}^{(n-1)})\tilde{\mathbf{T}}_0\\[1mm]
&+\int_0^t J^P(\tilde{X}^{(n-1)})\tilde{\zeta}^{(n-1)}\nabla\tilde{w}^{(n-1)}(\tilde{\mathbf{T}}_p^{(n)}-\tilde{\mathbf{T}}_p^{(n-1)})=\sum_{i=1}^7 I_{1,i}.
\end{align*}
\medskip

\noindent We show the estimates of the most significant terms $I_{1,1}, I_{1,6}, I_{1,7}$. We use the definition of the $L^{\infty}_{\frac{1}{4}}-$norm in time, lemma \ref{Jp-est} or lemma \ref{Jp-dif-est}, lemma \ref{zeta-est}  and the estimates \eqref{flux-estim} and \eqref{elastic-estim}.

\begin{align*}
&\|I_{1,1}\|_{L^{\infty}_{\frac{1}{4}}H^{s}}\leq T^{\frac{1}{4}}\|(J^P(\tilde{X}^{(n)})-J^P(\tilde{X}^{(n-1)}))\tilde{\zeta}^{(n)}\nabla\tilde{w}^{(n)} (\tilde{\mathbf{T}}_p^{(n)}-\tilde{\mathbf{T}}_0)\|_{L^2H^s}\\[1mm]
&\hspace{0.5cm}\leq T^{\frac{1}{4}}\|J^P(\tilde{X}^{(n)})-J^P(\tilde{X}^{(n-1)})\|_{L^{\infty}H^s}\|\tilde{\zeta}^{(n)}\|_{L^{\infty}H^s}
\|\nabla\tilde{w}^{(n)}\|_{L^2H^s}\\[1mm]
&\hspace{0.5cm}\cdot\|\tilde{\mathbf{T}}_p^{(n)}-\tilde{\mathbf{T}}_0\|_{L^{\infty}H^s}\\[1mm]
&\hspace{0.5cm}\leq C(\tilde{v}_0,\tilde{\mathbf{T}}_0,\kappa)\left(1+\frac{1}{\We}\right) T^{\frac{1}{2}}\|\tilde{X}^{(n)}-\tilde{X}^{(n-1)}\|_{\mathcal{F}^{s+1,\gamma}},\\[4mm]
&\|I_{1,6}\|_{L^{\infty}_{\frac{1}{4}}H^{s}}\leq T^{\frac{1}{4}} \|J^P(\tilde{X}^{(n-1)})\tilde{\zeta}^{(n-1)}(\nabla\tilde{w}^{(n)}-\nabla\tilde{w}^{(n-1)})\tilde{\mathbf{T}}_0\|_{L^2H^s}\\[1mm]
&\hspace{0.5cm}\leq T^{\frac{1}{4}} \|J^P(\tilde{X}^{(n-1)})\|_{L^{\infty}H^s}\|\tilde{\zeta}^{(n-1)}\|_{L^{\infty}H^s}\|\nabla\tilde{w}^{(n)}-\nabla\tilde{w}^{(n-1)}\|_{L^2H^s}\|\tilde{\mathbf{T}}_0\|_{H^s}\\[1mm]
&\hspace{0.5cm}\leq C(\tilde{v}_0,\tilde{\mathbf{T}}_0)T^{\frac{1}{4}}\|\tilde{w}^{(n)}-\tilde{w}^{(n-1)}\|_{\mathcal{K}^{s+1}_{(0)}},\\[4mm]
&\|I_{1,7}\|_{L^{\infty}_{\frac{1}{4}}H^s}\leq T^{\frac{1}{4}}\|J^P(\tilde{X}^{(n-1)})\tilde{\zeta}^{(n-1)}\nabla\tilde{w}^{(n-1)}(\tilde{\mathbf{T}}_p^{(n)}-\tilde{\mathbf{T}}_p^{(n-1)})\|_{L^2H^s}\\[1mm]
&\hspace{0.5cm}\leq T^{\frac{1}{4}} \|J^P(\tilde{X}^{(n-1)})\|_{L^{\infty}H^s}\|\tilde{\zeta}^{(n-1)}\|_{L^{\infty}H^s}\|\nabla\tilde{w}^{(n-1)}\|_{L^2H^s}\|\tilde{\mathbf{T}}_p^{(n)}-\tilde{\mathbf{T}}_p^{(n-1)}\|_{L^{\infty}H^s}\\[1mm]
&\hspace{0.5cm}\leq C(\tilde{v}_0) T^{\frac{1}{2}}\|\tilde{\mathbf{T}}_p^{(n)}-\tilde{\mathbf{T}}_p^{(n-1)}\|_{\mathcal{F}^{s,\gamma-1}}.
\end{align*}
\medskip

\noindent For the remaining terms by applying the mentioned lemmas, we have \\
$\|I_{1,i}\|_{L^{\infty}_{\frac{1}{4}}H^s}\leq C(\tilde{v}_0,\tilde{\mathbf{T}}_0) T^{\frac{1}{2}}\|\tilde{X}^{(n)}-\tilde{X}^{(n-1)}\|_{\mathcal{F}^{s+1,\gamma}},$ for $i=2,4$ and\\
$\|I_{1,3}\|_{L^{\infty}_{\frac{1}{4}}H^s}\leq C(\tilde{v}_0,\tilde{\mathbf{T}}_0,\kappa) \left(1+\frac{1}{\We}\right)T^{\frac{1}{2}}\|\tilde{X}^{(n)}-\tilde{X}^{(n-1)}\|_{\mathcal{F}^{s+1,\gamma}}$, for $i=3,5$. The next step is to analyze $I_2$ and $I_3$, but we notice that these two integrals have the same splitting as $I_1$ except for the presence of the differences of the velocities. Thus, we rewrite them as follows 

\begin{align*}
&I_2=\int_0^t (J^P(\tilde{X}^{(n)})-J^P(\tilde{X}^{(n-1)}))\tilde{\zeta}^{(n)}\nabla\tilde{v}_0 (\tilde{\mathbf{T}}_p^{(n)}-\tilde{\mathbf{T}}_0)+\int_0^t (J^P(\tilde{X}^{(n)})-J^P(\tilde{X}^{(n-1)}))\tilde{\zeta}^{(n)}\nabla\tilde{v}_0\tilde{\mathbf{T}}_0\\[1mm]
&\hspace{0.5cm}+\int_0^t J^P(\tilde{X}^{(n-1)})(\tilde{\zeta}^{(n)}-\tilde{\zeta}^{(n-1)})\nabla\tilde{v}_0(\tilde{\mathbf{T}}_p^{(n)}-\tilde{\mathbf{T}}_0)+\int_0^t J^P(\tilde{X}^{(n-1)})(\tilde{\zeta}^{(n)}-\tilde{\zeta}^{(n-1)})\nabla\tilde{v}_0\tilde{\mathbf{T}}_0\\[1mm]
&\hspace{0.5cm}+\int_0^t J^P(\tilde{X}^{(n-1)})\tilde{\zeta}^{(n-1)}\nabla\tilde{v}_0(\tilde{\mathbf{T}}_p^{(n)}-\tilde{\mathbf{T}}_p^{(n-1)})=\sum_{i=1}^5 I_{2,i},\\[4mm]
&I_3=\int_0^t (J^P(\tilde{X}^{(n)})-J^P(\tilde{X}^{(n-1)}))\tilde{\zeta}^{(n)}t\nabla\hat{\phi} (\tilde{\mathbf{T}}_p^{(n)}-\tilde{\mathbf{T}}_0)+\int_0^t (J^P(\tilde{X}^{(n)})-J^P(\tilde{X}^{(n-1)}))\tilde{\zeta}^{(n)}t\nabla\hat{\phi}\tilde{\mathbf{T}}_0\\[1mm]
&\hspace{0.5cm}+\int_0^t J^P(\tilde{X}^{(n-1)})(\tilde{\zeta}^{(n)}-\tilde{\zeta}^{(n-1)})t\nabla\hat{\phi}(\tilde{\mathbf{T}}_p^{(n)}-\tilde{\mathbf{T}}_0)+\int_0^t J^P(\tilde{X}^{(n-1)})(\tilde{\zeta}^{(n)}-\tilde{\zeta}^{(n-1)})t\nabla\hat{\phi}\tilde{\mathbf{T}}_0\\[1mm]
&\hspace{0.5cm}+\int_0^t J^P(\tilde{X}^{(n-1)})\tilde{\zeta}^{(n-1)}t\nabla\hat{\phi}(\tilde{\mathbf{T}}_p^{(n)}-\tilde{\mathbf{T}}_p^{(n-1)})=\sum_{i=1}^5 I_{3,i}.
\end{align*}
\medskip

\noindent For the estimate in $L^{\infty}_{\frac{1}{4}}H^s$, we resume as follows, for details it is enough to check $I_1$

\begin{align*}
&\|I_2\|_{L^{\infty}_{\frac{1}{4}}H^s}+\|I_3\|_{L^{\infty}_{\frac{1}{4}}H^s}\leq C(\tilde{v}_0,\tilde{\mathbf{T}}_0,\re,\kappa) \left(1+\frac{1}{\We}\right)T^{\frac{1}{2}}\\
&\hspace{1cm}\cdot\left(\|\tilde{X}^{(n)}-\tilde{X}^{(n-1)}\|_{\mathcal{F}^{s+1,\gamma}}
+\|\tilde{\mathbf{T}}_p^{(n)}-\tilde{\mathbf{T}}_p^{(n-1)}\|_{\mathcal{F}^{s,\gamma-1}}\right)
\end{align*}
\noindent The integrals $I_4, I_5$ and $I_6$ can be treated in the same way as $I_1, I_2$ and $I_3$ so we skip their estimate. We pass to the study of $I_7$, that can be easily managed 

\begin{align*}
\|I_7\|_{L^{\infty}_{\frac{1}{4}}H^s}\leq \frac{1}{\We}T^{\frac{1}{4}}\|\tilde{\mathbf{T}}_p^{(n)}-\tilde{\mathbf{T}}_p^{(n-1)}\|_{L^2H^s}\leq \frac{1}{\We}T\|\tilde{\mathbf{T}}_p^{(n)}-\tilde{\mathbf{T}}_p^{(n-1)}\|_{\mathcal{F}^{s,\gamma-1}}.
\end{align*}
\medskip

\noindent Finally, we focus on $I_8, I_9, I_{10}$ and we avoid $I_{11}, I_{12}, I_{13}$. For $I_8$ we have the following 

\begin{align*}
I_8&=\frac{\kappa}{\We}\int_0^t(J^P(\tilde{X}^{(n)})-J^P(\tilde{X}^{(n-1)}))\tilde{\zeta}^{(n)}\nabla\tilde{w}^{(n)}+\frac{\kappa}{\We}\int_0^t J^P(\tilde{X}^{(n-1)})(\tilde{\zeta}^{(n)}-\tilde{\zeta}^{(n-1)})\nabla\tilde{w}^{(n)}\\[1mm]
&+\frac{\kappa}{\We}\int_0^t J^P(\tilde{X}^{(n-1)})\tilde{\zeta}^{(n-1)}(\nabla\tilde{w}^{(n)}-\nabla\tilde{w}^{(n-1)})=I_{8,1}+I_{8,2}+I_{8,3}.
\end{align*}

\noindent By means of lemma \ref{Jp-est} or lemma \ref{Jp-dif-est}, lemma \ref{zeta-est} or lemma \ref{zeta-dif-est} and the estimate for the flux \eqref{flux-estim}.

\begin{align*}
&\|I_{8,1}\|_{L^{\infty}_{\frac{1}{4}}H^s}\leq \frac{\kappa}{\We}T^{\frac{1}{4}}\|(J^P(\tilde{X}^{(n)})-J^P(\tilde{X}^{(n-1)}))\tilde{\zeta}^{(n)}\nabla\tilde{w}^{(n)}\|_{L^2H^s} \\[2mm]
&\hspace{0.5cm}\leq\frac{\kappa}{\We}T^{\frac{1}{4}}\|J^P(\tilde{X}^{(n)})-J^P(\tilde{X}^{(n-1)})\|_{L^{\infty}H^s}\|\tilde{\zeta}^{(n)}\|_{L^{\infty}H^s}\|\nabla\tilde{w}^{(n)}\|_{L^2H^s}\\[2mm]
&\hspace{0.5cm}\leq C(\tilde{v}_0,\kappa)\frac{1}{\We}T^{\frac{1}{2}}\|\tilde{X}^{(n)}-\tilde{X}^{(n-1)}\|_{\mathcal{F}^{s+1,\gamma}}\\[4mm]
&\|I_{8,2}\|_{L^{\infty}_{\frac{1}{4}}H^s}\leq \frac{\kappa}{\We}T^{\frac{1}{4}}\|J^P(\tilde{X}^{(n-1)})(\tilde{\zeta}^{(n)}-\tilde{\zeta}^{(n-1)})\nabla\tilde{w}^{(n)}\|_{L^2H^s}\\[2mm]
&\hspace{0.5cm}\leq\frac{\kappa}{\We}T^{\frac{1}{4}}\|J^P(\tilde{X}^{(n-1)})\|_{L^{\infty}H^s}\|\tilde{\zeta}^{(n)}-\tilde{\zeta}^{(n-1)}\|_{L^{\infty}H^s}\|\nabla\tilde{w}^{(n)}\|_{L^2H^s}\\[2mm]
&\hspace{0.5cm}\leq C(\tilde{v}_0,\kappa) \frac{1}{\We}T^{\frac{1}{2}}\|\tilde{X}^{(n)}-\tilde{X}^{(n-1)}\|_{\mathcal{F}^{s+1,\gamma}}\\[4mm]
&\|I_{8,3}\|_{L^{\infty}_{\frac{1}{4}}H^s}\leq \frac{\kappa}{\We}T^{\frac{1}{4}}\|J^P(\tilde{X}^{(n-1)})\tilde{\zeta}^{(n-1)}(\nabla\tilde{w}^{(n)}-\nabla\tilde{w}^{(n-1)})\|_{L^2H^s}\\[2mm]
&\hspace{0.5cm}\leq \frac{\kappa}{\We}T^{\frac{1}{4}}\|J^P(\tilde{X}^{(n-1)})\|_{L^{\infty}H^s}\|\tilde{\zeta}^{(n-1)}\|_{L^{\infty}H^s}\|\tilde{w}^{(n)}-\tilde{w}^{(n-1)}\|_{L^2H^{s+1}}\\[2mm]
&\hspace{0.5cm}\leq C(\tilde{v}_0,\kappa) \frac{1}{\We}T^{\frac{1}{4}}\|\tilde{w}^{(n)}-\tilde{w}^{(n-1)}\|_{\mathcal{K}^{s+1}_{(0)}}.
\end{align*}
\medskip

\noindent We rewrite $I_9$ and $I_{10}$ as follows

\begin{align*}
&I_9=\frac{\kappa}{\We}\int_0^t(J^P(\tilde{X}^{(n)})-J^P(\tilde{X}^{(n-1)}))\tilde{\zeta}^{(n)}\nabla\tilde{v}_0+\frac{\kappa}{\We}\int_0^t J^P(\tilde{X}^{(n-1)})(\tilde{\zeta}^{(n)}-\tilde{\zeta}^{(n-1)})\nabla\tilde{v}_0\\[2mm]
&I_{10}=\frac{\kappa}{\We}\int_0^t(J^P(\tilde{X}^{(n)})-J^P(\tilde{X}^{(n-1)}))\tilde{\zeta}^{(n)}t\nabla\hat{\phi}+\frac{\kappa}{\We}\int_0^t J^P(\tilde{X}^{(n-1)})(\tilde{\zeta}^{(n)}-\tilde{\zeta}^{(n-1)})t\nabla\hat{\phi}.
\end{align*}
\medskip

\noindent The estimates make use of the lemmas already mentioned for $I_8$ and we assume the initial data $\tilde{v}_0, \tilde{\mathbf{T}}_0$ to be enough regular. Then we get

\begin{align*}
&\|I_{9,1}\|_{L^{\infty}_{\frac{1}{4}}H^s}\leq \frac{\kappa}{\We}T^{\frac{1}{4}}\|(J^P(\tilde{X}^{(n)})-J^P(\tilde{X}^{(n-1)}))\tilde{\zeta}^{(n)}\nabla\tilde{v}_0\|_{L^2H^s}\\[2mm]
&\hspace{1.8cm}\leq\frac{\kappa}{\We}T^{\frac{3}{4}}\|J^P(\tilde{X}^{(n)})-J^P(\tilde{X}^{(n-1)})\|_{L^{\infty}H^s}\|\tilde{\zeta}^{(n)}\|_{L^{\infty}H^s}\|\tilde{v}_0\|_{H^{s+1}}\\[2mm]
&\hspace{1.8cm}\leq C(\tilde{v}_0,\kappa)\frac{1}{\We}T\|\tilde{X}^{(n)}-\tilde{X}^{(n-1)}\|_{\mathcal{F}^{s+1,\gamma}}\\[4mm]
&\|I_{9,2}\|_{L^{\infty}_{\frac{1}{4}}H^s}\leq \frac{\kappa}{\We}T^{\frac{1}{4}}\|J^P(\tilde{X}^{(n-1)})(\tilde{\zeta}^{(n)}-\tilde{\zeta}^{(n-1)})\nabla\tilde{v}_0\|_{L^2H^s}\\[2mm]
&\hspace{1.8cm}\leq\frac{\kappa}{\We}T^{\frac{3}{4}}\|J^P(\tilde{X}^{(n-1)})\|_{L^{\infty}H^s}\|\tilde{\zeta}^{(n)}-\tilde{\zeta}^{(n-1)}\|_{L^{\infty}H^s}\|\tilde{v}_0\|_{H^{s+1}}\\[2mm]
&\hspace{1.8cm}\leq C(\tilde{v}_0,\kappa) \frac{1}{\We}T \|\tilde{X}^{(n)}-\tilde{X}^{(n-1)}\|_{\mathcal{F}^{s+1,\gamma}}\\[4mm]
&\|I_{10,1}\|_{L^{\infty}_{\frac{1}{4}}H^s}\leq \frac{\kappa}{\We}T^{\frac{1}{4}}\|(J^P(\tilde{X}^{(n)})-J^P(\tilde{X}^{(n-1)}))\tilde{\zeta}^{(n)}t\nabla\hat{\phi}\|_{L^2H^s} \\[2mm]
&\hspace{1.8cm}\leq\frac{\kappa}{\We}T^{\frac{1}{4}}\|J^P(\tilde{X}^{(n)})-J^P(\tilde{X}^{(n-1)})\|_{L^{\infty}H^s}\|\tilde{\zeta}^{(n)}\|_{L^{\infty}H^s}\|t\nabla\hat{\phi}\|_{L^2H^s}\\[2mm]
&\hspace{1.8cm}\leq C(\tilde{v}_0,\tilde{\mathbf{T}}_0,\re,\kappa)\frac{1}{\We}T^{\frac{9}{4}}\|\tilde{X}^{(n)}-\tilde{X}^{(n-1)}\|_{\mathcal{F}^{s+1,\gamma}}\\[4mm]
&\|I_{10,2}\|_{L^{\infty}_{\frac{1}{4}}H^s}\leq \frac{\kappa}{\We}T^{\frac{1}{4}}\|J^P(\tilde{X}^{(n-1)})(\tilde{\zeta}^{(n)}-\tilde{\zeta}^{(n-1)})t\nabla\hat{\phi}\|_{L^2H^s}\\[2mm]
&\hspace{1.8cm}\leq\frac{\kappa}{\We}T^{\frac{1}{4}}\|J^P(\tilde{X}^{(n-1)})\|_{L^{\infty}H^s}\|\tilde{\zeta}^{(n)}-\tilde{\zeta}^{(n-1)}\|_{L^{\infty}H^s}\|t\nabla\hat{\phi}\|_{L^2H^s}\\[2mm]
&\hspace{1.8cm}\leq C(\tilde{v}_0,\tilde{\mathbf{T}}_0,\re,\kappa) \frac{1}{\We}T^{\frac{9}{4}}\|\tilde{X}^{(n)}-\tilde{X}^{(n-1)}\|_{\mathcal{F}^{s+1,\gamma}}
\end{align*}

\noindent We deduce similar results for the remaining integrals. We can move to analysis of the $H^2_{(0)}H^{\gamma-1}-$norm. In order to deal with this space we have to guarantee enough cancelations at  time zero since this assumption is required in the space $H^2_{(0)}([0,T])$. We start with $I_1$, that can be splitted as follows

\begin{align*}
&I_{1}=\int_0^t (J^P(\tilde{X}^{(n)})-J^P(\tilde{X}^{(n-1)}))(\tilde{\zeta}^{(n)}-\mathcal{I})\nabla\tilde{w}^{(n)} (\tilde{\mathbf{T}}_p^{(n)}-\tilde{\mathbf{T}}_0)\\[1mm]
&\hspace{0.5cm}+\int_0^t (J^P(\tilde{X}^{(n)})-J^P(\tilde{X}^{(n-1)}))(\tilde{\zeta}^{(n)}-\mathcal{I})\nabla\tilde{w}^{(n)}\tilde{\mathbf{T}}_0\\[1mm]
&\hspace{0.5cm}+\int_0^t (J^P(\tilde{X}^{(n)})-J^P(\tilde{X}^{(n-1)}))\nabla\tilde{w}^{(n)} (\tilde{\mathbf{T}}_p^{(n)}-\tilde{\mathbf{T}}_0)\\[1mm]
&\hspace{0.5cm}+\int_0^t (J^P(\tilde{X}^{(n)})-J^P(\tilde{X}^{(n-1)}))\nabla\tilde{w}^{(n)} \tilde{\mathbf{T}}_0\\[1mm]
&\hspace{0.5cm}+\int_0^t (J^P(\tilde{X}^{(n-1)})-J^P)(\tilde{\zeta}^{(n)}-\tilde{\zeta}^{(n-1)})\nabla\tilde{w}^{(n)} (\tilde{\mathbf{T}}_p^{(n)}-\tilde{\mathbf{T}}_0)\\[1mm]
&\hspace{0.5cm}+\int_0^t (J^P(\tilde{X}^{(n-1)})-J^P)(\tilde{\zeta}^{(n)}-\tilde{\zeta}^{(n-1)})\nabla\tilde{w}^{(n)}\tilde{\mathbf{T}}_0\\[1mm]
&\hspace{0.5cm}+\int_0^t J^P(\tilde{\zeta}^{(n)}-\tilde{\zeta}^{(n-1)})\nabla\tilde{w}^{(n)} (\tilde{\mathbf{T}}_p^{(n)}-\tilde{\mathbf{T}}_0)+\int_0^t J^P(\tilde{\zeta}^{(n)}-\tilde{\zeta}^{(n-1)})\nabla\tilde{w}^{(n)}\tilde{\mathbf{T}}_0\\[1mm]
&\hspace{0.5cm}+\int_0^t (J^P(\tilde{X}^{(n-1)})-J^P)(\tilde{\zeta}^{(n-1)}-\mathcal{I})(\nabla\tilde{w}^{(n)}-\nabla\tilde{w}^{(n-1)})(\tilde{\mathbf{T}}_p^{(n)}-\tilde{\mathbf{T}}_0)\\[1mm]
&\hspace{0.5cm}+\int_0^t(J^P(\tilde{X}^{(n-1)})-J^P)(\tilde{\zeta}^{(n-1)}-\mathcal{I})(\nabla\tilde{w}^{(n)}-\nabla\tilde{w}^{(n-1)})\tilde{\mathbf{T}}_0\\[1mm]
&\hspace{0.5cm}+\int_0^t J^P(\tilde{\zeta}^{(n-1)}-\mathcal{I})(\nabla\tilde{w}^{(n)}-\nabla\tilde{w}^{(n-1)})(\tilde{\mathbf{T}}_p^{(n)}-\tilde{\mathbf{T}}_0)\\[1mm]
&\hspace{0.5cm}+\int_0^t J^P(\tilde{\zeta}^{(n-1)}-\mathcal{I})(\nabla\tilde{w}^{(n)}-\nabla\tilde{w}^{(n-1)})\tilde{\mathbf{T}}_0\\[1mm]
&\hspace{0.5cm}+\int_0^t(J^P(\tilde{X}^{(n-1)})-J^P)(\nabla\tilde{w}^{(n)}-\nabla\tilde{w}^{(n-1)})(\tilde{\mathbf{T}}_p^{(n)}-\tilde{\mathbf{T}}_0)\\[1mm]
&\hspace{0.5cm}+\int_0^t(J^P(\tilde{X}^{(n-1)})-J^P)(\nabla\tilde{w}^{(n)}-\nabla\tilde{w}^{(n-1)})\tilde{\mathbf{T}}_0\\[1mm]
&\hspace{0.5cm}+\int_0^t J^P(\nabla\tilde{w}^{(n)}-\nabla\tilde{w}^{(n-1)})(\tilde{\mathbf{T}}_p^{(n)}-\tilde{\mathbf{T}}_0)+\int_0^t J^P(\nabla\tilde{w}^{(n)}-\nabla\tilde{w}^{(n-1)})\tilde{\mathbf{T}}_0\\[1mm]
&\hspace{0.5cm}+\int_0^t (J^P(\tilde{X}^{(n-1)})-J^P)(\tilde{\zeta}^{(n-1)}-\mathcal{I})\nabla\tilde{w}^{(n-1)}(\tilde{\mathbf{T}}_p^{(n)}-\tilde{\mathbf{T}}_p^{(n-1)})\\[1mm]
&\hspace{0.5cm}+\int_0^t J^P(\tilde{\zeta}^{(n-1)}-\mathcal{I})\nabla\tilde{w}^{(n-1)}(\tilde{\mathbf{T}}_p^{(n)}-\tilde{\mathbf{T}}_p^{(n-1)})\\[1mm]
&\hspace{0.5cm}+\int_0^t(J^P(\tilde{X}^{(n-1)})-J^P)\nabla\tilde{w}^{(n-1)}(\tilde{\mathbf{T}}_p^{(n)}-\tilde{\mathbf{T}}_p^{(n-1)})\\[1mm]
&\hspace{0.5cm}+\int_0^t J^P\nabla\tilde{w}^{(n-1)}(\tilde{\mathbf{T}}_p^{(n)}-\tilde{\mathbf{T}}_p^{(n-1)})=\sum_{i=1}^{20} I_{1,i}.
\end{align*}
\medskip

\noindent For the estimate we focus on the most relevant integrals, as $I_{1,1}, I_{1,10}$ and $I_{1,17}$, which give us all the desired differences. We use lemma \ref{lem2}, with $\varepsilon=0$, lemma \ref{lem3} with $\gamma>1$, lemma \ref{lem5}, in order to separate each term. Then for the terms with $J^P(\tilde{X})$, we use lemma \ref{Jp-est} or lemma \ref{Jp-dif-est}. For the estimate of the velocity $\tilde{w}$ Lemma \ref{lem1} , for $\tilde{\zeta}$ lemma \ref{zeta-est} or lemma \ref{zeta-dif-est}. In the end to get the results we use the fundamental estimates for the flux and the elastic stress \eqref{flux-estim} and \eqref{elastic-estim} and lemma \ref{lem2}, with $0<\delta_i<\eta_i$, for $i=1,\ldots,8,17,\ldots,20$ and $\delta_i<\eta_i<\frac{s-1-\gamma}{2}$, for $i=9,\ldots,16$.

\begin{align*}
&\|I_{1,1}\|_{H^2_{(0)}H^{\gamma-1}}\leq \|(J^P(\tilde{X}^{(n)})-J^P(\tilde{X}^{(n-1)}))(\tilde{\zeta}^{(n)}-\mathcal{I})\nabla\tilde{w}^{(n)} (\tilde{\mathbf{T}}_p^{(n)}-\tilde{\mathbf{T}}_0)\|_{H^1_{(0)}H^{\gamma-1}}\\[1mm]
&\hspace{0.5cm}\leq \|J^P(\tilde{X}^{(n)})-J^P(\tilde{X}^{(n-1)})\|_{H^1_{(0)}H^{\gamma}}\|(\tilde{\zeta}^{(n)}-\mathcal{I})\nabla\tilde{w}^{(n)} (\tilde{\mathbf{T}}_p^{(n)}-\tilde{\mathbf{T}}_0)\|_{H^1_{(0)}H^{\gamma-1}}\\[1mm]
&\hspace{0.5cm}\leq C(\tilde{v}_0)\|\tilde{X}^{(n)}-\tilde{X}^{(n-1)}\|_{H^1_{(0)}H^{\gamma}}\|\tilde{\zeta}^{(n)}-\mathcal{I}\|_{H^1_{(0)}H^{\gamma-1}}\|\tilde{w}^{(n)}\|_{H^1_{(0)}H^{\gamma}}\|\tilde{\mathbf{T}}_p^{(n)}-\tilde{\mathbf{T}}_0\|_{H^1_{(0)}H^{\gamma-1}}\\[1mm]
&\hspace{0.5cm}\leq C(\tilde{v}_0, \tilde{\mathbf{T}}_0, \kappa)\left(1+\frac{1}{\We}\right)\left\|\int_0^t \partial_t(\tilde{X}^{(n)}-\tilde{X}^{(n-1)})\right\|_{H^{1+\eta_1-\delta_1}_{(0)}H^{\gamma}}\\[1mm]
&\hspace{0.5cm}\leq  C(\tilde{v}_0, \tilde{\mathbf{T}}_0, \kappa)\left(1+\frac{1}{\We}\right) T^{\delta_1}\|\tilde{X}^{(n)}-\tilde{X}^{(n-1)}\|_{H^{1+\eta_1}_{(0)}H^{\gamma}}\\[1mm]
&\hspace{0.5cm}\leq  C(\tilde{v}_0, \tilde{\mathbf{T}}_0, \kappa)\left(1+\frac{1}{\We}\right) T^{\delta_1}\|\tilde{X}^{(n)}-\tilde{X}^{(n-1)}\|_{\mathcal{F}^{s+1,\gamma}},\\[4mm]
&\|I_{1,10}\|_{H^2_{(0)}H^{\gamma-1}}\leq \|(J^P(\tilde{X}^{(n-1)})-J^P)(\tilde{\zeta}^{(n-1)}-\mathcal{I})(\nabla\tilde{w}^{(n)}-\nabla\tilde{w}^{(n-1)})\tilde{\mathbf{T}}_0\|_{H^1_{(0)}H^{\gamma-1}}\\[1mm]
&\hspace{0.5cm}\leq \|J^P(\tilde{X}^{(n-1)})-J^P\|_{H^1_{(0)}H^{\gamma}}\|(\tilde{\zeta}^{(n-1)}-\mathcal{I})(\nabla\tilde{w}^{(n)}-\nabla\tilde{w}^{(n-1)})\tilde{\mathbf{T}}_0\|_{H^1_{(0)}H^{\gamma-1}}\\[1mm]
&\hspace{0.5cm}\leq C(\tilde{v}_0)\|\tilde{w}^{(n)}-\tilde{w}^{(n-1)}\|_{H^{1}_{(0)}H^{\gamma}}\|(\tilde{\zeta}^{(n-1)}-\mathcal{I})\tilde{\mathbf{T}}_0\|_{H^1_{(0)}H^{\gamma-1}}\\[1mm]
&\hspace{0.5cm}\leq C(\tilde{v}_0)\|\tilde{w}^{(n)}-\tilde{w}^{(n-1)}\|_{H^{1}_{(0)}H^{\gamma}}\|\tilde{\zeta}^{(n-1)}-\mathcal{I}\|_{H^1_{(0)}H^{\gamma-1}}\|\tilde{\mathbf{T}}_0\|_{H^{\gamma}}\\[1mm]
&\hspace{0.5cm}\leq C(\tilde{v}_0, \tilde{\mathbf{T}}_0)\left\|\int_0^t \partial_t(\tilde{w}^{(n)}-\tilde{w}^{(n-1)})\right\|_{H^{1+\eta_{10}-\delta_{10}}_{(0)}H^{\gamma}}\\[1mm]
&\hspace{0.5cm}\leq C(\tilde{v}_0, \tilde{\mathbf{T}}_0) T^{\delta_{10}}\|\tilde{w}^{(n)}-\tilde{w}^{(n-1)}\|_{H^{1+\eta_{10}}_{(0)}H^{\gamma}}\\[1mm]
&\hspace{0.5cm}\leq C(\tilde{v}_0, \tilde{\mathbf{T}}_0) T^{\delta_{10}}\|\tilde{w}^{(n)}-\tilde{w}^{(n-1)}\|_{\mathcal{K}^{s+1}_{(0)}},\\[4mm]
&\|I_{1,17}\|_{H^2_{(0)}H^{\gamma-1}}\leq\|(J^P(\tilde{X}^{(n-1)})-J^P)(\tilde{\zeta}^{(n-1)}-\mathcal{I})\nabla\tilde{w}^{(n-1)}(\tilde{\mathbf{T}}_p^{(n)}-\tilde{\mathbf{T}}_p^{(n-1)})\|_{H^1_{(0)}H^{\gamma-1}}\\[1mm]
&\hspace{0.5cm}\leq \|J^P(\tilde{X}^{(n-1)})-J^P\|_{H^1_{(0)}H^{\gamma}}\|(\tilde{\zeta}^{(n-1)}-\mathcal{I})\nabla\tilde{w}^{(n-1)}(\tilde{\mathbf{T}}_p^{(n)}-\tilde{\mathbf{T}}_p^{(n-1)})\|_{H^1_{(0)}H^{\gamma-1}}\\[1mm]
&\hspace{0.5cm}\leq C(\tilde{v}_0)\|\tilde{\zeta}^{(n-1)}-\mathcal{I}\|_{H^1_{(0)}H^{\gamma-1}}\|\tilde{w}^{(n-1)}\|_{H^1_{(0)}H^{\gamma}}\|\tilde{\mathbf{T}}_p^{(n)}-\tilde{\mathbf{T}}_p^{(n-1)}\|_{H^1_{(0)}H^{\gamma-1}}\\[1mm]
&\hspace{0.5cm}\leq C(\tilde{v}_0)\left\|\int_0^t\partial_t(\tilde{\mathbf{T}}_p^{(n)}-\tilde{\mathbf{T}}_p^{(n-1)})\right\|_{H^{1+\eta_{17}-\delta_{17}}_{(0)}H^{\gamma-1}}\\[1mm]
&\hspace{0.5cm}\leq C(\tilde{v}_0)T^{\delta_{17}}\|\tilde{\mathbf{T}}_p^{(n)}-\tilde{\mathbf{T}}_p^{(n-1)})\|_{H^{1+\eta_{17}}_{(0)}H^{\gamma-1}}\leq C(\tilde{v}_0)T^{\delta_{17}}\|\tilde{\mathbf{T}}_p^{(n)}-\tilde{\mathbf{T}}_p^{(n-1)})\|_{\mathcal{F}^{s,\gamma-1}}.
\end{align*}
\medskip

\noindent For $i=2,\ldots,8$ we have $I_{1,i}\sim I_{1,1}$, for $i=9,11,\dots,16$, $I_{1,i}\sim I_{1,10}$ and finally for $i=18,19,20$ we have $I_{1,i}\sim I_{1,17}$. Now we move to the study of $I_2$ and we have the following splitting

\begin{align*}
I_{2}&=\int_0^t (J^P(\tilde{X}^{(n)})-J^P(\tilde{X}^{(n-1)}))(\tilde{\zeta}^{(n)}-\mathcal{I})\nabla\tilde{v}_0 (\tilde{\mathbf{T}}_p^{(n)}-\tilde{\mathbf{T}}_0)\\[1mm]
&+\int_0^t (J^P(\tilde{X}^{(n)})-J^P(\tilde{X}^{(n-1)}))(\tilde{\zeta}^{(n)}-\mathcal{I})\nabla\tilde{v}_0\tilde{\mathbf{T}}_0\\[1mm]
&+\int_0^t (J^P(\tilde{X}^{(n)})-J^P(\tilde{X}^{(n-1)}))\nabla\tilde{v}_0 (\tilde{\mathbf{T}}_p^{(n)}-\tilde{\mathbf{T}}_0)+\int_0^t (J^P(\tilde{X}^{(n)})-J^P(\tilde{X}^{(n-1)}))\nabla\tilde{v}_0 \tilde{\mathbf{T}}_0\\[1mm]
&+\int_0^t (J^P(\tilde{X}^{(n-1)})-J^P)(\tilde{\zeta}^{(n)}-\tilde{\zeta}^{(n-1)})\nabla\tilde{v}_0(\tilde{\mathbf{T}}_p^{(n)}-\tilde{\mathbf{T}}_0)\\[1mm]
&+\int_0^t (J^P(\tilde{X}^{(n-1)})-J^P)(\tilde{\zeta}^{(n)}-\tilde{\zeta}^{(n-1)})\nabla\tilde{v}_0\tilde{\mathbf{T}}_0\\[1mm]
&+\int_0^t J^P(\tilde{\zeta}^{(n)}-\tilde{\zeta}^{(n-1)})\nabla\tilde{v}_0 (\tilde{\mathbf{T}}_p^{(n)}-\tilde{\mathbf{T}}_0)+\int_0^t J^P(\tilde{\zeta}^{(n)}-\tilde{\zeta}^{(n-1)})\nabla\tilde{v}_0\tilde{\mathbf{T}}_0\\[1mm]
&+\int_0^t (J^P(\tilde{X}^{(n-1)})-J^P)(\tilde{\zeta}^{(n-1)}-\mathcal{I})\nabla\tilde{v}_0(\tilde{\mathbf{T}}_p^{(n)}-\tilde{\mathbf{T}}_p^{(n-1)})\\[1mm]
&+\int_0^t J^P(\tilde{\zeta}^{(n-1)}-\mathcal{I})\nabla\tilde{v}_0(\tilde{\mathbf{T}}_p^{(n)}-\tilde{\mathbf{T}}_p^{(n-1)})+\int_0^t(J^P(\tilde{X}^{(n-1)})-J^P)\nabla\tilde{v}_0(\tilde{\mathbf{T}}_p^{(n)}-\tilde{\mathbf{T}}_p^{(n-1)})\\[1mm]
&+\int_0^t J^P\nabla\tilde{v}_0(\tilde{\mathbf{T}}_p^{(n)}-\tilde{\mathbf{T}}_p^{(n-1)})=\sum_{i=1}^{12} I_{2,i}.
\end{align*}

\noindent For these integrals we observe that the results are the same as for $I_1$ but without the differences for the velocities. Then we want  to skip these estimates and to give a summary of the final result. 

\begin{align*}
&\|I_2\|_{H^2_{(0)}H^{\gamma-1}}\leq C(\tilde{v}_0,\tilde{\mathbf{T}}_0,\kappa)\left(1+\frac{1}{\We}\right)T^{\varrho_1}\left(\|\tilde{X}^{(n)}-\tilde{X}^{(n-1)}\|_{\mathcal{F}^{s+1,\gamma}}+\|\tilde{\mathbf{T}}_p^{(n)}-\tilde{\mathbf{T}}_p^{(n-1)})\|_{\mathcal{F}^{s,\gamma-1}}\right).
\end{align*}
\medskip

\noindent However we focus on the next term $I_3$ that contains the approximated velocity $t\hat{\phi}$ and we mark that in this case we use the fact that $t\in H^1_{(0)}([0,T])$. But before the estimates we make the following splitting

\begin{align*}
I_{3}&=\int_0^t (J^P(\tilde{X}^{(n)})-J^P(\tilde{X}^{(n-1)}))(\tilde{\zeta}^{(n)}-\mathcal{I})t\nabla\hat{\phi} (\tilde{\mathbf{T}}_p^{(n)}-\tilde{\mathbf{T}}_0)\\[1mm]
&+\int_0^t (J^P(\tilde{X}^{(n)})-J^P(\tilde{X}^{(n-1)}))(\tilde{\zeta}^{(n)}-\mathcal{I})t\nabla\hat{\phi} \tilde{\mathbf{T}}_0\\[1mm]
&+\int_0^t (J^P(\tilde{X}^{(n)})-J^P(\tilde{X}^{(n-1)}))t\nabla\hat{\phi}  (\tilde{\mathbf{T}}_p^{(n)}-\tilde{\mathbf{T}}_0)+\int_0^t (J^P(\tilde{X}^{(n)})-J^P(\tilde{X}^{(n-1)}))t\nabla\hat{\phi}  \tilde{\mathbf{T}}_0
\end{align*}
\begin{align*}
&+\int_0^t (J^P(\tilde{X}^{(n-1)})-J^P)(\tilde{\zeta}^{(n)}-\tilde{\zeta}^{(n-1)})t\nabla\hat{\phi} (\tilde{\mathbf{T}}_p^{(n)}-\tilde{\mathbf{T}}_0)\\[1mm]
&+\int_0^t (J^P(\tilde{X}^{(n-1)})-J^P)(\tilde{\zeta}^{(n)}-\tilde{\zeta}^{(n-1)})t\nabla\hat{\phi} \tilde{\mathbf{T}}_0\\[1mm]
&+\int_0^t J^P(\tilde{\zeta}^{(n)}-\tilde{\zeta}^{(n-1)})t\nabla\hat{\phi}  (\tilde{\mathbf{T}}_p^{(n)}-\tilde{\mathbf{T}}_0)+\int_0^t J^P(\tilde{\zeta}^{(n)}-\tilde{\zeta}^{(n-1)})t\nabla\hat{\phi} \tilde{\mathbf{T}}_0\\[1mm]
&+\int_0^t (J^P(\tilde{X}^{(n-1)})-J^P)(\tilde{\zeta}^{(n-1)}-\mathcal{I})t\nabla\hat{\phi} (\tilde{\mathbf{T}}_p^{(n)}-\tilde{\mathbf{T}}_p^{(n-1)})\\[1mm]
&+\int_0^t J^P(\tilde{\zeta}^{(n-1)}-\mathcal{I})t\nabla\hat{\phi} (\tilde{\mathbf{T}}_p^{(n)}-\tilde{\mathbf{T}}_p^{(n-1)})\\[1mm]
&+\int_0^t(J^P(\tilde{X}^{(n-1)})-J^P)t\nabla\hat{\phi} (\tilde{\mathbf{T}}_p^{(n)}-\tilde{\mathbf{T}}_p^{(n-1)})+\int_0^t J^Pt\nabla\hat{\phi} (\tilde{\mathbf{T}}_p^{(n)}-\tilde{\mathbf{T}}_p^{(n-1)})=\sum_{i=1}^{12} I_{3,i}.
\end{align*}
\medskip

\noindent We estimate one term for each difference $I_{3,1}, I_{3,9}$. We use lemma \ref{lem2} with $\varepsilon=0$, lemma \ref{lem3}, lemma \ref{lem5}. Then lemma \ref{Jp-est} or lemma \ref{Jp-dif-est} and lemma \ref{zeta-est} or lemma \ref{zeta-dif-est}, with the estimates \eqref{flux-estim} and \eqref{elastic-estim}. Finally, to conclude lemma \ref{lem2} with $0<\delta_i<\beta_i$, for $i=1,12$.

\begin{align*}
&\|I_{3,1}\|_{H^{2}_{(0)}H^{\gamma-1}}\leq \|(J^P(\tilde{X}^{(n)})-J^P(\tilde{X}^{(n-1)}))(\tilde{\zeta}^{(n)}-\mathcal{I})t\nabla\hat{\phi} (\tilde{\mathbf{T}}_p^{(n)}-\tilde{\mathbf{T}}_0)\|_{H^{1}_{(0)}H^{\gamma-1}}\\[1mm]
&\hspace{0.5cm}\leq \|J^P(\tilde{X}^{(n)})-J^P(\tilde{X}^{(n-1)})\|_{H^{1}_{(0)}H^{\gamma}} \|(\tilde{\zeta}^{(n)}-\mathcal{I})t\nabla\hat{\phi} (\tilde{\mathbf{T}}_p^{(n)}-\tilde{\mathbf{T}}_0)\|_{H^{1}_{(0)}H^{\gamma-1}}\\[1mm]
&\hspace{0.5cm}\leq C(\tilde{v}_0)\|\tilde{X}^{(n)}-\tilde{X}^{(n-1)}\|_{H^{1}_{(0)}H^{\gamma}}\|\tilde{\zeta}^{(n)}-\mathcal{I}\|_{H^{1}_{(0)}H^{\gamma-1}}\|t\nabla\hat{\phi}\|_{H^{1}_{(0)}H^{\gamma-1}}\|\tilde{\mathbf{T}}_p^{(n)}-\tilde{\mathbf{T}}_0\|_{H^{1}_{(0)}H^{\gamma-1}}\\[1mm]
&\hspace{0.5cm}\leq C(\tilde{v}_0,\tilde{\mathbf{T}}_0,\kappa)\left(1+\frac{1}{\We}\right)\|\tilde{X}^{(n)}-\tilde{X}^{(n-1)}\|_{H^{1}_{(0)}H^{\gamma}} \|t\|_{H^{1}_{(0)}}\|\hat{\phi}\|_{H^{\gamma}}\\[1mm]
&\hspace{0.5cm}\leq C(\tilde{v}_0,\tilde{\mathbf{T}}_0,\re,\kappa)\left(1+\frac{1}{\We}\right) T^{\frac{1}{2}}\left\|\int_0^t\partial_t(\tilde{X}^{(n)}-\tilde{X}^{(n-1)})\right\|_{H^{1+\eta_1-\beta_1}_{(0)}H^{\gamma}}\\[1mm]
&\hspace{0.5cm}\leq C(\tilde{v}_0,\tilde{\mathbf{T}}_0,\re,\kappa)\left(1+\frac{1}{\We}\right) T^{\beta_1}\|\tilde{X}^{(n)}-\tilde{X}^{(n-1)}\|_{H^{1+\eta_1}_{(0)}H^{\gamma}}\\[1mm]
&\hspace{0.5cm}\leq C(\tilde{v}_0,\tilde{\mathbf{T}}_0,\re,\kappa)\left(1+\frac{1}{\We}\right) T^{\beta_1}\|\tilde{X}^{(n)}-\tilde{X}^{(n-1)}\|_{\mathcal{F}^{s+1,\gamma}},\\[4mm]
&\|I_{3,9}\|_{H^{2}_{(0)}H^{\gamma-1}}\leq\|(J^P(\tilde{X}^{(n-1)})-J^P)(\tilde{\zeta}^{(n-1)}-\mathcal{I})t\nabla\hat{\phi} (\tilde{\mathbf{T}}_p^{(n)}-\tilde{\mathbf{T}}_p^{(n-1)})\|_{H^1_{(0)}H^{\gamma-1}}\\[1mm]
&\hspace{0.5cm}\leq \|J^P(\tilde{X}^{(n-1)})-J^P\|_{H^1_{(0)}H^{\gamma}}\|(\tilde{\zeta}^{(n-1)}-\mathcal{I})t\nabla\hat{\phi} (\tilde{\mathbf{T}}_p^{(n)}-\tilde{\mathbf{T}}_p^{(n-1)})\|_{H^1_{(0)}H^{\gamma-1}}\\[1mm]
&\hspace{0.5cm}\leq C(\tilde{v}_0)\|\tilde{\zeta}^{(n-1)}-\mathcal{I}\|_{H^1_{(0)}H^{\gamma-1}}\|t\|_{H^1_{(0)}}\|\hat{\phi}\|_{H^{\gamma}}
\|\tilde{\mathbf{T}}_p^{(n)}-\tilde{\mathbf{T}}_p^{(n-1)}\|_{H^1_{(0)}H^{\gamma-1}}\\[1mm]
&\hspace{0.5cm}\leq C(\tilde{v}_0,\tilde{\mathbf{T}}_0,\re,\kappa)\left\|\int_0^t\partial_t(\tilde{\mathbf{T}}_p^{(n)}-\tilde{\mathbf{T}}_p^{(n-1)})\right\|_{H^{1+\eta_9-\beta_9}_{(0)}H^{\gamma-1}}\\[1mm]
&\hspace{0.5cm}\leq C(\tilde{v}_0,\tilde{\mathbf{T}}_0,\re,\kappa)T^{\beta_9}\|\tilde{\mathbf{T}}_p^{(n)}-\tilde{\mathbf{T}}_p^{(n-1)}\|_{H^{1+\eta_9}_{(0)}H^{\gamma-1}}\\[1mm]
&\hspace{0.5cm}\leq C(\tilde{v}_0,\tilde{\mathbf{T}}_0,\re,\kappa)T^{\beta_9}\|\tilde{\mathbf{T}}_p^{(n)}-\tilde{\mathbf{T}}_p^{(n-1)}\|_{\mathcal{F}^{s,\gamma}}.
\end{align*}

\noindent We avoid to write the splitting of $I_4, I_5$ and $I_6$, that are similar to $I_1, I_2$ and $I_3$ respectively. But we summarize the resulting estimates.

\begin{align*}
&\|I_i\|_{H^{2}_{(0)}H^{\gamma-1}}\leq C(\tilde{v}_0,\tilde{\mathbf{T}}_0,\kappa)\left(1+\frac{1}{\We}\right)T^{\varrho_{i-2}}\left(\|\tilde{X}^{(n)}-\tilde{X}^{(n-1)}\|_{\mathcal{F}^{s+1,\gamma}}\right.\\[1mm]
&\hspace{1cm}\left.+\|\tilde{w}^{(n)}-\tilde{w}^{(n-1)}\|_{\mathcal{K}^{s+1}_{(0)}}+\|\tilde{\mathbf{T}}_p^{(n)}-\tilde{\mathbf{T}}_p^{(n-1)}\|_{\mathcal{F}^{s,\gamma}}\right),\hspace{0.3cm}\textrm{for}\hspace{0.3cm} i=4,5,6.
\end{align*}
\medskip

\noindent The estimate for $I_7$ is immediate by lemma  \ref{lem2}, with $\varepsilon=0$ and again lemma \ref{lem2} with $0<\varrho_5<\eta$.

\begin{align*}
&\|I_7\|_{H^2_{(0)}H^{\gamma-1}}\leq \frac{1}{\We}\|\tilde{\mathbf{T}}_p^{(n)}-\tilde{\mathbf{T}}_p^{(n-1)}\|_{H^1_{(0)}H^{\gamma-1}}\leq \frac{1}{\We}\left\|\int_0^t\partial_t(\tilde{\mathbf{T}}_p^{(n)}-\tilde{\mathbf{T}}_p^{(n-1)})\right\|_{H^{1+\eta-\varrho_5}_{(0)}H^{\gamma-1}}\\[1mm]
&\hspace{0.5cm}\leq  \frac{1}{\We}T^{\varrho_5}\|\tilde{\mathbf{T}}_p^{(n)}-\tilde{\mathbf{T}}_p^{(n-1)}\|_{H^{1+\eta}_{(0)}H^{\gamma-1}}\leq  \frac{1}{\We}T^{\varrho_5}\|\tilde{\mathbf{T}}_p^{(n)}-\tilde{\mathbf{T}}_p^{(n-1)}\|_{\mathcal{F}^{s,\gamma-1}}.
\end{align*}
\medskip

\noindent For the last six integrals we focus on $I_8, I_9, I_{10}$. For $I_8$ we have 

\begin{align*}
I_8&=\frac{\kappa}{\We}\int_0^t (J^P(\tilde{X}^{(n)})-J^P(\tilde{X}^{(n-1)}))(\tilde{\zeta}^{(n)}-\mathcal{I})\nabla\tilde{w}^{(n)}\\[1mm]
&+\frac{\kappa}{\We}\int_0^t (J^P(\tilde{X}^{(n)})-J^P(\tilde{X}^{(n-1)}))\nabla\tilde{w}^{(n)}\\[1mm]
&+\frac{\kappa}{\We}\int_0^t (J^P(\tilde{X}^{(n-1)})-J^P)(\tilde{\zeta}^{(n)}-\tilde{\zeta}^{(n-1)})\nabla\tilde{w}^{(n)}\\[1mm]
&+\frac{\kappa}{\We}\int_0^t J^P(\tilde{\zeta}^{(n)}-\tilde{\zeta}^{(n-1)})\nabla\tilde{w}^{(n)}\\[1mm]
&+\frac{\kappa}{\We}\int_0^t(J^P(\tilde{X}^{(n-1)})-J^P)(\tilde{\zeta}^{(n-1)}-\mathcal{I})(\nabla\tilde{w}^{(n)}-\nabla\tilde{w}^{(n-1)})\\[1mm]
&+\frac{\kappa}{\We}\int_0^t J^P(\tilde{\zeta}^{(n-1)}-\mathcal{I})(\nabla\tilde{w}^{(n)}-\nabla\tilde{w}^{(n-1)})\\[1mm]
&+\frac{\kappa}{\We}\int_0^t(J^P(\tilde{X}^{(n-1)})-J^P)(\nabla\tilde{w}^{(n)}-\nabla\tilde{w}^{(n-1)})\\[1mm]
&+\frac{\kappa}{\We}\int_0^t J^P(\nabla\tilde{w}^{(n)}-\nabla\tilde{w}^{(n-1)})=\sum_{i=1}^{8} I_{8,i}.
\end{align*}

\noindent We estimate $I_{8,1}$ and $I_{8,5}$ to give the idea of the final result. We use lemma \ref{lem2} with $\varepsilon=0$, in order to separate the terms we use lemma \ref{lem3} with $\gamma>1$, lemma \ref{lem5}, lemma \ref{Jp-est} or lemma \ref{Jp-dif-est} and lemma \ref{zeta-est} or lemma \ref{zeta-dif-est} and to conclude lemma \ref{lem2}, with $0<\theta_i<\eta_i$, for $i=1,\ldots,4$ and $\theta_i<\eta_i<\frac{s-1-\gamma}{2}$, for $i=5,\ldots,8$.

\begin{align*}
&\|I_{8,1}\|_{H^2_{(0)}H^{\gamma-1}}\leq \frac{\kappa}{\We}\|(J^P(\tilde{X}^{(n)})-J^P(\tilde{X}^{(n-1)}))(\tilde{\zeta}^{(n)}-\mathcal{I})\nabla\tilde{w}^{(n)}\|_{H^1_{(0)}H^{\gamma-1}}\\[2mm]
&\hspace{0.5cm}\leq \frac{\kappa}{\We}\|J^P(\tilde{X}^{(n)})-J^P(\tilde{X}^{(n-1)})\|_{H^1_{(0)}H^{\gamma}}\|(\tilde{\zeta}^{(n)}-\mathcal{I})\nabla\tilde{w}^{(n)}\|_{H^1_{(0)}H^{\gamma-1}}\\[2mm]
&\hspace{0.5cm}\leq \frac{\kappa}{\We} \|\tilde{X}^{(n)}-\tilde{X}^{(n-1)}\|_{H^1_{(0)}H^{\gamma}}\|\tilde{\zeta}^{(n)}-\mathcal{I}\|_{H^1_{(0)}H^{\gamma}}\|\tilde{w}^{(n)}\|_{H^1_{(0)}H^{\gamma}}\\[2mm]
&\hspace{0.5cm}\leq C(\tilde{v}_0,\kappa) \frac{1}{\We}\left\|\int_0^t\partial_t(\tilde{X}^{(n)}-\tilde{X}^{(n-1)})\right\|_{H^{1+\eta_1-\theta_1}_{(0)}H^{\gamma}}\\[2mm]
&\hspace{0.5cm}\leq C(\tilde{v}_0,\kappa) \frac{1}{\We}T^{\theta_1}\|\tilde{X}^{(n)}-\tilde{X}^{(n-1)}\|_{H^{1+\eta_1}_{(0)}H^{\gamma}}\\[1mm]
&\hspace{0.5cm}\leq C(\tilde{v}_0,\kappa) \frac{1}{\We}T^{\theta_1}\|\tilde{X}^{(n)}-\tilde{X}^{(n-1)}\|_{\mathcal{F}^{s+1,\gamma}}\\[4mm]
&\|I_{8,5}\|_{H^2_{(0)}H^{\gamma-1}}\leq \frac{\kappa}{\We}\|(J^P(\tilde{X}^{(n-1)})-J^P)(\tilde{\zeta}^{(n-1)}-\mathcal{I})(\nabla\tilde{w}^{(n)}-\nabla\tilde{w}^{(n-1)})\|_{H^1_{(0)}H^{\gamma-1}}\\[2mm]
&\hspace{0.5cm}\leq \frac{\kappa}{\We}\|J^P(\tilde{X}^{(n-1)})-J^P\|_{H^1_{(0)}H^{\gamma}}\|(\tilde{\zeta}^{(n-1)}-\mathcal{I})(\nabla\tilde{w}^{(n)}-\nabla\tilde{w}^{(n-1)})\|_{H^1_{(0)}H^{\gamma-1}}\\[2mm]
&\hspace{0.5cm}\leq C(\tilde{v}_0,\kappa)\frac{1}{\We}\|\tilde{\zeta}^{(n-1)}-\mathcal{I}\|_{H^1_{(0)}H^{\gamma-1}}\|\tilde{w}^{(n)}-\tilde{w}^{(n-1)}\|_{H^1_{(0)}H^{\gamma}}\\[2mm]
&\hspace{0.5cm}\leq C(\tilde{v}_0,\kappa)\frac{1}{\We}\left\|\int_0^t\partial_t(\tilde{w}^{(n)}-\tilde{w}^{(n-1)})\right\|_{H^{1+\eta_5-\theta_5}_{(0)}H^{\gamma}}\\[2mm]
&\hspace{0.5cm}\leq C(\tilde{v}_0,\kappa)\frac{1}{\We}T^{\theta_5}\|\tilde{w}^{(n)}-\tilde{w}^{(n-1)}\|_{H^{1+\eta_5}_{(0)}H^{\gamma}}\\[1mm]
&\hspace{0.5cm}\leq C(\tilde{v}_0,\kappa)\frac{1}{\We}T^{\theta_5}\|\tilde{w}^{(n)}-\tilde{w}^{(n-1)}\|_{\mathcal{K}_{(0)}^{s+1}}.
\end{align*}
\medskip

\noindent Now, we analyze $I_9, I_{10}$.

\begin{align*}
&I_9=\frac{\kappa}{\We}\int_0^t (J^P(\tilde{X}^{(n)})-J^P(\tilde{X}^{(n-1)}))(\tilde{\zeta}^{(n)}-\mathcal{I})\nabla\tilde{v}_0+\frac{\kappa}{\We}\int_0^t (J^P(\tilde{X}^{(n)})-J^P(\tilde{X}^{(n-1)}))\nabla\tilde{v}_0\\[1mm]
&\hspace{0.5cm}+\frac{\kappa}{\We}\int_0^t (J^P(\tilde{X}^{(n-1)})-J^P)(\tilde{\zeta}^{(n)}-\tilde{\zeta}^{(n-1)})\nabla\tilde{v}_0+\frac{\kappa}{\We}\int_0^t J^P(\tilde{\zeta}^{(n)}-\tilde{\zeta}^{(n-1)})\nabla\tilde{v}_0=\sum_{i=1}^{4} I_{9,i},\\[4mm]
&I_{10}=\frac{\kappa}{\We}\int_0^t (J^P(\tilde{X}^{(n)})-J^P(\tilde{X}^{(n-1)}))(\tilde{\zeta}^{(n)}-\mathcal{I})t\nabla\hat{\phi}+\frac{\kappa}{\We}\int_0^t (J^P(\tilde{X}^{(n)})-J^P(\tilde{X}^{(n-1)}))t\nabla\hat{\phi}\\[1mm]
&\hspace{0.5cm}+\frac{\kappa}{\We}\int_0^t (J^P(\tilde{X}^{(n-1)})-J^P)(\tilde{\zeta}^{(n)}-\tilde{\zeta}^{(n-1)})t\nabla\hat{\phi}+\frac{\kappa}{\We}\int_0^t J^P(\tilde{\zeta}^{(n)}-\tilde{\zeta}^{(n-1)})t\nabla\hat{\phi}=\sum_{i=1}^{4} I_{10,i}.
\end{align*}
\medskip

\noindent By using lemma \ref{lem2}, with $\varepsilon=0$, lemma \ref{lem3}, with $\gamma>1$ and we require enough regularity fot the initial data $\tilde{v}_0$. To get the final result we use lemma \ref{lem2}, with $0<\nu_i<\eta_i$ and $0<\tau_i<\eta_i$, for $i=1,\ldots,4$

\begin{align*}
&\|I_{91}\|_{H^2_{(0)}H^{\gamma-1}}\leq \frac{\kappa}{\We}\|(J^P(\tilde{X}^{(n)})-J^P(\tilde{X}^{(n-1)}))(\tilde{\zeta}^{(n)}-\mathcal{I})\nabla\tilde{v}_0\|_{H^1_{(0)}H^{\gamma-1}}\\[2mm]
&\hspace{0.5cm}\leq\frac{\kappa}{\We}\|J^P(\tilde{X}^{(n)})-J^P(\tilde{X}^{(n-1)})\|_{H^1_{(0)}H^{\gamma}}\|(\tilde{\zeta}^{(n)}-\mathcal{I})\nabla\tilde{v}_0\|_{H^1_{(0)}H^{\gamma-1}}\\[2mm]
&\hspace{0.5cm}\leq C(\tilde{v}_0,\kappa) \frac{1}{\We}\|\tilde{X}^{(n)}-\tilde{X}^{(n-1)}\|_{H^1_{(0)}H^{\gamma}}\|\tilde{\zeta}^{(n)}-\mathcal{I}\|_{H^1_{(0)}H^{\gamma-1}}\|\tilde{v}_0\|_{H^{\gamma+1}}\\[2mm]
&\hspace{0.5cm}\leq C(\tilde{v}_0,\kappa) \frac{1}{\We}\left\|\int_0^t\partial_t(\tilde{X}^{(n)}-\tilde{X}^{(n-1)})\right\|_{H^{1+\eta_1-\nu_1}_{(0)}H^{\gamma}}\\[2mm]
&\hspace{0.5cm}\leq C(\tilde{v}_0,\kappa) \frac{1}{\We} T^{\nu_1}\|\tilde{X}^{(n)}-\tilde{X}^{(n-1)}\|_{H^{1+\eta_1}_{(0)}H^{\gamma}}\\[2mm]
&\hspace{0.5cm}\leq C(\tilde{v}_0,\kappa) \frac{1}{\We} T^{\nu_1}\|\tilde{X}^{(n)}-\tilde{X}^{(n-1)}\|_{\mathcal{F}^{s+1,\gamma}}\\[4mm]
&\|I_{10,1}\|_{H^2_{(0)}H^{\gamma-1}}\leq \frac{\kappa}{\We}\|(J^P(\tilde{X}^{(n)})-J^P(\tilde{X}^{(n-1)}))(\tilde{\zeta}^{(n)}-\mathcal{I})t\nabla\hat{\phi}\|_{H^1_{(0)}H^{\gamma-1}}\\[2mm]
&\hspace{0.5cm}\leq \frac{\kappa}{\We}\|(J^P(\tilde{X}^{(n)})-J^P(\tilde{X}^{(n-1)}))(\tilde{\zeta}^{(n)}-\mathcal{I})t\nabla\hat{\phi}\|_{H^1_{(0)}H^{\gamma-1}}\\[2mm]
&\hspace{0.5cm}\leq\frac{\kappa}{\We}\|J^P(\tilde{X}^{(n)})-J^P(\tilde{X}^{(n-1)})\|_{H^1_{(0)}H^{\gamma}}\|(\tilde{\zeta}^{(n)}-\mathcal{I})t\nabla\hat{\phi}\|_{H^1_{(0)}H^{\gamma-1}}\\[2mm]
&\hspace{0.5cm}\leq C(\tilde{v}_0,\kappa) \frac{1}{\We}\|\tilde{X}^{(n)}-\tilde{X}^{(n-1)}\|_{H^1_{(0)}H^{\gamma}}\|\tilde{\zeta}^{(n)}-\mathcal{I}\|_{H^1_{(0)}H^{\gamma-1}}\|t\|_{H^1_{(0)}}\|\hat{\phi}\|_{H^{\gamma+1}}\\[2mm]
&\hspace{0.5cm}\leq C(\tilde{v}_0,\tilde{\mathbf{T}}_0,\re,\kappa) \frac{1}{\We}\left\|\int_0^t\partial_t(\tilde{X}^{(n)}-\tilde{X}^{(n-1)})\right\|_{H^{1+\eta_1-\tau_1}_{(0)}H^{\gamma}}\\[2mm]
&\hspace{0.5cm}\leq C(\tilde{v}_0,\kappa) \frac{1}{\We} T^{\tau_1}\|\tilde{X}^{(n)}-\tilde{X}^{(n-1)}\|_{H^{1+\eta_1}_{(0)}H^{\gamma}}\\[1mm]
&\hspace{0.5cm}\leq C(\tilde{v}_0,\kappa) \frac{1}{\We} T^{\nu_1}\|\tilde{X}^{(n)}-\tilde{X}^{(n-1)}\|_{\mathcal{F}^{s+1,\gamma}}.
\end{align*}
\medskip

\noindent For the remaining terms we have similar estimates, precisely $\|I_{9,i}\|_{H^2_{(0)}H^{\gamma-1}}\leq C(\tilde{v}_0,\kappa) \frac{1}{\We} T^{\nu_i}\|\tilde{X}^{(n)}-\tilde{X}^{(n-1)}\|_{\mathcal{F}^{s+1,\gamma}}$ and $\|I_{10,i}\|_{H^2_{(0)}H^{\gamma-1}}\leq C(\tilde{v}_0,\kappa) \frac{1}{\We} T^{\tau_i}\cdot\|\tilde{X}^{(n)}-\tilde{X}^{(n-1)}\|_{\mathcal{F}^{s+1,\gamma}}$, for $i=1,\ldots,4$. The proof of the second part of the proposition follows by choosing $\delta=\min\{\frac{1}{4},\delta_i, \varrho_m,\beta_j, \theta_k,\nu_l, \tau_l\}$, for $i=1,\ldots,20$, for $m=1,\ldots,5$, for $j=1,\ldots,12$, for $k=1,\ldots,8$ and  for $l=1,\ldots,4$.
\end{proof}

\subsection{Proof of Proposition \ref{fixed point}}
 
\noindent In order to prove the local existence theorem \ref{local-existence conf-lag} we have to put together the results obtained in Proposition \ref{estimate-conf-lag-(v,q)}, Proposition \ref{estimate-conf-lag-flux} and Proposition \ref{We-G-conf-lag-estimate}.

\begin{align*}
\bullet&\left\|\tilde{w}^{(n+1)}-\tilde{w}^{(n)}\right\|_{\mathcal{K}^{s+1}_{(0)}}+\left\|\tilde{q}_w^{(n+1)}-\tilde{q}_w^{(n)}\right\|_{\mathcal{K}^{s}_{pr(0)}} \leq C(\tilde{v}_0,\tilde{\mathbf{T}}_0,\re,\kappa)\left(1+\frac{1}{\We}\right)T^{\varrho}\\[1mm]
&\cdot\left(\left\|\tilde{X}^{(n)}-\tilde{X}^{(n-1)}\right\|_{\mathcal{F}^{s+1,\gamma}}+\left\|\tilde{w}^{(n)}-\tilde{w}^{(n-1)}\right\|_{\mathcal{K}^{s+1}_{(0)}}+\left\|\tilde{q}_w^{(n)}-\tilde{q}_w^{(n-1)}\right\|_{\mathcal{K}^{s}_{pr(0)}}\right.\\[1mm]
&\hspace{1cm}\left.+\left\|\tilde{\mathbf{T}}_p^{(n)}-\tilde{\mathbf{T}}_p^{(n-1)}\right\|_{\mathcal{F}^{s,\gamma-1}}\right),\\[4mm]
\bullet&\left\|\tilde{X}^{(n+1)}- \tilde{X}^{(n)}\right\|_{\mathcal{F}^{s+1,\gamma}}\leq C(\tilde{v}_0)T^{\eta}\left(\|\tilde{w}^{(n)}-\tilde{w}^{(n-1)}\|_{\mathcal{K}^{s+1}_{(0)}}\right.\\[1mm]
&\hspace{1cm}\left.+\left\|\tilde{X}^{(n)}- \tilde{X}^{(n-1)}\right\|_{\mathcal{F}^{s+1,\gamma}}\right),\\[4mm]
\bullet&\left\|\tilde{\mathbf{T}}_p^{(n+1)}-\tilde{\mathbf{T}}_p^{(n)}\right\|_{\mathcal{F}^{s,\gamma-1}}\leq C(\tilde{v}_0,\tilde{\mathbf{T}}_0,\re,\kappa)\left(1+\frac{1}{\We}\right)T^{\delta}\\[1mm]
&\cdot\left(\left\|\tilde{X}^{(n)}-\tilde{X}^{(n-1)}\right\|_{\mathcal{F}^{s+1,\gamma}}+\left\|\tilde{w}^{(n)}-\tilde{w}^{(n-1)}\right\|_{\mathcal{K}^{s+1}_{(0)}}+\left\|\tilde{\mathbf{T}}_p^{(n)}-\tilde{\mathbf{T}}_p^{(n-1)}\right\|_{\mathcal{F}^{s,\gamma-1}}\right).
\end{align*}
\medskip

\noindent By summing all these results an by taking $\mu=\min\{\varrho,\eta,\beta\}$ we have

\begin{align*}
&\left\|\tilde{w}^{(n+1)}-\tilde{w}^{(n)}\right\|_{\mathcal{K}^{s+1}_{(0)}}+\left\|\tilde{q}_w^{(n+1)}-\tilde{q}_w^{(n)}\right\|_{\mathcal{K}^{s}_{pr(0)}}+\left\|\tilde{X}^{(n+1)}- \tilde{X}^{(n)}\right\|_{\mathcal{F}^{s+1,\gamma}}+\left\|\tilde{\mathbf{T}}_p^{(n+1)}-\tilde{\mathbf{T}}_p^{(n)}\right\|_{\mathcal{F}^{s,\gamma-1}}\\[2mm]
&\leq \tilde{C}(\tilde{v}_0,\tilde{\mathbf{T}}_0,\re,\kappa)\left(1+\frac{1}{\We}\right)T^{\mu}\left(\left\|\tilde{w}^{(n)}-\tilde{w}^{(n-1)}\right\|_{\mathcal{K}^{s+1}_{(0)}}+\left\|\tilde{q}_w^{(n)}-\tilde{q}_w^{(n-1)}\right\|_{\mathcal{K}^{s}_{pr(0)}}\right.\\[2mm]
&\left.+\left\|\tilde{X}^{(n)}-\tilde{X}^{(n-1)}\right\|_{\mathcal{F}^{s+1,\gamma}}+\left\|\tilde{\mathbf{T}}_p^{(n)}-\tilde{\mathbf{T}}_p^{(n-1)}\right\|_{\mathcal{F}^{s,\gamma-1}}\right).
\end{align*}
\medskip

\noindent Then in order to have the contraction we impose

\begin{equation}\label{conf-lag-time}
\displaystyle T<\left(\frac{\We}{\tilde{C}(\tilde{v}_0,\tilde{\mathbf{T}}_0,\re,\kappa)(1+\We)}\right)^{\frac{1}{\mu}}.
\end{equation}
\bigskip

\section{Stability results for \eqref{undimsys}}\label{sec:3}

\noindent The existence of splash singularity is a consequence of the stability result, as we will explain later. Thus we introduce a one parameter family $\tilde{\Omega}_{\varepsilon}(0)$ of initial domains, defined as follows 

$$\tilde{\Omega}_{\varepsilon}(0)=\tilde{\Omega}_0+\varepsilon b,$$

\noindent where $b$ is a unit vector, such that $P^{-1}(\tilde{\Omega}_{\varepsilon}(0))$ is a regular domain, see \ref{fig:1}(a). We consider also a perturbation of the velocity $\tilde{v}_{\varepsilon}(0)$, which has a positive normal component at the splash points. In a rough way the stability results can be resumed as follows.

\begin{equation}\label{closeness}
\textrm{dist}(\tilde{\Omega}_{\varepsilon}(t),\tilde{\Omega}(t))\leq\varepsilon\quad\textrm{hence}\quad \textrm{dist}(P^{-1}(\tilde{\Omega}_{\varepsilon}(t)),P^{-1}(\tilde{\Omega}(t)))\leq\varepsilon,
\end{equation}

\noindent for sufficiently small $\varepsilon$. In particular to deduce \eqref{closeness}, we have to prove the following theorem, related to the flux, since it governs the evolution of the interface.

\begin{theorem}\label{flux-stab}
Let $2<s<\frac{5}{2}$ and a suitable $\delta>0$. If $\displaystyle T<\left(\frac{\We}{C(\tilde{v}_0,\tilde{\mathbf{T}}_0,\re,\kappa)(1+\We)}\right)^{\frac{1}{\delta}}$ then
$$\|\tilde{X}-\tilde{X}_{\varepsilon}\|_{L^{\infty}H^{s+1}}\leq 3C(\tilde{v}_0,\tilde{\mathbf{T}}_0,\re,\kappa)\left(1+\frac{1}{\We}\right)\varepsilon.$$
\end{theorem}

\noindent To obtain theorem \ref{flux-stab}, we take the following differences

\begin{equation}\label{w-stab}
\left\{\begin{array}{lll}
\displaystyle\re\hspace{0.1cm}\partial_t(\tilde{w}-\tilde{w}_{\varepsilon})-(1-\kappa)Q^2\Delta(\tilde{w}-\tilde{w}_{\varepsilon})+(J^P)^T\nabla(\tilde{q}_w-\tilde{q}_{w,\varepsilon})=\tilde{F}_{\varepsilon}\\ [3mm]
\displaystyle \trace(\nabla(\tilde{w}-\tilde{w}_{\varepsilon})J^P)=\tilde{K}_{\varepsilon}\\ [3mm]
[-(\tilde{q}_w-\tilde{q}_{w,\varepsilon})\mathcal{I}+(1-\kappa)\left(\nabla(\tilde{w}-\tilde{w}_{\varepsilon})J^P+(\nabla(\tilde{w}-\tilde{w}_{\varepsilon})J^P)^T\right)](J^P)^{-1}\tilde{n}_0=\tilde{H}_{\varepsilon}\\[3mm]
\displaystyle \tilde{w}_0-\tilde{w}_{\varepsilon,0}=0,
\end{array}\right.
\end{equation}

\noindent where
\begin{align*}
&\tilde{F}_{\varepsilon}=\tilde{f}-\tilde{f}_{\varepsilon}+\tilde{f}_{\phi}^L-\tilde{f}_{\phi,\varepsilon}^L+(1-\kappa)(Q^2-Q^2_{\varepsilon})\Delta \tilde{w}_{\varepsilon}-((J^P)^T-(J_{\varepsilon}^P)^T)\nabla \tilde{q}_{w,\varepsilon},\\[2mm]
&\tilde{K}_{\varepsilon}=\tilde{g}-\tilde{g}_{\varepsilon}+\tilde{g}_{\phi}^L-\tilde{g}_{\phi,\varepsilon}^L-\trace(\nabla\tilde{w}_{\varepsilon}(J^P-J^P_{\varepsilon})),\\[2mm]
&\tilde{H_{\varepsilon}}=\tilde{h}-\tilde{h}_{\varepsilon}+\tilde{h}_{\phi}^L-\tilde{h}_{\phi,\varepsilon}^L+\tilde{q}_{w,\varepsilon}((J^P)^{-1}-(J^P_{\varepsilon})^{-1})\tilde{n}_0\\[1mm]
&\hspace{0.5cm}-(1-\kappa)(\nabla\tilde{w}_{\varepsilon}J^P)(J^P)^{-1}\tilde{n}_0-(1-\kappa)(\nabla\tilde{w}_{\varepsilon}J^P)^T (J^P)^{-1}\tilde{n}_0\\[1mm]
&\hspace{0.5cm}+(1-\kappa)(\nabla\tilde{w}_{\varepsilon}J^P_{\varepsilon})(J^P_{\varepsilon})^{-1}\tilde{n}_0+(1-\kappa)(\nabla\tilde{w}_{\varepsilon}J^P_{\varepsilon})^T(J^P_{\varepsilon})^{-1}\tilde{n}_0,
\end{align*}

\noindent with 

\begin{align*}
&\tilde{f}_{\phi}^L-\tilde{f}_{\phi,\varepsilon}^L=-\re\hspace{0.2cm}\frac{d}{dt}(\phi-\phi_{\varepsilon})+(1-\kappa)Q^2\Delta\phi-(1-\kappa)Q^2_{\varepsilon}\Delta\phi_{\varepsilon}-(J^P)^T\nabla \tilde{q}_{\phi}+(J^P_{\varepsilon})^T\nabla \tilde{q}_{\phi,\varepsilon},\\[2mm]
&\tilde{g}_{\phi}^L-\tilde{g}_{\phi,\varepsilon}^L=-\trace(\nabla\phi J^P)+\trace(\nabla\phi_{\varepsilon}J^P_{\varepsilon}),\\[2mm]
&\tilde{h}_{\phi}^L-\tilde{h}_{\phi,\varepsilon}^L=\tilde{q}_{\phi}(J^P)^{-1}n_0-\tilde{q}_{\phi,\varepsilon}(J^P_{\varepsilon})^{-1}n_0-(1-\kappa)[(\nabla\phi J^P)+(\nabla\phi J^P)^T](J^P)^{-1}n_0\\[2mm]
&\hspace{1.8cm}+(1-\kappa)[(\nabla\phi_{\varepsilon}J^P_{\varepsilon})
+(\nabla\phi_{\varepsilon} J^P_{\varepsilon})^T](J^P_{\varepsilon})^{-1}n_0.
\end{align*}
\medskip

 \noindent The function $\phi_{\varepsilon}=\tilde{v}_0+\frac{1}{\re}t\left((1-\kappa)Q^2_{\varepsilon}\Delta\tilde{v}_0-(J_{\varepsilon}^P)^T\nabla q_{\phi,\varepsilon}+\trace(\nabla\tilde{\mathbf{T}}_{0,\varepsilon}J^P_{\varepsilon}\right)=\tilde{v}_0+t\hat{\phi}_{\varepsilon}$ allows us to invert the operator $L$ defined in \eqref{L} and $\tilde{f}$, $\tilde{f}_{\varepsilon}$, $\tilde{g}$, $\tilde{g}_{\varepsilon}$, $\tilde{h}$, $\tilde{h}_{\varepsilon}$ are the same defined in Section \ref{sec:2} as $\tilde{f}^{(n)}$, $\tilde{g}^{(n)}$ and $\tilde{h}^{(n)}$. For the flux we have that $\tilde{X}_{\varepsilon}(t,\tilde{\alpha})$ satisfies

\begin{equation}
\left\{\begin{array}{lll}
\displaystyle \frac{d}{dt} \tilde{X}_{\varepsilon}(t,\tilde{\alpha})=J^P(\tilde{X}_{\varepsilon}(t,\tilde{\alpha}))\tilde{ v}_{\varepsilon}(t,\tilde{\alpha})\\[3mm]
\displaystyle\tilde{ X}_{\varepsilon}(0,\tilde{\alpha})=\tilde{\alpha}+\varepsilon b,
\end{array}\right.
\end{equation}
\medskip

\noindent and so 

$$\tilde{X}(t,\tilde{\alpha})-\tilde{X}_{\varepsilon}(t,\tilde{\alpha})=-b\varepsilon+\int_0^t \left(J^P( \tilde{X}) \tilde{v}-J^P(\tilde{X}_{\varepsilon}) \tilde{v}_{\varepsilon}\right)(t,\tilde{\alpha})\,d\tau.$$

\medskip

\noindent The perturbed elastic stress tensor $\tilde{\mathbf{T}}_{p,\varepsilon}$ satisfies the following ODE.
\medskip

\begin{equation}\label{We-pert-G}
\left\{\begin{array}{lll}
\displaystyle\partial_t \tilde{\mathbf{T}}_{p,\varepsilon}-J^P(\tilde{X}_{\varepsilon})\tilde{\zeta}_{\varepsilon}\nabla\tilde{v}_{\varepsilon} \tilde{\mathbf{T}}_{p,\varepsilon}-\tilde{\mathbf{T}}_{p,\varepsilon}\left(J^P(\tilde{X}_{\varepsilon})\tilde{\zeta}_{\varepsilon}\nabla\tilde{v}_{\varepsilon}\right)^T=\\[2mm]
\displaystyle-\frac{\tilde{\mathbf{T}}_{p,\varepsilon}}{\We}
+\frac{\kappa\left(J^P(\tilde{X}_{\varepsilon})\tilde{\zeta}_{\varepsilon}\nabla\tilde{v}_{\varepsilon}+\left(J^P(\tilde{X}_{\varepsilon})\tilde{\zeta}_{\varepsilon}\nabla \tilde{v}_{\varepsilon}\right)^T\right)}{\We}\\[5mm]
\tilde{\mathbf{T}}_{p,\varepsilon}(0, \tilde{\alpha})= \tilde{\mathbf{T}}_0.
\end{array}\right.
\end{equation}
\medskip

\noindent hence
\begin{equation}\label{We(G-Geps)}
\begin{split}
&\tilde{\mathbf{T}}_p-\tilde{\mathbf{T}}_{p,\varepsilon}=\int_0^t\left( J^P(\tilde{X})\tilde{\zeta}\nabla\tilde{v} \tilde{\mathbf{T}}_p-J^P(\tilde{X}_{\varepsilon})\tilde{\zeta}_{\varepsilon}\nabla\tilde{v}_{\varepsilon} \tilde{\mathbf{T}}_{p,\varepsilon}\right)\\[2mm]
&\hspace{0.5cm}+\int_0^t\left(\tilde{\mathbf{T}}_p\left(J^P(\tilde{X})\tilde{\zeta}\nabla\tilde{v}\right)^T-\tilde{\mathbf{T}}_{p,\varepsilon}\left(J^P(\tilde{X}_{\varepsilon})\tilde{\zeta}_{\varepsilon}\nabla\tilde{v}_{\varepsilon}\right)^T\right)\\[2mm]
&\hspace{0.5cm}-\frac{1}{\We}\int_0^t(\tilde{\mathbf{T}}_p-\tilde{\mathbf{T}}_{p,\varepsilon})+\frac{\kappa}{\We}\int_0^t \left(J^P(\tilde{X})\tilde{\zeta}\nabla\tilde{v}-J^P(\tilde{X}_{\varepsilon})\tilde{\zeta}_{\varepsilon}\nabla\tilde{v}_{\varepsilon}\right)\\[2mm]
&\hspace{0.5cm}+\frac{\kappa}{\We}\int_0^t\left(\left(J^P(\tilde{X})\tilde{\zeta}\nabla \tilde{v}\right)^T-\left(J^P(\tilde{X}_{\varepsilon})\tilde{\zeta}_{\varepsilon}\nabla \tilde{v}_{\varepsilon}\right)^T\right)\\[2mm]
&\hspace{0.5cm}=I_1+I_2+I_3+I_4+I_5.
\end{split}
\end{equation}
\medskip

\noindent The proof of theorem \ref{flux-stab} is a consequence of the following result.

\begin{proposition}\label{stability}
For $2<s<\frac{5}{2}$, given the initial data $\tilde{v}_0, \tilde{\mathbf{T}}_0\in H^r$ with $r$ big enough and suitable $\delta>0$ we have
\medskip

\begin{enumerate}

\item$\|J^P-J^P_{\varepsilon}\|_{H^p}\leq C\varepsilon, \hspace{0.4cm} \|Q^2- Q_{\varepsilon}^2\|_{H^p}\leq C \varepsilon \hspace{0.3cm}\textrm{for all}\hspace{0.2 cm} p,$ since $Q^2$ and $J^P$ are $C^{\infty}$ functions.
\bigskip

\item $\|\phi-\phi_{\varepsilon}\|_{L^{\infty}H^{s+1}}\leq C(\tilde{v}_0,\tilde{\mathbf{T}}_0,\re,\kappa)\varepsilon,\hspace{0.3cm} \|\phi-\phi_{\varepsilon}\|_{H^1H^{\gamma}}\leq C(\tilde{v}_0,\tilde{\mathbf{T}}_0,\re,\kappa)\varepsilon$.
\bigskip

\item $\|\tilde{q}_{\phi}-\tilde{q}_{\phi,\varepsilon}\|_{H^{r+1}}\leq C(\tilde{v}_0,\tilde{\mathbf{T}}_0,\re,\kappa)\varepsilon\hspace{0.3cm}\forall r\geq 0.$
\bigskip

\item $\displaystyle\|\tilde{X}-\tilde{X}_{\varepsilon}+\varepsilon b-t(J^P-J^P_{\varepsilon})\tilde{v}_0\|_{\mathcal{F}^{s+1,\gamma}}\leq C\varepsilon + C T^{\delta}\\[2mm]
\cdot\left(\|\tilde{X}-\tilde{X}_{\varepsilon}+\varepsilon b-t(J^P-J^P_{\varepsilon})\tilde{v}_0\|_{\mathcal{F}^{s+1,\gamma}}+ \|\tilde{w}-\tilde{w}_{\varepsilon}\|_{\mathcal{K}^{s+1}_{(0)}}\right).$
\end{enumerate}
\medskip

\noindent Then

\begin{align*}
&\|\tilde{w}-\tilde{w}_{\varepsilon}\|_{\mathcal{K}^{s+1}_{(0)}} +\|\tilde{q}_w-\tilde{q}_{w,\varepsilon}\|_{\mathcal{K}^{s}_{pr(0)}}+ \|\tilde{X}-\tilde{X}_{\varepsilon}+\varepsilon b-t(J^P-J^P_{\varepsilon})\tilde{v}_0\|_{\mathcal{F}^{s+1,\gamma} }\\[1mm]
&+\|\tilde{\mathbf{T}}_p-\tilde{\mathbf{T}}_{p,\varepsilon}-t\hat{\mathbf{T}}_{\varepsilon}\|_{\mathcal{F}^{s,\gamma-1}}\leq 3C(\tilde{v}_0,\tilde{\mathbf{T}}_0,\re,\kappa)\varepsilon\left(1+\frac{1}{\We}\right)\\[1mm]
&+3C(\tilde{v}_0,\tilde{\mathbf{T}}_0,\re,\kappa)\left(1+\frac{1}{\We}\right)T^{\delta}\left(\|\tilde{w}-\tilde{w}_{\varepsilon}\|_{\mathcal{K}^{s+1}_{(0)}}+\|\tilde{q}_w-\tilde{q}_{w,\varepsilon}\|_{\mathcal{K}^{s}_{pr(0)}}\right.\\[1mm]
&\left.+\|\tilde{X}-\tilde{X}_{\varepsilon}+\varepsilon b-t(J^P-J^P_{\varepsilon})\tilde{v}_0\|_{\mathcal{F}^{s+1,\gamma}}+\|\tilde{\mathbf{T}}_p-\tilde{\mathbf{T}}_{p,\varepsilon}-t\hat{\mathbf{T}}_{\varepsilon}\|_{\mathcal{F}^{s,\gamma-1}}\right).
\end{align*}
\end{proposition}
\medskip

\noindent The hypothesis $(1)$-$(4)$ hold from \cite[Lemma 6.1]{CCFGG2}, with the only difference related to the definition of the approximated velocity. Indeed in $\phi, \phi_{\varepsilon}$ there is the presence of both the initial velocity $\tilde{v}_0$ and  the initial elastic stress tensor $ \tilde{\mathbf{T}}_0$, then the latter appears in the constants.

\subsection{Proof of Proposition \ref{stability}} 
The proof of this proposition is a result of the  two lemmas presented below, which have the aim to prove stability for both the velocity-pressure system and for the elastic stress tensor. At this point we have to figure out the fact that the difference $\tilde{\mathbf{T}}_p-\tilde{\mathbf{T}}_{p,\varepsilon}$ does not belong to the space $H^2_{(0)}H^{\gamma-1}$ since it does not have the cancelation at time zero required by the definition of $H^2_{(0)}([0,T])$. For this reason instead of considering  $\tilde{\mathbf{T}}_p-\tilde{\mathbf{T}}_{p,\varepsilon}$, we consider $\tilde{\mathbf{T}}_p-\tilde{\mathbf{T}}_{p,\varepsilon}-t\hat{\mathbf{T}}_{\varepsilon}$, where 
\begin{align*}
\hat{\mathbf{T}}_{\varepsilon}=&(J^P-J^P_{\varepsilon})\nabla\tilde{v}_0\tilde{\mathbf{T}}_0+\tilde{\mathbf{T}}_0((J^P-J^P_{\varepsilon})\nabla\tilde{v}_0)^T+\frac{\kappa}{\We}(J^P-J^P_{\varepsilon})\nabla\tilde{v}_0+\frac{\kappa}{\We}\left((J^P-J^P_{\varepsilon})\nabla\tilde{v}_0\right)^T.
\end{align*}
\medskip

\begin{lemma}
For a suitable $\beta>0$ and $2<s<\frac{5}{2}$, we have 

\begin{equation}\label{lemma(G-G_eps)}
\begin{split}
&\|\tilde{\mathbf{T}}_p-\tilde{\mathbf{T}}_{p,\varepsilon}-t\hat{\mathbf{T}}_{\varepsilon}\|_{\mathcal{F}^{s,\gamma-1}}\leq C(\tilde{v}_0,\tilde{\mathbf{T}}_0,\re,\kappa)\left(1+ \frac{1}{\We}\right)\varepsilon+C(\tilde{v}_0,\tilde{\mathbf{T}}_0,\re,\kappa)\left(1+ \frac{1}{\We}\right)T^{\beta}\\[2mm]
&\hspace{0.3cm}\cdot\left(\|\tilde{w}-\tilde{w}_{\varepsilon}\|_{\mathcal{K}_{(0)}^{s+1}}+\|\tilde{\mathbf{T}}_p-\tilde{\mathbf{T}}_{p,\varepsilon}-t\hat{\mathbf{T}}_{\varepsilon}\|_{\mathcal{F}^{s,\gamma-1}}
+\|\tilde{X}-\tilde{X}_{\varepsilon}+b\varepsilon-t(J^p-J^P_{\varepsilon})\tilde{v}_0\|_{\mathcal{F}^{s+1,\gamma}}\right).
\end{split}
\end{equation}
\end{lemma}

\begin{proof}
By definition we have

\begin{align*}
&\tilde{\mathbf{T}}_p-\tilde{\mathbf{T}}_{p,\varepsilon}-t\hat{\mathbf{T}}_{\varepsilon}=\int_0^t \left(J^P(\tilde{X})\tilde{\zeta}\nabla\tilde{w}\tilde{\mathbf{T}}_p-J^P(\tilde{X}_{\varepsilon})\tilde{\zeta}_{\varepsilon}\nabla\tilde{w}_{\varepsilon} \tilde{\mathbf{T}}_{p,\varepsilon}\right)\\[2mm]
&+\int_0^t \left(J^P(\tilde{X})\tilde{\zeta}\nabla\tilde{v}_0 \tilde{\mathbf{T}}_p-J^P(\tilde{X}_{\varepsilon})\tilde{\zeta}_{\varepsilon}\nabla\tilde{v}_0 \tilde{\mathbf{T}}_{p,\varepsilon}-(J^P-J^P_{\varepsilon})\nabla\tilde{v}_0\tilde{\mathbf{T}}_0\right)\\[2mm]
&+\int_0^t \left(J^P(\tilde{X})\tilde{\zeta}t\nabla\hat{\phi} \tilde{\mathbf{T}}_p-J^P(\tilde{X}_{\varepsilon})\tilde{\zeta}_{\varepsilon}t\nabla\hat{\phi}_{\varepsilon}\tilde{\mathbf{T}}_{p,\varepsilon}\right)\\[2mm]
&+\int_0^t\left(\tilde{\mathbf{T}}_p\left(J^P(\tilde{X})\tilde{\zeta}\nabla\tilde{w}\right)^T-\tilde{\mathbf{T}}_{p,\varepsilon}\left(J^P(\tilde{X}_{\varepsilon})\tilde{\zeta}_{\varepsilon}\nabla\tilde{w}_{\varepsilon}\right)^T\right)\\[2mm]
&+\int_0^t\left(\tilde{\mathbf{T}}_p\left(J^P(\tilde{X})\tilde{\zeta}\nabla\tilde{v}_0\right)^T-\tilde{\mathbf{T}}_{p,\varepsilon}\left(J^P(\tilde{X}_{\varepsilon})\tilde{\zeta}_{\varepsilon}\nabla\tilde{v}_0\right)^T-\tilde{\mathbf{T}}_0((J^P-J^P_{\varepsilon})\nabla\tilde{v}_0)^T\right)\\[2mm]
&+\int_0^t\left(\tilde{\mathbf{T}}_p\left(J^P(\tilde{X})\tilde{\zeta}t\nabla\hat{\phi}\right)^T-\tilde{\mathbf{T}}_{p,\varepsilon}\left(J^P(\tilde{X}_{\varepsilon})\tilde{\zeta}_{\varepsilon} t\nabla\hat{\phi}_{\varepsilon}\right)^T\right)-\frac{1}{\We}\int_0^t \left(\tilde{\mathbf{T}}_p-\tilde{\mathbf{T}}_{p,\varepsilon}\right)\\[2mm]
&+\frac{\kappa}{\We}\int_0^t\left(J^P(\tilde{X})\tilde{\zeta}\nabla\tilde{w}-J^P(\tilde{X}_{\varepsilon})\tilde{\zeta}_{\varepsilon}\nabla\tilde{w}_{\varepsilon}\right)\\[2mm]
&+\frac{\kappa}{\We}\int_0^t\left(J^P(\tilde{X})\tilde{\zeta}\nabla\tilde{v}_0-J^P(\tilde{X}_{\varepsilon})\tilde{\zeta}_{\varepsilon}\nabla\tilde{v}_0-(J^P-J^P_{\varepsilon})\nabla\tilde{v}_0\right)
\end{align*}
\begin{align*}
&+\frac{\kappa}{\We}\int_0^t\left(J^P(\tilde{X})\tilde{\zeta}t\nabla\hat{\phi}-J^P(\tilde{X}_{\varepsilon})\tilde{\zeta}_{\varepsilon}t\nabla\hat{\phi}_{\varepsilon}\right)+\frac{\kappa}{\We}\int_0^t\left(\left(J^P(\tilde{X})\tilde{\zeta}\nabla\tilde{w}\right)^T-\left(J^P(\tilde{X}_{\varepsilon})\tilde{\zeta}_{\varepsilon}\nabla\tilde{w}_{\varepsilon}\right)^T\right)\\[2mm]
&+\frac{\kappa}{\We}\int_0^t\left(\left(J^P(\tilde{X})\tilde{\zeta}\nabla\tilde{v}_0\right)^T-\left(J^P(\tilde{X}_{\varepsilon})\tilde{\zeta}_{\varepsilon}\nabla\tilde{v}_0\right)^T-((J^P-J^P_{\varepsilon})\nabla\tilde{v}_0)^T\right)\\[2mm]
&+\frac{\kappa}{\We}\int_0^t\left(\left(J^P(\tilde{X})\tilde{\zeta}t\nabla\hat{\phi}\right)^T-\left(J^P(\tilde{X}_{\varepsilon})\tilde{\zeta}_{\varepsilon}\nabla\hat{\phi}_{\varepsilon}\right)^T\right)=\sum_{i=1}^{13} I_i
\end{align*}
\medskip

\noindent We have to show the estimates for these terms in $L^{\infty}_{\frac{1}{4}}H^s$ and $H^2_{(0)}H^{\gamma-1}$. As we have already seen in the proof of iterative bounds the way to get the results in these spaces are different, indeed the space $H^2_{(0)}([0,T])$ requires more attention. First of all we focus on $L^{\infty}_{\frac{1}{4}}H^s$ and we show the results for some terms. We start with $I_1$, that can be written as follows

\begin{align*}
&I_1=\int_0^t (J^P(\tilde{X})-J^P(\tilde{X}_{\varepsilon})-J^P+J^P_{\varepsilon})\tilde{\zeta}\nabla\tilde{w}(\tilde{\mathbf{T}}_p-\tilde{\mathbf{T}}_0)\\[1mm]
&+\int_0^t (J^P(\tilde{X})-J^P(\tilde{X}_{\varepsilon})-J^P+J^P_{\varepsilon})\tilde{\zeta}\nabla\tilde{w}\tilde{\mathbf{T}}_0+\int_0^t (J^P-J^P_{\varepsilon})\tilde{\zeta}\nabla\tilde{w}(\tilde{\mathbf{T}}_p-\tilde{\mathbf{T}}_0)\\[1mm]
&+\int_0^t (J^P-J^P_{\varepsilon})\tilde{\zeta}\nabla\tilde{w}\tilde{\mathbf{T}}_0+\int_0^t J^P(\tilde{X}_{\varepsilon})(\tilde{\zeta}-\tilde{\zeta}_{\varepsilon})\nabla\tilde{w}(\tilde{\mathbf{T}}_p-\tilde{\mathbf{T}}_0)\\[1mm]
&+\int_0^t J^P(\tilde{X}_{\varepsilon})(\tilde{\zeta}-\tilde{\zeta}_{\varepsilon})\nabla\tilde{w}\tilde{\mathbf{T}}_0+\int_0^t J^P(\tilde{X}_{\varepsilon})\tilde{\zeta}_{\varepsilon}(\nabla\tilde{w}-\nabla\tilde{w}_{\varepsilon})(\tilde{\mathbf{T}}_p-\tilde{\mathbf{T}}_0)\\[1mm]
&+\int_0^t J^P(\tilde{X}_{\varepsilon})\tilde{\zeta}_{\varepsilon}(\nabla\tilde{w}-\nabla\tilde{w}_{\varepsilon})\tilde{\mathbf{T}}_0+\int_0^t J^P(\tilde{X}_{\varepsilon})\tilde{\zeta}_{\varepsilon}\nabla\tilde{w}_{\varepsilon}(\tilde{\mathbf{T}}_p-\tilde{\mathbf{T}}_{p,\varepsilon})=\sum_{i=1}^9 I_{1,i}.
\end{align*}

\noindent We study $I_{1,1}, I_{1,7}, I_{1,9}$ in order to have all the required differences. We use lemma \ref{Jp-est} or lemma \ref{Jp-dif-est}, lemma \ref{zeta-est} or lemma \ref{zeta-dif-est}. Moreover we choose sufficiently smooth initial data. The most difficult part is in $I_{1,1}$ when we need to apply lemma \ref{Jp-dif-est} and we need to be careful in the way as we use it.

\begin{align*}
&\|I_{1,1}\|_{L^{\infty}_{\frac{1}{4}}H^s}\leq \sup_{t\in [0,T]} t^{-\frac{1}{4}}\int_0^t\| (J^P(\tilde{X})-J^P(\tilde{X}_{\varepsilon})-J^P+J^P_{\varepsilon})\tilde{\zeta}\nabla\tilde{w}(\tilde{\mathbf{T}}_p-\tilde{\mathbf{T}}_0)\|_{H^s}\\[2mm]
&\leq T^{\frac{1}{4}}\| (J^P(\tilde{X})-J^P(\tilde{X}_{\varepsilon})-J^P+J^P_{\varepsilon})\tilde{\zeta}\nabla\tilde{w}(\tilde{\mathbf{T}}_p-\tilde{\mathbf{T}}_0)\|_{L^2H^s}\\[2mm]
&\leq T^{\frac{1}{4}} \|J^P(\tilde{X})-J^P(\tilde{X}_{\varepsilon})-J^P+J^P_{\varepsilon}\|_{L^{\infty}H^s}\|\tilde{\zeta}\|_{L^{\infty}H^s}\|\nabla\tilde{w}\|_{L^2H^s}\|\tilde{\mathbf{T}}_p-\tilde{\mathbf{T}}_0\|_{L^{\infty}H^s}\\[2mm]
&\leq T^{\frac{1}{4}}\left(\|J^P(\tilde{X}+\varepsilon b)-J^P(\tilde{X}_{\varepsilon})\|_{L^{\infty}H^s}
+\|J^P(\tilde{X})-J^P-J^P(\tilde{X}+\varepsilon b)+J^P_{\varepsilon}\|_{L^{\infty}H^s}\right)\\[1mm]
&\hspace{1cm}\cdot\|\tilde{X}-\tilde{\alpha}\|_{L^{\infty}H^{s+1}}\|\tilde{w}_{\varepsilon}\|_{L^2H^{s+1}}\|\tilde{\mathbf{T}}_p-\tilde{\mathbf{T}}_0\|_{L^{\infty}H^s}\\[2mm]
&\leq C(\tilde{v}_0,\tilde{\mathbf{T}}_0, \kappa)\left(1+\frac{1}{\We}\right) T^{\frac{1}{4}}\left(\|\tilde{X}-\tilde{X}_{\varepsilon}+\varepsilon b\|_{L^{\infty}H^s}+C(\tilde{v}_0)\varepsilon\right)\\[2mm]
\end{align*}
\begin{align*}
&\hspace{0.5cm}\leq C(\tilde{v}_0,\tilde{\mathbf{T}}_0, \kappa)\left(1+\frac{1}{\We}\right)\varepsilon+C(\tilde{v}_0,\tilde{\mathbf{T}}_0, \kappa)\left(1+\frac{1}{\We}\right) T^{\frac{1}{4}}\\[1mm]
&\hspace{1cm}\cdot\left(\|\tilde{X}-\tilde{X}_{\varepsilon}+\varepsilon b-t(J^P-J^P_{\varepsilon})\tilde{v}_0\|_{L^{\infty}H^s}+T\|(J^P-J^P_{\varepsilon})\tilde{v}_0\|_{H^s}\right)\\[2mm]
&\hspace{0.5cm}\leq C(\tilde{v}_0,\tilde{\mathbf{T}}_0, \kappa)\left(1+\frac{1}{\We}\right) \varepsilon +C(\tilde{v}_0,\tilde{\mathbf{T}}_0, \kappa)\left(1+\frac{1}{\We}\right) T^{\frac{3}{4}}\\[1mm]
&\hspace{1cm}\cdot\left\|\tilde{X}-\tilde{X}_{\varepsilon}+\varepsilon b-t(J^P-J^P_{\varepsilon})\tilde{v}_0\right\|_{\mathcal{F}^{s+1,\gamma}}\\[4mm]
&\|I_{1,7}\|_{L^{\infty}_{\frac{1}{4}}H^s}\leq T^{\frac{1}{4}}\|J^P(\tilde{X}_{\varepsilon})\tilde{\zeta}_{\varepsilon}(\nabla\tilde{w}-\nabla\tilde{w}_{\varepsilon})(\tilde{\mathbf{T}}_p-\tilde{\mathbf{T}}_0)\|_{L^2H^s}\\[2mm]
&\hspace{0.5cm}\leq T^{\frac{1}{4}} \|J^P(\tilde{X}_{\varepsilon})\|_{L^{\infty}H^s}\|\tilde{\zeta}_{\varepsilon}\|_{L^{\infty}H^s}\|\nabla\tilde{w}-\nabla\tilde{w}_{\varepsilon}\|_{L^2H^s}\|\tilde{\mathbf{T}}_p-\tilde{\mathbf{T}}_0\|_{L^{\infty}H^s}\\[2mm]
&\hspace{0.5cm}\leq C(\tilde{v}_0,\tilde{\mathbf{T}}_0, \kappa)\left(1+\frac{1}{\We}\right) T^{\frac{1}{4}}\|\tilde{w}-\tilde{w}_{\varepsilon}\|_{\mathcal{K}_{(0)}^{s+1}},\\[4mm]
&\|I_{1,9}\|_{L^{\infty}_{\frac{1}{4}}H^s}\leq T^{\frac{1}{4}}\|J^P(\tilde{X}_{\varepsilon})\tilde{\zeta}_{\varepsilon}\nabla\tilde{w}_{\varepsilon}(\tilde{\mathbf{T}}_p-\tilde{\mathbf{T}}_{p,\varepsilon})\|_{L^2H^s}\\[2mm]
&\hspace{0.5cm}\leq T^{\frac{1}{4}} \|J^P(\tilde{X}_{\varepsilon})\|_{L^{\infty}H^s}\|\tilde{\zeta}_{\varepsilon}\|_{L^{\infty}H^s}\|\tilde{w}_{\varepsilon}\|_{L^2H^{s+1}}\left(\|\tilde{\mathbf{T}}_p-\tilde{\mathbf{T}}_{p,\varepsilon}-t\hat{\mathbf{T}}_{\varepsilon}\|_{L^{\infty}H^s}+\|t\hat{\mathbf{T}}_{\varepsilon}\|_{L^{\infty}H^s}\right)\\[2mm]
&\hspace{0.5cm}\leq C(\tilde{v}_0,\tilde{\mathbf{T}}_0,\kappa)\left(1+\frac{1}{\We}\right)\varepsilon+C(\tilde{v}_0,\tilde{\mathbf{T}}_0,\kappa)\left(1+\frac{1}{\We}\right)T^{\frac{3}{4}}\|\tilde{\mathbf{T}}_p-\tilde{\mathbf{T}}_{p,\varepsilon}-t\hat{\mathbf{T}}_{\varepsilon}\|_{\mathcal{F}^{s,\gamma-1}}.
\end{align*}
\medskip

\noindent Furthermore the estimates for the other terms can be resumed as follows, for  i=2,6

\begin{align*}
&\|I_{1,i}\|_{L^{\infty}_{\frac{1}{4}}H^s}\leq C(\tilde{v}_0,\tilde{\mathbf{T}}_0)\varepsilon+C(\tilde{v}_0,\tilde{\mathbf{T}}_0) T^{\frac{3}{4}}\|\tilde{X}-\tilde{X}_{\varepsilon}+\varepsilon b-t(J^P-J^P_{\varepsilon})\tilde{v}_0\|_{\mathcal{F}^{s+1,\gamma}}
\end{align*}

\noindent and for all the others, we have
\begin{align*}
&\|I_{1,3}\|_{L^{\infty}_{\frac{1}{4}}H^s}\leq C(\tilde{v}_0,\tilde{\mathbf{T}}_0, \kappa)\left(1+\frac{1}{\We}\right) \varepsilon\\[2mm]
&\|I_{1,4}\|_{L^{\infty}_{\frac{1}{4}}H^s}\leq C(\tilde{v}_0,\tilde{\mathbf{T}}_0)\varepsilon\\[2mm]
&\|I_{1,5}\|_{L^{\infty}_{\frac{1}{4}}H^s}\leq  C(\tilde{v}_0,\tilde{\mathbf{T}}_0, \kappa)\left(1+\frac{1}{\We}\right) \varepsilon +C(\tilde{v}_0,\tilde{\mathbf{T}}_0, \kappa)\left(1+\frac{1}{\We}\right) T^{\frac{3}{4}}\\[2mm]
&\hspace{1cm}\cdot\|\tilde{X}-\tilde{X}_{\varepsilon}+\varepsilon b-t(J^P-J^P_{\varepsilon})\tilde{v}_0\|_{\mathcal{F}^{s+1,\gamma}}\\[2mm]
&\|I_{1,8}\|_{L^{\infty}_{\frac{1}{4}}H^s}\leq C(\tilde{v}_0,\tilde{\mathbf{T}}_0) T^{\frac{1}{4}}\|\tilde{w}-\tilde{w}_{\varepsilon}\|_{\mathcal{K}_{(0)}^{s+1}}
\end{align*}
\noindent Moreover, we want to figure out the estimate of $\hat{T}_{\varepsilon}$, where is hidden the requirement of enough regularity for the initial data indeed we have

\begin{equation}\label{hat-T}
\begin{split}
\|\hat{T}_{\varepsilon}\|_{H^s}&\leq \|(J^P-J^P_{\varepsilon})\nabla\tilde{v}_0\tilde{\mathbf{T}}_0\|_{H^s}+\|\tilde{\mathbf{T}}_0((J^P-J^P_{\varepsilon})\nabla\tilde{v}_0)^T\|_{H^s}+\frac{\kappa}{\We}\|(J^P-J^P_{\varepsilon})\nabla\tilde{v}_0\|_{H^s}\\[2mm]
&+\frac{\kappa}{\We}\|((J^P-J^P_{\varepsilon})\nabla\tilde{v}_0)^T\|_{H^s}\leq C(\tilde{v}_0,\tilde{\mathbf{T}}_0)\varepsilon+C(\tilde{v}_0,\kappa)\frac{1}{\We}\varepsilon\\[2mm]
&\leq C(\tilde{v}_0,\tilde{\mathbf{T}}_0,\kappa)\left(1+\frac{1}{\We}\right)\varepsilon. 
\end{split}
\end{equation}

\medskip

\noindent We rewrite below the way for splitting $I_2$ and $I_3$

\begin{align*}
&I_2=\int_0^t (J^P(\tilde{X})-J^P(\tilde{X}_{\varepsilon})-J^P+J^P_{\varepsilon})\tilde{\zeta}\nabla\tilde{v}_0(\tilde{\mathbf{T}}_p-\tilde{\mathbf{T}}_0)\\[1mm]
&\hspace{0.5cm}+\int_0^t (J^P(\tilde{X})-J^P(\tilde{X}_{\varepsilon})-J^P+J^P_{\varepsilon})\tilde{\zeta}\nabla\tilde{v}_0\tilde{\mathbf{T}}_0+\int_0^t (J^P-J^P_{\varepsilon})\tilde{\zeta}\nabla\tilde{v}_0(\tilde{\mathbf{T}}_p-\tilde{\mathbf{T}}_0)\\[1mm]
&\hspace{0.5cm}+\int_0^t (J^P-J^P_{\varepsilon})(\tilde{\zeta}-\mathcal{I})\nabla\tilde{v}_0\tilde{\mathbf{T}}_0+\int_0^t J^P(\tilde{X}_{\varepsilon})(\tilde{\zeta}-\tilde{\zeta}_{\varepsilon})\nabla\tilde{v}_0(\tilde{\mathbf{T}}_p-\tilde{\mathbf{T}}_0)\\[1mm]
&\hspace{0.5cm}+\int_0^t J^P(\tilde{X}_{\varepsilon})(\tilde{\zeta}-\tilde{\zeta}_{\varepsilon})\nabla\tilde{v}_0\tilde{\mathbf{T}}_0+\int_0^t J^P(\tilde{X}_{\varepsilon})\tilde{\zeta}_{\varepsilon}\nabla\tilde{v}_0(\tilde{\mathbf{T}}_p-\tilde{\mathbf{T}}_{p,\varepsilon})=\sum_{i=1}^7 I_{2,i}.\\[4mm]
&I_3=\int_0^t (J^P(\tilde{X})-J^P(\tilde{X}_{\varepsilon})-J^P+J^P_{\varepsilon})\tilde{\zeta}t\nabla\hat{\phi}(\tilde{\mathbf{T}}_p-\tilde{\mathbf{T}}_0)\\[1mm]
&\hspace{0.5cm}+\int_0^t (J^P(\tilde{X})-J^P(\tilde{X}_{\varepsilon})-J^P+J^P_{\varepsilon})\tilde{\zeta}t\nabla\hat{\phi}\tilde{\mathbf{T}}_0+\int_0^t (J^P-J^P_{\varepsilon})\tilde{\zeta}t\nabla\hat{\phi}(\tilde{\mathbf{T}}_p-\tilde{\mathbf{T}}_0)\\[1mm]
&\hspace{0.5cm}+\int_0^t (J^P-J^P_{\varepsilon})\tilde{\zeta}t\nabla\hat{\phi}\tilde{\mathbf{T}}_0+\int_0^t J^P(\tilde{X}_{\varepsilon})(\tilde{\zeta}-\tilde{\zeta}_{\varepsilon})t\nabla\hat{\phi}(\tilde{\mathbf{T}}_p-\tilde{\mathbf{T}}_0)\\[1mm]
&\hspace{0.5cm}+\int_0^t J^P(\tilde{X}_{\varepsilon})(\tilde{\zeta}-\tilde{\zeta}_{\varepsilon})t\nabla\hat{\phi}\tilde{\mathbf{T}}_0+\int_0^t J^P(\tilde{X}_{\varepsilon})\tilde{\zeta}_{\varepsilon}t(\nabla\hat{\phi}-\nabla\hat{\phi}_{\varepsilon})(\tilde{\mathbf{T}}_p-\tilde{\mathbf{T}}_0)\\[1mm]
&\hspace{0.5cm}+\int_0^t J^P(\tilde{X}_{\varepsilon})\tilde{\zeta}_{\varepsilon}t(\nabla\hat{\phi}-\nabla\hat{\phi}_{\varepsilon})\tilde{\mathbf{T}}_0+\int_0^t J^P(\tilde{X}_{\varepsilon})\tilde{\zeta}_{\varepsilon}t\nabla\hat{\phi}_{\varepsilon}(\tilde{\mathbf{T}}_p-\tilde{\mathbf{T}}_{p,\varepsilon})=\sum_{i=1}^9 I_{3,i}.
\end{align*}
 \medskip
 
\noindent However, we notice that in $I_{2,4}$ there is also the presence of one term of $\hat{\mathbf{T}}_{\varepsilon}$. We summarize the results for $I_{2,2}, I_{2,7}$ and $I_{3,8}$ and we use for $J^P(\tilde{X})-J^P(\tilde{X}_{\varepsilon})-J^P+J^P_{\varepsilon}$, the same estimate obtained for $I_{1,1}$ and for $\hat{\phi}-\hat{\phi}_{\varepsilon}$, which does not depend on time but only on the initial data, we use  proposition \ref{stability}.

\begin{align*}
&\|I_{2,2}\|_{L^{\infty}_{\frac{1}{4}}H^s}\leq T^{\frac{1}{4}}\| (J^P(\tilde{X})-J^P(\tilde{X}_{\varepsilon})-J^P+J^P_{\varepsilon})\tilde{\zeta}\nabla\tilde{v}_0\tilde{\mathbf{T}}_0\|_{L^2H^s}\\[2mm]
&\hspace{0.5cm}\leq T^{\frac{1}{4}}\|J^P(\tilde{X})-J^P(\tilde{X}_{\varepsilon})-J^P+J^P_{\varepsilon}\|_{L^{\infty}H^s}\|\tilde{\zeta}\|_{L^{\infty}H^s}\|\nabla\tilde{v}_0\|_{H^s}\|\tilde{\mathbf{T}}_0\|_{H^s}\\[2mm]
&\hspace{0.5cm}\leq C(\tilde{v}_0,\tilde{\mathbf{T}}_0)T^{\frac{1}{2}} \|\tilde{X}-\tilde{X}_{\varepsilon}+\varepsilon b-t(J^P-J^P_{\varepsilon})\tilde{v}_0\|_{\mathcal{F}^{s+1,\gamma}}
\end{align*}
\begin{align*}
&\|I_{2,7}\|_{L^{\infty}_{\frac{1}{4}}H^s}\leq T^{\frac{1}{4}}\|J^P(\tilde{X}_{\varepsilon})\tilde{\zeta}_{\varepsilon}\nabla\tilde{v}_0(\tilde{\mathbf{T}}_p-\tilde{\mathbf{T}}_{p,\varepsilon})\|_{L^2H^s}\\[2mm]
&\hspace{0.5cm}\leq  T^{\frac{1}{4}} \|J^P(\tilde{X}_{\varepsilon})\|_{L^{\infty}H^s}\|\tilde{\zeta}_{\varepsilon}\|_{L^{\infty}H^s}\|\tilde{v}_0\|_{H^{s+1}}\|\tilde{\mathbf{T}}_p-\tilde{\mathbf{T}}_{p,\varepsilon}\|_{L^2H^s}\\[2mm]
&\hspace{0.5cm}\leq C(\tilde{v}_0) T^{\frac{1}{2}}\left(\|\tilde{\mathbf{T}}_p-\tilde{\mathbf{T}}_{p,\varepsilon}-t\hat{\mathbf{T}}_{\varepsilon}\|_{L^{\infty}H^s}+\|t\hat{\mathbf{T}}_{\varepsilon}\|_{L^{\infty}H^s}\right)\\[2mm]
&\hspace{0.5cm}\leq C(\tilde{v}_0,\tilde{\mathbf{T}}_0,\kappa)\left(1+\frac{1}{\We}\right)\varepsilon+C(\tilde{v}_0,\tilde{\mathbf{T}}_0,\kappa)\left(1+\frac{1}{\We}\right)T\|\tilde{\mathbf{T}}_p-\tilde{\mathbf{T}}_{p,\varepsilon}-t\hat{\mathbf{T}}\|_{\mathcal{F}^{s,\gamma-1}}\\[4mm]
&\|I_{3,8}\|_{L^{\infty}_{\frac{1}{4}}H^s}\leq T^{\frac{1}{4}}\|J^P(\tilde{X}_{\varepsilon})\tilde{\zeta}_{\varepsilon}t(\nabla\hat{\phi}-\nabla\hat{\phi}_{\varepsilon})\tilde{\mathbf{T}}_0\|_{L^2H^s}\\[2mm]
&\hspace{0.5cm}\leq T^{\frac{1}{4}}\|J^P(\tilde{X}_{\varepsilon})\|_{L^{\infty}H^s}\|\tilde{\zeta}_{\varepsilon}\|_{L^{\infty}H^s}\|t(\hat{\phi}-
\hat{\phi}_{\varepsilon})\|_{L^{2}H^{s+1}}\|\tilde{\mathbf{T}}_0\|_{H^s}\\[2mm]
&\hspace{0.5cm}\leq C(\tilde{v}_0,\tilde{\mathbf{T}}_0)T^{\frac{7}{4}}\|\hat{\phi}-
\hat{\phi}_{\varepsilon}\|_{H^{s+1}}\leq C(\tilde{v}_0,\tilde{\mathbf{T}}_0,\re,\kappa)\varepsilon.
\end{align*}
\medskip

\noindent  For all the remaing terms we refer to these estimates and to the estimate of $I_1$. Indeed for$I_4, I_5$ and $I_6$, we proceed as $I_1, I_2$ and $I_3$ and we summarize the result 

\begin{align*}
&\sum_{i=4}^6 \|I_i\|_{L^{\infty}_{\frac{1}{4}}H^s}\leq C(\tilde{v}_0,\tilde{\mathbf{T}}_0,\re,\kappa)\left(1+\frac{1}{\We}\right)\varepsilon+C(\tilde{v}_0,\tilde{\mathbf{T}}_0,\re,\kappa)\left(1+\frac{1}{\We}\right) T^{\frac{1}{4}}\\[1mm]
&\hspace{0.5cm}\cdot\left(\|\tilde{X}-\tilde{X}_{\varepsilon}+\varepsilon b-t(J^P-J^P_{\varepsilon})\tilde{v}_0\|_{\mathcal{F}^{s+1,\gamma}}+\|\tilde{\mathbf{T}}_p-\tilde{\mathbf{T}}_{p,\varepsilon}-t\hat{\mathbf{T}}_{\varepsilon}\|_{\mathcal{F}^{s,\gamma-1}}+\|\tilde{w}-\tilde{w}_{\varepsilon}\|_{\mathcal{K}^{s+1}_{(0)}}\right).
\end{align*}
\medskip

\noindent For $I_7$ we add and subtract the missing terms and we get

\begin{align*}
&\|I_7\|_{L^{\infty}_{\frac{1}{4}}H^s}\leq T^{\frac{1}{4}}\frac{1}{\We}\|\tilde{\mathbf{T}}_p-\tilde{\mathbf{T}}_{p,\varepsilon}\|_{L^2H^s}\leq  T^{\frac{3}{4}}\frac{1}{\We}\left(\|\tilde{\mathbf{T}}_p-\tilde{\mathbf{T}}_{p,\varepsilon}-t\hat{\mathbf{T}}_{\varepsilon}\|_{L^{\infty}H^s}+\|t\hat{\mathbf{T}}_{\varepsilon}\|_{L^{\infty}H^s}\right)\\[2mm]
&\hspace{0.5cm}\leq C(\tilde{v}_0,\tilde{\mathbf{T}}_0,\kappa)\left(1+\frac{1}{\We}\right)\varepsilon+\frac{1}{\We}T \|\tilde{\mathbf{T}}_p-\tilde{\mathbf{T}}_{p,\varepsilon}-t\hat{\mathbf{T}}_{\varepsilon}\|_{\mathcal{F}^{s,\gamma-1}}.
\end{align*}
\medskip

\noindent Finally, we remark that for the estimates $I_i$, for $i=8,\ldots,13$ we can adapt the results obtained for $I_i$, with $i=1,\ldots,6$ but without the presence of the elastic stress tensor and we resume the result below. We will focus on these terms when we will do estimates in $H^2_{(0)}H^{\gamma-1}$, since it will be more complex.

\begin{align*}
&\sum_{i=8}^{13}\|I_i\|_{L^{\infty}_{\frac{1}{4}}H^s}\leq C(\tilde{v}_0,\tilde{\mathbf{T}}_0,\re,\kappa)\left(1+ \frac{1}{\We}\right)\varepsilon+C(\tilde{v}_0,\tilde{\mathbf{T}}_0,\re,\kappa)\left(1+ \frac{1}{\We}\right)T^{\frac{1}{4}}\\[2mm]
&\hspace{0.5cm}\cdot\left(\|\tilde{w}-\tilde{w}_{\varepsilon}\|_{\mathcal{K}^{s+1}_{(0)}}+\|\tilde{\mathbf{T}}_p-\tilde{\mathbf{T}}_{p,\varepsilon}-t\hat{\mathbf{T}}_{\varepsilon}\|_{\mathcal{F}^{s,\gamma-1}}+\|\tilde{X}-\tilde{X}_{\varepsilon}+b\varepsilon-t(J^p-J^P_{\varepsilon})\tilde{v}_0\|_{\mathcal{F}^{s+1,\gamma}}\right).
\end{align*}
\medskip

\noindent We pass to analyze $\tilde{\mathbf{T}}_p-\tilde{\mathbf{T}}_{p,\varepsilon}-t\hat{\mathbf{T}}_{\varepsilon}$ in $H^2_{(0)}H^{\gamma-1}$ and we need to pay attention in order to guarantee enough cancelations at time zero. We start by rewriting $I_1$ in a convenient way.

\begin{align*}
I_1&=\int_0^t (J^P( \tilde{X})-J^P( \tilde{X}_{\varepsilon})-J^P+J^P_{\varepsilon})(\tilde{\zeta}-\mathcal{I})\nabla \tilde{w} (\tilde{\mathbf{T}}_p-\tilde{\mathbf{T}}_0)\\[1mm]
&+\int_0^t(J^P( \tilde{X})-J^P( \tilde{X}_{\varepsilon})-J^P+J^P_{\varepsilon})(\tilde{\zeta}-\mathcal{I})\nabla \tilde{w}\tilde{\mathbf{T}}_0\\[1mm]
&+\int_0^t (J^P-J^P_{\varepsilon})(\tilde{\zeta}-\mathcal{I})\nabla \tilde{w} (\tilde{\mathbf{T}}_p-\tilde{\mathbf{T}}_0)+\int_0^t (J^P-J^P_{\varepsilon})(\tilde{\zeta}-\mathcal{I})\nabla \tilde{w}\tilde{\mathbf{T}}_0\\[1mm]
&+\int_0^t (J^P( \tilde{X})-J^P( \tilde{X}_{\varepsilon})-J^P+J^P_{\varepsilon})\nabla \tilde{w} (\tilde{\mathbf{T}}_p-\tilde{\mathbf{T}}_0)\\[1mm]
&+\int_0^t (J^P( \tilde{X})-J^P( \tilde{X}_{\varepsilon})-J^P+J^P_{\varepsilon})\nabla \tilde{w}\tilde{\mathbf{T}}_0+\int_0^t (J^P-J^P_{\varepsilon})\nabla \tilde{w} (\tilde{\mathbf{T}}_p-\tilde{\mathbf{T}}_0)\\[1mm]
&+\int_0^t (J^P-J^P_{\varepsilon})\nabla \tilde{w}\tilde{\mathbf{T}}_0+\int_0^t (J^P(\tilde{X}_{\varepsilon})-J^P_{\varepsilon})(\tilde{\zeta}-\tilde{\zeta}_{\varepsilon})\nabla\tilde{w}(\tilde{\mathbf{T}}_p-\tilde{\mathbf{T}}_0)\\[1mm]
&+\int_0^t (J^P(\tilde{X}_{\varepsilon})-J^P_{\varepsilon})(\tilde{\zeta}-\tilde{\zeta}_{\varepsilon})\nabla\tilde{w}\tilde{\mathbf{T}}_0+\int_0^t J^P_{\varepsilon}(\tilde{\zeta}-\tilde{\zeta}_{\varepsilon})\nabla\tilde{w}(\tilde{\mathbf{T}}_p-\tilde{\mathbf{T}}_0)\\[1mm]
&+\int_0^t J^P_{\varepsilon}(\tilde{\zeta}-\tilde{\zeta}_{\varepsilon})\nabla\tilde{w}\tilde{\mathbf{T}}_0+\int_0^t (J^P(\tilde{X}_{\varepsilon})-J^P_{\varepsilon})(\tilde{\zeta}_{\varepsilon}-\mathcal{I})(\nabla\tilde{w}-\nabla\tilde{w}_{\varepsilon})(\tilde{\mathbf{T}}_p-\tilde{\mathbf{T}}_0)\\[1mm]
&+\int_0^t (J^P(\tilde{X}_{\varepsilon})-J^P_{\varepsilon})(\tilde{\zeta}_{\varepsilon}-\mathcal{I})(\nabla\tilde{w}-\nabla\tilde{w}_{\varepsilon})\tilde{\mathbf{T}}_0+\int_0^t J^P_{\varepsilon}(\tilde{\zeta}_{\varepsilon}-\mathcal{I})(\nabla\tilde{w}-\nabla\tilde{w}_{\varepsilon})(\tilde{\mathbf{T}}_p-\tilde{\mathbf{T}}_0)\\[1mm]
&+\int_0^t J^P_{\varepsilon}(\tilde{\zeta}_{\varepsilon}-\mathcal{I})(\nabla\tilde{w}-\nabla\tilde{w}_{\varepsilon})\tilde{\mathbf{T}}_0+\int_0^t (J^P(\tilde{X}_{\varepsilon})-J^P_{\varepsilon})(\nabla\tilde{w}-\nabla\tilde{w}_{\varepsilon})(\tilde{\mathbf{T}}_p-\tilde{\mathbf{T}}_0)\\[1mm]
&+\int_0^t (J^P(\tilde{X}_{\varepsilon})-J^P_{\varepsilon})(\nabla\tilde{w}-\nabla\tilde{w}_{\varepsilon})\tilde{\mathbf{T}}_0+\int_0^t J^P_{\varepsilon}(\nabla\tilde{w}-\nabla\tilde{w}_{\varepsilon})(\tilde{\mathbf{T}}_p-\tilde{\mathbf{T}}_0)\\[2mm]
&+\int_0^t J^P_{\varepsilon}(\nabla\tilde{w}-\nabla\tilde{w}_{\varepsilon})\tilde{\mathbf{T}}_0+\int_0^t (J^P(\tilde{X}_{\varepsilon})-J^P_{\varepsilon})(\tilde{\zeta}_{\varepsilon}-\mathcal{I})\nabla\tilde{w}_{\varepsilon}(\tilde{\mathbf{T}}_p-\tilde{\mathbf{T}}_{p,\varepsilon})\\[2mm]
&+\int_0^t J^P_{\varepsilon}(\tilde{\zeta}_{\varepsilon}-\mathcal{I})\nabla\tilde{w}_{\varepsilon}(\tilde{\mathbf{T}}_p-\tilde{\mathbf{T}}_{p,\varepsilon})+\int_0^t (J^P(\tilde{X}_{\varepsilon})-J^P_{\varepsilon})\nabla\tilde{w}_{\varepsilon}(\tilde{\mathbf{T}}_p-\tilde{\mathbf{T}}_{p,\varepsilon})\\[2mm]
&+\int_0^t J^P_{\varepsilon}\nabla\tilde{w}_{\varepsilon}(\tilde{\mathbf{T}}_p-\tilde{\mathbf{T}}_{p,\varepsilon})=\sum_{i=1}^{24} I_{1,i}
\end{align*}
\medskip

\noindent We show the estimate of $I_{1,1}$ for the flux differences. We use lemma \ref{lem2}, with $\varepsilon=0$, lemma \ref{lem3} with $\gamma>1$, lemma \ref{lem5}. Then for $J^P( \tilde{X})-J^P( \tilde{X}_{\varepsilon})-J^P+J^P_{\varepsilon}$ we use lemma \ref{Jp-est} and lemma \ref{Jp-dif-est} and lemma \ref{zeta-est} for $\tilde{\zeta}-\mathcal{I}$, lemma \ref{lem1} for the velocity and for the elastic tensor \eqref{elastic-estim}. In conclusion lemma \ref{lem2} with $0<\delta_1<\eta_1$.

\begin{align*}
&\|I_{1,1}\|_{H^2_{(0)}H^{\gamma-1}}\leq \|(J^P( \tilde{X})-J^P( \tilde{X}_{\varepsilon})-J^P+J^P_{\varepsilon})(\tilde{\zeta}-\mathcal{I})\nabla \tilde{w} (\tilde{\mathbf{T}}_p-\tilde{\mathbf{T}}_0)\|_{H^1_{(0)}H^{\gamma-1}}\\[2mm]
&\hspace{0.5cm}\leq \|J^P( \tilde{X})-J^P( \tilde{X}_{\varepsilon})-J^P+J^P_{\varepsilon}\|_{H^1_{(0)}H^{\gamma}}\|(\tilde{\zeta}-\mathcal{I})\nabla \tilde{w} (\tilde{\mathbf{T}}_p-\tilde{\mathbf{T}}_0)\|_{H^1_{(0)}H^{\gamma-1}}\\[2mm]
\end{align*}
\begin{align*}
&\leq \left(\|J^P(\tilde{X}+\varepsilon b)-J^P(\tilde{X}_{\varepsilon})\|_{H^1_{(0)}H^{\gamma}}+\|J^P(\tilde{X})-J^P-J^P(\tilde{X}+\varepsilon b)+J^P_{\varepsilon}\|_{H^1_{(0)}H^{\gamma}}\right)
\\[2mm]
&\hspace{1cm}\cdot\|\tilde{\zeta}-\mathcal{I}\|_{H^1_{(0)}H^{\gamma-1}}\|\nabla \tilde{w}\|_{H^1_{(0)}H^{\gamma-1}}\|\tilde{\mathbf{T}}_p-\tilde{\mathbf{T}}_0\|_{H^1_{(0)}H^{\gamma-1}}\\[2mm]
&\leq C(\tilde{v}_0,\tilde{\mathbf{T}}_0,\kappa)\left(1+\frac{1}{\We}\right)\left(\|\tilde{X}-\tilde{X}_{\varepsilon}+\varepsilon b-t(J^P-J^P_{\varepsilon})\tilde{v}_0\|_{H^{1}_{(0)}H^{\gamma}}\right.\\
&\hspace{1cm}\left.+\|t(J^P-J^P_{\varepsilon})\tilde{v}_0\|_{H^{1}_{(0)}H^{\gamma}}+\varepsilon\right)\\[2mm]
&\leq  C(\tilde{v}_0,\tilde{\mathbf{T}}_0,\kappa)\left(1+\frac{1}{\We}\right)\varepsilon+  C(\tilde{v}_0,\tilde{\mathbf{T}}_0,\kappa)\left(1+\frac{1}{\We}\right)\\[2mm]
&\hspace{1cm}\cdot\left\|\int_0^t\partial_t(\tilde{X}-\tilde{X}_{\varepsilon}+\varepsilon b-t(J^P-J^P_{\varepsilon})\tilde{v}_0)\right\|_{H^{1+\eta_1-\delta_1}_{(0)}H^{\gamma}}\\[2mm]
&\leq  C(\tilde{v}_0,\tilde{\mathbf{T}}_0,\kappa)\left(1+\frac{1}{\We}\right)\varepsilon+  C(\tilde{v}_0,\tilde{\mathbf{T}}_0,\kappa) T^{\delta_1}\|\tilde{X}-\tilde{X}_{\varepsilon}+\varepsilon b-t(J^P-J^P_{\varepsilon})\tilde{v}_0\|_{H^{1+\eta_1}H^{\gamma}}\\[2mm]
&\leq C(\tilde{v}_0,\tilde{\mathbf{T}}_0,\kappa)\left(1+\frac{1}{\We}\right)\varepsilon+  C(\tilde{v}_0,\tilde{\mathbf{T}}_0,\kappa) T^{\delta_1}\|\tilde{X}-\tilde{X}_{\varepsilon}+\varepsilon b-t(J^P-J^P_{\varepsilon})\tilde{v}_0\|_{\mathcal{F}^{s+1,\gamma}}.
\end{align*}
 
 \noindent Then we remark that for the other terms involving the difference of the flux we have the following results
 
 \begin{align*}
 &\|I_{1,i}\|_{H^2_{(0)}H^{\gamma-1}}\leq C(\tilde{v}_0,\tilde{\mathbf{T}}_0) \varepsilon+C(\tilde{v}_0,\tilde{\mathbf{T}}_0) T^{\delta_i}\|\tilde{X}-\tilde{X}_{\varepsilon}+\varepsilon b-t(J^P-J^P_{\varepsilon})\tilde{v}_0\|_{\mathcal{F}^{s+1,\gamma}},\\[1mm]
 &\hspace{10cm}\textrm{for}\hspace{0.3cm} i=2,6,10,12\\[2mm]
 &\|I_{1,i}\|_{H^2_{(0)}H^{\gamma-1}}\leq C(\tilde{v}_0,\tilde{\mathbf{T}}_0,\kappa)\left(1+\frac{1}{\We}\right) \varepsilon,\hspace{0.3cm}\textrm{for}\hspace{0.3cm} i=3,7\\[2mm]
 &\|I_{1,i}\|_{H^2_{(0)}H^{\gamma-1}}\leq C(\tilde{v}_0,\tilde{\mathbf{T}}_0) \varepsilon, \hspace{0.3cm}\textrm{for}\hspace{0.3cm} i=4,8\\[2mm]
 &\|I_{1,i}\|_{H^2_{(0)}H^{\gamma-1}}\leq C(\tilde{v}_0,\tilde{\mathbf{T}}_0,\kappa)\left(1+\frac{1}{\We}\right)\varepsilon+  C(\tilde{v}_0,\tilde{\mathbf{T}}_0,\kappa) \left(1+\frac{1}{\We}\right)T^{\delta_i}\\[2mm]
&\cdot \|\tilde{X}-\tilde{X}_{\varepsilon}+\varepsilon b-t(J^P-J^P_{\varepsilon})\tilde{v}_0\|_{\mathcal{F}^{s+1,\gamma}},\hspace{0.3cm}\textrm{for}\hspace{0.3cm} i=5,9,11.
 \end{align*}
\medskip

\noindent Concerning the terms with the difference of the velocity we study $I_{1,13}$, by using lemma \ref{lem2}, with $\varepsilon=0$, then in order to separate the terms we use lemma \ref{lem3} with $\gamma>1$ and lemma \ref{lem5}. Futhermore estimate \eqref{flux-estim} and \eqref{elastic-estim} with lemma \ref{lem2} and lemma \ref{lem1} with $\delta_{13}<\eta_{13}<\frac{s-1-\gamma}{2}$ give the final result.
 
\begin{align*}
&\|I_{1,13}\|_{H^2_{(0)}H^{\gamma-1}}\leq \|(J^P(\tilde{X}_{\varepsilon})-J^P_{\varepsilon})(\tilde{\zeta}_{\varepsilon}-\mathcal{I})(\nabla\tilde{w}-\nabla\tilde{w}_{\varepsilon})(\tilde{\mathbf{T}}_p-\tilde{\mathbf{T}}_0)\|_{H^1_{(0)}H^{\gamma-1}}\\[2mm]
&\leq \|J^P(\tilde{X}_{\varepsilon})-J^P_{\varepsilon}\|_{H^1_{(0)}H^{\gamma}}\|(\tilde{\zeta}_{\varepsilon}-\mathcal{I})(\nabla\tilde{w}-\nabla\tilde{w}_{\varepsilon})(\tilde{\mathbf{T}}_p-\tilde{\mathbf{T}}_0)\|_{H^1_{(0)}H^{\gamma-1}}\\[2mm]
&\leq C(\tilde{v}_0)\|\tilde{X}_{\varepsilon}-\tilde{\alpha}-\varepsilon b\|_{H^{1}_{(0)}H^{\gamma}}\|\tilde{\zeta}_{\varepsilon}-\mathcal{I}\|_{H^1_{(0)}H^{\gamma-1}}\|\nabla\tilde{w}-\nabla\tilde{w}_{\varepsilon}\|_{H^1_{(0)}H^{\gamma-1}}\|\tilde{\mathbf{T}}_p-\tilde{\mathbf{T}}_0\|_{H^1_{(0)}H^{\gamma-1}}\\[2mm]
\end{align*}
\begin{align*}
&\leq C(\tilde{v}_0,\tilde{\mathbf{T}}_0,\kappa)\left(1+\frac{1}{\We}\right)\|\tilde{w}-\tilde{w}_{\varepsilon}\|_{H^1_{(0)}H^{\gamma}}\\[2mm]
&\leq C(\tilde{v}_0,\tilde{\mathbf{T}}_0,\kappa)\left(1+\frac{1}{\We}\right)\left\|\int_0^t \partial_t(\tilde{w}-\tilde{w}_{\varepsilon})\right\|_{H^{1+\eta_{13}-\delta_{13}}_{(0)}H^{\gamma}}\\[2mm]
&\leq C(\tilde{v}_0,\tilde{\mathbf{T}}_0,\kappa)\left(1+\frac{1}{\We}\right)T^{\delta_{13}}\|\tilde{w}-\tilde{w}_{\varepsilon}\|_{\mathcal{K}^{s+1}_{(0)}}.
\end{align*}

\noindent For $i=15,17,19$ we have the same estimate as $I_{1,13}$ and for $i=14,16,18,20$ we have $\|I_{1,i}\|_{H^{2}_{(0)}H^{\gamma-1}}\leq C(\tilde{v}_0,\tilde{\mathbf{T}}_0) T^{\delta_i}\|\tilde{w}-\tilde{w}_{\varepsilon}\|_{\mathcal{K}^{s+1}_{(0)}}.$ We conclude with the difference of the elastic stress tensor, by estimating $I_{1,21}$ and use in a key way the fact that $\|t\|_{H^1_{(0)}}\leq T^{\frac{1}{2}}$.

\begin{align*}
&\|I_{1,21}\|_{H^2_{(0)}H^{\gamma-1}}\leq \|(J^P(\tilde{X}_{\varepsilon})-J^P_{\varepsilon})(\tilde{\zeta}_{\varepsilon}-\mathcal{I})\nabla\tilde{w}_{\varepsilon}(\tilde{\mathbf{T}}_p-\tilde{\mathbf{T}}_{p,\varepsilon})\|_{H^1_{(0)}H^{\gamma-1}}\\[2mm]
&\leq \|J^P(\tilde{X}_{\varepsilon})-J^P_{\varepsilon}\|_{H^1_{(0)}H^{\gamma}}\|(\tilde{\zeta}_{\varepsilon}-\mathcal{I})\nabla\tilde{w}_{\varepsilon}(\tilde{\mathbf{T}}_p-\tilde{\mathbf{T}}_{p,\varepsilon})\|_{H^1_{(0)}H^{\gamma-1}}\\[2mm]
&\leq C(\tilde{v}_0)\|\tilde{X}_{\varepsilon}-\tilde{\alpha}-\varepsilon b\|_{H^{1}_{(0)}H^{\gamma}}\|\tilde{\zeta}_{\varepsilon}-\mathcal{I}\|_{H^1_{(0)}H^{\gamma-1}}\|\tilde{w}_{\varepsilon}\|_{H^1_{(0)}H^{\gamma}}\|\tilde{\mathbf{T}}_p-\tilde{\mathbf{T}}_{p,\varepsilon}\|_{H^1_{(0)}H^{\gamma-1}}\\[2mm]
&\leq C(\tilde{v}_0)\left(\|\tilde{\mathbf{T}}_p-\tilde{\mathbf{T}}_{p,\varepsilon}-t\hat{\mathbf{T}}_{\varepsilon}\|_{H^1_{(0)}H^{\gamma-1}}+\|t\hat{\mathbf{T}}_{\varepsilon}\|_{H^1_{(0)}H^{\gamma-1}}\right)\\[2mm]
&\leq C(\tilde{v}_0,\tilde{\mathbf{T}}_0,\kappa)\left(1+\frac{1}{\We}\right)\varepsilon+C(\tilde{v}_0)\left\|\int_0^t\partial_t(\tilde{\mathbf{T}}_p-\tilde{\mathbf{T}}_{p,\varepsilon}-t\hat{\mathbf{T}}_{\varepsilon})\right\|_{H^{1+\eta_{21}-\delta_{21}}_{(0)}H^{\gamma-1}}\\[2mm]
&\leq C(\tilde{v}_0,\tilde{\mathbf{T}}_0,\kappa)\left(1+\frac{1}{\We}\right)\varepsilon+C(\tilde{v}_0) T^{\delta_{21}}\|\tilde{\mathbf{T}}_p-\tilde{\mathbf{T}}_{p,\varepsilon}-t\hat{\mathbf{T}}_{\varepsilon}\|_{\mathcal{F}^{s,\gamma-1}}.
\end{align*}

\noindent For $i=22,\ldots,24$ we have the same result as $I_{1,21}$, but we use less lemmas. Now, we can pass to $I_2$.

\begin{align*}
I_2&=\int_0^t (J^P( \tilde{X})-J^P( \tilde{X}_{\varepsilon})-J^P+J^P_{\varepsilon})(\tilde{\zeta}-\mathcal{I})\nabla \tilde{v}_0 (\tilde{\mathbf{T}}_p-\tilde{\mathbf{T}}_0)\\[1mm]
&+\int_0^t(J^P( \tilde{X})-J^P( \tilde{X}_{\varepsilon})-J^P+\hspace{0.5cm}J^P_{\varepsilon})(\tilde{\zeta}-\mathcal{I})\nabla \tilde{v}_0\tilde{\mathbf{T}}_0\\[1mm]
&+\int_0^t (J^P-J^P_{\varepsilon})(\tilde{\zeta}-\mathcal{I})\nabla \tilde{v}_0 (\tilde{\mathbf{T}}_p-\tilde{\mathbf{T}}_0)+\int_0^t (J^P-J^P_{\varepsilon})(\tilde{\zeta}-\mathcal{I})\nabla \tilde{v}_0\tilde{\mathbf{T}}_0\\[1mm]
&+\int_0^t (J^P( \tilde{X})-J^P( \tilde{X}_{\varepsilon})-J^P+J^P_{\varepsilon})\nabla \tilde{v}_0 (\tilde{\mathbf{T}}_p-\tilde{\mathbf{T}}_0)\\[1mm]
&+\int_0^t (J^P( \tilde{X})-J^P( \tilde{X}_{\varepsilon})-J^P+J^P_{\varepsilon})\nabla \tilde{v}_0\tilde{\mathbf{T}}_0+\int_0^t (J^P-J^P_{\varepsilon})\nabla \tilde{v}_0 (\tilde{\mathbf{T}}_p-\tilde{\mathbf{T}}_0)\\[1mm]
&+\int_0^t (J^P(\tilde{X}_{\varepsilon})-J^P_{\varepsilon})(\tilde{\zeta}-\tilde{\zeta}_{\varepsilon})\nabla \tilde{v}_0(\tilde{\mathbf{T}}_p-\tilde{\mathbf{T}}_0)+\int_0^t (J^P(\tilde{X}_{\varepsilon})-J^P_{\varepsilon})(\tilde{\zeta}-\tilde{\zeta}_{\varepsilon})\nabla \tilde{v}_0\tilde{\mathbf{T}}_0\\[1mm]
&+\int_0^t J^P_{\varepsilon}(\tilde{\zeta}-\tilde{\zeta}_{\varepsilon})\nabla \tilde{v}_0(\tilde{\mathbf{T}}_p-\tilde{\mathbf{T}}_0)+\int_0^t J^P_{\varepsilon}(\tilde{\zeta}-\tilde{\zeta}_{\varepsilon})\nabla \tilde{v}_0\tilde{\mathbf{T}}_0
\end{align*}
\begin{align*}
&+\int_0^t (J^P(\tilde{X}_{\varepsilon})-J^P_{\varepsilon})(\tilde{\zeta}_{\varepsilon}-\mathcal{I})\nabla \tilde{v}_0(\tilde{\mathbf{T}}_p-\tilde{\mathbf{T}}_{p,\varepsilon})+\int_0^t J^P_{\varepsilon}(\tilde{\zeta}_{\varepsilon}-\mathcal{I})\nabla \tilde{v}_0(\tilde{\mathbf{T}}_p-\tilde{\mathbf{T}}_{p,\varepsilon})\\[1mm]
&+\int_0^t (J^P(\tilde{X}_{\varepsilon})-J^P_{\varepsilon})\nabla \tilde{v}_0(\tilde{\mathbf{T}}_p-\tilde{\mathbf{T}}_{p,\varepsilon})+\int_0^t J^P_{\varepsilon}\nabla \tilde{v}_0(\tilde{\mathbf{T}}_p-\tilde{\mathbf{T}}_{p,\varepsilon})
\end{align*}
\medskip

\noindent The estimate of $I_2$ is related to $I_1$ estimate by the fact that we obtain this term from the definitions of  $\tilde{v}=\tilde{w}+\phi=\tilde{w}+\tilde{v}_0+t\hat{\phi}$ and  $\tilde{v}_{\varepsilon}=\tilde{w}_{\varepsilon}+\phi_{\varepsilon}=\tilde{w}_{\varepsilon}+\tilde{v}_0+t\hat{\phi}_{\varepsilon}$. But here we do not have the velocities difference. Moreover in this term is included one of the terms of $\hat{\mathbf{T}}_{\varepsilon}$. Then  
\begin{align*}
&\|I_2\|_{H^{2}_{(0)}H^{\gamma-1}}\leq C(\tilde{v}_0,\tilde{\mathbf{T}}_0,\kappa)\left(1+\frac{1}{\We}\right)\varepsilon+  C(\tilde{v}_0,\tilde{\mathbf{T}}_0,\kappa) T^{\varrho_1}\\
&\hspace{0.5cm}\cdot\left(\|\tilde{X}-\tilde{X}_{\varepsilon}+\varepsilon b-t(J^P-J^P_{\varepsilon})\tilde{v}_0\|_{\mathcal{F}^{s+1,\gamma}}+\|\tilde{\mathbf{T}}_p-\tilde{\mathbf{T}}_{p,\varepsilon}-t\hat{\mathbf{T}}_{\varepsilon}\|_{\mathcal{F}^{s,\gamma-1}}\right).
\end{align*}

\noindent The splitting and the estimate of $I_3$ is summarized below

\begin{align*}
I_3&=\int_0^t (J^P( \tilde{X})-J^P( \tilde{X}_{\varepsilon})-J^P+J^P_{\varepsilon})(\tilde{\zeta}-\mathcal{I})t\nabla\hat{\phi} (\tilde{\mathbf{T}}_p-\tilde{\mathbf{T}}_0)\\[1mm]
&+\int_0^t(J^P( \tilde{X})-J^P( \tilde{X}_{\varepsilon})-J^P+J^P_{\varepsilon})(\tilde{\zeta}-\mathcal{I})t\nabla\hat{\phi}\tilde{\mathbf{T}}_0\\[1mm]
&+\int_0^t (J^P-J^P_{\varepsilon})(\tilde{\zeta}-\mathcal{I})t\nabla\hat{\phi}(\tilde{\mathbf{T}}_p-\tilde{\mathbf{T}}_0)+\int_0^t (J^P-J^P_{\varepsilon})(\tilde{\zeta}-\mathcal{I})t\nabla\hat{\phi}\tilde{\mathbf{T}}_0\\[1mm]
&+\int_0^t (J^P( \tilde{X})-J^P( \tilde{X}_{\varepsilon})-J^P+J^P_{\varepsilon})t\nabla\hat{\phi} (\tilde{\mathbf{T}}_p-\tilde{\mathbf{T}}_0)\\[1mm]
&+\int_0^t (J^P( \tilde{X})-J^P( \tilde{X}_{\varepsilon})-J^P+J^P_{\varepsilon})t\nabla\hat{\phi}\tilde{\mathbf{T}}_0+\int_0^t (J^P-J^P_{\varepsilon})t\nabla\hat{\phi}(\tilde{\mathbf{T}}_p-\tilde{\mathbf{T}}_0)\\[1mm]
&+\int_0^t (J^P-J^P_{\varepsilon})t\nabla\hat{\phi}\tilde{\mathbf{T}}_0+\int_0^t (J^P(\tilde{X}_{\varepsilon})-J^P_{\varepsilon})(\tilde{\zeta}-\tilde{\zeta}_{\varepsilon})t\nabla\hat{\phi}(\tilde{\mathbf{T}}_p-\tilde{\mathbf{T}}_0)\\[1mm]
&+\int_0^t (J^P(\tilde{X}_{\varepsilon})-J^P_{\varepsilon})(\tilde{\zeta}-\tilde{\zeta}_{\varepsilon})t\nabla\hat{\phi}\tilde{\mathbf{T}}_0+\int_0^t J^P_{\varepsilon}(\tilde{\zeta}-\tilde{\zeta}_{\varepsilon})t\nabla\hat{\phi}(\tilde{\mathbf{T}}_p-\tilde{\mathbf{T}}_0)\\[1mm]
&+\int_0^t J^P_{\varepsilon}(\tilde{\zeta}-\tilde{\zeta}_{\varepsilon})t\nabla\hat{\phi}\tilde{\mathbf{T}}_0+\int_0^t (J^P(\tilde{X}_{\varepsilon})-J^P_{\varepsilon})(\tilde{\zeta}_{\varepsilon}-\mathcal{I})t(\nabla\hat{\phi}-\nabla\hat{\phi}_{\varepsilon})(\tilde{\mathbf{T}}_p-\tilde{\mathbf{T}}_0)\\[1mm]
&+\int_0^t (J^P(\tilde{X}_{\varepsilon})-J^P_{\varepsilon})(\tilde{\zeta}_{\varepsilon}-\mathcal{I})t(\nabla\hat{\phi}-\nabla\hat{\phi}_{\varepsilon})\tilde{\mathbf{T}}_0+\int_0^t J^P_{\varepsilon}(\tilde{\zeta}_{\varepsilon}-\mathcal{I})t(\nabla\hat{\phi}-\nabla\hat{\phi}_{\varepsilon})(\tilde{\mathbf{T}}_p-\tilde{\mathbf{T}}_0)\\[1mm]
&+\int_0^t J^P_{\varepsilon}(\tilde{\zeta}_{\varepsilon}-\mathcal{I})t(\nabla\hat{\phi}-\nabla\hat{\phi}_{\varepsilon})\tilde{\mathbf{T}}_0+\int_0^t (J^P(\tilde{X}_{\varepsilon})-J^P_{\varepsilon})t(\nabla\hat{\phi}-\nabla\hat{\phi}_{\varepsilon})(\tilde{\mathbf{T}}_p-\tilde{\mathbf{T}}_0)\\[1mm]
&+\int_0^t (J^P(\tilde{X}_{\varepsilon})-J^P_{\varepsilon})t(\nabla\hat{\phi}-\nabla\hat{\phi}_{\varepsilon})\tilde{\mathbf{T}}_0+\int_0^t J^P_{\varepsilon}t(\nabla\hat{\phi}-\nabla\hat{\phi}_{\varepsilon})(\tilde{\mathbf{T}}_p-\tilde{\mathbf{T}}_0)\\[1mm]
&+\int_0^t J^P_{\varepsilon}t(\nabla\hat{\phi}-\nabla\hat{\phi}_{\varepsilon})\tilde{\mathbf{T}}_0+\int_0^t (J^P(\tilde{X}_{\varepsilon})-J^P_{\varepsilon})(\tilde{\zeta}_{\varepsilon}-\mathcal{I})t\nabla\hat{\phi}_{\varepsilon}(\tilde{\mathbf{T}}_p-\tilde{\mathbf{T}}_{p,\varepsilon})\\[1mm]
&+\int_0^t J^P_{\varepsilon}(\tilde{\zeta}_{\varepsilon}-\mathcal{I})t\nabla\hat{\phi}_{\varepsilon}(\tilde{\mathbf{T}}_p-\tilde{\mathbf{T}}_{p,\varepsilon})+\int_0^t (J^P(\tilde{X}_{\varepsilon})-J^P_{\varepsilon})t\nabla\hat{\phi}_{\varepsilon}(\tilde{\mathbf{T}}_p-\tilde{\mathbf{T}}_{p,\varepsilon})+\int_0^t J^P_{\varepsilon}t\nabla\hat{\phi}_{\varepsilon}(\tilde{\mathbf{T}}_p-\tilde{\mathbf{T}}_{p,\varepsilon})
\end{align*}

\medskip

\noindent The estimate of this term is obtained in the same way as $I_1, I_2$ but here we have the difference $\hat{\phi}-\hat{\phi}_{\varepsilon}$. First of all we remark that $\hat{\phi}$ does not depend on time but only on $\tilde{v}_0, \tilde{\mathbf{T}}_0$ and $t\in H^1_{(0)}([0,T])$ allows us to treat properly $t\hat{\phi}, t\hat{\phi}_{\varepsilon}$ and $t(\hat{\phi}-\hat{\phi}_{\varepsilon})$. Then the final estimates is 
\begin{align*}
&\|I_3\|_{H^{2}_{(0)}H^{\gamma-1}}\leq C(\tilde{v}_0,\tilde{\mathbf{T}}_0,\re,\kappa)\left(1+\frac{1}{\We}\right)\varepsilon+  C(\tilde{v}_0,\tilde{\mathbf{T}}_0,\re,\kappa) T^{\varrho_2}\\
&\hspace{0.5cm}\cdot\left(\|\tilde{X}-\tilde{X}_{\varepsilon}+\varepsilon b-t(J^P-J^P_{\varepsilon})\tilde{v}_0\|_{\mathcal{F}^{s+1,\gamma}}+\|\tilde{\mathbf{T}}_p-\tilde{\mathbf{T}}_{p,\varepsilon}-t\hat{\mathbf{T}}_{\varepsilon}\|_{\mathcal{F}^{s,\gamma-1}}\right).
\end{align*}

\noindent As we have already stated, concerning the estimates of $I_4, I_5, I_6$ we can observe that we will get the same results as $I_1, I_2$ and $I_3$, respectively. For the sake of simplicity we avoid the splitting and the estimates of these integrals. We pass to show how to manage $I_7$, by adding $t\hat{\mathbf{T}}_{\varepsilon}$ which guarantees enough cancelations and it can be estimated in $H^1_{(0)}H^{\gamma-1}$, in the way shown before but to conclude we use lemma \ref{lem2} with $0<\varrho_3<\eta$.

\begin{align*}
&\|I_7\|_{H^2_{(0)}H^{\gamma-1}}\leq \frac{1}{\We}\|\tilde{\mathbf{T}}_p-\tilde{\mathbf{T}}_{p,\varepsilon}\|_{H^1_{(0)}H^{\gamma-1}}\leq \frac{1}{\We}\left(\|\tilde{\mathbf{T}}_p-\tilde{\mathbf{T}}_{p,\varepsilon}-t\hat{\mathbf{T}}_{\varepsilon}\|_{H^1_{(0)}H^{\gamma-1}}+\|t\hat{\mathbf{T}}_{\varepsilon}\|_{H^1_{(0)}H^{\gamma-1}}\right)\\[2mm]
&\leq C(\tilde{v}_0,\tilde{\mathbf{T}}_0,\kappa)\left(1+\frac{1}{\We}\right)\varepsilon+\frac{1}{\We}\left\|\int_0^t\partial_t(\tilde{\mathbf{T}}_p-\tilde{\mathbf{T}}_{p,\varepsilon}-t\hat{\mathbf{T}}_{\varepsilon})\right\|_{H^{1+\varrho_3-\eta}_{(0)}H^{\gamma-1}}\\[2mm]
&\leq C(\tilde{v}_0,\tilde{\mathbf{T}}_0,\kappa)\left(1+\frac{1}{\We}\right)\varepsilon+\frac{1}{\We}T^{\varrho_3}\|\tilde{\mathbf{T}}_p-\tilde{\mathbf{T}}_{p,\varepsilon}-t\hat{\mathbf{T}}_{\varepsilon}\|_{\mathcal{F}^{s,\gamma-1}}.
\end{align*}

\noindent Now for the remaining six integrals, we observe that it is enough to study $I_8, I_9, I_{10}$ since the others are just their transpose. Moreover $I_8, I_9, I_{10}$ come from the definition of the velocity $\tilde{v}$ than we show how to rewrite these integrals to get all the required cancelations at time zero. In addition we remark that in $I_9$ there is the presence of another term of $\hat{\mathbf{T}}_{\varepsilon}$.

\begin{align*}
&I_8=\frac{\kappa}{\We}\int_0^t (J^P(\tilde{X})-J^P(\tilde{X}_{\varepsilon})-J^P+J^P_{\varepsilon})(\tilde{\zeta}-\mathcal{I})\nabla\tilde{w}+\frac{\kappa}{\We}\int_0^t (J^P-J^P_{\varepsilon})(\tilde{\zeta}-\mathcal{I})\nabla\tilde{w}\\[1mm]
&\hspace{0.5cm}+\frac{\kappa}{\We}\int_0^t (J^P(\tilde{X})-J^P(\tilde{X}_{\varepsilon})-J^P+J^P_{\varepsilon})\nabla\tilde{w}+\frac{\kappa}{\We}\int_0^t (J^P-J^P_{\varepsilon})\nabla\tilde{w}\\[1mm]
&\hspace{0.5cm}+\frac{\kappa}{\We}\int_0^t (J^P(\tilde{X}_{\varepsilon})-J^P_{\varepsilon})(\tilde{\zeta}-\tilde{\zeta}_{\varepsilon})\nabla\tilde{w}+\frac{\kappa}{\We}\int_0^t J^P_{\varepsilon}(\tilde{\zeta}-\tilde{\zeta}_{\varepsilon})\nabla\tilde{w}\\[1mm]
&\hspace{0.5cm}+\frac{\kappa}{\We}\int_0^t (J^P(\tilde{X}_{\varepsilon})-J^P_{\varepsilon})(\tilde{\zeta}_{\varepsilon}-\mathcal{I})(\nabla\tilde{w}-\nabla\tilde{w}_{\varepsilon})+\frac{\kappa}{\We}\int_0^t J^P_{\varepsilon}(\tilde{\zeta}_{\varepsilon}-\mathcal{I})(\nabla\tilde{w}-\nabla\tilde{w}_{\varepsilon})\\[1mm]
&\hspace{0.5cm}+\frac{\kappa}{\We}\int_0^t (J^P(\tilde{X}_{\varepsilon})-J^P_{\varepsilon})(\nabla\tilde{w}-\nabla\tilde{w}_{\varepsilon})+\frac{\kappa}{\We}\int_0^t J^P_{\varepsilon}(\nabla\tilde{w}-\nabla\tilde{w}_{\varepsilon})=\sum_{i=1}^{10} I_{8,i}\\[4mm]
&I_9=\frac{\kappa}{\We}\int_0^t (J^P(\tilde{X})-J^P(\tilde{X}_{\varepsilon})-J^P+J^P_{\varepsilon})(\tilde{\zeta}-\mathcal{I})\nabla\tilde{v}_0\\[1mm]
&\hspace{0.5cm}+\frac{\kappa}{\We}\int_0^t (J^P-J^P_{\varepsilon})(\tilde{\zeta}-\mathcal{I})\nabla\tilde{v}_0+\frac{\kappa}{\We}\int_0^t (J^P(\tilde{X})-J^P(\tilde{X}_{\varepsilon})-J^P+J^P_{\varepsilon})\nabla\tilde{v}_0\\[1mm]
&\hspace{0.5cm}+\frac{\kappa}{\We}\int_0^t (J^P(\tilde{X}_{\varepsilon})-J^P_{\varepsilon})(\tilde{\zeta}-\tilde{\zeta}_{\varepsilon})\nabla\tilde{v}_0+\frac{\kappa}{\We}\int_0^t J^P_{\varepsilon}(\tilde{\zeta}-\tilde{\zeta}_{\varepsilon})\nabla\tilde{v}_0=\sum_{i=1}^{5}I_{9,i}
\end{align*}
\begin{align*}
&I_{10}=\frac{\kappa}{\We}\int_0^t (J^P(\tilde{X})-J^P(\tilde{X}_{\varepsilon})-J^P+J^P_{\varepsilon})(\tilde{\zeta}-\mathcal{I})t\nabla\hat{\phi}+\frac{\kappa}{\We}\int_0^t (J^P-J^P_{\varepsilon})(\tilde{\zeta}-\mathcal{I})t\nabla\hat{\phi}\\[1mm]
&\hspace{0.5cm}+\frac{\kappa}{\We}\int_0^t (J^P(\tilde{X})-J^P(\tilde{X}_{\varepsilon})-J^P+J^P_{\varepsilon})t\nabla\hat{\phi}+\frac{\kappa}{\We}\int_0^t (J^P-J^P_{\varepsilon})t\nabla\hat{\phi}\\[2mm]
&\hspace{0.5cm}+\frac{\kappa}{\We}\int_0^t (J^P(\tilde{X}_{\varepsilon})-J^P_{\varepsilon})(\tilde{\zeta}-\tilde{\zeta}_{\varepsilon})t\nabla\hat{\phi}+\frac{\kappa}{\We}\int_0^t J^P_{\varepsilon}(\tilde{\zeta}-\tilde{\zeta}_{\varepsilon})t\nabla\hat{\phi}\\[2mm]
&\hspace{0.5cm}+\frac{\kappa}{\We}\int_0^t (J^P(\tilde{X}_{\varepsilon})-J^P_{\varepsilon})(\tilde{\zeta}_{\varepsilon}-\mathcal{I})t(\nabla\hat{\phi}-\nabla\hat{\phi}_{\varepsilon})+\frac{\kappa}{\We}\int_0^t J^P_{\varepsilon}(\tilde{\zeta}_{\varepsilon}-\mathcal{I})t(\nabla\hat{\phi}-\nabla\hat{\phi}_{\varepsilon})\\[2mm]
&\hspace{0.5cm}+\frac{\kappa}{\We}\int_0^t (J^P(\tilde{X}_{\varepsilon})-J^P_{\varepsilon})t(\nabla\hat{\phi}-\nabla\hat{\phi}_{\varepsilon})+\frac{\kappa}{\We}\int_0^t J^P_{\varepsilon}t(\nabla\hat{\phi}-\nabla\hat{\phi}_{\varepsilon})=\sum_{i=1}^{10} I_{10,i}.
\end{align*}

\noindent We analyze some of these integrals, in particular we focus on $I_{8,7}$, which gives the velocities difference, on $I_{9,1}$, which gives the flux difference and on $I_{10,7}$ in order to understand how to deal with $t(\hat{\phi}-\hat{\phi}_{\varepsilon})$ in $H^{1}_{(0)}H^{\gamma-1}-$norm. For all the integrals we use lemma \ref{lem2} with $\varepsilon=0$, lemma \ref{lem3} with $\gamma>1$ and lemma \ref{lem5} in order to separate each term. Then for $I_{8,7}$ we use lemma \ref{Jp-est}, lemma \ref{zeta-est} and for the velocity lemma \ref{lem2} and lemma \ref{lem1} with $\beta_7<\eta_7<\frac{s-1-\gamma}{2}$. For $I_{9,1}$ we require $\tilde{v}_0$ to be enough regular in order to separate it from $\tilde{\zeta}_{\varepsilon}-\mathcal{I}$ and we apply lemma \ref{Jp-est} and lemma \ref{Jp-dif-est}, then lemma \ref{zeta-est} and in the end lemma \ref{lem2} with $0<\theta_1<\eta_1$. For the last integral $I_{10,7}$, we use lemma \ref{Jp-est} and lemma \ref{zeta-est} and the difference of $\hat{\phi}-\hat{\phi}_{\varepsilon}$ give the dependence on $\tilde{\mathbf{T}}_0$ of the constant and to conclude it is fundamental the fact that $t\in H^1_{(0)}([0,T])$.

\begin{align*}
&\|I_{8,7}\|_{H^2_{(0)}H^{\gamma-1}}\leq \frac{\kappa}{\We}\|(J^P(\tilde{X}_{\varepsilon})-J^P_{\varepsilon})(\tilde{\zeta}_{\varepsilon}-\mathcal{I})(\nabla\tilde{w}-\nabla\tilde{w}_{\varepsilon})\|_{H^1_{(0)}H^{\gamma-1}}\\[2mm]
&\leq \frac{\kappa}{\We}\|J^P(\tilde{X}_{\varepsilon})-J^P_{\varepsilon}\|_{H^1_{(0)}H^{\gamma}}\|(\tilde{\zeta}_{\varepsilon}-\mathcal{I})(\nabla\tilde{w}-\nabla\tilde{w}_{\varepsilon})\|_{H^1_{(0)}H^{\gamma-1}}\\[2mm]
&\leq \frac{\kappa}{\We}\|\tilde{X}_{\varepsilon}-\tilde{\alpha}-\varepsilon b\|_{H^{1}_{(0)}H^{\gamma}}\|\tilde{\zeta}_{\varepsilon}-\mathcal{I}\|_{H^1_{(0)}H^{\gamma-1}}\|\nabla\tilde{w}-\nabla\tilde{w}_{\varepsilon}\|_{H^1_{(0)}H^{\gamma-1}}\\[2mm]
&\leq C(\tilde{v}_0,\kappa)\left\|\int_0^t\partial_t(\tilde{w}-\tilde{w}_{\varepsilon})\right\|_{H^{1+\eta_7-\beta_7}_{(0)}H^{\gamma-1}}\leq C(\tilde{v}_0,\kappa) T^{\beta_7}\|\tilde{w}-\tilde{w}_{\varepsilon}\|_{\mathcal{K}^{s+1}_{(0)}},\\[4mm]
&\|I_{9,1}\|_{H^2_{(0)}H^{\gamma-1}}\leq \frac{\kappa}{\We}\|(J^P(\tilde{X})-J^P(\tilde{X}_{\varepsilon})-J^P+J^P_{\varepsilon})(\tilde{\zeta}-\mathcal{I})\nabla\tilde{v}_0\|_{H^{1}_{(0)}H^{\gamma-1}}\\[2mm]
&\leq\frac{\kappa}{\We}\left(\|J^P(\tilde{X}+\varepsilon b)-J^P(\tilde{X}_{\varepsilon})\|_{H^{1}_{(0)}H^{\gamma}}+\|J^P(\tilde{X})-J^P-J^P(\tilde{X}+\varepsilon b)+J^P_{\varepsilon}\|_{H^{1}_{(0)}H^{\gamma}}\right)\\[1mm]
&\hspace{0.5cm}\cdot\|(\tilde{\zeta}-\mathcal{I})\nabla\tilde{v}_0\|_{H^{1}_{(0)}H^{\gamma-1}}\\[2mm]
&\leq\frac{\kappa}{\We}\left(\|\tilde{X}-\tilde{X}_{\varepsilon}+\varepsilon b\|_{H^{1}_{(0)}H^{\gamma}}+\varepsilon\right)\|\tilde{\zeta}-\mathcal{I}\|_{H^{1}_{(0)}H^{\gamma-1}}\|\nabla\tilde{v}_0\|_{H^{\gamma}}\\[2mm]
&\leq C(\tilde{v}_0,\kappa)\frac{1}{\We}\varepsilon+C(\tilde{v}_0,\kappa)\frac{1}{\We} \left(\|\tilde{X}-\tilde{X}_{\varepsilon}+\varepsilon b-t(J^P-J^P_{\varepsilon})\tilde{v}_0\|_{H^{1}_{(0)}H^{\gamma}}+\|t(J^P-J^P_{\varepsilon})\tilde{v}_0\|_{H^{1}_{(0)}H^{\gamma}}\right)\\[2mm]
\end{align*}
\begin{align*}
&\leq C(\tilde{v}_0,\kappa)\frac{1}{\We}\varepsilon+C(\tilde{v}_0,\kappa)\frac{1}{\We} \left\|\int_0^t\partial_t(\tilde{X}-\tilde{X}_{\varepsilon}+\varepsilon b-t(J^P-J^P_{\varepsilon})\tilde{v}_0\right\|_{H^{1+\eta_1-\theta_1}_{(0)}H^{\gamma}}\\[2mm]
&\leq C(\tilde{v}_0,\kappa)\frac{1}{\We}\varepsilon+C(\tilde{v}_0,\kappa)\frac{1}{\We} T^{\theta_1}\|\tilde{X}-\tilde{X}_{\varepsilon}+\varepsilon b-t(J^P-J^P_{\varepsilon})\tilde{v}_0\|_{\mathcal{F}^{s+1,\gamma}}\\[4mm]
&\|I_{10,7}\|_{H^2_{(0)}H^{\gamma-1}}\leq \frac{\kappa}{\We}\| (J^P(\tilde{X}_{\varepsilon})-J^P_{\varepsilon})(\tilde{\zeta}_{\varepsilon}-\mathcal{I})t(\nabla\hat{\phi}-\nabla\hat{\phi}_{\varepsilon})\|_{H^1_{(0)}H^{\gamma-1}}\\[2mm]
&\leq \frac{\kappa}{\We}\|J^P(\tilde{X}_{\varepsilon})-J^P_{\varepsilon}\|_{H^1_{(0)}H^{\gamma-1}}\|(\tilde{\zeta}_{\varepsilon}-\mathcal{I})t(\nabla\hat{\phi}-\nabla\hat{\phi}_{\varepsilon})\|_{H^1_{(0)}H^{\gamma-1}}\\[2mm]
&\leq \frac{\kappa}{\We}\|\tilde{X}_{\varepsilon}-\tilde{\alpha}-\varepsilon b\|_{H^{1}_{(0)}H^{\gamma}}\|\tilde{\zeta}_{\varepsilon}-\mathcal{I}\|_{H^1_{(0)}H^{\gamma-1}}\|t\|_{H^1_{(0)}}\|\hat{\phi}-\hat{\phi}_{\varepsilon}\|_{H^{\gamma}}\\[2mm]
&\leq C(\tilde{v}_0, \tilde{\mathbf{T}}_0,\re,\kappa)\frac{1}{\We}\varepsilon.
\end{align*}
\medskip

\noindent The estimates of the remaining terms can be deduced from these three estimates above. The proof of the lemma holds by choosing
$\beta=\min\{\frac{1}{4},\delta_i,\varrho_1, \varrho_2,\varrho_3, \beta_j, \theta_k\},$ where the $\delta_i$ come from $I_{1,i}$, the $\beta_j$ come from $I_{8,j}$ and the $\theta_k$ come from $I_{9,k}$.
\end{proof}

\noindent Now we pass to get stability estimates for the velocity and the pressure.

\begin{lemma}
For  $2<s<\frac{5}{2}$ and a suitable choice of $\varrho>0$,  we have
\begin{equation*}
\begin{split}
&\|\tilde{w}-\tilde{w}_{\varepsilon}\|_{\mathcal{K}^{s+1}_{(0)}} +\|\tilde{q}_w-\tilde{q}_{w,\varepsilon}\|_{\mathcal{K}^{s}_{pr(0)}}\leq C(\tilde{v}_0,\tilde{\mathbf{T}}_0,\re,\kappa)\left(1+\frac{1}{\We}\right)\varepsilon\\[2mm]
&+C(\tilde{v}_0,\tilde{\mathbf{T}}_0,\re,\kappa)\left(1+\frac{1}{\We}\right)T^{\varrho}\left( \|\tilde{w}-\tilde{w}_{\varepsilon}\|_{\mathcal{K}^{s+1}_{(0)}} +\|\tilde{q}_w-\tilde{q}_{w,\varepsilon}\|_{\mathcal{K}^{s}_{pr(0)}}\right.\\[2mm]
&\left.+\|\tilde{\mathbf{T}}_p-\tilde{\mathbf{T}}_{p,\varepsilon}-t\hat{\mathbf{T}}_{\varepsilon}\|_{\mathcal{F}^{s,\gamma-1}}+\|\tilde{X}-\tilde{X}_{\varepsilon}+\varepsilon b-t(J^P-J^P_{\varepsilon})\tilde{v}_0\|_{\mathcal{F}^{s+1,\gamma}}\right).
\end{split}
\end{equation*}
\end{lemma}

\begin{proof}
As we did for the iterative bounds we use the invertibility of the operator $L$ which defines the system \eqref{w-stab}, specifically $L(\tilde{w}-\tilde{w}_{\varepsilon},\tilde{q}_{w}-\tilde{q}_{w,\varepsilon})=(\tilde{F}_{\varepsilon},\tilde{K}_{\varepsilon},\tilde{H}_{\varepsilon})$, then we have 
\begin{align*}
&\|\tilde{w}-\tilde{w}_{\varepsilon}\|_{\mathcal{K}^{s+1}_{(0)}} +\|\tilde{q}_w-\tilde{q}_{w,\varepsilon}\|_{\mathcal{K}^{s}_{pr(0)}}
\leq \left( \|\tilde{F}_{\varepsilon}\|_{\mathcal{K}^{s-1}_{(0)}}+\|\tilde{K}_{\varepsilon}\|_{\mathcal{\bar{K}}^{s}_{(0)}}+\|\tilde{H}_{\varepsilon}\|_{\mathcal{K}^{s-\frac{1}{2}}_{(0)}}\right).
\end{align*}
\medskip

\noindent So what we have to show is the following
\begin{align*}
&\|\tilde{F}_{\varepsilon}\|_{\mathcal{K}^{s-1}_{(0)}}+\|\tilde{K}_{\varepsilon}\|_{\mathcal{\bar{K}}^{s}_{(0)}}+\|\tilde{H}_{\varepsilon}\|_{\mathcal{K}^{s-\frac{1}{2}}_{(0)}}\leq C(\tilde{v}_0,\tilde{\mathbf{T}}_0,\re,\kappa)\left(1+\frac{1}{\We}\right)\varepsilon\\[2mm]
&+C(\tilde{v}_0,\tilde{\mathbf{T}}_0,\re,\kappa)\left(1+\frac{1}{\We}\right)T^{\varrho}\left( \|\tilde{w}-\tilde{w}_{\varepsilon}\|_{\mathcal{K}^{s+1}_{(0)}} +\|\tilde{q}_w-\tilde{q}_{w,\varepsilon}\|_{\mathcal{K}^{s}_{pr(0)}}\right.\\[2mm]
&\left.+\|\tilde{\mathbf{T}}_p-\tilde{\mathbf{T}}_{p,\varepsilon}-t\hat{\mathbf{T}}_{\varepsilon}\|_{\mathcal{F}^{s,\gamma-1}}+\|\tilde{X}-\tilde{X}_{\varepsilon}+\varepsilon b-t(J^P-J^P_{\varepsilon})\tilde{v}_0\|_{\mathcal{F}^{s+1,\gamma}}\right).
\end{align*}

\underline{\textbf{Estimate for $\tilde{F}_{\varepsilon}$}}\\

\noindent As defined above $\tilde{F}_{\varepsilon}=\tilde{f}-\tilde{f}_{\varepsilon}+\tilde{f}_{\phi}^L-\tilde{f}_{\phi,\varepsilon}^L+(1-\kappa)(Q^2-Q^2_{\varepsilon})\Delta \tilde{w}_{\varepsilon}-((J^P)^T-(J_{\varepsilon}^P)^T)\nabla \tilde{q}_{w,\varepsilon}$ and we have to do estimates in $L^2H^s$ and $H^{\frac{s-1}{2}}_{(0)}L^2$. We start with the analysis of the  \\
$(1-\kappa)(Q^2-Q^2_{\varepsilon})\Delta \tilde{w}_{\varepsilon}-((J^P)^T-(J_{\varepsilon}^P)^T)\nabla \tilde{q}_{w,\varepsilon}$ and by using proposition \ref{stability} (1)-(5) we have the following estimates in $\mathcal{K}^{s-1}_{(0)}$.

\begin{align*}
&(1-\kappa)\|(Q^2-Q^2_{\varepsilon})\Delta w_{\varepsilon}\|_{L^2H^{s-1}}\leq (1-\kappa) \|Q^2-Q^2_{\varepsilon}\|_{L^{\infty}H^{s-1}}\|w_{\varepsilon}\|_{L^2H^{s+1}}\leq C(\tilde{v}_0,\kappa)\varepsilon\\[3mm]
&(1-\kappa)\|(Q^2-Q^2_{\varepsilon})\Delta w_{\varepsilon}\|_{H^{\frac{s-1}{2}}_{(0)}L^2}\leq(1-\kappa)\|Q^2-Q^2_{\varepsilon}\|_{H^{\frac{s-1}{2}}_{(0)}H^{1+\eta}}\|\Delta w_{\varepsilon}\|_{H^{\frac{s-1}{2}}_{(0)}L^2}\leq C(\tilde{v}_0,\kappa)\varepsilon,\\[3mm]
&\left\|((J^P)^T-(J^P_{\varepsilon})^T)\nabla {q}_{w,\varepsilon}\right\|_{L^2H^{s-1}}\leq \|(J^P)^T-(J^P_{\varepsilon})^T\|_{L^{\infty}H^{s-1}}\| {q}_{w,\varepsilon}\|_{L^2H^s}\leq C(\tilde{v}_0)\varepsilon\\[3mm]
&\|((J^P)^T-(J^P_{\varepsilon})^T)\nabla {q}_{w,\varepsilon}\|_{H^{\frac{s-1}{2}}_{(0)}L^2}\leq \|(J^P)^T-(J^P_{\varepsilon})^T\|_{H^{\frac{s-1}{2}}_{(0)}H^{1+\eta}}\| {q}_{w,\varepsilon}\|_{H^{\frac{s-1}{2}}_{(0)}H^1}\leq C(\tilde{v}_0)\varepsilon.
\end{align*}
\medskip

\noindent Now, we focus on the estimate of $\tilde{f}_{\phi}^L-\tilde{f}_{\phi,\varepsilon}^L=\trace(\nabla\tilde{\mathbf{T}}_0(J^P-J^P_{\varepsilon}))+t(1-\kappa)(Q^2-Q^2_{\varepsilon})\Delta\hat{\phi}+t(1-\kappa)Q^2_{\varepsilon}(\Delta\hat{\phi}-\Delta\hat{\phi}_{\varepsilon})=I_1+I_2+I_3.$ We observe that for $I_2$ and $I_3$ the presence of $t$ in front of $\hat{\phi}=\frac{1}{\re}\left((1-\kappa)Q^2\Delta \tilde{v}_0-(J^P)^T\nabla \tilde{q}_{\phi}+\trace(\nabla\tilde{\mathbf{T}}_0J^P)\right)$, which depend only on the initial data allows us to obtain estimates in $\mathcal{K}^{s-1}_{(0)}$. We cannot state the same for $I_1$ since it depends only on the initial data and it is impossible to get bounds in $\mathcal{K}^{s-1}_{(0)}$. For this reason we show below how to deal with $I_2$ and $I_3$ and it will be clear why we require enough regularity for $\tilde{v}_0,\tilde{\mathbf{T}}_0$. On the other hand we will put $I_1$ together with $\tilde{f}-\tilde{f}_{\varepsilon}$.

\begin{align*}
&\|I_2\|_{L^2H^{s-1}}\leq C(\kappa) \|t\|_{L^2}\|Q^2-Q^2_{\varepsilon}\|_{H^{s}}\|\hat{\phi}\|_{H^{s+1}}\leq  C(\tilde{v}_0,\tilde{\mathbf{T}}_0\re,\kappa)\varepsilon\\[2mm]
&\|I_2\|_{H^{\frac{s-1}{2}}_{(0)}L^2}\leq C(\kappa)\|t\|_{H^{\frac{s-1}{2}}_{(0)}}\|Q^2-Q^2_{\varepsilon}\|_{H^1}\|\hat{\phi}\|_{H^{2}}\leq  C(\tilde{v}_0,\tilde{\mathbf{T}}_0\re,\kappa)\varepsilon\\[2mm]
&\|I_3\|_{L^2H^{s-1}}\leq C(\kappa) \|t\|_{L^2}\|\hat{\phi}-\hat{\phi}_{\varepsilon}\|_{H^{s+1}}\leq  C(\tilde{v}_0,\tilde{\mathbf{T}}_0\re,\kappa)\varepsilon\\[2mm]
&\|I_3\|_{H^{\frac{s-1}{2}}_{(0)}L^2}\leq C(\kappa) \|t\|_{H^{\frac{s-1}{2}}_{(0)}}\|\hat{\phi}-\hat{\phi}_{\varepsilon}\|_{H^{2}}\leq  C(\tilde{v}_0,\tilde{\mathbf{T}}_0\re,\kappa)\varepsilon.
\end{align*}

\noindent As we did in the proof of Proposition \ref{estimate-conf-lag-(v,q)} we write $\tilde{f}-\tilde{f}_{\varepsilon}=\tilde{f}_w-\tilde{f}_{w,\varepsilon}+\tilde{f}_{\phi}-\tilde{f}_{\phi,\varepsilon}+\tilde{f}_{q}-\tilde{f}_{q,\varepsilon}+\tilde{f}_T-\tilde{f}_{T,\varepsilon}$ and by identifying $\tilde{f}^{(n)}$ with $\tilde{f}$ and $\tilde{f}^{(n-1)}$ with $\tilde{f}_{\varepsilon}$, we obtain similar results. In particular, we focus on $\tilde{f}_T-\tilde{f}_{T,\varepsilon}$, since the details for the other terms can be checked in \cite[Lemma 6.2]{CCFGG2} and the result is

\begin{align*}
&\|\tilde{f}_w-\tilde{f}_{w,\varepsilon}+\tilde{f}_{\phi}-\tilde{f}_{\phi,\varepsilon}+\tilde{f}_{q}-\tilde{f}_{q,\varepsilon}\|_{\mathcal{K}^{s-1}_{(0)}}\leq C(\tilde{v}_0,\tilde{\mathbf{T}}_0,\re,\kappa)\varepsilon+C(\tilde{v}_0,\tilde{\mathbf{T}}_0,\re,\kappa)T^{\delta_1}\\
&\cdot\left(\|\tilde{X}-\tilde{X}_{\varepsilon}+\varepsilon b-t(J^P-J^P_{\varepsilon})\tilde{v}_0\|_{\mathcal{F}^{s+1,\gamma}}+\|\tilde{w}-\tilde{w}_{\varepsilon}\|_{\mathcal{K}^{s+1}_{(0)}}+\|\tilde{q}_{w}-\tilde{q}_{w,\varepsilon}\|_{\mathcal{K}^s_{pr(0)}}\right)
\end{align*}

\noindent We consider $\tilde{f}_T-\tilde{f}_{T,\varepsilon}+I_1=\trace(J^P(\tilde{X})\tilde{\zeta}\nabla\tilde{\mathbf{T}}_p)-\trace(J^P(\tilde{X}_{\varepsilon})\tilde{\zeta}_{\varepsilon}\nabla\tilde{\mathbf{T}}_{p,\varepsilon})-\trace((J^P-J^P_{\varepsilon})\nabla\tilde{\mathbf{T}}_0)$. First we deal with the simplest norm, $L^2H^{s-1}$ and we split as we did in proposition \ref{estimate-conf-lag-(v,q)} for $\tilde{f}^{(n)}_T-\tilde{f}^{(n-1)}_T$, but here we have $I_1$ and we have to pay attention in the estimate of $J^P(\tilde{X})-J^P(\tilde{X}_{\varepsilon})$ which is not immediate. We split this term as follows

\begin{align*}
&\tilde{f}_T-\tilde{f}_{T,\varepsilon}+I_1=\trace((J^P(\tilde{X})-J^P(\tilde{X}_{\varepsilon})-J^P+J^P_{\varepsilon})\tilde{\zeta}(\nabla\tilde{\mathbf{T}}_p-\nabla\tilde{\mathbf{T}}_0))\\[1mm]
&+\trace((J^P-J^P_{\varepsilon})\tilde{\zeta}(\nabla\tilde{\mathbf{T}}_p-\nabla\tilde{\mathbf{T}}_0))+\trace((J^P(\tilde{X})-J^P(\tilde{X}_{\varepsilon})-J^P+J^P_{\varepsilon})\tilde{\zeta}\nabla\tilde{\mathbf{T}}_0)\\[1mm]
&+\trace((J^P-J^P_{\varepsilon})(\tilde{\zeta}-\mathcal{I})\nabla\tilde{\mathbf{T}}_0)+\trace((J^P(\tilde{X}_{\varepsilon})-J^P_{\varepsilon})(\tilde{\zeta}-\tilde{\zeta}_{\varepsilon})(\nabla\tilde{\mathbf{T}}_p-\nabla\tilde{\mathbf{T}}_0))\\[1mm]
&+\trace(J^P_{\varepsilon}(\tilde{\zeta}-\tilde{\zeta}_{\varepsilon})(\nabla\tilde{\mathbf{T}}_p-\nabla\tilde{\mathbf{T}}_0))+\trace((J^P(\tilde{X}_{\varepsilon})-J^P_{\varepsilon})(\tilde{\zeta}-\tilde{\zeta}_{\varepsilon})\nabla\tilde{\mathbf{T}}_0)\\[1mm]
&+\trace(J^P_{\varepsilon}(\tilde{\zeta}-\tilde{\zeta}_{\varepsilon})\nabla\tilde{\mathbf{T}}_0)+\trace((J^P(\tilde{X}_{\varepsilon})-J^P_{\varepsilon})\tilde{\zeta}_{\varepsilon}(\nabla\tilde{\mathbf{T}}_p-\nabla\tilde{\mathbf{T}}_{p,\varepsilon}))\\[1mm]
&+\trace(J^P_{\varepsilon}\tilde{\zeta}_{\varepsilon}(\nabla\tilde{\mathbf{T}}_p-\nabla\tilde{\mathbf{T}}_{p,\varepsilon}))
\end{align*}
\medskip

\noindent The final estimate is the following

\begin{align*}
&\|\tilde{f}_T-\tilde{f}_{T,\varepsilon}+I_1\|_{L^2H^s}\leq C(\tilde{v}_0,\tilde{\mathbf{T}}_0,\re,\kappa)\left(1+\frac{1}{\We}\right)\varepsilon+C(\tilde{v}_0,\tilde{\mathbf{T}}_0,\re,\kappa)\left(1+\frac{1}{\We}\right)T^{\frac{3}{4}}\\
&\cdot\left(\|\tilde{X}-\tilde{X}_{\varepsilon}+\varepsilon b-t(J^P-J^P_{\varepsilon})\tilde{v}_0\|_{\mathcal{F}^{s+1,\gamma}}+\|\tilde{\mathbf{T}}_p-\tilde{\mathbf{T}}_{p,\varepsilon}-t\hat{\mathbf{T}}_{\varepsilon}\|_{\mathcal{F}^{s,\gamma-1}}\right),
\end{align*}

\noindent where the constant depends on some parameters that come from the definition of $\hat{\mathbf{T}}_{\varepsilon}$ and the presence of the Weissenberg number comes from the estimate \eqref{elastic-estim}. Moreover the critical term is $\|J^P(\tilde{X})-J^P(\tilde{X}_{\varepsilon})-J^P+J^P_{\varepsilon}\|_{L^{\infty}H^s}\leq \|J^P(\tilde{X}+\varepsilon b)-J^P(\tilde{X}_{\varepsilon})\|_{L^{\infty}H^s}+\|J^P(\tilde{X})-J^P-(J^P(\tilde{X}+\varepsilon b)-J^P_{\varepsilon})\|_{L^{\infty}H^s}\leq C(\tilde{v}_0)T^{\frac{1}{4}}\|\tilde{X}-\tilde{X}_{\varepsilon}+\varepsilon b-t(J^P-J^P_{\varepsilon})\tilde{v}_0\|_{\mathcal{F}^{s+1,\gamma}}+C(\tilde{v}_0)\varepsilon,$ by using both lemma \ref{Jp-est} and lemma \ref{Jp-dif-est}. Furthermore in $H^{\frac{s-1}{2}}_{(0)}L^2-$norm we need a more accurate splitting 

\begin{align*}
&\tilde{f}_T-\tilde{f}_{T,\varepsilon}+I_1=\trace((J^P(\tilde{X})-J^P(\tilde{X}_{\varepsilon})-J^P+J^P_{\varepsilon})(\tilde{\zeta}-\mathcal{I})(\nabla\tilde{\mathbf{T}}_p-\nabla\tilde{\mathbf{T}}_0))\\[1mm]
&+\trace((J^P-J^P_{\varepsilon})(\tilde{\zeta}-\mathcal{I})(\nabla\tilde{\mathbf{T}}_p-\nabla\tilde{\mathbf{T}}_0))+\trace((J^P(\tilde{X})-J^P(\tilde{X}_{\varepsilon})-J^P+J^P_{\varepsilon})(\tilde{\zeta}-\mathcal{I})\nabla\tilde{\mathbf{T}}_0)\\[1mm]
&+\trace((J^P-J^P_{\varepsilon})(\tilde{\zeta}-\mathcal{I})\nabla\tilde{\mathbf{T}}_0)+\trace((J^P(\tilde{X})-J^P(\tilde{X}_{\varepsilon})-J^P+J^P_{\varepsilon})(\nabla\tilde{\mathbf{T}}_p-\nabla\tilde{\mathbf{T}}_0))\\[1mm]
&+\trace((J^P-J^P_{\varepsilon})(\nabla\tilde{\mathbf{T}}_p-\nabla\tilde{\mathbf{T}}_0))+\trace((J^P(\tilde{X})-J^P(\tilde{X}_{\varepsilon})-J^P+J^P_{\varepsilon})(\nabla\tilde{\mathbf{T}}_p-\nabla\tilde{\mathbf{T}}_0))\\[1mm]
&+\trace((J^P(\tilde{X}_{\varepsilon})-J^P_{\varepsilon})(\tilde{\zeta}-\tilde{\zeta}_{\varepsilon})(\nabla\tilde{\mathbf{T}}_p-\nabla\tilde{\mathbf{T}}_0))+\trace(J^P_{\varepsilon}(\tilde{\zeta}-\tilde{\zeta}_{\varepsilon})(\nabla\tilde{\mathbf{T}}_p-\nabla\tilde{\mathbf{T}}_0))\\[1mm]
&+\trace((J^P(\tilde{X}_{\varepsilon})-J^P_{\varepsilon})(\tilde{\zeta}-\tilde{\zeta}_{\varepsilon})\nabla\tilde{\mathbf{T}}_0)+\trace(J^P_{\varepsilon}(\tilde{\zeta}-\tilde{\zeta}_{\varepsilon})\nabla\tilde{\mathbf{T}}_0)\\[1mm]
&+\trace((J^P(\tilde{X}_{\varepsilon})-J^P_{\varepsilon})(\tilde{\zeta}_{\varepsilon}-\mathcal{I})(\nabla\tilde{\mathbf{T}}_p-\nabla\tilde{\mathbf{T}}_{p,\varepsilon}))+\trace(J^P_{\varepsilon}(\tilde{\zeta}_{\varepsilon}-\mathcal{I})(\nabla\tilde{\mathbf{T}}_p-\nabla\tilde{\mathbf{T}}_{p,\varepsilon}))\\[1mm]
&+\trace((J^P(\tilde{X}_{\varepsilon})-J^P_{\varepsilon})(\nabla\tilde{\mathbf{T}}_p-\nabla\tilde{\mathbf{T}}_{p,\varepsilon}))+\trace(J^P_{\varepsilon}(\nabla\tilde{\mathbf{T}}_p-\nabla\tilde{\mathbf{T}}_{p,\varepsilon})).
\end{align*}

\noindent The estimate is obtained in the same way as proposition  \ref{estimate-conf-lag-(v,q)} but taking into account the remarks explained for the previous result. Then by using in a key way lemma \ref{lem2} we have

\begin{align*}
&\|\tilde{f}_T-\tilde{f}_{T,\varepsilon}+I_1\|_{H^{\frac{s-1}{2}}_{(0)}L^2}\leq C(\tilde{v}_0,\tilde{\mathbf{T}}_0,\re,\kappa)\left(1+\frac{1}{\We}\right)\varepsilon+C(\tilde{v}_0,\tilde{\mathbf{T}}_0,\re,\kappa)\left(1+\frac{1}{\We}\right)T^{\varrho_1}\\[1mm]
&\cdot\left(\|\tilde{X}-\tilde{X}_{\varepsilon}+\varepsilon b-t(J^P-J^P_{\varepsilon})\tilde{v}_0\|_{\mathcal{F}^{s+1,\gamma}}+\|\tilde{\mathbf{T}}_p-\tilde{\mathbf{T}}_{p,\varepsilon}-t\hat{\mathbf{T}}_{\varepsilon}\|_{\mathcal{F}^{s,\gamma-1}}\right).
\end{align*}
\bigskip

\underline{\textbf{Estimate for $\tilde{K}_{\varepsilon}$}}\\

\noindent This term is defined as follows
$$\tilde{K}_{\varepsilon}=\tilde{g}-\tilde{g}_{\varepsilon}+\tilde{g}_{\phi}^L-\tilde{g}_{\phi,\varepsilon}^L-\trace(\nabla\tilde{w}_{\varepsilon}(J^P-J^P_{\varepsilon})).$$

\noindent It can be estimated by using proposition \ref{stability} (1)-(4)  in $\mathcal{\bar{K}}^{s}_{(0)}$ and we have

\begin{align*}
&\|\tilde{g}_{\phi}^L-\tilde{g}_{\phi,\varepsilon}^L\|_{\mathcal{\bar{K}}^{s}_{(0)}}\leq C(\tilde{v}_0,\tilde{\mathbf{T}}_0,\re,\kappa)\varepsilon,\\[1mm]
&\|\nabla\tilde{w}_{\varepsilon}(J^P-J^P_{\varepsilon})\|_{L^2H^s}\leq \|\tilde{w}_{\varepsilon}\|_{L^2H^{s+1}}\|J^P-J^P_{\varepsilon}\|_{L^{\infty}H^s}\leq C(\tilde{v}_0)\varepsilon,\\[1mm]
&\|\nabla\tilde{w}_{\varepsilon}(J^P-J^P_{\varepsilon})\|_{H^{\frac{s+1}{2}}_{(0)}H^{-1}}\leq\varepsilon \|\tilde{w}_{\varepsilon}\|_{H^{\frac{s+1}{2}}_{(0)}L^2}\leq C(\tilde{v}_0)\varepsilon.
\end{align*}

\noindent For the difference $\tilde{g}-\tilde{g}_{\varepsilon}$, as we stated before , it is the same as proposition \ref{estimate-conf-lag-(v,q)}. Indeed by identifying $\tilde{g}^{(n)}$ with $\tilde{g}$ and $\tilde{g}^{(n-1)}$ with $\tilde{g}_{\varepsilon}$ we obtain

\begin{align*}
\|\tilde{g}-\tilde{g}_{\varepsilon}\|_{\mathcal{\bar{K}}^s_{(0)}}\leq C(\tilde{v}_0,\tilde{\mathbf{T}}_0,\re,\kappa) \varepsilon&+ C(\tilde{v}_0,\tilde{\mathbf{T}}_0,\re,\kappa) T^{\varrho_2}\left(\|\tilde{w}-\tilde{w}_{\varepsilon}\|_{\mathcal{K}^{s+1}_{(0)}}\right.\\
&\left.+\|\tilde{X}-\tilde{X}_{\varepsilon}+\varepsilon b-t(J^P-J^P_{\varepsilon})\tilde{v}_0\|_{\mathcal{F}^{s+1,\gamma}}\right).
\end{align*}

\underline{\textbf{Estimate for $\tilde{H}_{\varepsilon}$}}\\

\noindent As we did for the previous terms, also this term can be estimated as in proposition  \ref{estimate-conf-lag-(v,q)}. The term we have to study is
\begin{align*}
\tilde{H_{\varepsilon}}&=\tilde{h}-\tilde{h}_{\varepsilon}+\tilde{h}_{\phi}^L-\tilde{h}_{\phi,\varepsilon}^L+\tilde{q}_{w,\varepsilon}((J^P)^{-1}-(J^P_{\varepsilon})^{-1})\tilde{n}_0-(1-\kappa)(\nabla\tilde{w}_{\varepsilon}J^P)(J^P)^{-1}\tilde{n}_0\\[1mm]
& -(1-\kappa)(\nabla\tilde{w}_{\varepsilon}J^P)^T (J^P)^{-1}\tilde{n}_0+(1-\kappa)(\nabla\tilde{w}_{\varepsilon}J^P_{\varepsilon})(J^P_{\varepsilon})^{-1}\tilde{n}_0+(1-\kappa)(\nabla\tilde{w}_{\varepsilon}J^P_{\varepsilon})^T(J^P_{\varepsilon})^{-1}\tilde{n}_0\\[1mm]
&=\tilde{h}-\tilde{h}_{\varepsilon}+\tilde{h}_{\phi}^L-\tilde{h}_{\phi,\varepsilon}^L+\bar{H}_{\varepsilon}.
\end{align*}

\noindent We write the term $\bar{H}_{\varepsilon}$ in the following three terms
\begin{align*}
&I_1=q_{w,\varepsilon}((J^P)^{-1}-(J^P_{\varepsilon})^{-1}){n}_0,\\
&I_2=(1-\kappa)[\nabla w_{\varepsilon}(J^P_{\varepsilon}-J^P)+(\nabla w_{\varepsilon}(J^P_{\varepsilon}-J^P))^T](J^P)^{-1} {n}_0,\\
&I_3=(1-\kappa)[(\nabla w_{\varepsilon}J^P_{\varepsilon})+(\nabla w_{\varepsilon} J^P_{\varepsilon}))^T]((J^P_{\varepsilon})^{-1}- (J^P)^{-1}){n}_0,\\
\end{align*}

\noindent By using proposition \ref{stability} (1)-(4), the trace theorem \ref{parabolic-trace} and lemma \ref{lem1}, we have the following results in $\mathcal{K}^{s-\frac{1}{2}}_{(0)}$.

\begin{align*}
\|I_1\|_{\mathcal{K}^{s-\frac{1}{2}}_{(0)}}+\|I_2\|_{\mathcal{K}^{s-\frac{1}{2}}_{(0)}}+\|I_3\|_{\mathcal{K}^{s-\frac{1}{2}}_{(0)}}\leq C(\tilde{v}_0,\kappa)\varepsilon.
\end{align*}

\noindent We pass to the estimate of $\tilde{h}_{\phi}^L-\tilde{h}_{\phi,\varepsilon}^L=\tilde{\mathbf{T}}_0((J^P)^{-1}-(J^P_{\varepsilon})^{-1})\tilde{n}_0+(1-\kappa)t(\nabla\hat{\phi}_{\varepsilon}-\nabla\hat{\phi})\tilde{n}_0+(1-\kappa)t(\nabla\hat{\phi}_{\varepsilon} (J^P_{\varepsilon})^{-1})^T(J^P_{\varepsilon})^{-1}\tilde{n}_0-(1-\kappa)t(\nabla\hat{\phi} (J^P)^{-1})^T(J^P)^{-1}\tilde{n}_0$, which is obtained by using the definition of $\tilde{q}_{\phi}, \tilde{q}_{\phi_{\varepsilon}}$ in \eqref{q_phi}. We rewrite this difference in three terms 

\begin{align*}
&J_1=\tilde{\mathbf{T}}_0((J^P)^{-1}-(J^P_{\varepsilon})^{-1})\tilde{n}_0\\
&J_2=(1-\kappa)t(\nabla\hat{\phi}_{\varepsilon}-\nabla\hat{\phi})\tilde{n}_0\\
&J_3=(1-k)t\left((\nabla\hat{\phi}_{\varepsilon}-\nabla\hat{\phi})J^P_{\varepsilon}\right)^T(J^P_{\varepsilon})^{-1}\tilde{n}_0+(1-k)t\left(\nabla\hat{\phi}(J^P_{\varepsilon}-J^P)\right)^T(J^P_{\varepsilon})^{-1}\tilde{n}_0\\
&\hspace{0.5cm}+(1-k)t(\nabla\hat{\phi}J^P)^T\left((J^P_{\varepsilon})^{-1}-(J^P)^{-1}\right)\tilde{n}_0.
\end{align*}

\noindent For $J_2, J_3$ we do not have problems. Indeed even if $\hat{\phi}, \hat{\phi}_{\varepsilon}$ do not depend on time, we have a $t$ in front which allows us to make an analysis in the space $H^{\frac{s}{2}-\frac{1}{4}}_{(0)}([0,T])$, in particular $\|J_2\|_{\mathcal{K}^{s-\frac{1}{2}}_{(0)}}+\|J_3\|_{\mathcal{K}^{s-\frac{1}{2}}_{(0)}}\leq C(\tilde{v}_0,\tilde{\mathbf{T}}_0,\re,\kappa)\varepsilon$. On the contrary we are not able to deal with $J_1$ since there is not dependence on time, for this reason we will put this term together with $\tilde{h}-\tilde{h}_{\varepsilon}=\tilde{h}_{w}-\tilde{h}_{w,\varepsilon}+\tilde{h}_{\phi}-\tilde{h}_{\phi,\varepsilon}+\tilde{h}_{q}-\tilde{h}_{q,\varepsilon}+\tilde{h}_T-\tilde{h}_{T,\varepsilon}$. In particular, we focus on $\tilde{h}_T-\tilde{h}_{T,\varepsilon}$, since the details for the other terms can be found in \cite[Lemma 6.2]{CCFGG2} and we summarize the final result

\begin{align*}
&\|\tilde{h}_{w}-\tilde{h}_{w,\varepsilon}+\tilde{h}_{\phi}-\tilde{h}_{\phi,\varepsilon}+\tilde{h}_{q}-\tilde{h}_{q,\varepsilon}\|_{\mathcal{K}^{s-\frac{1}{2}}_{(0)}}\leq C(\tilde{v}_0,\tilde{\mathbf{T}}_0,\re,\kappa)\varepsilon+C(\tilde{v}_0,\tilde{\mathbf{T}}_0,\re,\kappa)T^{\delta_2}\\
&\cdot\left(\|\tilde{X}-\tilde{X}_{\varepsilon}+\varepsilon b-t(J^P-J^P_{\varepsilon})\tilde{v}_0\|_{\mathcal{F}^{s+1,\gamma}}+\|\tilde{w}-\tilde{w}_{\varepsilon}\|_{\mathcal{K}^{s+1}_{(0)}}+\|\tilde{q}_{w}-\tilde{q}_{w,\varepsilon}\|_{\mathcal{K}^s_{pr(0)}}\right).
\end{align*}

\noindent The remaining term is the following
\begin{align*}
\tilde{h}_T-\tilde{h}_{T,\varepsilon}+J_1&=\tilde{\mathbf{T}}_pJ^P(\tilde{X})^{-1}\nabla_{\Lambda}\tilde{X}\tilde{n}_0-\tilde{\mathbf{T}}_{p,\varepsilon}J^P(\tilde{X}_{\varepsilon})^{-1}\nabla_{\Lambda}\tilde{X}_{\varepsilon}\tilde{n}_0+\tilde{\mathbf{T}}_0\left((J^P)^{-1}-(J^P_{\varepsilon})^{-1}\right)\tilde{n}_0.
\end{align*}

\noindent At the beginning we deal with $L^2H^{s-\frac{1}{2}}$ and  we split as follows

\begin{align*}
&\tilde{h}_T-\tilde{h}_{T,\varepsilon}+J_1=(\tilde{\mathbf{T}}_p-\tilde{\mathbf{T}}_{p,\varepsilon})J^P(\tilde{X})^{-1}\nabla_{\Lambda}\tilde{X} \tilde{n}_0\\
&\hspace{0.5cm}+(\tilde{\mathbf{T}}_{p,\varepsilon}-\tilde{\mathbf{T}}_0)\left(J^P(\tilde{X})^{-1}-J^P(\tilde{X}_{\varepsilon})^{-1}-(J^P)^{-1}+(J^P_{\varepsilon})^{-1}\right)\nabla_{\Lambda}\tilde{X}\tilde{n}_0\\
&\hspace{0.5cm}+(\tilde{\mathbf{T}}_{p,\varepsilon}-\tilde{\mathbf{T}}_0)\left((J^P)^{-1}-(J^P_{\varepsilon})^{-1}\right)\nabla_{\Lambda}\tilde{X}\tilde{n}_0\\
&\hspace{0.5cm}+\tilde{\mathbf{T}}_0\left(J^P(\tilde{X})^{-1}-J^P(\tilde{X}_{\varepsilon})^{-1}-(J^P)^{-1}+(J^P_{\varepsilon})^{-1}\right)\nabla_{\Lambda}\tilde{X}\tilde{n}_0\\
&\hspace{0.5cm}+(\tilde{\mathbf{T}}_{p,\varepsilon}-\tilde{\mathbf{T}}_0)J^P(\tilde{X}_{\varepsilon})^{-1}(\nabla_{\Lambda}\tilde{X}-\nabla_{\Lambda}\tilde{X}_{\varepsilon})\tilde{n}_0+\tilde{\mathbf{T}}_0J^P(\tilde{X}_{\varepsilon})^{-1}(\nabla_{\Lambda}\tilde{X}-\nabla_{\Lambda}\tilde{X}_{\varepsilon})\tilde{n}_0.
\end{align*}
\medskip

\noindent To get the desired result we will use the trace theorem \ref{parabolic-trace}, lemma \ref{Jp-est} or lemma \ref{Jp-dif-est} and lemma \ref{zeta-est} or lemma \ref{zeta-dif-est}. Moreover to deal with the elastic part  we use or \eqref{elastic-estim} or we  add $t\hat{\mathbf{T}}_{\varepsilon}$.  In a similar, but  more accurate manner, we have to split $\tilde{h}_T-\tilde{h}_{T,\varepsilon}+J_1$ in order to manage the estimates in $H^{\frac{s}{2}-\frac{1}{4}}_{(0)}L^2$.

\begin{align*}
&\tilde{h}_T-\tilde{h}_{T,\varepsilon}+J_1=(\tilde{\mathbf{T}}_p-\tilde{\mathbf{T}}_{p,\varepsilon})(J^P(\tilde{X})^{-1}-(J^P)^{-1})(\nabla_{\Lambda}\tilde{X}-\mathcal{I}) \tilde{n}_0\\[1mm]
&+(\tilde{\mathbf{T}}_p-\tilde{\mathbf{T}}_{p,\varepsilon})(J^P)^{-1}(\nabla_{\Lambda}\tilde{X}-\mathcal{I}) \tilde{n}_0+(\tilde{\mathbf{T}}_p-\tilde{\mathbf{T}}_{p,\varepsilon})(J^P(\tilde{X})^{-1}-(J^P)^{-1})\tilde{n}_0\\[1mm]
&+(\tilde{\mathbf{T}}_p-\tilde{\mathbf{T}}_{p,\varepsilon})(J^P)^{-1}\tilde{n}_0+(\tilde{\mathbf{T}}_{p,\varepsilon}-\tilde{\mathbf{T}}_0)\left(J^P(\tilde{X})^{-1}-J^P(\tilde{X}_{\varepsilon})^{-1}-(J^P)^{-1}+(J^P_{\varepsilon})^{-1}\right)(\nabla_{\Lambda}\tilde{X}-\mathcal{I})\tilde{n}_0\\[1mm]
&+(\tilde{\mathbf{T}}_{p,\varepsilon}-\tilde{\mathbf{T}}_0)((J^P)^{-1}-(J^P_{\varepsilon})^{-1})(\nabla_{\Lambda}\tilde{X}-\mathcal{I})\tilde{n}_0\\[1mm]
&+\tilde{\mathbf{T}}_0\left(J^P(\tilde{X})^{-1}-J^P(\tilde{X}_{\varepsilon})^{-1}-(J^P)^{-1}+(J^P_{\varepsilon})^{-1}\right)(\nabla_{\Lambda}\tilde{X}-\mathcal{I})\tilde{n}_0
\end{align*}
\begin{align*}
&+\tilde{\mathbf{T}}_0((J^P)^{-1}-(J^P_{\varepsilon})^{-1})(\nabla_{\Lambda}\tilde{X}-\mathcal{I})\tilde{n}_0\\[1mm]
&+(\tilde{\mathbf{T}}_{p,\varepsilon}-\tilde{\mathbf{T}}_0)\left(J^P(\tilde{X})^{-1}-J^P(\tilde{X}_{\varepsilon})^{-1}-(J^P)^{-1}+(J^P_{\varepsilon})^{-1}\right)\tilde{n}_0\\[1mm]
&+\tilde{\mathbf{T}}_0\left(J^P(\tilde{X})^{-1}-J^P(\tilde{X}_{\varepsilon})^{-1}-(J^P)^{-1}+(J^P_{\varepsilon})^{-1}\right)\tilde{n}_0+(\tilde{\mathbf{T}}_{p,\varepsilon}-\tilde{\mathbf{T}}_0)((J^P)^{-1}-(J^P_{\varepsilon})^{-1})\tilde{n}_0\\[1mm]
&+(\tilde{\mathbf{T}}_{p,\varepsilon}-\tilde{\mathbf{T}}_0)(J^P(\tilde{X}_{\varepsilon})^{-1}-(J^P_{\varepsilon})^{-1})(\nabla_{\Lambda}\tilde{X}-\nabla_{\Lambda}\tilde{X}_{\varepsilon})\tilde{n}_0\\[1mm]
&+\tilde{\mathbf{T}}_0(J^P(\tilde{X}_{\varepsilon})^{-1}-(J^P_{\varepsilon})^{-1})(\nabla_{\Lambda}\tilde{X}-\nabla_{\Lambda}\tilde{X}_{\varepsilon})\tilde{n}_0\\[1mm]
&+(\tilde{\mathbf{T}}_{p,\varepsilon}-\tilde{\mathbf{T}}_0)(J^P_{\varepsilon})^{-1}(\nabla_{\Lambda}\tilde{X}-\nabla_{\Lambda}\tilde{X}_{\varepsilon})\tilde{n}_0+\tilde{\mathbf{T}}_0(J^P_{\varepsilon})^{-1}(\nabla_{\Lambda}\tilde{X}-\nabla_{\Lambda}\tilde{X}_{\varepsilon})\tilde{n}_0.
\end{align*}
\medskip

\noindent To obtain the estimate of each term, we refer to proposition \ref{estimate-conf-lag-(v,q)}. The final result is the following

\begin{align*}
&\|\tilde{h}_T-\tilde{h}_{T,\varepsilon}+J_1\|_{\mathcal{K}^{s-\frac{1}{2}}_{(0)}}\leq C(\tilde{v}_0,\tilde{\mathbf{T}}_0,\re,\kappa)\left(1+\frac{1}{\We}\right)\varepsilon+C(\tilde{v}_0,\tilde{\mathbf{T}}_0,\re,\kappa)\left(1+\frac{1}{\We}\right)T^{\varrho_3}\\
&\cdot\left(\|\tilde{X}-\tilde{X}_{\varepsilon}+\varepsilon b-t(J^P-J^P_{\varepsilon})\tilde{v}_0\|_{\mathcal{F}^{s+1,\gamma}}+\|\tilde{\mathbf{T}}_p-\tilde{\mathbf{T}}_{p,\varepsilon}-t\hat{\mathbf{T}}_{\varepsilon}\|_{\mathcal{F}^{s,\gamma-1}}\right)
\end{align*}
\medskip

\noindent The lemma holds by choosing $\varrho=\min\{\delta_1,\varrho_1,\varrho_2,\delta_2,\varrho_3\}$.
\end{proof}

\noindent The proof of proposition \ref{stability} follows from these two lemmas by choosing $\delta=\min\{\beta, \varrho\}$.
\bigskip

\section{Proof of Theorem \ref{flux-stab} and existence of splash}\label{sec:4}

\noindent The final goal is to prove that 

\begin{equation}\label{flux-close}
\|\tilde{X}-\tilde{X}_{\varepsilon}\|_{L^{\infty}H^{s+1}}\leq 3C(\tilde{v}_0,\tilde{\mathbf{T}}_0,\re,\kappa)\varepsilon\left(1+\frac{1}{\We}\right).
\end{equation}

\noindent In particular if $\|\tilde{X}-\tilde{X}_{\varepsilon}-\varepsilon b-t(J^P-J^P_{\varepsilon})\tilde{v}_0\|_{L^{\infty}H^{s+1}}\leq 3C(\tilde{v}_0,\tilde{\mathbf{T}}_0,\re,\kappa)\varepsilon\left(1+\frac{1}{\We}\right)$ then \eqref{flux-close} holds. To get this result we use  proposition \ref{stability}. In addition if  $$1-3C(\kappa,M,\re)T^{\delta}\left(1+\frac{1}{\We}\right)>0,$$ 
then we have

\begin{align*}
&\|\tilde{X}-\tilde{X}_{\varepsilon}-\varepsilon b-t(J^P-J^P_{\varepsilon})\tilde{v}_0\|_{L^{\infty}H^{s+1}}\leq \left(1-3C(\tilde{v}_0,\tilde{\mathbf{T}}_0,\re,\kappa)T^{\delta}\left(1+\frac{1}{\We}\right)\right)\\[1mm]
&\cdot\left(\|\tilde{w}-\tilde{w}_{\varepsilon}\|_{\mathcal{K}^{s+1}_{(0)}}+\|\tilde{q}_w-\tilde{q}_{w,\varepsilon}\|_{\mathcal{K}^{s}_{pr(0)}}+ \|\tilde{X}-\tilde{X}_{\varepsilon}-\varepsilon b-t(J^P-J^P_{\varepsilon})\tilde{v}_0\|_{\mathcal{F}^{s+1,\gamma}}\right.\\[1mm]
&\hspace{0.5cm}\left.+\|\tilde{\mathbf{T}}_p-\tilde{\mathbf{T}}_{p,\varepsilon}-t\hat{\mathbf{T}}_{\varepsilon}\|_{\mathcal{F}^{s,\gamma-1}}\right)\leq 3C(\tilde{v}_0,\tilde{\mathbf{T}}_0,\re,\kappa)\varepsilon\left(1+\frac{1}{\We}\right),
\end{align*}
\medskip

\noindent The assumption on $T$ is equivalent to \eqref{conf-lag-time} up to a constant.\\

\noindent This result states that the fluxes, which govern the evolution of the domain, are close and it implies that the two interfaces are close. Then we can conclude that starting with a regular domain $P^{-1}(\tilde{\Omega}_{\varepsilon}(0))$, we end up in a self-intersecting one, since $P^{-1}(\tilde{\Omega}(T))$ is self-intersecting. This argument works if we have a right initial velocity. In particular, we use the same argument explained in \cite{DMS1}. \\
\noindent We are looking for initial data that satisfy the compatibility conditions \eqref{undim-compatibility}. In particular
$$n^{\perp}\left((1-\kappa)(\nabla u_0+\nabla u_0^T)+\tau_0\right)n=0.$$

\noindent At the beginning we consider the Navier-Stokes system, without the presence of the elastic stress tensor $\tau_0$. In this specific case, as can be seen in \cite{CCFGG2}, the compatibility condition for the initial velocity $u_0$ is given by 
\begin{equation}\label{NS-compcond}
n^{\perp} \left((1-\kappa)(\nabla u_0+\nabla u_0^T)\right) n=0.
\end{equation}

\noindent We take into account the analysis of the Navier-Stokes case and we recall the arguments in \cite{CCFGG2}. Let us consider a neighborhood $U$ of the boundary $\partial\Omega$, we can use  a coordinates system $(s,\lambda)$ given by $x(s,\lambda)=z(s)+\lambda z_s^{\perp}(s)$ and we define a stream function $\psi$ by using the following quadratic expansion

\begin{equation}\label{stream}
\psi(x(s,\lambda))=\bar{\psi}(s,\lambda)=\psi_0(s)+\lambda\psi_1(s)+\frac{1}{2}\lambda^2\psi_2(s).
\end{equation}
\medskip

\noindent Consequently we extend on $U$ both $n^{\perp}$ and $n$ in the following way

\begin{equation*}
\left\{\begin{array}{lll}
N^{\perp}(s,\lambda)=x_s(s,\lambda)=z_s(s)+\lambda z_{ss}^{\perp}(s)=(1-\lambda k(s))z_s(s)\\[3mm]
N(s,\lambda)=x_{\lambda}=z_s^{\perp},
\end{array}\right.
\end{equation*}
\noindent where $k(s)=z_{ss}\cdot z_s^{\perp}$  is the scalar curvature.
\medskip

\noindent Since $u_0$ is divergence free, we define $u_0=\nabla^{\perp}\psi$ and then we substitute this definition in \eqref{NS-compcond} and we get

\begin{equation}\label{cc}
N^{\perp}\left((1-\kappa)(\nabla u_0+\nabla u_0)\right) N=(1-\kappa)(\partial_s^2\psi_0(s)-\psi_2(s))=0.
\end{equation}

\noindent As $u_0\cdot N=\partial_s \psi_0(s)$ and we need a positive normal component in order to apply the argument explained in the Introduction, first of all we take $\psi_0(s)$ in such a way $\partial_s\psi_0(s)>0$ and consequently $\psi_2(s)$ in such a way that condition (\ref{cc}) is satisfied. We can immediately observe that the normal component of the velocity depends only on the stream function and does not depend on the boundary conditions, it suggests for the viscoelastic problem, that $u_0\cdot n$ does not depend on the elastic stress tensor, so the compatibility condition \eqref{undim-compatibility} can be written as follows

\begin{equation}\label{newcompcond}
(1-\kappa)\left(\partial_s^2\psi_0(s)- \psi_2(s)\right)=-\left(N^{\perp}\tau_0 N\right)_{|\lambda=0},
\end{equation}

\noindent Furthermore we can state that for a given $\psi_0$ such that $\partial_s\psi_0>0$ and for any $\tau_0$, there exist $\psi_1, \psi_2$ such that \eqref{newcompcond} is satisfied.\\

\noindent In conclusion we observe that the two main ingredients for proving  the existence of splash singularities are the stability result and the construction of the initial data $(u_0, \tau_0)$. These key results allow us to pass from a regular domain $P^{-1}(\partial\tilde{\Omega}_{\varepsilon}(0))$ into a self-intersecting $P^{-1}(\partial\tilde{\Omega}_{\varepsilon}(T))$. Thus we can define the splash time $t^*$ as follows
$$t^*=\inf\{t\in(0,T): P^{-1}(\partial\tilde{\Omega}_{\varepsilon}(t))\hspace{0.2cm} \textrm{is as in fig. \ref{fig:1} (b)}\}.$$
\noindent We get the the final result

\begin{theorem}
For $2<s<\frac{5}{2}$ there exists a solution of  \eqref{undimsys} $\{\Omega_{\varepsilon}, u_{\varepsilon}, p_{\varepsilon}, \tau_{p,\varepsilon}\}$ on $[0, t^*]$, which forms a splash singularity at time $t^*$. 
\end{theorem}

\subsection*{Acknowledgements:} 
The authors would like to thank \'Angel Castro and Diego C\'ordoba for the helpful convertions and the anonymous referee for their comments and suggestions.

\appendix\section{}\label{appendix}
\subsection{Functional Spaces}
\noindent The functional spaces used throughout this paper were introduced by J. T. Beale in \cite{B}. In particular, we define the Sobolev spaces with fractional derivatives in time $H^s_{(0)}([0,T])$, for $0<s<1$ as the interpolation space between $L^2([0,T])$ and $H^1_{(0)}([0,T])$ through the operator $S=(1-\partial_t^2)$. The domain of $S$ is the following

\begin{equation*}
\{v\in H^2([0,T]): v(0)=0, \partial_t v(T)=0\},
\end{equation*}

\noindent and the $H^s_{(0)}([0,T])$ - norm is the graph norm of the operator $S^{\frac{s}{2}}$. In particular, $v\in  H^s_{(0)}([0,T])$ if $v\in L^2([0,T])$ and

\begin{align*}
\|v\|_{H^s_{(0)}}^2=\sum_{n=0}^{\infty}\left(\int_0^T v(t)\sin\left(\frac{(2n+1)\pi}{2T}t\right)\sqrt\frac{2}{T} \,dt\right)^2\left(\frac{(2n+1)\pi}{2T}\right)^{2s}<\infty.
\end{align*}

\medskip

\noindent For larger exponents, the definition of $H^{s+m}_{(0)}([0,T])$, for $0<s<1$ and $m\in\mathbb{N}$ is 

\begin{center}
$\{v\in H^{s+m}([0,T]):\partial_t^k v(0)=0, k=0,1,\ldots,m-1\},\hspace{0.3cm}\textrm{with}\hspace{0.3cm}\partial_t^m v\in H^s_{(0)}([0,T])$.
\end{center}
\medskip

\noindent This space has the following norm

$$\|v\|_{H^{s+m}_{(0)}}^2=\sum_{k=0}^{m-1}\|\partial_t^k v\|_{ L^2([0,T])}^2+\|\partial_t^m v\|_{ H^s_{(0)}([0,T])}^2.$$
\medskip

\noindent The space we used through the paper are the following

\begin{align*}
&\mathcal{K}^s_{(0)}([0,T]; \Omega)= L^2([0,T]; H^s( \Omega))\cap H^{\frac{s}{2}}_{(0)}([0,T]; L^2( \Omega)),\\[2mm]
&\mathcal{K}^s_{pr(0)}([0,T]; \Omega)=\{q\in L^{\infty}([0,T]; \dot{H}^1( \Omega)): \nabla q\in\mathcal{K}^{s-1}_{(0)}([0,T]; \Omega),
q\in\mathcal{K}^{s-\frac{1}{2}}_{(0)}([0,T]; \partial\Omega)\},\\[2mm]
&\mathcal{\bar{K}}^{s}_{(0)}([0,T]; \Omega)= L^2([0,T];H^s(\Omega)) \cap H^{\frac{s+1}{2}}_{(0)}([0,T]; H^{-1}(\Omega)),\\[2mm]
&\mathcal{F}^{s+1,\gamma}([0,T]; \Omega)= L^{\infty}_{\frac{1}{4}}([0,T]; H^{s+1}( \Omega)) \cap H^2_{(0)}([0,T]; H^{\gamma}( \Omega)),\hspace{0.3cm}\textrm{for}\hspace{0.2cm} s-1-\varepsilon<\gamma<s-1,
\end{align*}
\noindent with 

$$\|f\|_{L^{\infty}_{\frac{1}{4}}H^s}=\sup\limits_{t\in [0,T]}t^{-\frac{1}{4}}\|f(t)\|_{H^s}.$$
\medskip

\subsection{Preliminary Lemmas}

\noindent The spaces defined above are fundamental for the use of the following embedding theorems and interpolation estimates, in order to get constants independent of time. For details, see \cite{B}, \cite{CCFGG2} and \cite{LM}.
\bigskip

\begin{lemma}\label{extentheorem}
Let $B$ a Hilbert space
\begin{enumerate}
\item For $s\geq 0$, there is a bounded extension operator from $H^s((0,T);B)\rightarrow H^s((-\infty,\infty);B)$.
\item For $0\leq s\leq 2$ , $s-\frac{1}{2}$ is not an integer, there is an extension operator from
$$\left\lbrace v\in H^s((0,T);B); \partial_t^k v(0)=0, 0\leq k < s-\frac{1}{2}\right\rbrace \rightarrow H^s((-\infty,\infty);B)$$
with a norm bounded independently on $T$. Moreover, if $Ev$ is the extention of $v$ then 

$$\|Ev\|_{ H^s((-\infty,\infty);B)}\leq C \|v\|_{ H_{(0)}^s((0,T);B)}$$
\end{enumerate}
\end{lemma}
\bigskip

\begin{lemma}\label{lem1}
Suppose $0\leq r\leq 4$.
\begin{enumerate}
\item The Identity extends to a bounded operator
$$\mathcal{K}^r((0,T);\Omega)\rightarrow H^p(0,T)H^{r-2p}(\Omega),$$
$p\leq \frac{r}{2}$.
\item If $r$ is not an odd integer, the restriction of this operator to the subspace with $\partial_t^k v(0)=0$, $0\leq k< \frac{r-1}{2}$ is bounded independently on $T$, indeed
\begin{equation*}
\|v\|_{H^p_{(0)}H^{r-2p}}\leq C \|v\|_{\mathcal{K}^r_{(0)}}.
\end{equation*}
\end{enumerate}
\end{lemma}
\bigskip

\begin{lemma}\label{lem2}
Let $\bar{T}>0$ be arbitrary, $B$ a Hilbert space and choose $T\leq\bar{T}$.
\begin{enumerate}
\item 
For $v\in L^2((0,T); B)$, we define $V\in H^1((0,T); B)$ by
$$V(t)=\int_0^t v(\tau)\, d\tau.$$
For $0<s<\frac{1}{2}$ and $0\leq\varepsilon<s$, then the map $v\rightarrow V$ is a bounded operator from $H^s((0,T); B)$ to $H^{s+1-\varepsilon}((0,T);B)$, and

$$\|V\|_{H^{s+1-\varepsilon}((0,T); B)}\leq C_0T^{\varepsilon}\|v\|_{H^s((0,T); B)},$$
where $C_0$ is independent of $T$ for $0< T\leq\bar {T}$.
\medskip

\item For $\frac{1}{2}< s<1$, we impose $v(0)=0$ and $0\leq\varepsilon<s$. Then $v\rightarrow V$ is a bounded operator from $H^s_{(0)}((0,T); B)$ to $H^{s+1-\varepsilon}_{(0)}((0,T); B)$ and 

$$\|V\|_{H^{s+1-\varepsilon}_{(0)}((0,T); B)}\leq C_0T^{\varepsilon}\|v\|_{H^s_{(0)}((0,T); B)},$$
where $C_0$ is independent of $T$ for $0< T\leq\bar{T}$.
\end{enumerate}
\end{lemma}
\bigskip

\begin{lemma}\label{lem3}
Suppose $r>1$ and $r\geq s\geq 0.$ If $v\in H^r(\Omega)$ and $w\in H^s(\Omega)$, then $vw\in H^s(\Omega)$ and
$$\|vw\|_{H^s}\leq C\|v\|_{H^r} \|w\|_{H^s}.$$
\end{lemma}
\bigskip

\begin{lemma}\label{lem4}
If $v\in H^{\frac{1}{q}}$ and $w\in H^{\frac{1}{p}}$ with $\frac{1}{p}+\frac{1}{q}=1$ and $1<p<\infty$ then 
$$\|vw\|_{L^2}\leq C\|v\|_{H^{\frac{1}{q}}}\|w\|_{H^{\frac{1}{p}}}.$$
\end{lemma}
\bigskip

\begin{lemma}\label{lem5}
Suppose  $B, Y, Z$ are Hilbert spaces, and $M:B\times Y\rightarrow Z$ is a bounded, bilinear, multiplication operator. Suppose $w\in H^s((0,T); B)$ and $v\in H^s((0,T); Y)$, where $s>\frac{1}{2}$. If $vw$ is defined by $M(v,w)$, then $vw\in H^s((0,T); Z)$ and the following hold
\begin{enumerate}
\item $$\|vw\|_{H^s((0,T); Z)}\leq C\|v\|_{H^s((0,T); Y)} \|w\|_{H^s((0,T); B)}.$$
\medskip

\item In addition, if $s\leq 2$ and $\partial_t^k v(0)=\partial_t^k w(0)=0$, $0\leq k<s-\frac{1}{2}$ and $s-\frac{1}{2}$ is not an integer, then the constant $C$ in $(1)$ can be chosen independently on $T$. Indeed
$$\|vw\|_{H^s_{(0)}((0,T); Z)}\leq C \|v\|_{H^s_{(0)}((0,T); Y)}\|w\|_{H^s_{(0)}((0,T); B)}.$$

\end{enumerate}
\end{lemma}
\bigskip

\begin{lemma}\label{lem6}
For $2<s<\frac{5}{2}$, $\varepsilon, \delta$ positive and small enough and $v\in\mathcal{F}^{s+1,\gamma}$ the following estimates hold
\medskip

\begin{enumerate}
\item $\|v\|_{H_{(0)}^{\frac{s+1}{2}}H^{1-\varepsilon}}\leq C\|v\|_{\mathcal{F}^{s+1,\gamma}}$,\\
\item  $\|v\|_{H_{(0)}^{\frac{s+1}{2}+\varepsilon}H^{1+\delta}}\leq C\|v\|_{\mathcal{F}^{s+1,\gamma}}$,\\
\item  $\|v\|_{H_{(0)}^{\frac{s-1}{2}+\varepsilon}H^{2+\delta}}\leq C\|v\|_{\mathcal{F}^{s+1,\gamma}}$,\\
\item  $\|v\|_{H_{(0)}^{\frac{s}{2}-\frac{1}{4}+\varepsilon}H^{2+\delta}}\leq C\|v\|_{\mathcal{F}^{s+1,\gamma}}$,\\
\item  $\|v\|_{H_{(0)}^1H^{s-1}}\leq C\|v\|_{\mathcal{F}^{s+1,\gamma}}$,\\
\item  $\|v\|_{H_{(0)}^{\frac{1}{2}+2\varepsilon}H^{s}}\leq C\|v\|_{\mathcal{F}^{s+1,\gamma}}$.
\end{enumerate}
\end{lemma}
\medskip

\begin{remark}
We notice that in the same way as Lemma \ref{lem6} we can deduce the embeddings for the space $\mathcal{F}^{s,\gamma-1}$. For instance, $\mathcal{F}^{s,\gamma-1}\subset H_{(0)}^{\frac{s-1}{2}+\delta}H^{1+\eta}$, $\mathcal{F}^{s,\gamma-1}\subset H_{(0)}^{\frac{s}{2}-\frac{1}{4}+\delta}H^{1+\eta}$, for $\delta, \eta>0$ and small enough.
\end{remark}

\bigskip

\begin{lemma}\label{parabolic-trace}
Let $\Omega$ be a bounded set with a sufficient smooth boundary then the following trace theorems hold
\begin{enumerate}
\item Suppose $\frac{1}{2}< s\leq 5$. The mapping $v\rightarrow\partial_n^j v$ extends to a bounded operator\\ $\mathcal{K}^s([0,T];\Omega)\rightarrow \mathcal{K}^{s-j-\frac{1}{2}}([0,T];\partial\Omega)$, where $j$ is an integer $0\leq j<s-\frac{1}{2}$. The mapping $v\rightarrow \partial_t^k v(\alpha,0)$ extends to a bounded operator 
$\mathcal{K}^s([0,T];\Omega)\rightarrow H^{s-2k-1}(\Omega)$, if $k$ is an integer $ 0\leq k<\frac{1}{2}(s-1)$.

\item  Suppose $\frac{3}{2}<s<5$, $s\neq 3$ and $s-\frac{1}{2}$ not an integer. Let
$$\mathcal{W}^s=\prod_{0\leq j\leq s-\frac{1}{2}}\mathcal{K}^{s-j-\frac{1}{2}}([0,T];\partial\Omega)\times \prod_{0\leq k< \frac{s-1}{2}}H^{s-2k-1}(\Omega),$$
and let $\mathcal{W}^s_0$ the subspace consisting of $\{a_j,w_k\}$, which are the traces described in the previous point, so that $ \partial_t^k a_j(\alpha,0)=\partial_n^j w_k(\alpha)$, $\alpha\in\partial\Omega$, for $j+2k<s-\frac{3}{2}$. Then the traces in the previous point form a bounded operator $\mathcal{K}^s([0,T];\Omega)\rightarrow \mathcal{W}^s_0$ and this operator has a bounded right inverse.
\end{enumerate}
\end{lemma}
\medskip

\subsection{Auxiliary estimates}
In this section we resume the estimates we use throught the proofs in order to avoid too many computations. In the specific we need the following results for $\tilde{X}-\tilde{\alpha}$, by using lemma \ref{lem2}, with $0<\delta<\eta$.

\begin{equation}\label{flux-estim}
\begin{split}
&1.\hspace{0.5cm} \|\tilde{X}-\tilde{\alpha}\|_{L^{\infty}H^{s+1}}\leq \|\tilde{X}-\tilde{\alpha}-tJ^p \tilde{v}_0\|_{L^{\infty}H^{s+1}}+\|tJ^p \tilde{v}_0\|_{L^{\infty}H^{s+1}}\\[2mm]
&\hspace{1.5cm}\leq T^{\frac{1}{4}}\|\tilde{X}-\tilde{\alpha}-tJ^p \tilde{v}_0\|_{L^{\infty}_{\frac{1}{4}}H^{s+1}}+T\|J^p \tilde{v}_0\|_{H^{s+1}}\\[2mm]
&\hspace{1.5cm}\leq C(\tilde{v}_0)T^{\frac{1}{4}}\|\tilde{X}-\tilde{\alpha}-tJ^p \tilde{v}_0\|_{\mathcal{F}^{s+1,\gamma}}+C(\tilde{v}_0)T.\\[5mm]
&2.\hspace{0.5cm} \|\tilde{X}-\tilde{\alpha}\|_{H^1_{(0)}H^{\gamma}}\leq \|\tilde{X}-\tilde{\alpha}-tJ^p \tilde{v}_0\|_{H^1_{(0)}H^{\gamma}}+\|tJ^p \tilde{v}_0\|_{H^1_{(0)}H^{\gamma}}\\[2mm]
&\hspace{1.5cm}\leq \left\|\int_0^t\partial_t(\tilde{X}-\tilde{\alpha}-tJ^p \tilde{v}_0)\right\|_{H^{1+\eta-\delta}_{(0)}H^{\gamma}}+C(\tilde{v}_0)\|t\|_{H^1_{(0)}}\\[2mm]
&\hspace{1.5cm}\leq C(\tilde{v}_0)T^{\delta}\|\tilde{X}-\tilde{\alpha}-tJ^p \tilde{v}_0\|_{H^{1+\eta}_{(0)}H^{\gamma}}+C(\tilde{v}_0)T^{\frac{1}{2}}\\[2mm]
&\hspace{1.5cm}\leq C(\tilde{v}_0)T^{\delta}\|\tilde{X}-\tilde{\alpha}-tJ^p \tilde{v}_0\|_{\mathcal{F}^{s+1,\gamma}}+C(\tilde{v}_0)T^{\frac{1}{2}}.
\\[5mm]
&3.\hspace{0.5cm} \|\tilde{X}-\tilde{\alpha}\|_{H^{\frac{s-1}{2}}_{(0)}H^{2+\mu}}\leq \|\tilde{X}-\tilde{\alpha}-tJ^p \tilde{v}_0\|_{H^{\frac{s-1}{2}}_{(0)}H^{2+\mu}}+\|tJ^p \tilde{v}_0\|_{H^{\frac{s-1}{2}}_{(0)}H^{2+\mu}}\\[2mm]
&\hspace{1.5cm}\leq \left\|\int_0^t\partial_t(\tilde{X}-\tilde{\alpha}-tJ^p \tilde{v}_0)\right\|_{H^{\frac{s-1}{2}+\delta-\delta}_{(0)}H^{2+\mu}}+C(\tilde{v}_0)\|t\|_{H^{\frac{s-1}{2}}_{(0)}}\\[2mm]
&\hspace{1.5cm}\leq C(\tilde{v}_0)T^{\delta}\|\tilde{X}-\tilde{\alpha}-tJ^p \tilde{v}_0\|_{H^{\frac{s-1}{2}+\delta}_{(0)}H^{2+\mu}}+C(\tilde{v}_0)T^{\frac{1}{2}}\\[2mm]
&\hspace{1.5cm}\leq C(\tilde{v}_0)T^{\delta}\|\tilde{X}-\tilde{\alpha}-tJ^p \tilde{v}_0\|_{\mathcal{F}^{s+1,\gamma}}+C(\tilde{v}_0)T^{\frac{1}{2}}.
\end{split}
\end{equation}

\noindent And  we get the following estimates for $\mathbf{\tilde{T}}_p-\mathbf{\tilde{T}}_0$, by using lemma \ref{lem2}, with $0<\delta<\eta$.

\begin{equation}\label{elastic-estim}
\begin{split}
&1.\hspace{0.5cm} \|\mathbf{\tilde{T}}_p-\mathbf{\tilde{T}}_0\|_{L^{\infty}H^{s}}\leq \|\mathbf{\tilde{T}}_p-\mathbf{\tilde{T}}_0-t\mathbf{\hat{T}}\|_{L^{\infty}H^{s}}+\|t\mathbf{\hat{T}}\|_{L^{\infty}H^{s}}\\[2mm]
&\hspace{1cm}\leq T^{\frac{1}{4}}\|\mathbf{\tilde{T}}_p-\mathbf{\tilde{T}}_0-t\mathbf{\hat{T}}\|_{L^{\infty}_{\frac{1}{4}}H^{s}}+T\|\mathbf{\hat{T}}\|_{H^{s}}\\[2mm]
&\hspace{1cm}\leq  C(\tilde{v}_0,\mathbf{\tilde{T}}_0,\kappa,\|\mathbf{\tilde{T}}_p-\mathbf{\tilde{T}}_0-t\mathbf{\hat{T}}\|_{\mathcal{F}^{s,\gamma-1}})\left(1+\frac{1}{\We}\right)T^{\frac{1}{4}}.
\end{split}
\end{equation}
\begin{align*}
&2.\hspace{0.5cm} \|\mathbf{\tilde{T}}_p-\mathbf{\tilde{T}}_0\|_{H^1_{(0)}H^{\gamma-1}}\leq \|\mathbf{\tilde{T}}_p-\mathbf{\tilde{T}}_0-t\mathbf{\hat{T}}\|_{H^1_{(0)}H^{\gamma-1}}+\|t\mathbf{\hat{T}}\|_{H^1_{(0)}H^{\gamma-1}}\\[2mm]
&\hspace{1cm}\leq \left\|\int_0^t\partial_t(\mathbf{\tilde{T}}_p-\mathbf{\tilde{T}}_0-t\mathbf{\hat{T}})\right\|_{H^{1+\eta-\delta}_{(0)}H^{\gamma-1}}+C(\tilde{v}_0,\mathbf{\tilde{T}}_0,\kappa)\left(1+\frac{1}{\We}\right)\|t\|_{H^1_{(0)}}\\[2mm]
&\hspace{1cm}\leq C(\tilde{v}_0,\mathbf{\tilde{T}}_0)T^{\delta}\|\mathbf{\tilde{T}}_p-\mathbf{\tilde{T}}_0-t\mathbf{\hat{T}}\|_{H^{1+\eta}_{(0)}H^{\gamma-1}}+ C(\tilde{v}_0,\mathbf{\tilde{T}}_0,\kappa)\left(1+\frac{1}{\We}\right)T^{\frac{1}{2}}\\[2mm]
&\hspace{1cm}\leq C(\tilde{v}_0,\mathbf{\tilde{T}}_0,\kappa, \|\mathbf{\tilde{T}}_p-\mathbf{\tilde{T}}_0-t\mathbf{\hat{T}}\|_{\mathcal{F}^{s,\gamma-1}})\left(1+\frac{1}{\We}\right)T^{\delta}.\\[5mm]
&3.\hspace{0.5cm} \|\mathbf{\tilde{T}}_p-\mathbf{\tilde{T}}_0\|_{H^{\frac{s-1}{2}}_{(0)}H^{1+\mu}}\leq \|\mathbf{\tilde{T}}_p-\mathbf{\tilde{T}}_0-t\mathbf{\hat{T}}\|_{H^{\frac{s-1}{2}}_{(0)}H^{1+\mu}}+\|t\mathbf{\hat{T}}\|_{H^{\frac{s-1}{2}}_{(0)}H^{1+\mu}}\\[2mm]
&\hspace{1cm}\leq \left\|\int_0^t\partial_t(\mathbf{\tilde{T}}_p-\mathbf{\tilde{T}}_0-t\mathbf{\hat{T}})\right\|_{H^{\frac{s-1}{2}+\delta-\delta}_{(0)}H^{1+\mu}}+C(\tilde{v}_0,\mathbf{\tilde{T}}_0,\kappa)\left(1+\frac{1}{\We}\right)\|t\|_{H^{\frac{s-1}{2}}_{(0)}}\\[2mm]
&\hspace{1cm}\leq C(\tilde{v}_0,\mathbf{\tilde{T}}_0)T^{\delta}\|\mathbf{\tilde{T}}_p-\mathbf{\tilde{T}}_0-t\mathbf{\hat{T}}\|_{H^{\frac{s-1}{2}+\delta}_{(0)}H^{1+\mu}}+C(\tilde{v}_0,\mathbf{\tilde{T}}_0,\kappa)\left(1+\frac{1}{\We}\right)T^{\frac{1}{2}}\\[2mm]
&\hspace{1cm}\leq C(\tilde{v}_0,\mathbf{\tilde{T}}_0,\kappa,\|\mathbf{\tilde{T}}_p-\mathbf{\tilde{T}}_0-t\mathbf{\hat{T}}\|_{\mathcal{F}^{s,\gamma-1}})\left(1+\frac{1}{\We}\right)T^{\delta}.
\end{align*}

\noindent  In this estimate we use the fact that $\mathbf{\hat{T}}$ depends only on the initial data. Indeed 
$$\mathbf{\hat{T}}=J^P\nabla\tilde{v}_0\mathbf{\tilde{T}}_0+\mathbf{\tilde{T}}_0(J^P\nabla\tilde{v}_0)^T-\frac{1}{\We}\mathbf{\tilde{T}}_0+\frac{\kappa}{\We}\left(J^P\nabla\tilde{v}_0+(J^P\nabla\tilde{v}_0)^T\right).$$

\noindent Then, when we estimate its $H^s-$norm we require sufficient regularity for $\tilde{v}_0$ and $\mathbf{\tilde{T}}_0$.

\begin{align*}
\|\mathbf{\hat{T}}\|_{H^s}&=\left\|J^P\nabla\tilde{v}_0\mathbf{\tilde{T}}_0+\mathbf{\tilde{T}}_0(J^P\nabla\tilde{v}_0)^T-\frac{1}{\We}\mathbf{\tilde{T}}_0+\frac{\kappa}{\We}\left(J^P\nabla\tilde{v}_0+(J^P\nabla\tilde{v}_0)^T\right)\right\|_{H^s}\\[2mm]
&\leq  \|J^P\nabla\tilde{v}_0\mathbf{\tilde{T}}_0\|_{H^s}+\frac{1}{\We}\|\mathbf{\tilde{T}}_0\|_{H^s}+\frac{\kappa}{\We}\|J^P\nabla\tilde{v}_0\|_{H^s}\leq C(\tilde{v}_0,\mathbf{\tilde{T}}_0,\kappa)\left(1+\frac{1}{\We}\right).
\end{align*}

\medskip

\noindent All the estimates for proving the local existence and the stability results make use of the following lemmas. For details see \cite{CCFGG2}.

\begin{lemma}\label{Jp-est}
Let $2<s<\frac{5}{2}$ and $\tilde{X}-\tilde{\alpha}-tJ^P\tilde{v}_0\in\mathcal{F}^{s+1,\gamma}$. Then, for $T$ small enough we have
\begin{align*}
&1.\hspace{0.5cm}\|J^P(\tilde{X})\|_{L^{\infty}H^{s+1}}\leq C(M, \tilde{v}_0, \|\tilde{\alpha}\|_{L^2},\|\tilde{X}-\tilde{\alpha}-tJ^P\tilde{v}_0\|_{\mathcal{F}^{s+1,\gamma}})\\[3mm]
&2.\hspace{0.5cm}\|J^P(\tilde{X})-J^P\|_{L^{\infty}H^{s+1}}\leq C(M, \tilde{v}_0, \|\tilde{\alpha}\|_{L^2},\|\tilde{X}-\tilde{\alpha}-tJ^P\tilde{v}_0\|_{\mathcal{F}^{s+1,\gamma}})\\[2mm]
&\hspace{4.5cm}\cdot\left(\|\tilde{X}-\tilde{\alpha}-tJ^P\tilde{v}_0\|_{L^{\infty}H^{s+1}}+\|tJ^P\tilde{v}_0\|_{L^{\infty}H^{s+1}}\right)\\[3mm]
&3.\hspace{0.5cm}\|J^P(\tilde{X})-J^P\|_{H^1_{(0)}H^{\gamma}}\leq  C(M, \tilde{v}_0,\|\tilde{X}-\tilde{\alpha}-tJ^P\tilde{v}_0\|_{\mathcal{F}^{s+1,\gamma}})\|\tilde{X}-\tilde{\alpha}\|_{H^1_{(0)}H^{\gamma}}\\[3mm]
&4.\hspace{0.5cm}\|J^P(\tilde{X})-J^P\|_{H^{\frac{s-1}{2}}_{(0)}H^{1+\delta}}\leq  C(M, \tilde{v}_0,\|\tilde{X}-\tilde{\alpha}-tJ^P\tilde{v}_0\|_{\mathcal{F}^{s+1,\gamma}})\\[1mm]
&\hspace{5.9cm}\cdot(\|\tilde{X}-\tilde{\alpha}-tJ^P\tilde{v}_0\|_{H^{\frac{s-1}{2}+\eta}_{(0)}H^{1+\delta}}+T),
\end{align*}
with 

$$M=\frac{1}{\inf_{\tilde{\alpha}}|\tilde{\alpha}|-C(\tilde{v}_0)T-T^{\frac{1}{4}}\|\tilde{X}-\tilde{\alpha}-tJ^P\tilde{v}_0\|_{\mathcal{F}^{s+1,\gamma}}}.$$
\end{lemma}

\begin{proof}
These results  have been obtained by using the definition of $J^P_{kj}=\partial_{X_j}P_k(P^{-1}(\tilde{X}))$, which contains terms as $\frac{\tilde{X}_i}{|\tilde{X}|}$ and by using estimates \eqref{flux-estim}. 
\end{proof}
\bigskip

\begin{lemma}\label{Jp-dif-est}
Let $2<s<\frac{5}{2}$ and $\tilde{X}-\tilde{\alpha}-tJ^P\tilde{v}_0, \tilde{Y}-\tilde{\alpha}-tJ^P\tilde{v}_0 \in\mathcal{F}^{s+1,\gamma}$. Then, for $T$ small enough we have

\begin{align*}
&1.\hspace{0.5cm}\|J^P(\tilde{X})-J^P(\tilde{Y})\|_{L^{\infty}H^{s+1}}\\[1mm]
&\hspace{0.5cm}\leq C(M, \|\tilde{X}-\tilde{\alpha}-tJ^P\tilde{v}_0\|_{\mathcal{F}^{s+1,\gamma}}, \|\tilde{Y}-\tilde{\alpha}-t J^P\tilde{v}_0\|_{\mathcal{F}^{s+1,\gamma}}) \|\tilde{X}-\tilde{Y}\|_{L^{\infty}H^{s+1}}\\[3mm]
&2.\hspace{0.5cm}\|J^P(\tilde{X})-J^P(\tilde{Y})\|_{H^1_{(0)}H^{\gamma}}\\[1mm]
&\hspace{0.5cm}\leq C(M, \|\tilde{X}-\tilde{\alpha}-t J^P\tilde{v}_0\|_{\mathcal{F}^{s+1,\gamma}}, \|\tilde{Y}-\tilde{\alpha}-tJ^P\tilde{v}_0\|_{\mathcal{F}^{s+1,\gamma}}) \|\tilde{X}-\tilde{Y}\|_{H^1_{(0)}H^{\gamma}},
\end{align*}
where 
\begin{align*}
&M=\max\left\lbrace\frac{1}{\inf_{\tilde{\alpha}}|\tilde{\alpha}|-C(\tilde{v}_0)T-T^{\frac{1}{4}}\|\tilde{X}-\tilde{\alpha}-tJ^P\tilde{v}_0\|_{\mathcal{F}^{s+1,\gamma}}},\right.\\[1mm]
&\hspace{2cm}\left.\frac{1}{\inf_{\tilde{\alpha}}|\tilde{\alpha}|-C(\tilde{v}_0)T-T^{\frac{1}{4}}\|\tilde{Y}-\tilde{\alpha}-tJ^P\tilde{v}_0\|_{\mathcal{F}^{s+1,\gamma}}}\right\rbrace.
\end{align*}
\end{lemma}
\medskip

\begin{remark}
The same estimates we obtain for $J^P(\tilde{X})$ and $J^P(\tilde{X})-J^P(\tilde{Y})$ hold also for $Q^2(\tilde{X})$ and $Q^2(\tilde{X})-Q^2(\tilde{Y})$.
\end{remark}
\bigskip

\begin{lemma}\label{zeta-est}
Let $2<s<\frac{5}{2}$, $\tilde{X}-\tilde{\alpha}-tJ^P\tilde{v}_0\in\mathcal{F}^{s+1,\gamma}$ and $\tilde{\zeta}=(\nabla\tilde{X})^{-1}$. Then for $T$ small enough, we have

\begin{align*}
&1.\hspace{0.5cm}\|\tilde{\zeta}\|_{L^{\infty}H^s}+\sum_{i=1}^{2}\|\partial_i\tilde{\zeta}\|_{L^{\infty}H^{s}}\leq C(M, \|\tilde{X}-\tilde{\alpha}-tJ^P\tilde{v}_0\|_{\mathcal{F}^{s+1,\gamma}})\\[3mm]
&2.\hspace{0.5cm}\|\tilde{\zeta}-\mathcal{I}\|_{L^{\infty}H^s}\leq C(M, \|\tilde{X}-\tilde{\alpha}-tJ^P\tilde{v}_0\|_{\mathcal{F}^{s+1,\gamma}})\|\tilde{X}-\tilde{\alpha}\|_{L^{\infty}H^{s+1}}\\[3mm]
&3.\hspace{0.5cm}\|\tilde{\zeta}-\mathcal{I}\|_{H^{\frac{s-1}{2}+\delta}_{(0)}H^{1+\eta}}\leq C(M, \|\tilde{X}-\tilde{\alpha}-tJ^P\tilde{v}_0\|_{\mathcal{F}^{s+1,\gamma}}) \|\tilde{X}-\tilde{\alpha}\|_{H^{\frac{s-1}{2}+\delta}_{(0)}H^{2+\eta}}\\[3mm]
&4.\hspace{0.5cm}\|\tilde{\zeta}-\mathcal{I}\|_{H^1_{(0)}H^{\gamma-1}}\leq C(M, \|\tilde{X}-\tilde{\alpha}-tJ^P\tilde{v}_0\|_{\mathcal{F}^{s+1,\gamma}}) \|\tilde{X}-\tilde{\alpha}\|_{H^1_{(0)}H^{\gamma}},
\end{align*}

where 

$$M=\frac{1}{1-C(\tilde{v}_0)T-CT^{\frac{1}{4}}\|\tilde{X}-\tilde{\alpha}-tJ^P\tilde{v}_0\|_{\mathcal{F}^{s+1,\gamma}}-CT^{\frac{1}{2}}|\tilde{X}-\tilde{\alpha}-tJ^P\tilde{v}_0\|_{\mathcal{F}^{s+1,\gamma}}^2}.$$
\end{lemma}
\medskip

\begin{proof}
1. Since $\tilde{\zeta}=\frac{1}{\det(\nabla\tilde{X})}(\cof(\nabla\tilde{X}))^T$,  we need to estimate the  $\det(\nabla\tilde{X})=1+\dive(\tilde{X}-\tilde{\alpha})+\det(\nabla(\tilde{X}-\tilde{\alpha}))$ from below, by using \eqref{flux-estim}.\\ \\
2. Since $\tilde{\zeta}-\mathcal{I}=\tilde{\zeta}(\mathcal{I}-\nabla\tilde{X})=\tilde{\zeta}\nabla(\tilde{\alpha}-\tilde{X})$ then by using the previous result and \eqref{flux-estim} we obtain also the second estimate.\\ \\
3. and 4. We use the previous definition of $\tilde{\zeta}$, $\det(\nabla(\tilde{X}))$ and $\tilde{\zeta}-\mathcal{I}$ in the right spaces.
\end{proof}
\bigskip

\begin{lemma}\label{zeta-dif-est}
Let $2<s<\frac{5}{2}$ and $\tilde{X}^{(n)}-\tilde{\alpha}-tJ^P\tilde{v}_0, \tilde{X}^{(n-1)}-\tilde{\alpha}-tJ^P\tilde{v}_0\in\mathcal{F}^{s+1,\gamma}$. Then for $T$ small enough, we have

\begin{align*}
&1.\hspace{0.5cm}\|\tilde{\zeta}^{(n)}-\tilde{\zeta}^{(n-1)}\|_{H^{\frac{s-1}{2}+\delta}_{(0)}H^{1+\eta}}\leq C(M, \tilde{v}_0)\|\tilde{X}^{(n)}-\tilde{X}^{(n-1)}\|_{H^{\frac{s-1}{2}+\delta}_{(0)}H^{2+\eta}}\\[3mm]
&2.\hspace{0.5cm}\|\tilde{\zeta}^{(n)}-\tilde{\zeta}^{(n-1)}\|_{H^1_{(0)}H^{\gamma-1}}\leq C(M, \tilde{v}_0)\|\tilde{X}^{(n)}-\tilde{X}^{(n-1)}\|_{H^1_{(0)}H^{\gamma}}
\end{align*}
where
$$M=\max_{m=n-1,n}\left\lbrace\frac{1}{1-C(\tilde{v}_0)T-CT^{\frac{1}{4}}\mathcal{A}^{(m)}-CT^{\frac{1}{2}}(\mathcal{A}^{(m)})^2}\right\rbrace.$$
with $\mathcal{A}^{(m)}=\|\tilde{X}^{(m)}-\tilde{\alpha}-tJ^P\tilde{v}_0\|_{\mathcal{F}^{s+1,\gamma}}$
\end{lemma}

\end{document}